%% file: main_field.tex
\documentclass[11pt]{article}

\usepackage[utf8]{inputenc}
\usepackage[T1]{fontenc}
\usepackage{textcomp}
\usepackage{amsmath,amssymb,amsthm}
\usepackage{lmodern}
\usepackage[a4paper, margin=0.8in]{geometry}
\usepackage{microtype}
\usepackage{listings}
\usepackage{moreverb}
\usepackage{hyperref}
\usepackage[sans]{dsfont}
\usepackage[font={sf,small}, labelfont={sf,bf,small}, margin=1cm]{caption}
\usepackage{enumitem}

\usepackage{hyperref}
\hypersetup{
    colorlinks=true,
    linkcolor=blue,
    citecolor=magenta,
    urlcolor=blue,
    pdfborder={0 0 0}
}

\usepackage{import}
\usepackage{xifthen}
\usepackage{pdfpages}
\usepackage{transparent}
\usepackage{chngcntr}
\counterwithin{figure}{section}

\newtheorem*{theorem*}{Theorem}
\newtheorem{theorem}{Theorem}[section]
\newtheorem{proposition}[theorem]{Proposition}
\newtheorem{lemma}[theorem]{Lemma}
\newtheorem{corollary}[theorem]{Corollary}
\newtheorem{notation}[theorem]{Notation}

\newtheorem{conjecture}[theorem]{Conjecture}

\newtheorem{assumption}[theorem]{Assumption}
\theoremstyle{remark}
\newtheorem{remark}[theorem]{Remark}

\newcommand{\R}{\mathbb{R}}
\newcommand{\N}{\mathbb{N}}
\newcommand{\Z}{\mathbb{Z}}

\newcommand{\C}{\mathbb{C}}

\newcommand{\D}{\mathbb{D}}

\newcommand{\Bc}{\mathcal{B}}
\newcommand{\Cc}{\mathcal{C}}

\newcommand{\Ec}{\mathcal{E}}
\newcommand{\Fc}{\mathcal{F}}

\newcommand{\Kc}{\mathcal{K}}
\newcommand{\Lc}{\mathcal{L}}
\newcommand{\Mc}{\mathcal{M}}

\newcommand{\Tc}{\mathcal{T}}

\newcommand{\Pc}{\mathcal{P}}

\newcommand{\Cf}{\mathfrak{C}}
\newcommand{\Qf}{\mathfrak{Q}}

\newcommand{\Expect}[1]{\mathbb{E} \left[ #1 \right] }
\newcommand{\EXPECT}[2]{\mathbb{E}_{#1} \left[ #2 \right] }

\newcommand{\Prob}[1]{\mathbb{P} \left( #1 \right) }
\newcommand{\PROB}[2]{\mathbb{P}_{#1} \left( #2 \right) }

\renewcommand{\P}{\mathbb{P}}
\newcommand{\E}{\mathbb{E}}

\newcommand{\Ps}{\mathds{P}}
\newcommand{\Es}{\mathds{E}}

\newcommand{\abs}[1]{\left\vert #1 \right\vert}
\newcommand{\norme}[1]{\left\| #1 \right\| }
\newcommand{\scalar}[1]{\left( #1 \right)}
\newcommand{\floor}[1]{\left\lfloor #1 \right\rfloor}

\newcommand{\indic}[1]{ \mathbf{1}_{ \left\{ #1 \right\} } }
\newcommand{\eps}{\varepsilon}

\DeclareMathOperator{\CR}{CR}
\renewcommand{\d}{\mathrm{d}}

\def \loopmeasure{\mu^{\rm loop}}

\renewcommand{\Im}{\mathrm{Im}}

\newcommand{\chrT}{\mathbf{T}}
\newcommand{\fav}{\mathtt{F}}
\newcommand{\coT}{\mathcal{T}}
\newcommand{\Range}{\operatorname{Range}}
\newcommand{\cyl}{\mathtt{C}}
\newcommand{\espace}{\mathcal{E}}
\newcommand{\Komp}{\mathcal{Q}}
\newcommand{\perm}{\mathfrak{P}}
\newcommand{\Law}{\operatorname{Law}}
\newcommand{\diam}{\operatorname{diam}}

\newcommand{\one}{\mathbf{1}}

\newcommand{\Pb}{\mathbb{P}}

\newcommand{\Hb}{\mathbb{H}}
\newcommand{\Rb}{\mathbb{R}}
\newcommand{\Eb}{\mathbb{E}}
\newcommand{\Qc}{\mathcal{Q}}
\newcommand{\wt}{\widetilde}
\newcommand{\wh}{\widehat}
\newcommand{\ol}{\overline}

\title{Conformally invariant fields out of Brownian loop soups}
\author{Antoine Jego\thanks{\'Ecole Polytechnique Fédérale de Lausanne; antoine.jego@epfl.ch} \and Titus Lupu\thanks{Sorbonne Université and Université Paris Cité, CNRS, Laboratoire de Probabilités, Statistique et Modélisation, F-75005 Paris, France; titus.lupu@sorbonne-universite.fr} \and Wei Qian\thanks{City University of Hong Kong; weiqian@cityu.edu.hk}}
\date {\today}

\numberwithin{equation}{section}

\begin{document}

\maketitle

\abstract{Consider a Brownian loop soup $\Lc_D^\theta$ with subcritical intensity $\theta \in (0,1/2]$ in some 2D bounded simply connected domain $D$. We define and study the properties of a conformally invariant field $h_\theta$ naturally associated to $\Lc_D^\theta$. Informally, this field is a signed version of the local time of $\Lc_D^\theta$ to the power $1-\theta$.
When $\theta = 1/2$, $h_\theta$ is a Gaussian free field (GFF) in $D$.

Our construction of $h_\theta$ relies on the multiplicative chaos $\Mc_\gamma$ associated with $\Lc_D^\theta$, as introduced in \cite{ABJL21}. Assigning independent symmetric signs to each cluster, we restrict $\Mc_\gamma$ to positive clusters. We prove that, when $\theta = 1/2$, the resulting measure $\Mc_\gamma^+$ corresponds to the exponential of $\gamma$ times a GFF. At this intensity, the GFF can be recovered by differentiating at $\gamma = 0$ the measure $\Mc_\gamma^+$.
When $\theta < 1/2$, we show that $\Mc_\gamma^+$ has a nondegenerate \textit{fractional derivative} at $\gamma =0$ defining a random generalised function $h_\theta$.

We establish a result which is analogous to the recent work \cite{ALS4} in the GFF case ($\theta =1/2$),
but for $h_\theta$ with $\theta \in (0,1/2]$.
Relying on the companion article \cite{JLQ23a}, we prove that each cluster of $\Lc_D^\theta$ possesses a nondegenerate Minkowski content in some non-explicit gauge function $r \mapsto r^2 |\log r|^{1-\theta+o(1)}$. We then prove that $h_\theta$ agrees a.s.\ with the sum of the Minkowski content of each cluster multiplied by its sign.

We further extend the couplings between CLE$_4$, SLE$_4$ and the GFF \cite{schramm2013contour,MS} to $h_\theta$ for $\theta\in(0,1/2]$.
We show that the (non-nested) CLE$_\kappa$ loops form level lines for $h_\theta$  
and that there exists a constant height gap $c(\theta)>0$ between the values of the field on either side of the CLE loops. 
However, unless $\theta=1/2$, $h_\theta$ is not equal to the CLE nesting field.
We finally study the Wick powers of $h_\theta$ and the expansion of $\Mc_\gamma$ and $\Mc_\gamma^+$ in power series of $\gamma$.
}

\setcounter{tocdepth}{1}
\tableofcontents

\section{Introduction}
\input{./subfiles/intro_new.tex}

\subsection{A conformally invariant field out of the Brownian loop soup}
\input{./subfiles/intro_field_new.tex}

\subsection{\texorpdfstring{$h_\theta$}{h theta} as a sum of Minkowski contents}
\input{./subfiles/intro_minkowski_new.tex}

\subsection{Connecting two points with a cluster}\label{SS:intro_conditioning}
\input{./subfiles/intro_conditioning_new.tex}

\subsection{Conjectural construction of \texorpdfstring{$h_\theta$}{h theta} from a discrete loop soup}\label{SS:intro_discrete}
\input{./subfiles/intro_discrete_new.tex}

\subsection{Relation to other conformally invariant and covariant fields}\label{SS:intro_other}
\input{./subfiles/other.tex}

\paragraph*{Acknowledgements}
The authors thank Wendelin Werner for discussions that ultimately inspired our investigations.
This material is based upon work supported by the National Science Foundation under Grant No. DMS-1928930 while AJ and WQ participated in a program hosted by the Mathematical Sciences Research Institute in Berkeley, California, during the Spring 2022 semester.
Some aspects of this work were discussed while AJ visited the City University of Hong Kong and the LPSM in Paris VI. The hospitality of these institutions is greatly acknowledged.
AJ is supported by Eccellenza grant 194648 of the Swiss National Science Foundation and is a member of NCCR SwissMAP.
WQ was in CNRS and Laboratoire de Math\'ematiques d’Orsay, Universit\'e Paris-Saclay, when this work was initiated.

\section{Preliminaries}\label{S:preliminaries}
\input{./subfiles/preliminaries.tex}

\section{Convergence of renormalised crossing probabilities and coupling of clusters}\label{S:convergence_ratio}
\input{./subfiles/coupling.tex}

\section{Construction of \texorpdfstring{$h_\theta$}{h theta}: proof of Theorems \ref{T:convergenceL2},  \ref{T:cluster2point} and \ref{T:cluster_xy_doob}}\label{S:construction_h}
\input{./subfiles/field_new.tex}


\section{\texorpdfstring{$h_\theta$}{h theta} as a sum of Minkowski measures: proof of Theorems \ref{T:measure_cluster} and \ref{T:h_and_minkowski}}\label{S:minkowski}
\input{./subfiles/minkowski_new.tex}

\section{Excursions from the boundary of a cluster}\label{sec:exc_bdy}
\input{./subfiles/boundary_excursions.tex}

\section{Properties of \texorpdfstring{$h_\theta$}{h theta} and height gap}\label{S:properties}
\input{./subfiles/properties.tex}

\section{Approximation of \texorpdfstring{$\Mc_a^+$}{M a +} from the discrete}\label{S:discrete}

\input{./subfiles/discrete.tex}

\section{Wick powers and expansion of \texorpdfstring{$\Mc_\gamma^+$}{Mgamma+}}\label{S:2D_Wick}

\input{./subfiles/Wick.tex}


\bibliographystyle{alpha}
\bibliography{bibliography}

\end{document}

%% file: subfiles/intro_new.tex

The two-dimensional Brownian loop soup, introduced by Lawler and Werner \cite{Lawler04}, is a fundamental object with numerous connections to other important conformally invariant/covariant models. 
Brownian loops in the form of bubbles first appeared in \cite{LSW03Restriction}, where Lawler, Schramm and Werner constructed chordal conformal restriction measures by attaching Brownian loops to an SLE path. 
A Brownian loop soup is a random collection of infinitely many Brownian-type loops in a given domain $D$ of the plane, sampled according to a Poisson point process with intensity $\theta \loopmeasure_D$. $\theta$ is a positive real parameter and $\loopmeasure_D$ is a measure on loops defined by \eqref{Eq loop meas}. The parameter $c=2\theta$ corresponds to the central charge in conformal field theory.
In their seminal work \cite{SheffieldWernerCLE}, Sheffield and Werner 
showed the following phase transition:
\begin{itemize}[leftmargin=*]
\item
Supercritical:
when $\theta>1/2$, there is a.s. only one cluster of loops;
\item
Critical and subcritical:
when $\theta \leq 1/2$, there are a.s. infinitely many clusters. In this case, the outermost boundaries of the outermost clusters are distributed according to a (non-nested) conformal loop ensemble CLE$_\kappa$ where $\kappa = \kappa(\theta) \in (8/3, 4]$ is related to $\theta$ by
\begin{equation}
\label{E:kappa}
c=2 \theta = (6-\kappa)(3\kappa -8)/(2\kappa).
\end{equation}
\end{itemize}
SLE (by Schramm \cite{MR1776084}) and its loop variant CLE (by Sheffield \cite{MR2494457}) were introduced to describe the scaling limits of discrete models from statistical mechanics, and have now become fundamental objects in random conformal geometry. 

The intensity $\theta=1/2$ is special for another (related) reason. Indeed, when $\theta=1/2$,
Le Jan \cite{LeJan2010LoopsRenorm, LeJan2011Loops} identified the occupation field of the loop soup with half of the square of the Gaussian free field (GFF), on discrete graphs and in dimensions $1,2,3$ in the continuum. (In the 2D and 3D continuum, both the GFF and the occupation field are not defined pointwise and additive renormalisations are needed to define rigourously the square of the GFF and the occupation field of the loop soup.)
This result is part of the random walk/Brownian motion
representations of the GFF, 
also known as \emph{isomorphism theorems}
\cite{Symanzik65Scalar,Symanzik66Scalar,Symanzik1969QFT,
BFS82Loop,Dynkin1984Isomorphism,
Dynkin1984IsomorphismPresentation,
Wolpert2,
EKMRS2000,
MarcusRosen2006MarkovGaussianLocTime,
Sznitman2012LectureIso}.

The \emph{critical} intensity $\theta=1/2$ corresponds to $\kappa=4$, which is also the critical parameter for SLE in terms of self intersection \cite{MR2153402}. 
A remarkable relation between SLE$_4$ and the GFF was established by Schramm and Sheffield in \cite{schramm2013contour}, where SLE$_4$ appeared as \emph{level lines} of the GFF. There, SLE$_4$ was recovered as a \emph{local set} of the GFF, so that the GFF satisfies a certain spatial Markov property when we condition on such sets. The value of the GFF is constant on either side of the level line, with a constant height gap $2\lambda>0$ between the two sides. Such a coupling was also extended to CLE$_4$ by Miller and Sheffield \cite{MS} (see also \cite{MR3708206,MR3936643}), so that conditionally on the CLE$_4$, the GFF restricted to the domain encircled by each CLE$_4$ loop is an independent GFF with boundary conditions 
$\pm 2\lambda$.

\begin{figure}
\centering
\includegraphics[width=.85\textwidth]{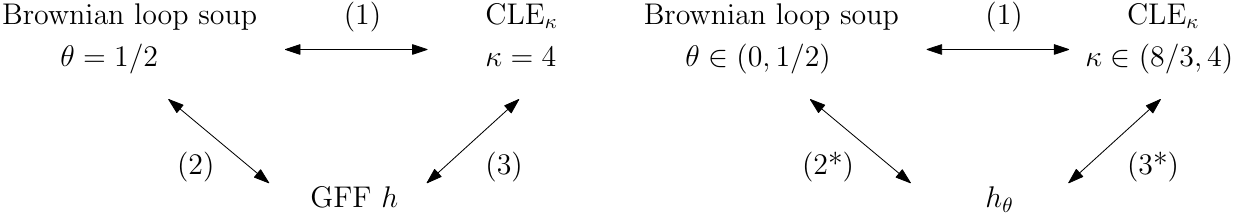}
\caption{\label{fig:triangle} Relation (1) \cite{SheffieldWernerCLE} states that the outer boundaries of the outermost clusters in a loop soup are CLE. Relation (2) \cite{LeJan2010LoopsRenorm, LeJan2011Loops} states that the renormalised $h^2/2$ is equal to the occupation time field of the loop soup. Relation (3) \cite{MS} states that CLE$_4$ form the level loops of the GFF. 
In this paper, we construct a field $h^\theta$ out of a loop soup with intensity $\theta\in(0,1/2)$,  and conjecture that the occupation time field of the loop soup is equal to a constant times the renormalised $|h_\theta|^{1/(1-\theta)}$. We prove that  SLE$_\kappa$/CLE$_\kappa$ appear as level lines of  $h^\theta$.} 
\end{figure}

The three relations between the critical loop soup, CLE$_4$, and GFF in Figure~\ref{fig:triangle} are proved using three different approaches, but are shown to commute in \cite{QianWerner19Clusters}, using the metric graph approximation by \cite{LupuIsomorphism, lupu2018convergence} as a key input. The metric graph is locally one dimensional, so that the local times of a loop soup are defined pointwise. It was shown in \cite{LupuIsomorphism} that the GFF on the metric graph can be obtained by taking the square root of the occupation field of the critical loop soup, and then assigning i.i.d. symmetric signs in $\{\pm 1\}$ to each cluster. In the 2D continuum, the local times are not defined pointwise,
so it is not immediately clear what the square root of the local time means.
Yet, the article \cite{ALS4} gives the precise analogue of this procedure in the 2D continuum, and shows that the GFF is equal to the sum of the signed Minkowski contents of the clusters.

\medskip

Relying on the connection with the GFF, one is able to deduce many properties about the loop soup at $\theta=1/2$. For example, the excursions induced by the Brownian loops that touch the boundary of a cluster were shown to 
have a surprisingly simple distribution, forming a Poisson point process
\cite{QianWerner19Clusters, qian2018}.
In \cite{ALS1}, the Minkowski gauge function of the clusters was shown to be $r \mapsto |\log r|^{1/2} r^2$. For $\theta\in (0,1/2)$, without the powerful relation to the GFF, much less is known,
and the picture is expected to be overall more complicated.
This paper, together with the companion article \cite{JLQ23a}, aims to fill this gap.
The main question addressed in the current paper is:
\begin{center}
\emph{What is the natural field associated to the Brownian loop soup when $\theta \in(0, 1/2)$?}
\end{center}
More precisely, we would like to have a field $h_\theta$ that is related to the occupation field of a loop soup in a similar way that the GFF is related to the critical loop soup. 
It was shown in \cite{LeJan2011Loops, MR3298468} that the occupation field of the loop soup is a permanental field, an object introduced by Vere-Jones \cite{MR1450811} long before the Brownian loop soup (also see \cite{MR3077522}).
The field $h_\theta$ that we are after would be a more fundamental object behind these permanental fields.

A natural idea is to proceed as in the critical case  \cite{LupuIsomorphism}, i.e., to take the square root of the occupation field of a loop soup, and then assign i.i.d.\ signs to each cluster. 
Although one can always do this on the metric graph, the resulting field is not guaranteed to have any nice property nor to converge to a nondegenerate limit in the continuum. As it turns out, surprisingly, the relevant field will not correspond to a signed version of the square root of the local time of the loop soup, as for $\theta = 1/2$, but instead to a signed version of the local time to the power $1-\theta$.

\medskip

Our construction of the field $h_\theta$ is done directly in the continuum, relying on two key inputs. The first one is the computation of the crossing exponent and related estimates obtained in the companion paper \cite{JLQ23a}. The second key input is the construction in \cite{ABJL21} of a multiplicative chaos associated to the Brownian loop soup with any intensity $\theta >0$.
Gaussian multiplicative chaos (GMC), initially introduced by Kahane \cite{kahane} in the eighties and extensively studied in the past decade \cite{RobertVargas2010, DuplantierSheffield, RhodesVargasGMC, ShamovGMC, BerestyckiGMC}, is a theory that defines and studies properties of random measures informally defined as the exponential of a real parameter $\gamma \in \R$ times a Gaussian logarithmically-correlated field. The two-dimensional GFF is an archetypal example of such a field and its GMC measure is sometimes called Liouville measure. 
GMC appears to be a universal feature of log-correlated fields going beyond the Gaussian setting: \cite{Junnila18} studies fields defined as random Fourier series with i.i.d. coefficients and \cite{jegoBMC, jegoCritical} makes sense of the exponential of the square root of the local time of planar Brownian motion.
The article \cite{ABJL21} provides a third example of a non-Gaussian multiplicative chaos and defines what can be thought of as the exponential of $\gamma$ times the square root of twice the local time of the Brownian loop soup, for any intensity $\theta$. In view of Le Jan's isomorphism, it is natural to expect, and has been proven in \cite{ABJL21}, that when $\theta=1/2$ the resulting measure agrees with an unsigned version of Liouville measure: $e^{\gamma |\text{GFF}|} \d x = e^{\gamma \text{GFF}} \d x+ e^{-\gamma \text{GFF}} \d x$, i.e. a sum of two \emph{correlated} Liouville measures.

\paragraph{Construction of $h_\theta$} ((2*) in Figure \ref{fig:triangle})
Concretely, we assign i.i.d.\ symmetric signs to each cluster in the loop soup, and then consider the multiplicative chaos $\Mc_\gamma$ of the loop soup restricted to positive clusters. We prove that, when $\theta = 1/2$, the resulting measure $\Mc_\gamma^+$ agrees with $e^{\gamma \text{GFF}} \d x$. The GFF can therefore be recovered from the loop soup, together with the independent spins, by differentiating $\Mc_\gamma^+$ at 
$\gamma = 0$. 
This idea is used for instance in \cite{ALS1}.
When $\theta < 1/2$, we show that the same procedure yields a nondegenerate conformally invariant field $h_\theta$, with the surprising and important difference that one has to take a \emph{fractional derivative} of $\Mc_\gamma^+$ at $\gamma=0$. This fractional power is equal to $2(1-\theta)$ which is twice the crossing exponent computed in \cite{JLQ23a}.

We provide a second and equivalent approach to $h_\theta$. We show that each cluster has a nondegenerate Minkowski content in some non-explicit gauge function $r \mapsto r^2 |\log r|^{1-\theta+o(1)}$ and we show that $h_\theta$ is equal to the sum of the signed Minkowski contents of all the clusters for all $\theta\in(0,1/2)$, similarly to the case $\theta=1/2$ \cite{ALS4}.

We emphasise that, when $\theta\not=1/2$, $h_\theta$ is not log-correlated 
(its correlations blow-up as a power of log, see  Theorem~\ref{T:prop}) nor is supposed to be Gaussian. 
A conjectural expansion of $\Mc_\gamma^+$ for $\theta\in(0,1/2)$ near $\gamma=0$ in terms of $h_\theta$  and the local times is given in \eqref{E:conj_expansion}. This conjecture is closely related to Conjecture~\ref{Conj:intro} below. 

\paragraph{Level lines and height gap} ((3*) in Figure \ref{fig:triangle})
In addition, we establish a coupling between $h_\theta$ and CLE$_\kappa$, where $\kappa$ and $\theta$ are related by \eqref{E:kappa}. 
For each $\theta\in(0,1/2]$, there is a height gap $c(\theta)>0$ between the values of $h_\theta$ on both sides of the level lines.
We show that our field (with $0$ boundary conditions) can also be constructed by first sampling a CLE$_\kappa$, and then sampling an independent $h_\theta$ with wired boundary conditions 
in each domain encircled by a CLE$_\kappa$ loop.
Although the field $h_\theta$ with wired boundary condition has constant value $\pm c(\theta)$ on the boundary, we emphasise that, when $\theta <1/2$ and unlike the GFF, this field is \emph{not} equal to $\pm c(\theta)$ plus $h_\theta$ with zero boundary conditions.
In particular, our field is \emph{not} equal to the CLE nesting field constructed by  \cite{miller2015conformal}.
By definition, the CLE nesting field should also admit CLE as its level lines, but it has no apparent relation to the loop soup occupation field to our knowledge (except when $\theta=1/2$, where the field is the GFF).

The fact that CLE$_\kappa$ or SLE$_\kappa$ are level lines of $h_\theta$ is not at all clear form our definition. To prove the level line coupling, we need to prove additional properties about the Brownian excursions induced by the loops that touch the boundary of a cluster, and crucially rely on the conformal restriction property of such excursions deduced in \cite{MR3901648}.

This level line coupling can also be made to work in the chordal case, where an SLE$_\kappa$ curve can be coupled with a field $h_\theta$ with mixed wired/zero boundary conditions.
This generalises the SLE$_4$/GFF coupling.
Note that a coupling between SLE$_\kappa$ (for all $\kappa>0$) and GFF was established in \cite{MR2525778} and in the celebrated imaginary geometry series \cite{MR3477777,MR3502592,MR3548530,MR3719057}, which was another natural generalisation of the SLE$_4$/GFF coupling. In the imaginary geometry coupling, when $\kappa\not=4$, the value of the GFF on both sides of the SLE$_\kappa$ curve are not constant, but is equal to a constant plus a winding term, which depends on the choice of the starting point and does not have pointwise value on the curve. In comparison, the value of $h_\theta$ is constant along the curves, as in the SLE$_4$/GFF coupling.

\medskip

In a nutshell, we construct for any $\theta \in (0,1/2]$ a coupling between a field $h_\theta$, a Brownian loop soup and a CLE$_\kappa$ such that:
\begin{itemize}[leftmargin=*]
\item (Figure \ref{fig:triangle}, (1))
The CLE loops are the outermost boundaries of the outermost clusters of the loop soup.
\item (Figure \ref{fig:triangle}, (2*))
$h_\theta$ is a signed version of the local time of the loop soup to the power $1-\theta$. The precise relation is more complicated as described above.
\item (Figure \ref{fig:triangle}, (3*))
The CLE loops form level lines of $h_\theta$ with a constant height gap on either sides of the loops.
\end{itemize}
When $\theta=1/2$, this coupling possesses an additional feature: the centred occupation field of the loop soup is half the Wick square of the GFF and the CLE loops are some two-valued sets of the GFF. In particular both the occupation field and the CLE are measurable w.r.t.\ the GFF. To keep the paper of a reasonable size, we do not attempt to prove such a measurability feature when $\theta \in (0,1/2)$.

%% file: subfiles/intro_field_new.tex

Let $D \subset \C$ be a bounded simply connected domain and $\theta \in (0,1/2]$. Consider a Brownian loop soup $\Lc_D^\theta$ in $D$ with intensity $\theta$.
Let us denote by $\Cf$ the collection of all clusters $\Cc$ of $\Lc_D^\theta$. Conditionally on $\Cf$, let $(\sigma_\Cc)_{\Cc \in \Cf}$ be i.i.d. signs taking values in $\{ \pm 1 \}$ with equal probability 1/2. If a point $x \in D$ belongs to some cluster $\Cc$, we will write $\sigma_x$ for the sign $\sigma_\Cc$ of that cluster.

Let $\gamma \in (0,2)$ and let $\Mc_\gamma$ be the multiplicative chaos of the Brownian loop soup $\Lc_D^\theta$ constructed in \cite{ABJL21}, normalised so that $\Expect{ \Mc_\gamma(\d x)} = 2\d x$; see Section \ref{S:multiplicative_chaos} for details. Let $\Mc_\gamma^+ (\d x) = \indic{\sigma_x = +1} \Mc_\gamma (\d x)$ be the multiplicative chaos of $\Lc_D^\theta$ restricted to positive clusters. Note that this definition is meaningful since every $\Mc_\gamma$--typical point is visited by a loop (actually by infinitely many loops, see \cite[Theorem 1.8]{ABJL21}).
Let $h_\gamma$ be the field obtained from $\Mc_\gamma^+$ by
\begin{equation}\label{eq:mc_def_field}
(h_\gamma,f) := \frac{1}{Z_\gamma} \int_D f(x) (\Mc_\gamma^+(\d x) - \d x),
\end{equation}
for any bounded measurable function $f: \C \to \R$. Here $Z_\gamma$ is a normalising constant defined in \eqref{E:Zgamma} below in terms of some crossing probability.
With some abuse of notation, we will often write $(h_\gamma,f) = \int h_\gamma(x) f(x) \d x$. Note that $h_\gamma$ is defined as a random generalised function on the whole complex plane, vanishing outside of $D$.

To define the normalising constants we will use in this article, we first introduce the following notations:

\begin{notation}
For any family $\Lc$ of loops and any sets $A, B \subset D$, we will denote by $\{ A \overset{\Lc}{\leftrightarrow} B \}$ the event that there is a cluster of loops in $\Lc$ that intersects both $A$ and $B$. If $x$ and $y$ are points, we will simply write $\{ x \overset{\Lc}{\leftrightarrow} A \}$ and $\{ x \overset{\Lc}{\leftrightarrow} y \}$ instead of $\{ \{x\} \overset{\Lc}{\leftrightarrow} A \}$ and $\{ \{x\} \overset{\Lc}{\leftrightarrow} \{y\} \}$.
\end{notation}

We will then use two normalising constants $Z_r$ and $Z_\gamma$ defined by
\begin{equation}
\label{E:Zgamma}
Z_r := \P \Big( e^{-1} \partial \D \overset{\Lc_\D^\theta}{\longleftrightarrow} r \D \Big), \quad r >0,
\quad \quad \text{and} \quad \quad
Z_\gamma := \P \Big( e^{-1} \partial \D \overset{\Lc_\D^\theta \cup \Xi_a^{0,\D}}{\longleftrightarrow} 0 \Big), \quad \gamma \in (0,2),
\end{equation}
where $\Xi_a^{0,\D}$ is an $a$-thick loop ($a = \gamma^2/2$) at the origin defined in Section \ref{SS:preliminaries_paths} that is independent from $\Lc_\D^\theta$.

\begin{theorem}[Construction of $h_\theta$]\label{T:convergenceL2}
Let $\theta \in (0,1/2]$.
For any bounded measurable function $f : \C \to \R$, $(h_\gamma, f)$ converges in $L^2$ as $\gamma \to 0$.
Moreover, for any $\eps >0$, $h_\gamma$ converges in $L^2$ as $\gamma \to 0$ in the Sobolev space $H^{-\eps}(\C)$ \eqref{E:Sobolev} to some random field $h_\theta \in H^{-\eps}(\C)$.

Furthermore, the normalising constant $Z_\gamma$ satisfies $Z_\gamma = \gamma^{2(1-\theta) + o(1)}$.
\end{theorem}

\begin{theorem}[Properties of $h_\theta$]\label{T:prop}
Let $\theta \in (0,1/2]$.
\begin{enumerate}
\item
\label{I:covariance}
Covariance:
There exists a measurable function $C_\theta : D \times D \to [0,\infty]$ such that for all test functions $f$,
\begin{equation}
\label{E:Ctheta}
\Expect{ (h_\theta,f)^2 } = \int_{D \times D} f(x) C_\theta(x,y) f(y) \d x \d y.
\end{equation}
The blow-up of $C_\theta$ on the diagonal is given by
\[
C_\theta(x,y) = | \log |x-y||^{2(1-\theta) + o(1)}
\]
as $x-y \to 0$, while staying at a positive distance from $\partial D$.
\item
\label{I:conformal_invariance}
Conformal invariance: Let $\psi : D \to \widetilde{D}$ be a conformal map between two bounded simply connected domains and let $h_{\theta, D}$ and 
$h_{\theta, \widetilde{D}}$ denote the fields from Theorem \ref{T:convergenceL2} in the domains $D$ and $\widetilde{D}$. 
Then\footnote{Recall that by definition $(h_{\theta, D} \circ \psi^{-1}, \tilde{f}) = (h_{\theta,D}, \tilde{f} \circ \psi^{-1})$ for any test function $\tilde{f}$.}
$h_{\theta, D} \circ \psi^{-1}$ and $h_{\theta, \widetilde{D}}$ 
have the same law.
\item
\label{I:symmetry}
Symmetry: Let $h_\theta^-$ denote the field obtained in a similar way as $h_\theta$ with negative clusters instead of positive ones. Then $h_\theta^- = - h_\theta$ a.s. In particular, $h_\theta \overset{\mathrm{(d)}}{=} - h_\theta$.
\item
\label{I:nondegenerate}
Nondegeneracy: Let $f : D \to \R$ be a smooth test function, non identically zero. Then the law of $(h_\theta,f)$ is nonatomic. In particular, $(h_\theta,f) \neq 0$ a.s.
\end{enumerate}
\end{theorem}

As already mentioned, when $\theta = 1/2$, $\Mc_\gamma$ agrees with the hyperbolic cosine $:e^{\gamma h}: \d x$ $+ :e^{-\gamma h}:\d x$ of the GFF $h$ with Dirichlet boundary condition. Here $h$ is normalised so that $\Expect{h(x) h(y)} \sim -\log|x-y|$ as $x-y \to 0$ and $:e^{\pm \gamma h}: \d x$ are GMC measures normalised so that $\Expect{:e^{\pm \gamma h}: \d x} = \d x$. See \cite[Theorem 1.5]{ABJL21}. Because the isomorphism relating the local time to the GFF is much stronger in the discrete, the proof of this result was obtained by first proving a discrete approximation of $\Mc_\gamma$. In Theorem \ref{T:one_cluster}, we show a stronger version of this convergence by showing that the discrete approximation holds cluster by cluster, for any $\theta \in (0,1/2]$. It implies that, when $\theta =1/2$, the measure $\Mc_\gamma^+(\d x) = \indic{\sigma_x = +1} \Mc_\gamma(\d x)$ agrees with $:e^{\gamma h}: \d x$. As a corollary,

\begin{theorem}\label{T:identification_field1/2}
When $\theta = 1/2$, $h_\theta$ has the law of a multiple of the GFF in $D$.
More precisely, there is a deterministic constant $c_{1/2} \in (0,\infty)$ and a coupling between a critical loop soup $\Lc_D^{1/2}$ together with i.i.d. spins $(\sigma_\Cc)_{\Cc \in \Cf}$ and a Gaussian free field $h$ in $D$ such that almost surely,
\begin{itemize}
\item
$\Mc_\gamma^+ = :e^{\gamma h}: \d x$ and $\Mc_\gamma = 2:\cosh (\gamma h): \d x$ for all $\gamma \in (0,2)$;
\item
$h_{1/2} = c_{1/2} h$ and $:L: = \frac12 :h^2:$.
\end{itemize}
\end{theorem}

The appearance of the constant $c_{1/2}$ comes from the fact that we do not renormalise $\Mc_{\gamma}^+(\d x) - \d x$ by $\gamma$, but by a normalising constant $Z_\gamma$ that behaves asymptotically like a multiple of $\gamma$ (see \cite[Lemma 2.6]{JLQ23a} for a similar result where the $a$-thick loop is replaced by a small disc).

\begin{remark}
As a consequence of Theorem \ref{T:identification_field1/2}, the covariance $C_\theta$ is a constant multiple of the Green function when $\theta=1/2$. When $\theta <1/2$,
we do not know any close formula for $C_\theta$.
We intend to investigate this question in the future.
However, note that the conformal invariance of $C_\theta$ implies that there is some measurable function $F : [0,\infty) \to [0,\infty)$ such that for all $x,y \in D$, $C_\theta(x,y) = F(G_D(x,y))$. This simply comes from the fact that $G_D(x,y)$ characterises the conformal type of a disc with two punctures.
\end{remark}

\subsection{Level lines and the height gap of \texorpdfstring{$h_\theta$}{h theta}}

The outermost boundaries of the outermost clusters form natural interfaces coupled with $h_\theta$. 
In Theorem \ref{T:level_line} below, we show that, for general intensities $\theta \in (0,1/2]$, the (non-nested) CLE$_\kappa$ loops are level lines of $h_\theta$, where $\kappa$ and $\theta$ are related by \eqref{E:kappa}, 
and that there is a constant height gap $c(\theta)>0$.

The field $h_\theta$ given by Theorem~\ref{T:convergenceL2} is a field with $0$ boundary conditions.
In order to state  Theorem \ref{T:level_line}, we need to define a field $h_\theta^\mathrm{wired}$ out of a loop soup with wired boundary conditions, which will turn out to have constant value $\pm c(\theta)$ on the boundary. 
Let $x_0 \in D$ and let $\Cc$ be the outermost cluster surrounding $x_0$. Let $O$ be the simply connected domain delimited by the outermost boundary of $\Cc$ and let $\sigma_{\partial O}$ be the sign associated to $\Cc$.
We recall the following fact from \cite[Theorem 1]{QianWerner19Clusters}. The law of the loop soup inside $O$, conditionally given $O$, is that of the union of two conditionally independent collections of loops. The loops that do not touch the boundary of $O$ are simply distributed according to an unconditioned loop soup $\Lc_O^\theta$ in $O$. The loops that touch the boundary form a random collection $\mathcal{E}^\theta_{\partial O}$ of excursions. The law of $\mathcal{E}^\theta_{\partial O}$ is explicit only when $\theta = 1/2$. However, for general intensities $\theta \leq 1/2$, $\mathcal{E}^\theta_{\partial O}$ is conformally invariant in the following strong sense: if $\mathcal{E}^\theta_{\partial O}$ is conformally mapped to the unit disc, then the resulting collection of excursions $\Ec_{\partial \D}^\theta$ is independent of $O$ and invariant in law under M\"{o}bius transformations.

Let $h_\theta^\mathrm{wired}$ be the random field in $\D$ obtained by restricting $h_\theta$ to test functions compactly supported in $O$ and then by conformally mapping $O$ onto $\D$. More precisely, let $\psi : O \to \D$ be a conformal map. For any smooth test function $f : \D \to \R$ with compact support in $\D$, we define $(h_\theta^\mathrm{wired},f) = (h_\theta, f \circ \psi)$. By \cite[Theorem 1]{QianWerner19Clusters}, the law of $h_\theta^\mathrm{wired}$ does not depend on the choice of the conformal map $\psi$ and $h_\theta^\mathrm{wired}$ is independent of $O$.

In Theorem \ref{T:level_line} below, we will consider smooth test functions $f_\eps : \D \to [0,\infty)$, $\eps \in (0,1)$, whose supports concentrate on the circle $\partial \D$ as $\eps \to 0$. We will make the following assumptions:

\begin{assumption}\label{assumption_feps}
For all $\eps \in (0,1)$, $f_\eps$ is compactly supported in $\{ x \in \D: \d(x,\partial \D) < \eps \}$ and $\int_\D f_\eps =1$. Moreover,
\begin{gather}
\label{E:assumption_feps}
\sup_\eps \int_{\D \times \D} \left( C_{\theta,\D}(x,y) + \max(1, -\log|x-y|) \right) f_\eps(x) f_\eps(y) \d x \d y < \infty,\\
\label{E:assumption_feps2}
\text{and} \quad \lim_{\delta \to 0} \limsup_{\eps \to 0} \int_{\D \times \D} \indic{|x-y| < \delta} \left( C_{\theta,\D}(x,y) + \max(1, -\log|x-y|) \right) f_\eps(x) f_\eps(y) \d x \d y = 0
\end{gather}
where $C_{\theta,\D}$ denotes the covariance of the field $h_\theta$ in the domain $\D$ with zero boundary condition; see \eqref{E:Ctheta}.

If $u:[0,\infty) \to [0,\infty)$ is a smooth function with $u(0)=0$, $u(r)=0$ for $r \geq 1$ and $\int_{(0,\infty)} u =1$, then $f_\eps: x \in \D \mapsto \frac{1}{2\pi \eps |x|} u( \frac{1-|x|}{\eps}), \eps \in (0,1),$ satisfies the desired properties.
\end{assumption}

\begin{theorem}\label{T:level_line}
For $\theta\in(0,1/2]$, there exists a constant $c(\theta) >0$ depending only on $\theta$ such that:
\begin{itemize}
\item
(Constant expectation)
For all test function $f$ compactly supported in $\D$,
\begin{equation}
\label{E:T_level_line_expectation}
\Expect{ (h_\theta^\mathrm{wired},f) \vert \sigma_{\partial \D} } = \sigma_{\partial \D} c(\theta) \int_\D f,
\quad \quad \text{a.s.}
\end{equation}
\item
(Constant boundary conditions) Let $(f_\eps)_\eps$ be a sequence of test functions satisfying Assumption~\ref{assumption_feps}. Then
\begin{equation}
\label{E:T_level_line_bc}
(h_\theta^\mathrm{wired}, f_\eps) - \sigma_{\partial \D} c(\theta) \xrightarrow[\eps \to 0]{L^2} 0.
\end{equation}

\item (Spatial Markov property) We can construct $h_\theta$ in the domain $D$ as follows. We first sample a CLE$_\kappa = \{ \ell, \ell \in$ CLE$\}$ and i.i.d. signs $\{\sigma_\ell, \ell \in$ CLE$\}$.
Let $h_{\theta,\ell}^\mathrm{wired}$, $\ell \in$ CLE, be independent fields in $\D$ with wired boundary conditions and sign $\sigma_\ell$ on $\partial \D$. For $\ell \in$ CLE, let $\psi_\ell$ be a conformal map from $\D$ onto the domain encircled by $\ell$. 
Then 
$\sum_\ell h_{\theta,\ell}^\mathrm{wired} \circ \psi_\ell^{-1}$ 
has the law of the field $h_\theta$ in $D$ with zero boundary condition.
\end{itemize}
The height gap $c(\theta)$ is expressed as the limit of the ratio of the probabilities of some crossing event; see \eqref{E:c_theta}.
\end{theorem}

We will prove in Section~\ref{S:1D} that a stronger version of the ``constant expectation'' property holds in dimension one, i.e., the 1D analogue of this field is a martingale.
In Section~\ref{S:1D},  we define  $h_\theta$ in 1D to be a square Bessel process taken to the power $\nu=1-\theta$ and multiplied by an i.i.d.\ sign for each of its excursion away from $0$. In Lemma~\ref{L:1D_martingale2}, we show that this one-dimensional process is a martingale if and only if $\nu=1-\theta$, which again singles out the special role of the exponent $1-\theta$.

We believe that, if $\theta<1/2$, the ``spatial Markov'' decomposition for $h_\theta$ only holds along some special interfaces, 
including the boundaries of the clusters. This makes our field more complicated to analyse than the GFF. For example, we believe that we cannot couple a \emph{nested} CLE$_\kappa$ as the level lines of $h_\theta$, unless $\theta=1/2$. As already mentioned, for $\theta<1/2$, our fields $h_{\theta}$ are different from the \textit{nesting fields} built in \cite{miller2015conformal}.
Indeed, for the nesting fields, the correlation blows up logarithmically on the diagonal 
\cite[Theorem 1.3]{miller2015conformal},
and in our case it blows up as $\log$ to the power $2(1-\theta) + o(1)$.

\medskip

By \cite{SheffieldWernerCLE}, one can 
 define an exploration process of a CLE which stops in the middle of a loop. One can conformally map $\varphi : H \to \Hb$ the unexplored region $H$ onto the upper half plane $\Hb$ in such a way that the explored portion of this CLE loop is sent to $\R^-$ and the rest of the boundary is sent to $\R^+$. The portion of the CLE loop that remains unexplored is a chordal SLE$_\kappa$ curve $\eta$ from 0 to infinity. Suppose that this CLE is the outer boundaries of the clusters in a loop soup $\Lc^\theta_\Hb$. 
For any smooth test function $f : \Hb \to \R$ with compact support in $\Hb$, we define $(h_\theta^\mathrm{mixed},f) = (h_\theta, f \circ \varphi^{-1})$. By \cite{MR3901648}, the law of $h_\theta^\mathrm{mixed}$ does not depend on the choice of the conformal map $\varphi$ and $h_\theta^\mathrm{mixed}$ is independent of $H$. Moreover, the loop soup satisfies spatial Markov property as one continues to explore along $\eta$.
This implies a coupling between $h_\theta^\mathrm{mixed}$ and $\eta$, stated in the following.

\begin{theorem}\label{T:mixed}
Let $h_\theta^\mathrm{mixed}$ be a field in $\Hb$ with wired boundary conditions on $\R^-$ and zero on $\R^+$. There is a curve $\eta$ coupled with $h_\theta^\mathrm{mixed}$, such that $\eta$ is a chordal SLE$_\kappa$ from $0$ to $\infty$. For each $t>0$, let $g_t$ be a conformal map from $\Hb\setminus\eta([0,t])$ onto $\Hb$ that sends $\eta(t)$ to $0$ and $\infty$ to $\infty$.
Then conditionally on $\eta([0,t])$, $h_\theta^\mathrm{mixed}\circ g_t^{-1}$ is distributed as $h_\theta^\mathrm{mixed}$ and is independent from $\eta([0,t])$.
\end{theorem}
This is a generalisation of the SLE$_4$/GFF coupling. 
We can also immediately deduce from Theorem~\ref{T:level_line} that $h_\theta^\mathrm{mixed}$ has boundary value $c(\theta)$ (resp.\ $0$) on the left (resp. right) side of $\eta$.

Combining \eqref{E:T_level_line_expectation} from Theorem \ref{T:level_line} and Theorem \ref{T:mixed}, we see that for all $x \in \Hb$, $\Expect{h_\theta^\mathrm{mixed}(x)}$ is equal to $c(\theta)$ times the probability that a chordal SLE$_\kappa$ from 0 to $\infty$ goes to the right of $x$. It turns out that this probability is explicit, given by Schramm's formula \cite{10.1214/ECP.v6-1041}.

%% file: subfiles/intro_minkowski_new.tex
We now give another description of $h_\theta$, somewhat more geometric. We will show that the field $h_{\theta}$ is a limit of random signed Radon measures supported on the topological closures of clusters of 
$\Lc^{\theta}_{D}$,
and give an expression for these measures.
This establishes the analogue of the result in the critical case  \cite{ALS1,ALS4}, where $h_{1/2}$ is the GFF.

We start by defining a nondegenerate measure supported on the closure of a given cluster of $\Lc_D^\theta$.
Given a closed (fractal) set $A \subset \C$, it is a classical problem (that can be highly non-trivial) to define a nondegenerate measure supported on $A$. One possibility is to try to renormalise the Lebesgue measure of the $r$-neighbourhood of $A$ and take the limit as $r \to 0^+$. More precisely, one seeks an appropriate gauge function $\phi : (0,\infty) \to (0,\infty)$ such that the measure
$
r^{-2} \phi(r) \indic{\mathrm{dist}(x,A)<r} \d x
$
converges weakly as $r \to 0^+$ to a nondegenerate limit. If such a function can be found, the resulting measure is then called the Minkowski content of $A$ in the gauge $\phi$.

It is a well known fact that the Minkowski content of a single planar Brownian trajectory is well defined and nondegenerate in the gauge $r \mapsto r^2 |\log r|$. In this case, it coincides with the occupation measure of the path.
Because the non-renormalised local time of a given cluster of $\Lc_D^\theta$ (i.e. the sum of the local times of all the trajectories of a given cluster) is infinite, the Brownian gauge function $r \mapsto r^2 |\log r|$ is not appropriate for the clusters of $\Lc_D^\theta$. And indeed, \cite{ALS1} showed that, when $\theta = 1/2$, the Minkowski content of a given cluster of $\Lc_D^\theta$ is well-defined and nondegenerate in the gauge $r \mapsto r^2 |\log r|^{1/2}$.
In Theorem \ref{T:measure_cluster} below, we resolve the analogous problem when $\theta \in (0,1/2)$ by showing that the Minkowski content of a given cluster of $\Lc_D^\theta$ is well-defined and nondegenerate in some non-explicit gauge function $r \mapsto r^2 / Z_r$ that behaves like $r^2 |\log r|^{1-\theta + o(1)}$ as $r \to 0$.

We order the clusters of $\Lc^{\theta}_{D}$ as follows.
Let $(x_{k})_{k\geq 1}$ be a dense sequence in $D$.
We will denote by $\Cc_{k}$ the outermost cluster of $\Lc^{\theta}_{D}$
surrounding $x_{k}$, among the clusters different from
$\Cc_{1},\dots,\Cc_{k-1}$.
By construction, the clusters $(\Cc_{k})_{k\geq 1}$ are two by two distinct.
It is also easy to see that a.s. this is an enumeration of the clusters of
$\Lc^{\theta}_{D}$.
For $k \geq 1$, let $\mu_{k,r}$ and $\mu_{k, \gamma}$ be the random Radon measures defined by
\[
\mu_{k, r}(B) =
\dfrac{1}{Z_{r}} \int_B \indic{\mathrm{dist}(x,\Cc_{k})<r} \,\d x,
\quad
\mu_{k, \gamma}(B) =
\dfrac{1}{Z_{\gamma}}\int_B \indic{\mathrm{dist}(x,\Cc_{k})<1} (1-\mathrm{dist}(x,\Cc_{k})^{\gamma^{2}/2}) \,\d x,
\]
for all Borel set $B \subset D$. 
Recall that the normalising constants $Z_r$ and $Z_\gamma$ are defined in \eqref{E:Zgamma}.

\begin{theorem}[Minkowksi contents of clusters]\label{T:measure_cluster}
Let $\theta \in (0,1/2]$. For all $k\geq 1$, $\mu_{k,r}$ and $\mu_{k,\gamma}$ converge in probability w.r.t. the topology of weak convergence as $r \to 0$ and $\gamma \to 0$ to the same random Radon measure $\mu_k$. Moreover, $\mu_k$ is a.s. a non-negative finite measure supported on $\overline{\Cc}_{k}$ and $\mu_k(\overline{\Cc}_k) > 0$ a.s.
\end{theorem}

Recall that we denote by $\sigma_{\Cc_k}$ the spin associated to the $k$-th cluster $\Cc_k$.
The field $h_\theta$ is then obtained by:

\begin{theorem}[Construction of $h_\theta$ from the Minkowski contents]\label{T:h_and_minkowski}
Let $\theta \in (0,1/2]$. For all $\eps >0$ and all bounded measurable function $f: \C \to \R$, the following convergences hold a.s. and in $L^2$:
\[
\sum_{j=1}^k \sigma_{\Cc_j} \mu_j \xrightarrow[k \to \infty]{} h_\theta
\quad \quad \text{and} \quad \quad \sum_{j=1}^k \sigma_{\Cc_j} (\mu_j,f) \xrightarrow[k \to \infty]{} (h_\theta,f)
\]
in the spaces $H^{-\eps}(\C)$ and $\R$ respectively.
\end{theorem}

In other words, the ``restriction of $|h_\theta|$'' to $\Cc_k$ is a well-defined Radon measure that agrees with $\mu_k$.
Note however that this does not at all mean that the field $h_\theta$ is a signed measure.
Most likely, the associated total variation measure 
$
\sum_{j=1}^k \mu_j
$
diverges, as $k\to +\infty$, in every open subset of $D$;
so the compensations induced by the signs are necessary for the convergence.
This phenomenon is known to occur when $\theta=1/2$.

%% file: subfiles/intro_conditioning_new.tex
We now want to elaborate on the intimate relation between the field $h_\theta$ and the connectivity properties of the loop soup. 
We are first going to define the law $\P_{x \leftrightarrow y}$ of the loop soup conditioned on the event that two given points $x$ and $y$ belong to the closure of the same cluster. Because this event has a vanishing probability, we will define this law via a limiting procedure; see Theorem \ref{T:cluster2point}.
Theorem \ref{T:cluster_xy_doob} then provides a Doob-transform approach to the definition of $\P_{x \leftrightarrow y}$: $\P_{x \leftrightarrow y}$ can be informally obtained by reweighting the law of $\Lc_D^\theta$ by $h_\theta(x) h_\theta(y) / C_\theta(x,y)$.

Let $x, y \in D$ be two distinct points. For all $r >0$, let $\P_{x \leftrightarrow y, r}$ be the law of a loop soup $\Lc_D^\theta$ conditioned on the event $\{ D(x,r) \overset{ \Lc_D^\theta }{\longleftrightarrow} D(y,r) \}$ that a cluster of $\Lc_D^\theta$ connects the two discs $D(x,r)$ and $D(y,r)$.

\begin{theorem}\label{T:cluster2point}
Let $x, y \in D$ be two distinct points. Then $\P_{x \leftrightarrow y, r}$ converges weakly as $r \to 0$ to some probability measure $\P_{x \leftrightarrow y}$, with respect to the topology induced by $d_\mathfrak{L}$ \eqref{metric}.
\end{theorem}

We can now state the aforementioned relation between the field $h_\theta$ and the connectivity properties of the loop soup:

\begin{theorem}\label{T:cluster_xy_doob}
Let $F : D \times D \times \mathfrak{L} \to \R$ be a bounded measurable function such that for all $\Lc \in \mathfrak{L}$, $F(\cdot, \cdot, \Lc)$ is smooth. Then
\[
\Expect{
\int_D F(x,y,\Lc_D^\theta) h_\theta(x) h_\theta(y) \d x \d y
}
= \int_{D \times D} C_\theta(x,y) \EXPECT{x \leftrightarrow y}{ F(x,y,\Lc) } \d x \d y.
\]
\end{theorem}

There are many instances in statistical mechanics where such a relationship between the two-point function of a field and connection probabilities holds.
The Edwards--Sokal coupling between the Potts model and the random cluster model is a famous example \cite{fortuin1972random, edwards1988generalization}. 
These relations can be very powerful tools to study one model via the knowledge of the other.
As already alluded to, the clusters of the loop soup are well understood at criticality ($\theta = 1/2$) thanks to the GFF. On the other hand, the current paper, together with \cite{JLQ23a}, studies directly the clusters ($\theta \leq 1/2$) and then infers properties on $h_\theta$.

\begin{remark}\label{rmk:isomorphism1/2}
When $\theta = 1/2$, the law $\P_{x \leftrightarrow y}$ can be described precisely as follows. Consider the metric graph GFF $\tilde{\varphi}$ on some graph and let $x$ and $y$ be two distinct points. The BFS-Dynkin isomorphism \cite{BFS82Loop, Dynkin1984Isomorphism, Dynkin1984IsomorphismPresentation} tells us that for any test function $F$,
\begin{equation}
\label{E:BFS_Dynkin}
\Expect{ \tilde{\varphi}(x) \tilde{\varphi}(y) F( \tilde{\varphi}^2 / 2 ) }
= G(x,y) \Expect{ F( \tilde{\varphi}^2 / 2 + \ell_\wp ) }
\end{equation}
where $\wp$ is an independent Brownian excursion from $x$ to $y$ and $\ell_\wp$ is its local time. Note that by independence of the signs on different clusters, the left hand side can be rewritten as $\Expect{ \tilde{\varphi}(x) \tilde{\varphi}(y) F( \tilde{\varphi}^2 / 2 ) \indic{x \leftrightarrow y} }$.
In the scaling limit, the identity \eqref{E:BFS_Dynkin} suggests that the law of the cluster joining $x$ to $y$ under $\P_{x \leftrightarrow y}$ can be obtained as follows. Sample a critical Brownian loop soup $\Lc_D^{1/2}$ and an independent Brownian excursion $\wp$ from $x$ to $y$, then consider the cluster joining $x$ to $y$ in the collection of paths $\Lc_D^{1/2} \cup \{ \wp \}$.

Viewing the path $\wp$ as an SLE$_2$ curve decorated by Brownian loops coming from a loop soup with intensity $\theta =1$ \cite{LSW03Restriction}, we can see the above description as a notion of duality between the loop soup with intensity $1/2$ and the loop soup with intensity $1/2 + 1 = 3/2$. See also Lemma \ref{L:duality_criticality}.
As we will see in Section \ref{S:2D_Wick}, some notion of duality between the intensities $\theta$ and $2-\theta$ might hold more generally for all $\theta \in (0,1/2]$. We will see that this is reminiscent to the duality between Bessel processes of dimensions $d$ and $4-d$.
\end{remark}

%% file: subfiles/intro_discrete_new.tex
Let $D_N \subset \frac1N \Z^2 \cap D$ be a subset of the square lattice with mesh size $\frac1N$ which approximates the continuous domain $D$. See \eqref{eq:DN} for a precise definition.
Fix $\theta \in (0,1/2]$ and let $\Lc_{D_N}^\theta$ be a random walk loop soup in the metric graph $\widetilde{D}_N$ and $\Cf_N = \{ \Cc \}_{\Cc \in \Cf_N}$ be the collection of clusters $\Cc$ of $\Lc_{D_N}^\theta$. Conditionally on $\Cf_N$, let $\sigma_\Cc$, $\Cc \in \Cf_N$, be i.i.d. spins taking values in $\{\pm 1\}$ with equal probability.
For all $x \in \widetilde{D}_N$, let $\ell_x$ be the total local time accumulated by all loops at $x$, see \eqref{E:discrete_local_time}.
For orientation, we recall that for all $x$, $\ell_x$ is distributed according to Gamma$(\theta,G_{D_N}(x,x))$ where $G_{D_N}(x,x)$ is the discrete Green function which behaves, with our normalisation, like $G_{D_N}(x,x) \sim \frac{1}{2\pi} \log N$ as $N \to \infty$.
We can now define what we conjecture being the analogue of $h_\theta$ in the discrete:
\begin{equation}
\label{E:def_htheta_discrete}
h_{\theta,N}(x) := \frac{\Gamma(\theta)}{2^{1-\theta} \Gamma(2-\theta)} \sigma_{\Cc_x} \{ 2\pi \ell_x \}^{1-\theta},
\end{equation}
where $\Cc_x$ denotes the cluster containing $x$. Importantly, there is no renormalisation in the definition of $h_{\theta,N}$.

\begin{conjecture}\label{Conj:intro}
Let $\theta \in (0,1/2]$. There exists $c_{\mathrm{conj}} = c_{\mathrm{conj}}(\theta) >0$ such that for all bounded smooth test function $f$,
\begin{equation}
\frac{1}{N^2} \sum_{x \in D_N} f(x) h_{\theta,N}(x)
\xrightarrow[N \to \infty]{\mathrm{(d)}} c_{\mathrm{conj}} (h_\theta,f).
\end{equation}
\end{conjecture}

In Section \ref{S:2D_Wick}, we considerably elaborate on this conjecture and address the question of defining the renormalised powers of $h_\theta$. The powers of $h_\theta$ that reduce to integer powers $:L_x^n:$ of the local time $:L_x:$ have already been defined in \cite{LeJan2011Loops, MR3633290}. In order to define the remaining powers (the ``odd'' powers of $h_\theta$), we give a specific approximation scheme from the discrete and we conjecture that it possesses a nondegenerate scaling limit (Conjecture \ref{Conjecture1}). This is based on a connection with the one-dimensional Brownian loop soup (first revealed in \cite{JLQ23a}) and reveals a surprising notion of duality between the intensities $\theta$ and $\theta^* = 2-\theta$. If this conjecture held true, we would obtain the following expansion for $\Mc_\gamma^+$ around $\gamma =0$:
\[
\Mc_\gamma^+(\d x) = \Big( \sum_{k \geq 0} \frac{\gamma^{2k}}{2^k} \frac{\Gamma(\theta)}{k! \Gamma(k+\theta)} (2\pi)^k :L_x^k:
+ c_\mathrm{conj} \gamma^{2(1-\theta)} \frac{\gamma^{2k}}{2^k} \frac{\Gamma(\theta^*)}{k! \Gamma(k+\theta^*)} (2\pi)^k :h_\theta(x) L_x^k: \Big) \d x.
\]
See Section \ref{S:expansion_signed} for the definition of the different terms appearing in the above display.
We plan to prove these conjectures in a future work.

%% file: subfiles/other.tex

Let us review some other conformally invariant and covariant fields that appeared in the mathematical literature.

For $\theta = 1/2$, our field $h_{\theta}$ is the continuum GFF (up to a multiplicative constant).
We believe that for $\theta < 1/2$, the fields $h_{\theta}$
are neither Gaussian nor Markovian (physicists would say non-local fields).

The GFF is in particular a CLE$_{4}$ nesting field (Miller-Sheffield coupling).
In \cite{miller2015conformal}, the authors constructed other conformally invariant 
 fields using the nested CLE$_{\kappa}$ for
$\kappa\in (8/3,4]$.
However, as already mentioned, for $\theta<1/2$, 
our fields $h_{\theta}$ are different from the \textit{nesting fields} built in \cite{miller2015conformal}.
In the case $\theta=1/2$, the clusters of the corresponding Brownian loop soup can be described through the nested labelled CLE$_{4}$; see \cite[Remark 17]{ALS4}.
However, for $\theta<1/2$, we believe that the clusters of the Brownian loop soup cannot be described through the nested CLE$_{\kappa}$,
since there is too much independence in the latter.

For $\theta=1/4$, the outer boundaries of the Brownian loop soup clusters
are the CLE$_{3}$ loops,
which are also the scaling limits of interfaces in the critical 2D Ising model
\cite{KempainnenSmirnovCLE}.
But this coincidence is just for the outer boundaries.
The full scaling limit of Ising spin clusters is different from the Brownian loop soup clusters.
The dimension of the former is $187/96$ 
(dimension of the CLE$_{3}$ gasket \cite{MillerSunWilsonDimCLEgasket}),
and the dimension of the latter is $2$.
We further believe that our field is none of the Ising local fields studied in the Conformal Field Theory (CFT),
and in particular not the continuum \textit{Ising spin field} constructed
in \cite{CamiaGarbanNewmanIsing1,CamiaGarbanNewmanIsing2},
since the two-point correlations of these fields from the Ising CFT
blow up on the diagonal as a power (and not a power of the $\log$),
and these fields are conformally covariant but not invariant
(under conformal mappings, the modulus of the derivative to some power appears).
Still, it is an open question whether $h_{\theta}$ for $\theta=1/4$ is related to the Ising model in some other way, perhaps through a non-local transformation.

We would like to mention that other conformally covariant (but not invariant)
fields have been constructed using the Brownian loop soup.
First, there are the so-called \textit{winding fields}, obtained by renormalising
$e^{i\beta W(z)}$, where $W(z)$ is the sum of the indices of the Brownian loops around the point $z\in D$, and $\beta$ is a parameter in $[0,\pi]$ 
\cite{CamiaGandolfiKleban16, BrugCamiaLis, LeJanWinding}.
The correlations of the winding fields blow up as a power on the diagonal.
Informally speaking, the winding fields have constant values on the dual clusters of the Brownian loop soup, that is to say on the connected components of the domain minus the Brownian loops, all these connected components being
CLE$_{\kappa}$ gaskets in the scaling limit.
So this is different from our setting, since the fields $h_{\theta}$
are in a sense constant on the clusters of the Brownian loop soup,
not on their planar duals.
It has been observed that for $\theta =1/2$ and
$\beta = \pi$ (so that $e^{i\beta W(z)}\in\{-1,1\}$),
the corresponding winding field is exactly the Kramers-Wannier 
dual (a planar duality inherited from the duality in FK random clusters) 
of the 2D GFF \cite[Theorem 2.1]{BrugCamiaLis}.
Another family of fields constructed from the Brownian loop soups are the 
so called \textit{layering fields},
which are a variant of winding fields, where one considers only the outer boundaries of Brownian loops, which are SLE$_{8/3}$ loops.
The latter construction does not capture the structure of the clusters of Brownian loops, since one keeps only little information on those loops.
However, these layering fields are known to satisfy a lot of exact solvability
and have been related to the CFT \cite{CamiaGandolfiKleban16,CamiaGandolfiPeccatiReddyGMC,CamiaFoitGandolfiKleban20,
CamiaFoitGandolfiKleban22,CamiaFoitGandolfiKleban22SE}.

\subsection{Organisation of the paper}

The paper is organised as follows:

\begin{itemize}[leftmargin=*]
\item Section \ref{S:preliminaries}: we recall some definitions and results that we will need from the literature. We will also state and prove preliminary lemmas that we will be useful in the rest of the paper.
\item Section \ref{S:convergence_ratio}: we prove the convergence of the ratio of two hitting probabilities by clusters; see Theorem \ref{T:ratio_general}. To do so, we will develop in Theorem \ref{Thm abstract coupling} a general coupling result between two trajectories of the same Markov chain with distinct starting points.
\item Section \ref{S:construction_h}: we prove Theorems \ref{T:convergenceL2}, \ref{T:cluster2point} and \ref{T:cluster_xy_doob} about the construction of $h_\theta$, the conditioned probability measure $\P_{x \leftrightarrow y}$ and the link between the two.
\item Section \ref{S:minkowski}: we prove Theorems \ref{T:measure_cluster} and \ref{T:h_and_minkowski} concerning the definition of the Minkowski contents of the clusters and their relations to the field $h_\theta$.
\item Section \ref{sec:exc_bdy}: we obtain estimates on the loops that touch the outer boundary of a given cluster; see in particular Propositions \ref{prop:be_main}, \ref{prop:integral} and \ref{prop:sum_mm2}.
\item Section \ref{S:properties}: we prove the properties of $h_\theta$ stated in Theorem \ref{T:prop}. We also prove that CLE/SLE are level lines of $h_\theta$ as stated in Theorems \ref{T:level_line} and \ref{T:mixed}.
\item Section \ref{S:discrete}: Relying on a discrete approximation, we prove the identification of $h_\theta$ with the GFF when $\theta =1/2$ (Theorem \ref{T:identification_field1/2}). See in particular Theorem \ref{T:one_cluster}.
\item Section \ref{S:2D_Wick}: we prove an expansion of the unsigned measure $\Mc_\gamma$ in powers of $\gamma$; see Theorem \ref{T:expansion_unsigned}. We then investigate the Wick powers of $h_\theta$ and state a precise conjecture about a renormalisation procedure that would define them (Conjecture \ref{Conjecture1}). We further show that these Wick powers would allow one to expand the signed measure $\Mc_\gamma^+$; see Theorem \ref{T:expansion_signed_discrete} and \eqref{E:conj_expansion}. Finally, we study the analogue of the field $h_\theta$ in dimension 1; see Section \ref{S:1D}.
\end{itemize}
Most sections of the current paper contain an introduction that describes more thoroughly its content.

%% file: subfiles/preliminaries.tex

\subsection{Measures on paths and definition of the Brownian loop soup}\label{SS:preliminaries_paths}

\begin{itemize}[leftmargin=*]
\item
Brownian motion: in this paper, we will always consider Brownian motions with infinitesimal generator $\Delta$ rather than the standard Brownian motion,
which has generator $\frac{1}{2}\Delta$ (this is to have more tractable constants in isomorphism relations).
\item
$p_D(t, z, w), t>0, z,w \in D$: transition probability of Brownian motion killed upon leaving the domain $D$.
If $$
p_{\C}(t,z,w) = \frac1{4\pi t} \exp \Big( - \frac{|w-z|^2}{4\pi t }\Big), \quad \quad t >0, \quad z, w \in D,
$$ 
denotes the transition probabilities of Brownian motion in the full plane, and if $\pi_D(t, z, w)$ denotes the probability that a Brownian bridge of duration $t$ remains in the domain $D$ throughout, then 
$
p_D(t, z, w) = p_{\C}(t,z,w) \pi_D(t, z, w).
$
\item
Green function $G_D(z,w) = \int_0^\infty p_D(t,z,w) \d t$, $z,w \in D$. In our normalisation, 
$
G_D(z,w) \sim - \frac1{2\pi} \log (|w-z|)
$
as $|w-z| \to 0$.
\item
Gaussian free field $\varphi$: random generalised function such that $((\varphi, f))_{f \in \mathcal{S}(\C)}$ is a centred Gaussian process with covariance given by $\Expect{(\varphi, f)(\varphi, g)} = \int_{D \times D} f(z) G_D(z,w) f(w) \d z \d w$. See \cite{WernerPowell, BerestyckiPowell} for introductions to the GFF. We will denote by $h = \sqrt{2\pi} \varphi$ the GFF in $D$ normalised so that $\Expect{h_z h_w} \sim -\log |z-w|$ as $z-w \to 0$.
\item 
$\P_{D,t}^{z,w}, t>0, z, w \in D$:  probability measure on
Brownian bridges from $z$ to $w$ of duration $t$, conditioned on staying in $D$.
\item 
$\loopmeasure_D$: loop measure in $D$:
\begin{equation}
\label{Eq loop meas}
\loopmeasure_{D}(d\wp)=
\int_{D}\int_{0}^{+\infty}
\P_{D,t}^{z,z}(d\wp)
p_D(t, z, z)\dfrac{\d t}{t} \d z.
\end{equation}
\item
$\Lc_D^\theta$, $\theta >0$: Brownian loop soup in $D$ with intensity $\theta$. Random collection of loops distributed according to a Poisson point process with intensity $\theta \loopmeasure_D$.
\item
$\mu_{D}^{z,w}, z,w \in D$: measure on continuous paths
from $z$ to $w$ in $D$:
\begin{equation}
\label{Eq mu D z w}
\mu_{D}^{z,w}(d\wp)
=
\int_{0}^{+\infty}\P_{D,t}^{z,w}(d\wp)p_D(t, z, w) \d t.
\end{equation}
\item
$\Xi_a^{x,D}$, $x \in D$, $a>0$: random loop which is the concatenation at $x$ of all the excursions in a Poisson point process with intensity $a \mu_D^{x,x}$. If there is no ambiguity, we will simply write $\Xi_a^x$.
\end{itemize}

Depending on our needs, we will view clusters of loops as random collections of loops or as random compact subsets of $\overline{D}$.
We now define the metrics we will use in these two situations. In the latter case, we will denote by $(\mathcal{K},\d_\Kc)$ the collection of all compact subsets of $\overline{D}$, endowed with the Hausdorff metric:
\begin{equation}
\label{E:Hausdorff_distance}
\d_\Kc(A,B) = \max \left( \sup_{x \in A} d(x,B), \sup_{x \in B} d(x,A) \right), \quad \quad A, B \in \Kc.
\end{equation}
We now specify the distance we will use between collections of loops.

\paragraph*{Distance between collections of loops}

It is easier to consider the loops in the loop soup as rooted loops (we can sample a root uniformly at random w.r.t.\ the time parametrization of each loop).
Let $\Pc$ be the set of all parametrized continuous planar curves $\gamma$ defined on a finite time-interval $[0, t_\gamma]$. We endow $\Pc$ with the metric
\begin{align*}
d_\Pc(\gamma, \gamma')=|t_\gamma -t_{\gamma'}| + \sup|\gamma(\cdot/ t_\gamma) - \gamma'(\cdot/ t_{\gamma'})|.
\end{align*}
Let $\mathfrak{L}$ be the collection of locally finite families $\Lc$ of loops in $\Pc$. We define the following metric on $\mathfrak{L}$:
\begin{align}
\label{metric}
d_\mathfrak{L}(\Lc, \Lc') \le \eps \iff \, &\exists f: \Lc_\eps \to \Lc' \text{ injective such that } \forall \gamma\in\Lc_\eps, d_\Pc(\gamma, f(\gamma))\le \eps\\
\notag
&\text{ and similarly when we exchange } \Lc \text{ and } \Lc',
 \end{align}
where $\Lc_\eps$ is the collection of loops in $\Lc$ with a diameter larger than $\eps$.

\paragraph*{Random walk loop soup}

For a portion $D_N$ of the square lattice $\frac1N \Z^2$, we will also need to consider the random walk loop soup $\Lc_{D_N}^\theta$ in $D_N$ with intensity $\theta$, introduced by \cite{LawlerFerreras07RWLoopSoup}. Its definition is analogous to the definition of $\Lc_D^\theta$ with discrete versions of the heat kernel and Brownian bridges. See \cite[Section 2.2]{ABJL21} for details.

\subsection{Multiplicative chaos of the Brownian loop soup}\label{S:multiplicative_chaos}

In this section, we will recall the definition and some properties of the multiplicative chaos measure $\Mc_\gamma$ of the Brownian loop soup $\Lc_D^\theta$ in a given bounded simply connected domain $D \subset \C$. We will recall the construction from the discrete setting since it is the easiest one to explain and since it will be the approximation we will use in Section \ref{S:discrete}. Let us mention however that there is a construction directly in the continuum which is based on Brownian multiplicative chaos measures \cite{bass1994, AidekonHuShi2018, jegoBMC, jegoRW, jegoCritical}. We follow \cite{ABJL21}. Let us stress already that our normalisation of $\Mc_\gamma$ is given by $\Expect{\Mc_\gamma(\d z)} = 2 \d z$ which is different from the one in \cite{ABJL21}: if we denote by $\hat{\Mc}_\gamma$ the measure constructed in \cite[Theorem 1.1]{ABJL21}, we will have
\begin{equation}\label{E:difference_normalisation}
\Mc_\gamma(\d z) = a^{1-\theta} 2^{1+\theta} \Gamma(\theta) \CR(z,D)^{-\gamma^2/2} \hat{\Mc}_\gamma(\d z).
\end{equation}

Let $N \geq 1$ be a large integer and let $D_N \subset \frac1N \Z^2$ be a discrete approximation of $D$. Specifically, assuming without loss of generality that $D$ contains the origin, let
\begin{equation}
\label{eq:DN}
D_N := \Big\{ z \in D \cap \frac{1}{N} \Z^2 :
\begin{array}{c}
\mathrm{there~exists~a~path~in~} \tfrac{1}{N} \Z^2 \mathrm{~from~} z \mathrm{~to~the~origin} \\
\mathrm{whose~distance~to~the~boundary~of~} D \mathrm{~is~at~least~} \tfrac{1}{N}
\end{array}
\Big\}.
\end{equation}
Let $\Lc_{D_N}^\theta$ be a random walk loop soup with intensity
$\theta$. For any vertex $x \in D_N$ and any loop $\wp = (\wp(t))_{0 \leq t \leq T(\wp)} \in \Lc_{D_N}^\theta$, we denote by $\ell_x(\wp)$ the local time of $\wp$ at $x$ and $\ell_x$ the local time of the loop soup at $x$:
\begin{equation}
\label{E:discrete_local_time}
\ell_x(\wp) := N^2 \int_0^{T(\wp)} \indic{ \wp(t) = z } \d t
\quad \quad \text{and} \quad \quad \ell_x := \sum_{\wp \in \Lc_{D_N}^\theta} \ell_x(\wp).
\end{equation}
In our normalisation,
$
\Expect{\ell_x} \sim \frac{\theta}{2\pi} \log N
$
as $N \to \infty$,
for any $x$ in the bulk of $D_N$. We then define the set of $a$-thick points by
\begin{equation}
\Tc_N(a) := \left\{ x \in D_N: \ell_x \geq \frac{1}{2\pi} a (\log N)^2 \right\}.
\end{equation}
The parameter $a$ can be thought of as a thickness parameter. It is related to the GMC parameter $\gamma$ by
\begin{equation}
\label{E:gamma}
a = \frac{\gamma^2}{2},
\quad \quad \gamma = \sqrt{2a}.
\end{equation}
We finally define the discrete approximation $\Mc_a^N$ of $\Mc_a$ as the uniform measure on the set of thick points: for all Borel set $A \subset \C$,
\begin{equation}
\label{E:def_discrete_measure}
\Mc_a^N(A) := \frac{1}{c_*(a)} \frac{(\log N)^{1-\theta}}{N^{2-a}} \sum_{x \in \Tc_N(a)} \CR(x,D)^{-a} \indic{x \in A}.
\end{equation}
where the constant $c_*(a)$ is given by $c_*(a) = \frac{(2\sqrt{2})^a e^{a \gamma_{\text{EM}}}}{2 a^{1-\theta} \Gamma(\theta)}$ with $\gamma_{\text{EM}}$ being the Euler--Mascheroni constant.

\begin{theorem}[{\cite[Theorem 1.12]{ABJL21}}]
Let $\theta >0$ and $a \in (0,2)$. The following convergence holds in distribution
\[
(\Lc_{D_N}^\theta, \Mc_a^N) \to (\Lc_D^\theta, \Mc_a) \quad \quad \text{as} \quad N \to \infty,
\]
relatively to the topology induced by $d_{\mathfrak{L}}$ for $\Lc_{D_N}^\theta$ and the weak topology on $\C$ for $\Mc_a^N$.
\end{theorem}

With some abuse of notation, we will write $\Mc_a$ or $\Mc_\gamma$ depending on whether we view the measure as a function of $a$ or $\gamma$. In this paper $a$ and $\gamma$ will always be related by \eqref{E:gamma}.

\begin{remark}\label{rmk:multiplicative_chaos}
Let us mention that the previous results can easily be generalised to finitely connected domains. In that case, one would need to change in \eqref{E:def_discrete_measure} $\log \CR(z,D)$ by the harmonic extension of $\log |z - \cdot|$ from $\partial D$ to $D$. Indeed, this is the function that appears in the asymptotic behaviour of the Green function on the diagonal (see e.g. (1.4) in \cite{BerestyckiPowell}).
\end{remark}

\paragraph*{Some properties}

In the following lemma we recall some first and second moments computations related to $\Mc_a$. Let $a,a' \in (0,1)$ be two thickness parameters lying in the $L^2$-regime. Let $\{a_1, a_2, \dots \}$ and $\{a_1', a_2', \dots \}$ be two independent random partitions of $[0,a]$ and $[0,a']$, respectively, distributed according to a Poisson--Dirichlet distribution with parameter $\theta$. Conditionally on these partitions, let $\Xi_{a_i}^x, \Xi_{a_i'}^y$, $i \geq 1$, be independent loops whose distributions are described in Section \ref{SS:preliminaries_paths}. 
We also recall the definition of the modified Bessel function of the first kind
\begin{equation}
\label{E:BesselI}
I_{\theta-1}(u) = \sum_{n \geq 0} \frac{1}{n! \Gamma(\theta+n)} \left( \frac{u}{2} \right)^{2n+\theta-1},
\quad \quad \theta >0, \quad u >0.
\end{equation}

\begin{lemma}\label{L:moment_measure}
For $a \in (0,2)$ and $F: D \times \mathfrak{L} \to \R$ a bounded measurable admissible function, we have
\begin{equation}
\label{E:first_moment_measure}
\Expect{ \int_D F(x, \Lc_D^\theta) \Mc_a(\d x) } = 2 \int_D \Expect{F(x, \Lc_D^\theta \cup \{ \Xi_{a_i}^x \}_{i \geq 1} ) } \d x
\end{equation}
where the two collections of loops $\Lc_D^\theta$ and $\{\Xi_{a_i}^x \}_{i \geq 1}$ are independent.
Let $a, a' \in (0,1)$.
We have
\begin{equation}
\label{E:second_moment_measure1}
\Expect{ \Mc_{a}(\d x) \Mc_{a'}(\d y) } = 4 \left( 2 \pi \sqrt{a a'} G_D(x,y) \right)^{1-\theta} \Gamma(\theta) I_{\theta-1} \left(4\pi \sqrt{a a'} G_D(x,y) \right) \d x \d y.
\end{equation}
Moreover, if $F : D \times D \times \mathfrak{L} \to \R$ is a bounded measurable admissible function, then
\begin{align}
\label{E:second_moment_measure2}
& \Expect{F(x,y,\Lc_D^\theta) \Mc_{a}(\d x) \Mc_{a'}(\d y) \indic{\nexists \wp \in \Lc_D^\theta \text{~visiting~both~} x \text{~and~} y} } \\
& = 4 \Expect{F(x,y,\Lc_D^\theta \cup \{ \Xi_{a_i}^x \}_{i \geq 1} \cup \{ \Xi_{a_i'}^y \}_{i \geq 1} ) } \d x \d y
\nonumber
\end{align}
where the three collections of loops $\Lc_D^\theta$, $\{\Xi_{a_i}^x \}_{i \geq 1}$ and $\{ \Xi_{a_i'}^y \}_{i \geq 1}$ are independent.
\end{lemma}

\begin{proof}
Recall that our normalisation of $\Mc_a$ is different than in \cite{ABJL21} (see \eqref{E:difference_normalisation}).
\eqref{E:first_moment_measure} is the content of \cite[Theorem 1.8]{ABJL21} and 
\eqref{E:second_moment_measure1} can be found in \cite[Remark 1.3]{ABJL21}. It remains to prove \eqref{E:second_moment_measure2}. We are going to see that it is a quick consequence of \eqref{E:first_moment_measure}.
Indeed, the left hand side of \eqref{E:second_moment_measure2} is equal to
\begin{align*}
2 \Expect{ F(x,y,\Lc_D^\theta \cup \{ \Xi_{a_i}^x \}_{i \geq 1} ) \Mc_{a'}^{\Lc_D^\theta \cup \{ \Xi_{a_i}^x \}_{i \geq 1} }(\d y) \indic{\nexists \wp \in \Lc_D^\theta \cup \{ \Xi_{a_i}^x \}_{i \geq 1} \text{~visiting~both~} x \text{~and~} y} } \d x
\end{align*}
where $\Lc_D^\theta$ and $\{\Xi_{a_i}^x \}_{i \geq 1}$ are independent. In the above display, we denoted by $\Mc_{a'}^{\Lc_D^\theta \cup \{\Xi_{a_i}^x \}_{i \geq 1}}$ the multiplicative chaos generated by the loops in $\Lc_D^\theta \cup \{\Xi_{a_i}^x \}_{i \geq 1}$. For more details see \cite[Section 1.3]{ABJL21}. Here, we notice that on the event that none of the loops in $\{\Xi_{a_i}^x \}_{i \geq 1}$ visits $y$, none of these loops can contribute to the thickness at $y$:
\[
\Mc_{a'}^{\Lc_D^\theta \cup \{ \Xi_{a_i}^x \}_{i \geq 1} }(\d y) \indic{\nexists \wp \in \{ \Xi_{a_i}^x \}_{i \geq 1} \text{~visiting~} y} 
= \Mc_{a'}(\d y) \indic{\nexists \wp \in \{ \Xi_{a_i}^x \}_{i \geq 1} \text{~visiting~} y}
\quad \quad \text{a.s.}
\]
Hence, the left hand side of \eqref{E:second_moment_measure2} is equal to
\begin{align*}
& 2 \Expect{ F(x,y,\Lc_D^\theta \cup \{ \Xi_{a_i}^x \}_{i \geq 1} ) \Mc_{a'}(\d y) \indic{\nexists \wp \in \Lc_D^\theta \cup \{ \Xi_{a_i}^x \}_{i \geq 1} \text{~visiting~both~} x \text{~and~} y} } \d x \\
& = 4 \E \Big[ F(x,y,\Lc_D^\theta \cup \{ \Xi_{a_i}^x \}_{i \geq 1} \cup \{ \Xi_{a_i'}^y \}_{i \geq 1} ) \mathbf{1}_{ \{ \nexists \wp \in \Lc_D^\theta \cup \{ \Xi_{a_i}^x \}_{i \geq 1} \cup \{ \Xi_{a_i'}^y \}_{i \geq 1} \text{~visiting~both~} x \text{~and~} y \} } \Big] \d x ~\d y
\end{align*}
where in the last equality we applied once again \cite[Theorem 1.8]{ABJL21}. Finally, we notice that for Lebesgue-typical points $x$ and $y$, almost surely none of the loops in $\Lc_D^\theta \cup \{ \Xi_{a_i}^x \}_{i \geq 1} \cup \{ \Xi_{a_i'}^y \}_{i \geq 1}$ visits both $x$ and $y$. This concludes the proof.
\end{proof}

\subsection{Crossing exponent in the Brownian loop soup}

We now recall the main result of \cite{JLQ23a} concerning the decay rate of the probability of some crossing event in the loop soup. This decay rate is particularly important to us since it drives the blow-up of the correlations of $h_\theta$ (Theorem \ref{T:prop}). The decay of the normalising constant $Z_\gamma$ in Theorem \ref{T:convergenceL2} is also obtained from the asymptotic behaviour of these crossing probabilities.

\begin{theorem}[{\cite[Theorem 1.2]{JLQ23a}}]\label{T:large_crossing}
Let $\theta \in (0,1/2]$.
The probability that a cluster in a Brownian loop soup $\Lc_\D^\theta$ intersects both the microscopic circle $r \partial \D$ and the macroscopic circle $e^{-1} \partial \D$ decays like
\[
\Prob{ r \partial \D \overset{\Lc_\D^\theta}{\longleftrightarrow} e^{-1} \partial \D } = |\log r|^{-1+ \theta +o(1)}
\quad \quad \text{as~} \quad r \to 0.
\]
\end{theorem}

We will also need the following estimates from \cite{JLQ23a}. As opposed to Theorem \ref{T:large_crossing}, the next results concern some crossing events by \emph{loops} rather than \emph{clusters} (they are in particular much easier to establish).
In the first lemma below and in the rest of the paper, we will say that a loop $\wp \in \Lc_\D^\theta$ surrounds the disc $r \D$ if $\wp$ does not intersect $r \D$ but disconnects it from $\partial \D$.

\begin{lemma}[{\cite[Lemma 2.5]{JLQ23a}}] \label{L:surround}
There exists $c = c(\theta)>0$ such that for all $r \in (0,1/10)$,
\begin{equation}
\Prob{ \exists \wp \in \Lc_\D^\theta \mathrm{~surrounding~} r \D } \geq 1 - r^c.
\end{equation}
\end{lemma}

\begin{lemma}[{\cite[(2.3) in Lemma 2.2]{JLQ23a}}] \label{L:cross_loop}
Let $0<r_1<r_2<1$. We have
\[
\loopmeasure_\D (\{ \wp \text{~crossing~} r_2 \D \setminus r_1 \D \} ) = \left( 1 + O(1) \frac{r_1}{r_2} \right) \log \frac{\log (1/r_1)}{\log (r_2/r_1)}.
\]
In particular, for all $r_0 \in (0,e^{-1})$ and $s \geq 1$,
\[
\Prob{ \exists \wp \in \Lc_\D^\theta: r_0 \overset{\wp}{\longleftrightarrow} r_0^s } = (1+O(s^{-1})) \theta s^{-1}.
\]
\end{lemma}

We finally recall the FKG inequality which is an import tool that we will use frequently in the paper.
A function $f : \mathfrak{L} \to \R$ is said to be increasing if for all $\Lc, \Lc' \in \mathfrak{L}$ with $\Lc \subset \Lc'$, $f(\Lc) \leq f(\Lc')$.

\begin{lemma}[FKG inequality {\cite[Lemma 2.1]{Janson84}}]
For all increasing bounded measurable functions $f, g : \mathfrak{L} \to \R$,
\[
\Expect{ f(\Lc_D^\theta) g(\Lc_D^\theta) } \geq \Expect{ f(\Lc_D^\theta) } \Expect{ g(\Lc_D^\theta) }.
\]
\end{lemma}

In most places we will apply this result to indicator functions of increasing events.

\subsection{Some properties of \texorpdfstring{$\Xi_a^z$}{Xiaz}}
\label{SS:thick_loop}

We start by recalling a restriction property for the $a$-thick loop $\Xi_a^z$ defined in Section \ref{SS:preliminaries_paths}.

\begin{lemma}\label{L:restriction_thick_loop}
Let $z \in D$ and let $D' \subset D$ be a simply connected domain containing $z$. Then
\begin{equation}
\label{E:restriction_thick_loop}
\Expect{F(\Xi_a^{z,D}) \indic{ \Xi_a^{z,D} \subset D'} } = \frac{\CR(z,D')^a}{\CR(z,D)^a} \Expect{F(\Xi_a^{z,D'})}.
\end{equation}
\end{lemma}

\begin{proof}
This follows quickly from the restriction property of the measures $\mu_D^{z,w}$. A proof of this result can be found in \cite{ABJL21}; see equation (5.2) therein.
\end{proof}

As a direct consequence,

\begin{lemma}\label{L:diameter_thick}
Let $x \in D$ and $a>0$. For any $r< \d(x,\partial D)$,
\[
\Prob{ \norme{\Xi_a^x - x}_\infty \in \d r } = \frac{ar^{a-1}}{\CR(x,D)^a} \d r.
\]
\end{lemma}

\begin{proof}
By Lemma \ref{L:restriction_thick_loop}, the probability that $\norme{\Xi_a^x - x}_\infty < r $ is equal to $\frac{\CR(x,D \cap D(x,r))^a}{\CR(x,D)^a}$. When $r < \d(x,\partial D)$, this is simply equal to $\frac{r^a}{\CR(x,D)^a}$. The lemma then follows after differentiating with respect to $r$.
\end{proof}

We will also need the following two-point estimate:

\begin{lemma}\label{L:thick_loop_xy}
Let $x, y \in D$ with $|x-y| > e^{-1/a}$ and let $k \geq 0$.
\begin{equation}
\label{E:L:thick_loop_xy}
\Prob{ \norme{ \Xi_a^x - y }_\infty \frac{2}{|x-y|} \in [e^{-k-1},e^{-k}] } \leq Ca \frac{ \left( \log |x-y| \right)^2 }{ k \left( k + |\log |x-y|| \right)}.
\end{equation}
\end{lemma}

\begin{proof}
We will only explain where this bound comes from without providing all the details.
By Lemma \ref{L:diameter_thick}, the probability that $\Xi_a^x$ reaches the circle $\partial D(x,|x-y|/2)$ is given by
$
1 - \left( \frac{|x-y|}{2 \CR(x,D)} \right)^a.
$
Using the assumption that $|x-y|>e^{-1/a}$, we can bound this probability by $C a |\log |x-y||$. Conditioned on the event that $\Xi_a^x$ has reached $\partial D(x,|x-y|/2)$, the probability that $\frac{2}{|x-y|} \norme{\Xi_a^x - y}_\infty  \in [e^{-k-1},e^{-k}]$ can be compared with the probability for a simple random walk on $\Z$ starting at $0$ to reach $k$ before reaching $-\log |x-y|$ (hit the small ball before exiting the domain) and then to come back to $0$ without hitting $k+1$ (do not hit $\partial D(y,|x-y|e^{-k+1}/2)$ before going back to $\partial D(y,|x-y|/2)$).
The probabilities of these events are respectively equal to
\[
\frac{|\log |x-y||}{k + |\log |x-y||}
\quad \text{and} \quad
\frac{1}{k}.
\]
Putting things together, we obtain that \eqref{E:L:thick_loop_xy}.
\end{proof}

We finish this section by combining Theorem \ref{T:large_crossing} and Lemma \ref{L:diameter_thick} to show that, as stated in Theorem \ref{T:convergenceL2}, the normalising constant $Z_\gamma$ \eqref{E:Zgamma} behaves like $\gamma^{2(1-\theta)+o(1)}$.
Importantly, a slight modification of the proof below shows that for any $z \in D$ and compact set $K \subset D \setminus \{z\}$,
\begin{equation}
\label{E:thick_to_disc}
\P \Big( K \overset{\Lc_D^\theta \cup \Xi_a^z}{\longleftrightarrow} z \Big)= (1+o(1)) \P \Big( K \overset{\Lc_D^\theta}{\longleftrightarrow} D(z, \norme{\Xi_a^z - z}_\infty) \Big),
\end{equation}
i.e. we can replace the thick loop at $z$ by a disc centred at $z$ with random radius $\norme{\Xi_a^z - z}_\infty$.

\begin{proof}[Proof of Theorem \ref{T:convergenceL2} -- asymptotic behaviour of $Z_\gamma$]
We are going to show that $Z_\gamma = \gamma^{2(1-\theta)+o(1)}$ as $\gamma \to 0$.
We start with the upper bound.
If the origin is connected to the circle $e^{-1} \partial \D$ by a cluster of $\Xi_a^{0,\D} \cup \Lc_\D^\theta$, there must be a cluster of $\Lc_\D^\theta$ which intersects both $r \partial \D$ and $e^{-1} \partial \D$ where $r = \norme{ \Xi_a^{0,\D} }_\infty$.
By independence of $\Xi_a^{0,\D}$ and $\Lc_\D^\theta$ this gives
\begin{align*}
Z_\gamma & \leq \Prob{\norme{ \Xi_a^{0,\D} }_\infty \geq e^{-1} } + \int_0^{e^{-1}}  \Prob{\norme{ \Xi_a^{0,\D} }_\infty \in \d r} \Prob{ r \partial \D \overset{\Lc_\D^\theta}{\longleftrightarrow} e^{-1} \partial \D }.
\end{align*}
Let $\eps >0$. By Theorem \ref{T:large_crossing}, there exists $C_\eps >0$ such that the last probability in the above display is at most $C_\eps |\log r|^{-1+\theta+\eps}$ for all $r \in (0,e^{-1})$. Together with Lemma \ref{L:diameter_thick}, we obtain that
\[
Z_\gamma
\leq 1 - e^{-a} + \int_0^{e^{-1}} a r^{a-1} C_\eps |\log r|^{-1+\theta +\eps} \d r \sim C_\eps \Gamma(\theta + \eps) a^{1-\theta - \eps}
\]
as $a \to 0$. Since $a = \gamma^2/2$, this proves the desired upper bound: $Z_\gamma \leq \gamma^{2(1-\theta)+o(1)}$.

We now turn to the lower bound. The event $\{ 0 \overset{\Lc_\D^\theta \cup \Xi_a^{0,\D}}{\longleftrightarrow} e^{-1} \partial \D \}$ occurs as soon as
1) $\Xi_a^{0,\D}$ reaches $e^{-1/a} \partial \D$,
2) a cluster of $\Lc_\D^\theta$ intersects the circles $e^{-2/a} \partial \D$ and $e^{-1} \partial \D$ and
3) there is a loop in $\Lc_\D^\theta$ surrounding $e^{-2/a} \D$ while staying in $e^{-1/a} \D$.
By independence of $\Xi_a^{0,\D}$ and $\Lc_\D^\theta$ and by FKG, the probability of the intersection of these three events is at least the product of each probability. By Lemma \ref{L:diameter_thick}, the probability of the first event is $1 - e^{-1}$. By Theorem \ref{T:large_crossing}, the probability of the second event is $a^{1-\theta + o(1)}$. By Lemma \ref{L:surround}, the probability of the last event is $1 + o(1)$. Overall, we conclude that $Z_\gamma \geq a^{1-\theta + o(1)} = \gamma^{2(1-\theta)+o(1)}$.
\end{proof}

\subsection{Preliminary estimate}

We finish this preliminary section with a lemma that we state here for ease of future reference.

\begin{lemma}\label{L:Rgamma}
For $\gamma >0$, let
\begin{equation}
\label{E:Rgamma}
R_\gamma := \sup \{ r >0 : r \partial \D \overset{\Lc_\D^\theta \cup \Xi_a^{0,\D}}{\longleftrightarrow} 0 \}.
\end{equation}
Then for all $\eta >0$, there exists $C=C(\eta)>0$ such that for all $\gamma >0$ and $r \in (0,1)$,
\begin{equation}
\label{E:L_Rgamma1}
\frac{1}{Z_\gamma} \Prob{R_\gamma \geq r} \leq C |1+\log r|^{1-\theta+\eta}.
\end{equation}
Moreover,
\begin{equation}
\label{E:L_Rgamma2}
\limsup_{\gamma \to 0} \frac{1}{Z_\gamma} \Prob{R_\gamma \geq r} \to 0 \quad \quad \text{as} \quad r \to 1.
\end{equation}
\end{lemma}

\begin{proof}
By definition, $Z_\gamma = \Prob{R_\gamma \geq e^{-1}}$. See \eqref{E:Zgamma}. Hence, $\Prob{R_\gamma \geq r} / Z_\gamma \leq 1$ if $r \geq e^{-1}$ which concludes the proof of \eqref{E:L_Rgamma1} in this case. We now consider the case $r \in (0,e^{-1}).$ This time,
\[
\frac{1}{Z_\gamma} \Prob{R_\gamma \geq r} = \Prob{R_\gamma \geq e^{-1} \vert R_\gamma \geq r}^{-1}.
\]
Let $E_1 := \{ e^{-1} \partial \D \overset{\Lc_\D^\theta}{\longleftrightarrow} r/2 \partial \D \}$ and $E_2$ be the event that there is a loop in $\Lc_\D^\theta$ surrounding $r/2 \D$ while staying in $r\D$. 
Conditionally on $\{R_\gamma \geq r\}$, $\{ R_\gamma \geq e^{-1} \} \supset E_1 \cap E_2$. By FKG inequality, we deduce that
\[
\frac{1}{Z_\gamma} \Prob{R_\gamma \geq r} \leq \Prob{E_1}^{-1} \Prob{E_2}^{-1}.
\]
By Lemma \ref{L:surround}, $\Prob{E_2} \geq c$ and, by Theorem \ref{T:large_crossing}, for all $\eta >0$, there exists $c>0$ such that $\Prob{E_1} \geq c |\log r|^{-1+\theta-\eta}$. This concludes the proof of \eqref{E:L_Rgamma1}.

We now prove \eqref{E:L_Rgamma2}. Let $r \in (e^{-1},1)$. 
FKG inequality cannot be used directly to bound $\Prob{R_\gamma \geq r} / Z_\gamma = \Prob{R_\gamma \geq r \vert R_\gamma \geq e^{-1}}$ (it goes in the wrong direction). Let $E$ be the event that $\Xi_a^{0,\D} \subset e^{-1} \D$. By Lemma \ref{L:diameter_thick}, $\Prob{E^c} \leq C \gamma^2$. Hence
\[
\limsup_{\gamma \to 0} \frac{1}{Z_\gamma} \Prob{E^c, R_\gamma \geq r} = 0.
\]
We now work on the event $E$. Let $r' \in (e^{-1}, r)$ be an intermediate radius.
We decompose the loops of $\Lc_\D^\theta$ into two independent sets $\Lc_1$ and $\Lc_2$ of loops: the ones that are contained in $\D \setminus e^{-1} \overline{\D}$ and the ones that touch $e^{-1} \overline{\D}$ respectively.
Notice that the event $\{R_\gamma \geq e^{-1}\}$ is measurable w.r.t. $\Lc_2$. If none of the loops in $\Lc_2$ reaches $r' \partial \D$, then there must be a cluster of $\Lc_1$ that intersects both $r' \partial \D$ and $r \partial \D$ in order to have $R_\gamma \geq r$. That is,
\[
\frac{1}{Z_\gamma} \Prob{E,R_\gamma \geq r}
\leq \Prob{r \partial \D \overset{\Lc_1}{\longleftrightarrow} r' \partial \D \vert R_\gamma \geq e^{-1}} + \Prob{E,\exists \wp \in \Lc_2: r' \partial \D \overset{\wp}{\longleftrightarrow} e^{-1} \partial \D \vert R_\gamma \geq e^{-1}}.
\]
Since $\Lc_1$ is independent of $\{R_\gamma \geq e^{-1}\}$, the first probability on the right hand side is equal to $\Prob{r \partial \D \overset{\Lc_1}{\longleftrightarrow} r' \partial \D}$. $r'$ being fixed, this event can occur only if a macroscopic CLE loop (the outermost boundary of a cluster) in the domain $\D \setminus e^{-1} \overline{\D}$ reaches the circle $r \partial \D$. This probability vanishes as $r \to 1$ (this follows from \cite[Corollaries 4.3 and 6.6]{SheffieldWernerCLE} where the exponent of the probability that a macroscopic CLE loop gets close to a fixed boundary point is computed). 
We have shown that for all $r' \in (e^{-1},1)$,
\begin{equation}
\label{E:pf_q7}
\limsup_{r \to 1} \limsup_{\gamma \to 0} \frac{1}{Z_\gamma} \Prob{R_\gamma \geq r} \leq \limsup_{\gamma \to 0} \Prob{\exists \wp \in \Lc_2: r' \partial \D \overset{\wp}{\longleftrightarrow} e^{-1} \partial \D \vert R_\gamma \geq e^{-1}}.
\end{equation}
Since the left hand side is independent of $r'$, it is enough to show that the right hand side vanishes as $r' \to 1$ to conclude the proof of \eqref{E:L_Rgamma2}.
Let $R:= \inf \{ x>0: \exists \wp\in \Lc_2: r' \partial \D \overset{\wp}{\longleftrightarrow} x \partial \D \}$. The probability on the right hand side of \eqref{E:pf_q7} is equal to
\begin{align*}
\frac{1}{Z_\gamma}\int_0^{e^{-1}} \Prob{R \in \d x, R_\gamma \geq e^{-1}}
\leq \frac{1}{Z_\gamma}\int_0^{e^{-1}} \Prob{R \in \d x} \Prob{R_\gamma \geq x}
\end{align*}
using the independence of the loops that touch $x \D$ and the ones that do not.
By \eqref{E:L_Rgamma1} (applied to $\eta=\theta/2$), $\Prob{R_\gamma \geq x} / Z_\gamma \leq C |\log x|^{1-\theta/2}$. Injecting this estimate in the above integral and then integrating by parts shows that the right hand side of \eqref{E:pf_q7} is at most
\[
\Prob{R \leq e^{-1}} + (1-\theta/2) \int_0^{e^{-1}} \frac{\Prob{R \leq x}}{x|\log x|^{\theta/2}} \d x.
\]
By Lemma \ref{L:cross_loop}, $\Prob{R \leq x} \leq C |\log r'|/|\log x|, x \in (0,e^{-1}]$. The integral $\int_0^{e^{-1}} x^{-1} |\log x|^{-1-\theta/2} \d x$ being finite, the right hand side of \eqref{E:pf_q7} is thus at most $C |\log r'|$ which vanishes as $r' \to 1$. This concludes the proof of \eqref{E:L_Rgamma2}.
\end{proof}

%% file: subfiles/coupling.tex
The main goal of this section is to prove the following result. This will be a key result for the rest of the article.

\begin{theorem}\label{T:ratio_general}
Let $\theta \in (0,1/2]$ and $D$ be a bounded simply connected domain. Let $x \in D$ and $K$ and $\hat{K}$ be two connected compact subsets of $\overline{D} \setminus \{x \}$. The following limits exist and are identical
\begin{equation}
\label{E:ratio_D}
\lim_{r \to 0}
\P \Big( K \overset{\Lc_D^\theta}{\longleftrightarrow} D(x,r) \Big) \Big/ \P \Big( \hat{K} \overset{\Lc_D^\theta}{\longleftrightarrow} D(x,r) \Big)
= \lim_{a \to 0}
\P \Big( K \overset{\Lc_D^\theta \cup \Xi_a^{x,D}}{\longleftrightarrow} x \Big) \Big/ \P \Big( \hat{K} \overset{\Lc_D^\theta \cup \Xi_a^{x,D}}{\longleftrightarrow} x \Big).
\end{equation}
\end{theorem}

The proof of this result is based on the construction of a coupling between two processes of conditioned clusters. Let us describe this result now.

For $t\geq 0$, let $\cyl_{t}$ be the open cylinder $\cyl_{t} = (t,+\infty)\times\mathbb{S}^{1}$,
$\overline{\cyl}_{t}$ the corresponding closed cylinder,
and $\partial\cyl_{t}$ the circle
$\partial\cyl_{t} = \{ t\}\times\mathbb{S}^{1}\subset \overline{\cyl}_{t}$.
The cylinder $\cyl_{0}$ is conformally equivalent to
$\D\setminus\{ 0\}$ through the map
$(t,\omega)\mapsto e^{-t}\omega$.
Since the 2D Brownian loop soups are conformally invariant in law (up to time change), and $\{ 0\}$ is polar for Brownian motion,
working on $\cyl_{0}$ and working on $\D$ (and therefore on any simply connected domain $D$) will be equivalent.

Let $\theta \in (0,1/2]$ and consider a Brownian loop soup with intensity $\theta$ in $\cyl_0$ that we denote by $\Lc_0^\theta$. For $t >0$, let $\Lc_{\cap (0,t]}^\theta$ be the subset of $\Lc_0^\theta$ consisting of all the loops that intersect $(0,t] \times \mathbb{S}^1$.
Let $K_0 \subset \overline{\cyl}_0$ be a connected compact subset of $\overline{\cyl}_0$. For all integer $t \geq 1$, let $Q_0(t)$ be the union of all clusters of $\Lc_{\cap (0,t]}^\theta$ that intersect $K_0$. Let $K_0(t) := ( (\overline{Q_0(t)} \cap \cyl_t) - t) \cup \partial \cyl_0$ (for $A \subset \cyl_0$, $A-t$ stands for the longitudinal translation of $A$ by $t$, to the left)
and $T^\dagger := \inf\{ t \in \N : K_0(t) = \partial \cyl_0 \}$. The event $\{ T^\dagger >t \}$ corresponds to the event that there is a cluster intersecting both $K_0$ and $\{ t\} \times \mathbb{S}^1$. 
Let $\hat{K}_0$ be another initial connected compact subset of $\overline{\cyl}_0$ and define analogously $(\hat{K}_0(t))_{t \in \N}$ and $\hat{T}^\dagger$.
We would like to stress that $(K_0(t))_{t \in \N}$ is \emph{not} a Markov process; see Figure \ref{fig1}. For this reason, we will need to consider a larger filtration than the natural filtration associated to $(K_0(t))_{t \in \N}$; this filtration is introduced in \eqref{E:filtration} and will be denoted by $(\Fc_t)_{t \in \N}$.

\begin{figure}
   \centering
    \def\svgwidth{0.5\columnwidth}
   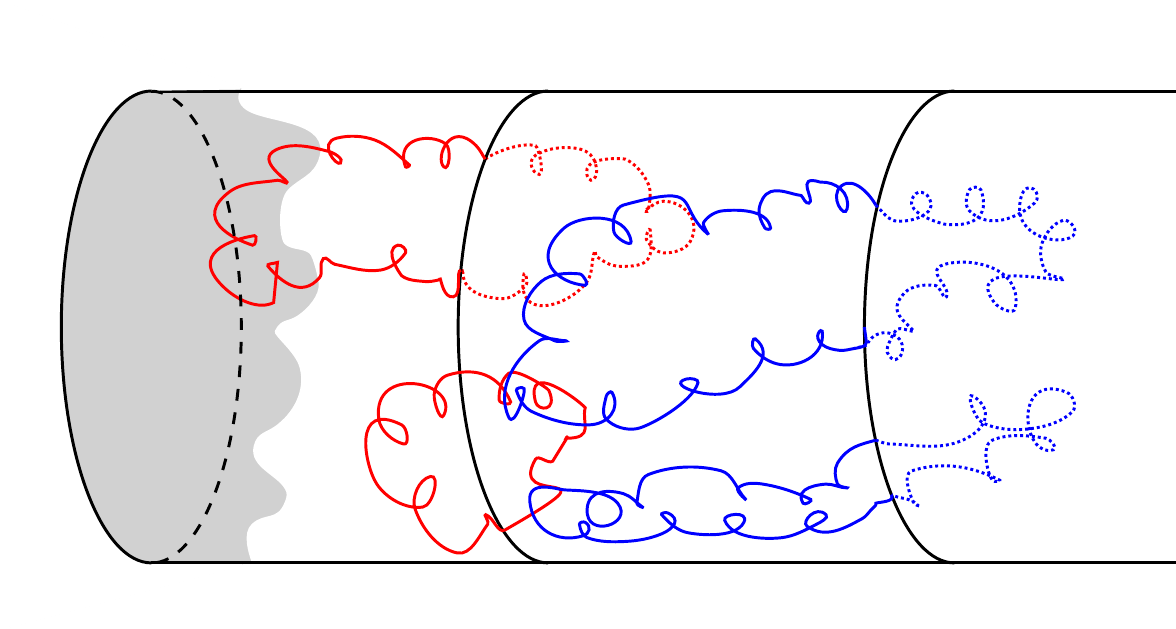
   \caption{Schematic representation of the sequence $(K_0(t))_{t \in \N}$. The filled grey region is the initial set $K_0$. The clusters of $\Lc_{\cap(0,t_1]}^\theta$ (resp. $\Lc_{\cap(0,t_2]}^\theta \setminus \Lc_{\cap(0,t_1]}^\theta$) are depicted in red (resp. blue). $Q_0(t_1) \cap \cyl_{t_1}$ and $Q_0(t_2) \cap \cyl_{t_2}$ correspond to the dotted red and blue regions respectively. $Q_0(t_2)$ is \emph{not} measurable w.r.t. $Q_0(t_1)$ and the loops in $\Lc_{\cap(0,t_2]}^\theta \setminus \Lc_{\cap(0,t_1]}^\theta$ since it also depends on the other clusters of $\Lc_{\cap(0,t_1]}^\theta$ (bottom red cluster in this picture).}\label{fig1}
\end{figure}

Our coupling result reads as follows:

\begin{theorem}\label{T:coupling_rough}
For every initial conditions $K_0$ and $\hat{K}_0$ not identical to $\partial \cyl_0$, and every $T \geq 1$ large enough, we can couple on the same probability space the two stochastic processes $(K_0(t))_{t = 0, \dots, T}$ under $\P(\cdot \vert T^\dagger > T)$ and $(\hat{K}_0(t))_{t = 0, \dots, T}$ under $\P(\cdot \vert \hat{T}^\dagger > T)$ such that there exists a stopping time $\Tc^T$ w.r.t.\ the filtration $(\Fc_t)_{t \in \N}$ such that the following holds:
\begin{enumerate}
\item Coalescence:
On the event $\{ \Tc^T < T \}$, $K_0(t) = \hat{K}_0(t)$ for all $t \in \{\Tc^T, \dots, T\}$;
\item In finite time:
Moreover $\lim_{n \to \infty} \sup_{T \geq n} \Prob{ \Tc^T \geq n } = 0$.
\end{enumerate}
\end{theorem}

This result will be crucial in the proof of the convergence of the ratio $\P ( T^\dagger > T ) / \P ( \hat{T}^\dagger > T )$ as $T \to \infty$ (see Theorem~\ref{T:ratio_general}). Indeed, thanks to the above coupling, as soon as the two processes $(K_0(t))_{t \in [0,T]}$ and $(\hat{K}_0(t))_{t \in [0,T]}$ have coalesced, surviving up to time $T$ will have the same cost. The cost will be different only before the coalescence time that we know being finite almost surely. Making this argument rigorous requires some extra work (especially on the event that the coalescence time is large), but most of the difficulty lies in Theorem \ref{T:coupling_rough} above. See Section \ref{SS:ratio} for details.

To prove Theorem \ref{T:coupling_rough}, we will develop a general abstract result that produces a coupling between two trajectories of the same Markov chain with distinct starting points. In this coupling the two trajectories coalesce in finite time. See Section \ref{SS:coupling_Markov} and Theorem \ref{Thm abstract coupling}.

As already mentioned, $(K_0(t))_{t \geq 1}$ is not Markovian. In Section \ref{Sucsec Markov chain clusters}, we introduce an enlarged process of clusters (actually a collection of clusters) that is Markovian. In Section \ref{Subsec couple cluster}, we show that we can apply Theorem \ref{Thm abstract coupling} to the framework described in Section \ref{Sucsec Markov chain clusters}, concluding the proof of Theorem \ref{T:coupling_rough}.

We will start in the next section by recalling some facts about the coupling of two random variables.

\subsection{Total variation and optimal coupling}
\label{Subsec TV}

Let $(\espace,d)$ be a Polish space, that is to say a separable and complete metric space, endowed with its Borel sigma-algebra $\mathcal{B}(\espace)$ generated by the open subsets.
Given two probability measures $\mu_{1}$ and $\mu_{2}$ on
$(\espace,\mathcal{B}(\espace))$,
the total variation distance $d_{\rm TV}(\mu_{1},\mu_{2})$
is given by
\begin{displaymath}
d_{\rm TV}(\mu_{1},\mu_{2}) = 
2\sup\{\vert\mu_{2}(A)-\mu_{1}(A)\vert ~: A\in \mathcal{B}(\espace)\} \leq 2.
\end{displaymath}
This is also the total mass of total variation measure corresponding to the signed measure $\mu_{2} - \mu_{1}$.
According to Lebesgue's decomposition theorem,
there are measures $\nu$, $\nu_{1}$, $\nu_{2}$,
two by two mutually singular, and non-negative measurable functions $f_{1}$,
$f_{2}$, such that
$
\mu_{1} = f_{1}\nu + \nu_{1}
$
and
$
\mu_{2} = f_{2}\nu + \nu_{2}.
$
One can also impose that $f_{1}$ and $f_{2}$ are $0$
$\nu_{1} + \nu_{2}$ almost everywhere,
and this is the convention we will use.
Then
\begin{equation}\label{E:defq}
d_{\rm TV}(\mu_{1},\mu_{2}) = 
\int_{E}\vert f_{2} - f_{1} \vert d\nu
+ \nu_{1}(\espace) + \nu_{2}(\espace)
\quad \text{and} \quad
1 - \dfrac{1}{2} d_{\rm TV}(\mu_{1},\mu_{2})
=
\int_{E} f_{1}\wedge f_{2}\,d\nu =: q.
\end{equation}

We are interested in the couplings $(X_{1},X_{2})$,
where the marginal distribution of $X_{i}$ is $\mu_{i}$,
$i=1,2$,
so as to maximize $\mathbb{P}(X_{1}=X_{2})$.
This maximal probability is actually given by $q$ defined in \eqref{E:defq} above.
An optimal coupling can be constructed as follows.
Let $Y$, $Y_{i}, i=1,2,$ be three independent r.v.s taking values in $\espace$ distributed respectively according to
\begin{displaymath}
\dfrac{1}{q} f_{1}\wedge f_{2}\,d\nu
\quad \text{and} \quad
\dfrac{1}{1-q}\big(
(f_{i}-f_{1}\wedge f_{2})d\nu
+ d\nu_{i}
\big), \quad i=1,2.
\end{displaymath}
Moreover, let $Z$ be a Bernouilli r.v. independent from
$(Y,Y_{1},Y_{2})$, with $\mathbb{P}(Z=1) = q$.
On the event $\{Z=1\}$, we set
$X_{1} = X_{2} =Y$.
On the event $\{Z=0\}$, we set
$X_{1} = Y_{1}$ and $X_{2} =Y_{2}$.
It is easy to check that each $X_{i}$, $i=1,2$,
has for distribution $\mu_{i}$.
Note that in this construction,
\begin{displaymath}
\mathbb{P}(X_{2}=X_{1}\vert X_{1}) = 
(f_{2}(X_{1})/f_{1}(X_{1}))\wedge 1.
\end{displaymath}
Moreover, conditionally on the value of $X_{1}$ and the event
$\{X_{2}\neq X_{1}\}$,
the conditional distribution of $X_{2}$ is that of $Y_{2}$.

\subsection{Coupling Markov processes conditioned to survive}\label{SS:coupling_Markov}

Let $(\espace,d)$ be a Polish space, endowed with its Borel sigma-algebra $\Bc(\espace)$ generated by the open subsets.
Let $(P_{n})_{n\geq 1}$ be a time-homogeneous Markov semi-group on $\espace$.
That is to say, for every $n\geq 1$ and $x\in \espace$,
$P_{n}(x,dy)$ is a probability measure on $\espace$ and the Chapman--Kolmogorov equation holds: for every $x\in\espace$ and $n_{1}, n_{2} \geq 1$, 
\begin{displaymath}
\int_{y\in\espace}P_{n_{1}}(x,dy)P_{n_{2}}(y,dz)
= P_{n_{1}+n_{2}}(x, dz).
\end{displaymath}
For $x\in\espace$, let $(\xi^{x}(n))_{n\geq 0}$ denote a Markov chain
starting from $x$, with transition semi-group $(P_{n})_{n\geq 1}$.

Let $\dagger$ be a measurable subset of $\espace$,
with $\dagger\neq \espace$.
We will denote $\espace^{\ast}=\espace\setminus\dagger$.
For $x\in\espace$, let
\begin{displaymath}
T_{\dagger}^{x} = 
\inf\{
n\geq 0:
\xi^{x}(n)\in \dagger
\}
\end{displaymath}
be the first hitting time of $\dagger$.
We assume that for every $x\in \espace^{\ast}$ and
$T\in\N\setminus\{ 0\}$,
$\mathbb{P}(T_{\dagger}^{x}>T)>0$.
We will denote by $(\xi^{x,T}(n))_{n\geq 0}$ the trajectory
$(\xi^{x}(n))_{n\geq 0}$ conditioned on the event
$\{T_{\dagger}^{x}>T\}$.
Let $P^{T}_{n}(x,dy)$ denote the law of
$\xi^{x,T}(n)$.
We have a time-inhomogeneous Markov property:
for every $x\in \espace^{\ast}$,
$n_{1},n_{2}\geq 1$ and $T\geq n_{1} + n_{2}$,
\begin{displaymath}
P^{T}_{n_{1} + n_{2}}(x,dz)
=\int_{y\in \espace^{\ast}}
P^{T}_{n_{1}}(x,dy)P^{T-n_{1}}_{n_{2}}(y,dz).
\end{displaymath}

Given $x,y\in\espace^{\ast}$, and $T\in\N\setminus\{ 0\}$,
we will denote by $(\xi^{x,y,T}(n))_{0\leq n\leq T}$
the process $(\xi^{x,T}(n))_{0\leq n\leq T}$
conditioned on $\xi^{x,T}(T)=y$ 
(Markovian bridge conditioned on not hitting $\dagger$).
This conditional distribution is defined for
$P^{T}_{T}(x,dy)$ almost every $y$.
For this, we refer to the existence of regular conditional distributions on Polish spaces \cite[Section 5.3.C, Theorem 3.19]{KaratzasShreve2010BMStochCalc}.
Further, for every $T'\geq T$,
conditionally on $\xi^{x,T'}(T)=y$,
the two processes  $(\xi^{x,T'}(n))_{0\leq n\leq T}$
and $(\xi^{x,T'}(T+n))_{n\geq 0}$
are independent, and distributed respectively as
$(\xi^{x,y,T}(n))_{0\leq n\leq T}$ and
$(\xi^{y,T'-T}(n))_{n\geq 0}$.

We are interested in coupling the conditioned processes
$(\xi^{x,T}(n))_{0\leq n\leq T}$ and
$(\xi^{y,T}(n))_{0\leq n\leq T}$,
for two distinct starting points $x\neq y\in\espace^{\ast}$,
on the same probability space,
such that the two processes coalesce with high probability for
$T$ large enough, and with a uniform control in $T$ over the coalescence time.
In our approach, we will need the following data:
\begin{itemize}
\item For every $n\in \N\setminus\{ 0\}$, a non-empty measurable subset 
$\fav_{n}$ of $\espace^{\ast}$.
This will be the subset of favourable configurations (favourable to attempt a
coalescence).
\item A measurable function
$\chrT: \espace^{\ast}\mapsto \N\setminus\{ 0\}$.
The value $\chrT(x)$ is a sort of characteristic time
that is roughly proportional to the time that the process $\xi^{x,T}$ needs in order to reach a favourable configuration.
We will also use the following notation:
$\chrT(x,y) := \chrT(x)\vee\chrT(y)$.
\end{itemize}

\begin{theorem}
\label{Thm abstract coupling}
With the notations above, we assume that
there are integer constants $c_{2}>c_{1}\geq 2$, such that the following holds.
\begin{enumerate}
\item\label{cond1} Control of coalescence starting from favourable configurations.
We have that
\begin{equation}
\label{Eq q}
\inf_{\substack{n\in\N\setminus\{ 0\}
\\ x',y'\in\fav_{n}}}
\inf_{T\geq c_{2}n}
\Big(
1- \frac{1}{2}
d_{\rm TV}
(P^{T-n}_{(c_{1}-1)n}(x',dz),P^{T-n}_{(c_{1}-1)n}(y',dz))
\Big)
>0.
\end{equation}
\item\label{cond2} Reaching the favourable set $\fav_n$.
We have that
\begin{equation}
\label{Eq p}
\inf_{x\in\espace^{\ast}}
\inf_{n\geq \chrT(x)}
\inf_{T\geq c_{2}n}
\mathbb{P}\big(
\xi^{x,T}(n)\in\fav_{n}
\big) >0.
\end{equation}
\item\label{cond3} Control on the characteristic time.
We have that
\begin{equation}
\label{Eq control bfT}
\lim_{u\to +\infty}
\sup_{x\in\espace^{\ast}}
\sup_{n\geq \chrT(x)}
\sup_{T\geq c_{2}n}
\mathbb{P}(\chrT(\xi^{x,T}(c_{1}n))\geq u n) = 0.
\end{equation}
\end{enumerate}
Then, for every $x,y\in\espace^{\ast}$ and 
every $T\geq c_{2}\chrT(x,y)$,
one can couple on the same probability space the stochastic processes
$(\xi^{x,T}(n))_{0\leq n\leq T}$ and 
$(\xi^{y,T}(n))_{0\leq n\leq T}$, such
that there exists a stopping time 
$\coT^{T}\in \{0,1,\dots,T\}$
with respect to the joint process 
$(\xi^{x,T}(n),\xi^{y,T}(n))_{0\leq n\leq T}$
such that the following holds.
\begin{enumerate}
\item On the event $\{\coT^{T}<T\}$,
for every $n\in \{\coT^{T},\dots,T\}$,
$\xi^{x,T}(n)=\xi^{y,T}(n)$ a.s.
\item On the event $\{\coT^{T}<T\}$,
conditionally on 
$(\coT^{T},(\xi^{x,T}(n),\xi^{y,T}(n))_{0\leq n\leq \coT^{T}})$,
the processes 
\\$(\xi^{x,T}(\coT^{T}+n))_{0\leq n\leq T-\coT^{T}}$ and $(\xi^{y,T}(\coT^{T}+n))_{0\leq n\leq T-\coT^{T}}$ coincide a.s and are distributed as
$(\xi^{z,T'}(n))_{0\leq n\leq T'}$,
with $z=\xi^{x,T}(\coT^{T})$
and $T'= T-\coT^{T}$.
\item We have that
\begin{equation}
\label{Eq unif control coupling time}
\lim_{n\to +\infty}
\sup_{T>n}
\mathbb{P}(\coT^{T}>n)=0.
\end{equation}
\end{enumerate}
In particular,
\begin{equation}
\label{Eq TV large t}
\lim_{n\to +\infty}
\sup_{T\geq n}
d_{\rm TV}(P^{T}_{n}(x, dz),P^{T}_{n}(y, dz)) = 0.
\end{equation}
\end{theorem}

\begin{proof}
Denote by $q$ the infimum \eqref{Eq q}, and $p$ the infimum \eqref{Eq p}.
Fix $x,y\in \espace^{\ast}$ and an integer
$T\geq c_{2}\chrT(x,y)$.
Let $(\xi^{x,T}(n))_{0\leq n\leq T}$ be a conditioned Markov trajectory
starting from $x$.
We will construct by induction a sequence of processes
$(\xi^{y,T}_{i}(n))_{0\leq n\leq T, i\geq 1}$,
two random sequences of points $(x_{i})_{i\geq 1}$ and $(y_{i})_{i\geq 1}$
in $\espace^{\ast}$, and a non-increasing sequence of random non-negative integer times $(T_{i})_{i\geq 1}$. We will also denote by $\tau_{i}=\chrT(x_{i},y_{i}), i \geq 1$.
We initiate these sequences by setting $x_{1}=x$, $y_{1}=y$ and $T_{1}=T$.
If $y = x$,
we set $x_{i}=x$, $y_{i} = y$, $T_{i}=T$ and
$(\xi^{y,T}_{i}(n))_{0\leq n\leq T}
=(\xi^{x,T}(n))_{0\leq n\leq T}$
for every $i\geq 1$.
If $y \neq x$, we take $(\xi^{y,T}_{1}(n))_{0\leq n\leq T}$
distributed as $(\xi^{y,T}(n))_{0\leq n\leq T}$
and independent from $(\xi^{x,T}(n))_{0\leq n\leq T}$.
We will explain in the next few paragraphs how our inductive definition works.

If $x_{i}\neq y_{i}$ and $c_{2}\tau_{i}\leq T_{i}$,
we will set $T_{i+1} = T_{i}-c_{1}\tau_{i}$.
Whenever $x_{i}=y_{i}$ or $c_{2}\tau_{i}> T_{i}$,
we will set $x_{j}=x_{i}$, $y_{j}=y_{i}$, $T_{j}=T_{i}$
and $(\xi^{y,T}_{j}(n))_{0\leq n\leq T}=
(\xi^{y,T}_{i}(n))_{0\leq n\leq T}$
for every $j\geq i$.
We will set
\begin{displaymath}
i^{\ast} = \inf\{i\geq 1  :x_{i}=y_{i} 
\text{ or } c_{2}\tau_{i}> T_{i}\}.
\end{displaymath}
We will see that $i^{\ast} <+\infty$ a.s.

We now consider the case where
$y\neq x$ and we explain our definition of $x_2$, $y_2$, $T_2$ and $(\xi^{y,T}_{2}(n))_{0\leq n\leq T}$.
We set $T_{2} = T - c_{1}\tau_{1}$.
Denote by $F_{1}$ the event
$\{\xi^{x,T}(\tau_{1})\in\fav_{\tau_{1}}
\text{ and } \xi^{y,T}_{1}(\tau_{1})\in\fav_{\tau_{1}}\}$,
and by $F_{1}^{\rm c}$ its complementary.
We have that
$
\mathbb{P}(F_{1})\geq p^{2}>0.
$
On the event $F_{1}^{\rm c}$,
then we set $\xi^{y,T}_{2}(n) = \xi^{y,T}_{1}(n)$ for every $n\in\{0,1,\dots,T\}$.
We set $x_{2}=\xi^{x,T}(c_{1}\tau_{1})$
and $y_{2}=\xi^{y,T}_{1}(c_{1}\tau_{1})$.
On the event $F_{1}$,
we set $x_{2}=\xi^{x,T}(c_{1}\tau_{1})$.
Let $x_{1}'= \xi^{x,T}(\tau_{1})$
and $y_{1}' = \xi^{y,T}_{1}(\tau_{1})$.
We consider the Lebesgue decomposition of the probability measure
$P^{T-\tau_{1}}_{(c_{1}-1)\tau_{1}}(y_{1}',dz)$ with respect to
$P^{T-\tau_{1}}_{(c_{1}-1)\tau_{1}}(x_{1}',dz)$:
\begin{displaymath}
P^{T-\tau_{1}}_{(c_{1}-1)\tau_{1}}(y_{1}',dz)
= f_{1}(z)P^{T-\tau_{1}}_{(c_{1}-1)\tau_{1}}(x_{1}',dz)
+\nu_{1}(dz),
\end{displaymath}
where the measure $\nu_{1}$ is singular with respect to
$P^{T-\tau_{1}}_{(c_{1}-1)\tau_{1}}(x_{1}',dz)$.
Let $Z_{1}$ be a Bernoulli random variable,
conditionally independent from 
$(\xi^{x,T}(n),\xi^{y_{1},T}(t))_{0\leq n\leq T}$
given $(x_{1}',y_{1}',x_{2})$, with
\begin{displaymath}
\mathbb{P}(Z_{1}=1\vert x_{1}',y_{1}',x_{2})
= f_{1}(x_{2})\wedge 1.
\end{displaymath}
Let $V_{1}$ be a random variable on $\espace^{\ast}$,
conditionally independent from 
$(Z_{1},(\xi^{x,T}(n),\xi^{y_{1},T}(n))_{0\leq n\leq T})$
given $(x_{1}',y_{1}',x_{2})$, with
the conditional distribution given $(x_{1}',y_{1}',x_{2})$
equal to
\begin{displaymath}
\dfrac{1}{1-q_{1}}
\big((f_{1}(z)-f_{1}(z)\wedge 1)
P^{T-\tau_{1}}_{(c_{1}-1)\tau_{1}}(x_{1}',dz)
+\nu_{1}(dz)
\big),
\end{displaymath}
where $1-q_{1}$ is the normalization constant.
On the event $F_{1}\cap\{Z_{1}=1\}$,
we further set $y_{2}=x_{2}$.
On the event $F_{1}\cap\{Z_{1}=0\}$, we set $y_{2}=V_{1}$.
In this way,
\begin{align*}
& \mathbb{P}\big(y_{2}=x_{2}\big\vert
(\xi^{x,T}(n),\xi^{y,T}_{1}(n))_{0\leq n\leq \tau_{1}},
F_{1}\big)
\\
& =
1-\dfrac{1}{2} d_{\rm TV}(P^{T-\tau_{1}}_{(c_{1}-1)\tau_{1}}(x_{1}',dz),
P^{T-\tau_{1}}_{(c_{1}-1)\tau_{1}}(y_{1}',dz))
= q_{1}\geq q>0,
\end{align*}
and the conditional distribution of $y_{2}$ given
$(\xi^{x,T}(n),\xi^{y,T}_{1}(n))_{0\leq n\leq \tau_{1}}$
on the event $F_{1}$ is
$P^{T-\tau_{1}}_{(c_{1}-1)\tau_{1}}(y_{1}',dz)$;
see Section \ref{Subsec TV}.
In this way,
the distribution of
$((\xi^{y,T}_{1}(n))_{0\leq n\leq \tau_{1}},y_{2})$
conditionally on the event $F_{1}$
coincides with the conditional distribution of
$((\xi^{y,T}_{1}(n))_{0\leq n\leq \tau_{1}},\xi^{y,T}_{1}(c_{1}\tau_{1}))$
given the same event.
Consider a conditioned Markov process 
$(\xi^{y_{2}, T-c_{1}\tau_{1}}(n))_{0\leq n\leq T-c_{1}\tau_{1}}$
conditionally independent from 
$((\xi^{x,T}(n),\xi^{y,T}_{1}(n))_{0\leq n\leq T},Z_{1},V_{1})$
given $y_{2}$.
Consider also a Markovian bridge conditioned on not hitting $\dagger$,
$(\xi^{y_{1}',y_{2},(c_{1}-1)\tau_{1}}(n))_{0\leq n\leq (c_{1}-1)\tau_{1}}$,
with the bridge being conditionally independent from
\[
((\xi^{x,T}(t),\xi^{y,T}_{1}(n))_{0\leq n\leq T},Z_{1},V_{1},
(\xi^{y_{2}, T-c_{1}\tau_{1}}(n))_{0\leq n\leq T-c_{1}\tau_{1}})\]
given $(y_{1}',y_{2})$.
On the event $F_{1}\cap\{Z_{1}=1\}$, we define
$(\xi^{y,T}_{2}(n))_{0\leq n\leq T}$ as follows:
\begin{displaymath}
\xi^{y,T}_{2}(n) =
\left\lbrace
\begin{array}{ll}
\xi^{y,T}_{1}(n) & \text{for } n\in\{0,\dots,\tau_{1}\}, \\ 
\xi^{y_{1}',y_{2},(c_{1}-1)\tau_{1}}(n-\tau_{1}) & \text{for } n\in 
\{\tau_{1}+1,\dots,c_{1}\tau_{1}-1\}, \\ 
\xi^{x,T}(n) & \text{for } n\in\{c_{1}\tau_{1},\dots, T\}.
\end{array} 
\right.
\end{displaymath}
On the event $F_{1}\cap\{Z_{1}=0\}$, we define
$(\xi^{y,T}_{2}(n))_{0\leq n\leq T}$ as follows:
\begin{displaymath}
\xi^{y,T}_{2}(n) =
\left\lbrace
\begin{array}{ll}
\xi^{y,T}_{1}(n) & \text{for } n\in\{0,\dots,\tau_{1}\}, \\ 
\xi^{y_{1}',y_{2},(c_{1}-1)\tau_{1}}(n-\tau_{1}) & \text{for } n\in 
\{\tau_{1}+1,\dots,c_{1}\tau_{1}-1\}, \\ 
\xi^{y_{2}, T-c_{1}\tau_{1}}(n-c_{1}\tau_{1}) & \text{for } 
n\in\{c_{1}\tau_{1},\dots, T\}.
\end{array} 
\right.
\end{displaymath}

The process $(\xi^{y,T}_{2}(n))_{0\leq n\leq T}$ has been constructed in such a way that it has the same distribution as $(\xi^{y,T}_{1}(n))_{0\leq n\leq T}$.
But now $(\xi^{y,T}_{2}(n))_{0\leq n\leq T}$ is correlated to
$(\xi^{x,T}(n))_{0\leq n\leq T}$.
What we have gained is that the two process $(\xi^{x,T}(n))_{0\leq n\leq T}$ and 
$(\xi^{y,T}_{2}(n))_{0\leq n\leq T}$ have a positive probability to coalesce.
More precisely,
\begin{displaymath}
\mathbb{P}
(\forall n\in \{c_{1}\tau_{1},\dots, T\}, \xi^{y,T}_{2}(n)=\xi^{x,T}(n))
=
\mathbb{P}(F_{1}\cap\{Z_{1}=1\})
\geq p^{2} q >0.
\end{displaymath}

Consider now the r.v. $\tau_{2}=\chrT(x_{2},y_{2})$.
The condition \eqref{Eq control bfT} provides a control on the tail of
$\tau_{2}/\tau_{1}$: for all $u>0$, $\mathbb{P}(\tau_{2}\geq u\tau_{1})$ is at most
\begin{align*}
\mathbb{P}(\chrT(x_{2})\geq u\tau_{1}) + \mathbb{P}(\chrT(y_{2})\geq u\tau_{1})
& =
\mathbb{P}(\chrT(\xi^{x,T}(c_{1}\tau_{1}))\geq u\tau_{1}) +
\mathbb{P}(\chrT(\xi^{y,T}_{2}(c_{1}\tau_{1}))\geq u\tau_{1}) \\
& \leq
2\sup_{z\in\espace^{\ast}}
\sup_{n\geq \chrT(z)}
\sup_{T'\geq c_{2}n}
\mathbb{P}(\chrT(\xi^{z,T'}(c_{1}n))\geq u n).
\end{align*}
Let 
\begin{displaymath}
\psi : u \in [0,\infty) \mapsto 1\wedge 2\sup_{z\in\espace^{\ast}}
\sup_{n\geq \chrT(z)}
\sup_{T'\geq c_{2}n}
\mathbb{P}(\chrT(\xi^{z,T'}(c_{1}n))\geq u n).
\end{displaymath}
The function $\psi$ is non-increasing, but not necessarily left-continuous.
We will denote by $\psi(u^{-})$ the limit to the left of $\psi(u)$.
Let $U_{2}$ be a r.v. in $[0+\infty)$ with distribution
characterised by
\begin{equation}
\label{E:defU2}
\mathbb{P}(U_{2}\geq u) = \psi(u^{-}).
\end{equation}
The condition \eqref{Eq control bfT} ensures that $U_{2}<+\infty$.
The r.v. $\tau_{2}$ is stochastically dominated by
$U_{2} \tau_{1}$.
Note however that conditionally on $y_{2}\neq x_{2}$,
the r.v. $\tau_{2}$ is not necessarily stochastically dominated by
$U_{2} \tau_{1}$.

On the event $\{ x_{2}=y_{2}\}$, we have $i^{\ast} =2$,
and for every $i\geq 3$,
$x_{i} = x_{2} = y_{i} = y_{2}$, $T_{i}=T_{2}$,
and 
$(\xi^{y,T}_{i}(n))_{0\leq n\leq T}=(\xi^{y,T}_{2}(n))_{0\leq n\leq T}$.
On the event $\{ x_{2}\neq y_{2}\}$,
conditionally on 
\[((\xi^{x,T}(n),\xi^{y,T}_{2}(n))_{0\leq n\leq T - T_{2}},
(\xi^{y,T}_{1}(n))_{0\leq n\leq T}),\]
the processes
$(\xi^{x,T}(T - T_{2}+ n))_{0\leq n\leq T_{2}}$
and $(\xi^{y,T}_{2}(T - T_{2}+ n))_{0\leq n\leq T_{2}}$
are independent,
distributed as
$(\xi^{x_{2},T_{2}}(n))_{0\leq n\leq T_{2}}$ and 
$(\xi^{y_{2},T_{2}}(n))_{0\leq n\leq T_{2}}$ respectively.
Thus, we can iterate by applying the previous construction
to $(\xi^{x,T}(T - T_{2}+ n),\xi^{y,T}_{2}(T - T_{2}+ n))_{0\leq n\leq T_{2}}$.

\medskip

By iterating the construction further, we get the sequences
$(x_{i})_{i\geq 1}$, $(y_{i})_{i\geq 1}$,
$(T_{i})_{i\geq 1}$ and $(\xi^{y,T}_{i}(n))_{0\leq n\leq T, i\geq 1}$
satisfying the following properties:
\begin{itemize}
\item For every $i\geq 1$,
the process $(\xi^{y,T}_{i}(n))_{0\leq n\leq T}$
is distributed as the conditioned Markov process 
$(\xi^{y,T}(n))_{0\leq n\leq T}$.
\item For every $i\geq 1$,
$x_{i}=\xi^{x,T}(T-T_{i})$ and
$y_{i}=\xi^{y,T}_{i}(T-T_{i})$.
\item For every $i\geq 2$,
a.s. for every $n\in\{0,\dots, T-T_{i-1}\}$, 
$\xi^{y,T}_{i}(n)=\xi^{y,T}_{i-1}(n)$.
\item For every $i\geq 1$,
on the event $\{ x_{i}=y_{i}\}$,
we have that a.s. for every $n\in\{T-T_{i},\dots,T\}$,
$\xi^{y,T}_{i}(n)=\xi^{x,T}(n)$,
and for every $j\geq i$, $T_{j}=T_{i}$,
$x_{j}=y_{j}=x_{i}=y_{i}$,
and $(\xi^{y,T}_{j}(n))_{0\leq n\leq T}=(\xi^{y,T}_{i}(n))_{0\leq n\leq T}$.
\item For every $i\geq 1$,
on the event $\{c_{2}\tau_{i}>T_{i}\}$
(where $\tau_{i}=\chrT(x_{i},y_{i})$),
we have for every $j\geq i$, $T_{j}=T_{i}$,
$x_{j}=x_{i}$, $y_{j}=y_{i}$,
and $(\xi^{y,T}_{j}(n))_{0\leq n\leq T}=(\xi^{y,T}_{i}(n))_{0\leq n\leq T}$.
\item For every $i\geq 1$,
on the event $\{ x_{i}=y_{i} \text{ and } c_{2}\tau_{i-1}\leq T_{i-1}\}$,
conditionally on 
\\$((\xi^{x,T}(n),\xi^{y,T}_{i}(n))_{0\leq n\leq T-T_{i}},
(\xi^{y,T}_{j}(n))_{0\leq n\leq T, 1\leq j\leq i-1})$,
the process
$(\xi^{x,T}(T - T_{i}+ n))_{0\leq n\leq T_{i}}$
is distributed as the conditioned Markov processes
$(\xi^{x_{i},T_{i}}(n))_{0\leq n\leq T_{i}}$.
\item For every $i\geq 2$,
on the event $\{ x_{i}\neq y_{i} \text{ and } c_{2}\tau_{i-1}\leq T_{i-1}\}$,
conditionally on 
\\$((\xi^{x,T}(n),\xi^{y,T}_{i}(n))_{0\leq n\leq T-T_{i}},
(\xi^{y,T}_{j}(n))_{0\leq n\leq T, 1\leq j\leq i-1})$,
the processes
$(\xi^{x,T}(T - T_{i}+ n))_{0\leq n\leq T_{i}}$
and $(\xi^{y,T}_{i}(T - T_{i}+ n))_{0\leq n\leq T_{i}}$
are independent,
distributed as the conditioned Markov processes
$(\xi^{x_{i},T_{i}}(n))_{0\leq n\leq T_{i}}$ and 
$(\xi^{y_{i},T_{i}}(n))_{0\leq n\leq T_{i}}$ respectively.
\item For every $i\geq 1$,
on the event $\{ x_{i}\neq y_{i} \text{ and } c_{2}\tau_{i}\leq T_{i}\}$
(i.e. $\{i^{\ast}>i\}$),
$T_{i+1} = T_{i}-c_{1}\tau_{i}$.
\item For every $i\geq 1$,
\begin{displaymath}
\mathbb{P}\big(x_{i+1}=y_{i+1}\big\vert i^{\ast}>i,
(\xi^{x,T}(n),\xi^{y,T}_{i}(n))_{0\leq n\leq T-T_{i}},
(\xi^{y,T}_{j}(n))_{0\leq n\leq T, 1\leq j\leq i-1}\big)
\geq p^{2}q.
\end{displaymath}
\end{itemize}

From the above properties one can derive the following.
First of all, for every $i\geq 2$,
\begin{displaymath}
\mathbb{P}(i^{\ast}> i)\leq (1-p^{2}q)^{i-1},
\end{displaymath}
and in particular, $i^{\ast}<+\infty$ a.s.
Further, for every $i\geq 2$,
the r.v. $T-T_{i}$ is stochastically dominated by
\begin{displaymath}
c_{1}(1+U_{2}+U_{2}U_{3}+ \dots +
U_{2}U_{3}\dots U_{i-1})\tau_{1},
\end{displaymath}
where the r.v.s $U_{j}$, $j\geq 3$, are i.i.d. copies of the r.v. $U_{2}$ \eqref{E:defU2}.
We will denote
\begin{displaymath}
\widehat{U}_{i} = 1+U_{2}+U_{2}U_{3}+ \dots +
U_{2}U_{3}\dots U_{i},
\end{displaymath}
with $\widehat{U}_{1}=1$,
so that $T-T_{i}$ is stochastically dominated by
$c_{1}\widehat{U}_{i-1}\tau_{1}$.
Note that this stochastic domination is uniform in
$T\geq c_{2}\tau_{1}$.

Then, for every $i\geq 2$,
\begin{displaymath}
\mathbb{P}(i^{\ast}>i, y_{i}\neq x_{i})
\leq \mathbb{P}(i^{\ast}>i)\leq (1-p^{2}q)^{i-1}
\end{displaymath}
and
\begin{align*}
\mathbb{P}(i^{\ast}\leq i, y_{i}\neq x_{i})
& \leq 
\mathbb{P}(c_{2}\tau_{i}>T_{i}) =
\mathbb{P}(T-T_{i} + c_{2}\tau_{i}>T)
\leq 
\mathbb{P}((c_{1}\widehat{U}_{i-1}+ c_{2}(\widehat{U}_{i}-\widehat{U}_{i-1}))\tau_{1}>T).
\end{align*}
In this way,
\begin{displaymath}
\mathbb{P}(y_{i}\neq x_{i}) \leq 
(1-p^{2}q)^{i-1} + 
\mathbb{P}((c_{1}\widehat{U}_{i-1}+ c_{2}(\widehat{U}_{i}-\widehat{U}_{i-1}))\tau_{1}>T).
\end{displaymath}

Let $I(T)$ be a \textbf{deterministic} integer-valued function such that
$\lim_{T\to +\infty}I(T)=+\infty$ and
\begin{displaymath}
\lim_{T\to +\infty}
\mathbb{P}((c_{1}\widehat{U}_{I(T)-1}+ c_{2}(\widehat{U}_{I(T)}-\widehat{U}_{I(T)-1}))\tau_{1}>T) = 0,
\end{displaymath}
so that
\begin{displaymath}
\lim_{T\to +\infty}\mathbb{P}(y_{I(T)}=x_{I(T)}) = 1.
\end{displaymath}
On the event $\{y_{I(T)}=x_{I(T)}\}$
we necessarily have $i^{\ast}\leq I(T)$, and we will set
$\coT^{T} = T-T_{i^{\ast}}$.
On the event $\{y_{I(T)}\neq x_{I(T)}\}$, we will set $\coT^{T} = T$.
Then we take the coupling
$(\xi^{x,T}(n),\xi^{y,T}_{I(T)}(n))_{0\leq n\leq T}$.
Since $I(T)$ is deterministic, the process
$(\xi^{y,T}_{I(T)}(n))_{0\leq n\leq T}$ is distributed as the conditioned
Markov process $(\xi^{y,T}(n))_{0\leq n\leq T}$.
Moreover, on the event $\{\coT^{T} < T\}$,
the processes $\xi^{x,T}(n)$ and $\xi^{y,T}_{I(T)}(n)$
coincide on $\{\coT^{T},\dots,T\}$, as desired.
Finally, to bound the tail of the distribution $\coT^{T}$,
we use the following.
\begin{multline*}
\mathbb{P}(\coT^{T}>n)
\leq 
\indic{T> n}(\mathbb{P}(\coT^{T}=T)
+\mathbb{P}(i^{\ast}>i)
+\mathbb{P}(T-T_{i}>n))
\\\leq \indic{T> n}(\mathbb{P}(y_{I(T)}\neq x_{I(T)})
+(1-p^{2}q)^{i-1}
+\mathbb{P}(c_{1}\widehat{U}_{i-1}\tau_{1}>n)),
\end{multline*}
where $i$ is arbitrary.
By taking $i=I(n)$, we get \eqref{Eq unif control coupling time}.

For \eqref{Eq TV large t}, we use that
$
d_{\rm TV}(P^{T}_{n}(x, dz),P^{T}_{n}(y, dz))
\leq 
2(1-\mathbb{P}(\coT^{T}<T, \coT^{T}\leq n)).
$
This concludes the proof.
\end{proof}

\subsection{Markov chain of clusters in a Brownian loop soup}
\label{Sucsec Markov chain clusters}

We now describe precisely the Markov chain we will consider in our Brownian loop soup setting. We start by introducing the relevant spaces. Recall that we denote by $\cyl_0$ the open cylinder $(0,\infty) \times \mathbb{S}^1$.

\paragraph*{The space $(\Komp,\d_\Komp)$.}
Denote
\begin{displaymath}
\Komp = 
\{
K \text{ compact subset of } \overline{\cyl}_{0} :
\partial\cyl_{0}\subset K, K \text{ connected}
\}.
\end{displaymath}
We endow $\Komp$ with the Hausdorff distance $\d_{\Komp}$ on compact subsets of
$\overline{\cyl}_{0}$.
The metric space $(\Komp, \d_{\Komp})$ is Polish, i.e. separable and complete.
Indeed, the space of compact subsets of $\overline{\cyl}_{0}$ endowed with
$\d_{\Komp}$ is Polish
(see e.g. \cite{molchanov2005theory}, Theorems C.2 and C.8),
and $\Komp$ is a closed subset of this space, thus also Polish.
Note that in the definition of $\Komp$ it is important to take
$K$ connected and not connected by arcs,
otherwise we do not get a complete space.
Let us also explain why we want our compacts to contain $\partial\cyl_{0}$.
Essentially, the behaviour of compacts on $\partial\cyl_{0}$
will not be important, and we will not distinguish between two compacts that differ only on $\partial\cyl_{0}$. 
So imposing the compacts to contain $\partial\cyl_{0}$ is a way to choose a representative in each equivalence class.

Given $K\in \Komp$, we will denote
\begin{equation}\label{E:rho_K}
\rho(K) = 
\max\{ t\geq 0 : \exists \omega\in \mathbb{S}^{1},
(t,\omega)\in K\} = \d_{\Komp}(K,\partial\cyl_{0}).
\end{equation}
Given $t\in\R$ and $K$ a compact subset of $\R\times \mathbb{S}^{1}$,
we will denote
\begin{displaymath}
K - t = \{(s,\omega)\in\R\times \mathbb{S}^{1} :  (s+t,\omega)\in K\}.
\end{displaymath}

\paragraph*{The space $(\espace,\d_\espace)$.}
In what follows we will consider sequences (infinite countable families)
of compacts in $\Komp$.
Actually we will consider the family of compacts up to permutation,
except the first one in the family, which will be marked.
Let $\perm_{0}(\N)$ denote the permutations $\sigma$ of 
$\N$ such that $\sigma(0)=0$.
Given $(K_{i})_{i\geq 0}\in\Komp^{\N}$,
we will identify this sequence with all the
$(K_{\sigma(i)})_{i\geq 0}$ for $\sigma\in \perm_{0}(\N)$,
and we will denote by $\sim$ the corresponding equivalence relation.
We will denote by $\Komp^{\N}/\sim$ the quotient space.
By an abuse of notation, we will still denote by
$(K_{i})_{i\geq 0}$ and element of $\Komp^{\N}/\sim$.
Let be
\begin{displaymath}
\espace = \{(K_{i})_{i\geq 0}\in \Komp^{\N}/\sim ~ :
\lim_{i\to +\infty} \rho(K_i)=0\}.
\end{displaymath}
Note that the condition $\lim_{i\to +\infty} \rho(K_i)=0$
is invariant under the permutations $\sigma\in \perm_{0}(\N)$.
We endow $\espace$ with the following distance:
\begin{displaymath}
\d_{\espace}((K_{i})_{i\geq 0},(K_{i}')_{i\geq 0})
=
\d_{\Komp}(K_{0},K_{0}') + 
\inf_{\sigma\in \perm_{0}(\N)}\sup_{i\geq 1}
\d_{\Komp}(K_{i},K_{i}').
\end{displaymath}

\begin{lemma}
\label{Lem E Polish}
The metric space $(\espace,\d_{\espace})$ is Polish,
i.e. separable and complete.
\end{lemma}

\begin{proof}
The fact that $(\espace,\d_{\espace})$ is Polish is a consequence of
$(\Komp, \d_{\Komp})$ being Polish, 
and does not rely on the specific nature of $(\Komp, \d_{\Komp})$.
Let $\Komp^{\#}$ be a countable dense subset of $\Komp$.
Let $\espace^{\#}$ be the subset of $\espace$ made of families
where each element is in $\Komp^{\#}$, and only finitely many elements are different from $\partial\cyl_{0}$.
Then $\espace^{\#}$ is countable and dense in $\espace$ for 
$\d_{\espace}$.
To see that $(\espace,\d_{\espace})$ is complete,
order the elements of $(K_{i})_{i\geq 1}$, for some element $(K_{i})_{i\geq 0}\in \espace$,
in a non-increasing order of $\rho(K_{i})$, and use the fact that
$(\Komp, \d_{\Komp})$ is complete.
\end{proof}

From now on we will fix a value $\theta\in (0,1/2]$ and
consider a Brownian loop soup $\Lc_0^\theta$ of intensity parameter $\theta$ in
the open cylinder $\cyl_{0}$.
For $t>0$, we set
\begin{displaymath}
\cyl_{t} = (t,+\infty)\times\mathbb{S}^{1},
\qquad
\overline{\cyl}_{t} = [t,+\infty)\times\mathbb{S}^{1},
\end{displaymath}
\begin{displaymath}
\Lc^{\theta}_{\cap (0,t]}
= \{\wp\in \Lc^{\theta}_{0} : \Range(\wp)\cap (0,t]\times \mathbb{S}^{1}
\neq \emptyset\},
\qquad
\Lc^{\theta}_{t} = 
\{\wp\in \Lc^{\theta}_{0} : \Range(\wp)\subset \cyl_{t}\}.
\end{displaymath}
The collections of loops $\Lc^{\theta}_{\cap (0,t]}$ and
$\Lc^{\theta}_{t}$ are independent and
$\Lc^{\theta}_{\cap (0,t]}\cup\Lc^{\theta}_{t} = \Lc^{\theta}_{0}$.
Moreover, $\Lc^{\theta}_{t}$ has the same distribution as $\Lc^{\theta}_{0}$
up to a translation by $t$ along the cylinder. The filtration we will consider is then defined by
\begin{equation}
\label{E:filtration}
\Fc_t = \sigma \{ \Lc_{\cap (0,t]} \}, \quad \quad t \geq 1.
\end{equation}

\paragraph*{A Markov chain of clusters}
Given $(K_{i})_{i\geq 0}\in\espace$, we will denote by
$(K_{i})_{i\geq 0}\cup \Lc^{\theta}_{\cap (0,t]}$ (by abuse of notation)
the family formed by the compacts $(K_{i})_{i\geq 0}$ and the loops in 
$\Lc^{\theta}_{\cap (0,t]}$,
where the Brownian loops are identified to their range and considered as compact subsets of $\cyl_{0}$.
Let $\sim_{t}$ be the following equivalence relation on the elements of
$(K_{i})_{i\geq 0}\cup \Lc^{\theta}_{\cap (0,t]}$.
\begin{itemize}[leftmargin=*]
\item Given $\wp, \wp'\in \Lc^{\theta}_{\cap (0,t]}$,
whenever $\Range(\wp)\cap\Range(\wp')\neq \emptyset$,
we have $\wp\sim_{t}\wp'$.
\item Given $\wp\in \Lc^{\theta}_{\cap (0,t]}$ and
$i\geq 0$,
whenever  $\wp\cap K_{i}\neq \emptyset$,
we have $\Range(\wp)\sim_{t}K_{i}$.
\item Given $i,j\geq 1$,
whenever $(K_{i}\cap K_{j})\setminus\partial \cyl_{0} \neq \emptyset$,
we have $K_{i}\sim_{t} K_{j}$.
We would like to emphasize that $K_{i}\cap K_{j}$ always contains 
$\partial \cyl_{0}$, but our condition is that this intersection is strictly larger than $\partial \cyl_{0}$.
\item Further, we define $\sim_{t}$ to be the minimal
equivalence relation satisfying the above three conditions.
That is to say, we identify into the same equivalence class the finite chains
of objects where each couple of consecutive elements satisfy one of the above three rules.
\end{itemize} 
In essence, the equivalence relation $\sim_{t}$ corresponds to clusters of
$(K_{i}\setminus \partial \cyl_{0})_{i\geq 0}\cup \Lc^{\theta}_{\cap (0,t]}$.
Note that a set $K_{i}\setminus \partial \cyl_{0}$ is not necessarily connected,
however, by definition, we put all its connected components into the same cluster.
Given an equivalence class $A$ of $\sim_{t}$,
we will consider the associated cluster $\Cc$ seen as a connected compact subset of
$\overline{\cyl}_{0}$:
\begin{displaymath}
\Cc = 
\overline{\bigcup_{\wp\in A}\Range(\wp)\cup\bigcup_{K_{i}\in A} K_{i}}.
\end{displaymath}
Let $\Cf_{t}$ denote the family of such clusters induced by $\sim_{t}$,
where each cluster is seen as a compact connected subset of $\overline{\cyl}_{0}$.
We will denote by $\Cc_{0,t}\in \Cf_{t}$ the marked cluster containing $K_{0}$.
Next, we will denote by $(K_{i}(t))_{i\geq 0}$ the following family of compacts:
\begin{displaymath}
\big(
((\Cc\cap \overline{\cyl}_{t}) - t)\cup\partial\cyl_{0}
\big)_{\Cc\in \Cf_{t}, \Cc\cap \overline{\cyl}_{t}\neq \emptyset}.
\end{displaymath}
That is to say, among the compacts of $\Cf_{t}$, we take those that intersect
$\overline{\cyl}_{t}$, 
consider only the parts inside $\overline{\cyl}_{t}$, translate them by $-t$
to make them adjacent to $\partial\cyl_{0}$,
and add the circle $\partial\cyl_{0}$ for the normalization.
Also, $K_{0}(t)$ will be the compact obtained from the marked cluster
$\Cc_{0,t}$ (the one containing $K_{0}$).
We see the family $(K_{i}(t))_{i\geq 0}$ as an element of $\espace$.
By convention, we set $(K_{i}(0))_{i\geq 0}=(K_{i})_{i\geq 0}$.


Now, we consider $t$ to be an integer $n$ and take the process
$(K_{i}(n))_{i\geq 0,n\in\N}$.
We see the latter as a stochastic process in $\espace$,
parametrized by discrete time $n$.
This process is actually time-homogeneous Markov with respect to the filtration $(\Fc_n)_{n \in \N}$ \eqref{E:filtration}.
This is because for each $n\geq 1$,
the loop soup $\Lc^{\theta}_{n}$ is independent from $\Lc^{\theta}_{\cap (0,n]}$,
and $\Lc^{\theta}_{n}$ has the same law as $\Lc^{\theta}_{0}$
up to a longitudinal translation by $n$.

\subsection{Coupling conditioned clusters in a Brownian loop soup}
\label{Subsec couple cluster}

We consider the state space $\espace$ and the Markov chain of Brownian loop soup clusters introduced in Section \ref{Sucsec Markov chain clusters}.
Let $\dagger$ be the subset
\begin{displaymath}
\dagger = \{(K_{i})_{i\geq 0}\in\espace : K_{0}=\partial\cyl_{0}\}.
\end{displaymath}
Note that the subset $\dagger$ is absorbing, i.e. is preserved by the Markov chain.
The subset $\dagger$ corresponds to the extinction of the marked cluster $K_{0}$.
Let $\espace^{\ast} = \espace\setminus\dagger$.
We will denote by $T_{\dagger}$ the first hitting time of $\dagger$.
Given an initial condition 
$(K_{i}(0))_{i\geq 0}\in \espace^{\ast}$,
for every $T\in\N\setminus\{ 0\}$,
$\mathbb{P}(\{T_{\dagger}>T)>0$.
Indeed, it is enough to have a single Brownian loop in $\Lc^{\theta}_{0}$
that connects $K_{0}(0)$ to $\cyl_{T}$.
Given $T\in\N\setminus\{ 0\}$,
we will denote by 
$(K_{i}^{T}(n))_{i\geq 0,0\leq n\leq T}$ the Markov chain of clusters
conditioned on $\{T_{\dagger}>T\}$.

In this section we will prove Theorem \ref{T:coupling_rough} (actually a stronger version; see Theorem \ref{Thm abstract coupling}) by showing that the conditioned Markov chain $(K_{i}^{T}(n))_{i\geq 0,0\leq n\leq T}$
satisfies the assumptions of Theorem \ref{Thm abstract coupling}.
For this we will need characteristic times
$\chrT$ and favourable configurations
$\fav_{n}\subset \espace^{\ast}$ for $n\geq 1$.
We set
\begin{displaymath}
\chrT((K_{i})_{i\geq 0})
=
\sup_{i\geq 0}
\lceil
\rho(K_{i})
\rceil
\in\N,
\end{displaymath}
where $\lceil~\rceil$ is the integer ceiling.
Let $c_{1}\in\N$, $c_{1}\geq 5$. We will tune the value of
$c_{1}$ later; see \eqref{Eq condition c1}.
We also set $c_{2} = c_{1} + 1$ and
\begin{displaymath}
\fav_{n} = 
\big\{
(K_{i})_{i\geq 0}\in \espace^{\ast} : 
\rho(K_{0})\geq 2n
\text{ and }
\forall i\geq 0, 
\rho(K_{i})\leq (c_{1}-2)n
\big\},
\quad n \geq 1.
\end{displaymath}

\begin{lemma}
\label{Lem F n clusters}
There is $p_{0}\in (0,1]$ such that for every $n\in\N\setminus\{ 0\}$,
for every $T\in \N$ with $T\geq n$, 
and for every initial condition
$(K_{i}^{T}(0))_{i\geq 0}\in \espace^{\ast}$,
\begin{displaymath}
\mathbb{P}(\rho(K^{T}_{0}(n))\geq 2n)\geq p_{0}.
\end{displaymath}
\end{lemma}

\begin{proof}
Let $E_{1,n}$ be the event
\begin{displaymath}
E_{1,n} = 
\big\{
\exists\wp\in \Lc^{\theta}_{0}, 
\partial\cyl_{n/2}\stackrel{\wp}{\longleftrightarrow}
\partial\cyl_{3n}
\big\}.
\end{displaymath}
Let $E_{2,n}$ be the event that there is a loop in $\Lc^{\theta}_{0}$,
contained in the cylinder $(n/2,n)\times\mathbb{S}^{1}$, and that surrounds this cylinder, i.e. is non-contractible in the topological sense.
Then
\begin{displaymath}
\mathbb{P}(\rho(K^{T}_{0}(n))\geq 2n)\geq
\mathbb{P}(E_{1,n}\cap E_{2,n}, T_{\dagger}>T)
/\mathbb{P}(T_{\dagger}>T).
\end{displaymath}
The three events $E_{1,n}$, $E_{2,n}$ and $\{T_{\dagger}>T\}$
are increasing for the Poisson point process $\Lc^{\theta}_{0}$.
By the FKG inequality for Poisson point processes,
\begin{displaymath}
\mathbb{P}(\rho(K^{T}_{0}(n))\geq 2n)\geq
\mathbb{P}(E_{1,n}\cap E_{2,n})
\geq\mathbb{P}(E_{1,n})\mathbb{P}(E_{2,n}).
\end{displaymath}
By Lemma \ref{L:cross_loop},
$\mathbb{P}(E_{1,n})\geq p_{1}$
for some $p_{1}\in (0,1)$.
Moreover,
$\mathbb{P}(E_{2,n})\geq 1- p_{2}^{\theta n}$
for some $p_{2}\in (0,1)$ (Lemma \ref{L:surround}).
Thus we get a uniform lower bound for 
$\mathbb{P}(\rho(K^{T}_{0}(n))\geq 2n)$.
\end{proof}

By Lemma \ref{L:cross_loop}, the probability that a single loop crosses $[1,t] \times \mathbb{S}^1$ decays like $t^{-1}$ as $t \to \infty$. In the next lemma, we show that if one conditions on the event that a cluster makes a large crossing, then this probability decays like $t^{-\theta+o(1)}$. This will in particular ensure that loops stay small (although larger) after this conditioning.

\begin{lemma}
\label{Lem exponent conditioned loops}
For every $\varepsilon\in (0,\theta)$,
there is a constant $C_{\varepsilon}>0$,
such that for every initial condition
$(K_{i}(0))_{i\geq 0}\in \espace^{\ast}$
for the Markov chain,
for every $T\in\N\setminus\{ 0\}$,
for every $t_{1}\geq 1$ and every 
$t_{2}\geq 2 t_{1}$,
\begin{equation}
\label{E:L_at}
\mathbb{P}\big(
\exists\wp\in\Lc^{\theta}_{0},
\partial\cyl_{t_{1}}\stackrel{\wp}{\longleftrightarrow}
\partial\cyl_{t_{2}}
\big\vert T_{\dagger}>T
\big)
\leq
C_{\varepsilon}
\Big(
\dfrac{t_{1}}{t_{2}}
\Big)^{\theta-\varepsilon}
.
\end{equation}
\end{lemma}

\begin{proof}
Let $\eps>0$ be as in the statement.
Let $\underline{\tau} t_1$ (resp. $\overline{\tau} t_1$) be the left most (resp. right most) level reached by the loops that cross $[t_1,t_2] \times \mathbb{S}^1$:
\[
\underline{\tau} := \inf \{ u \in (0,1): \exists \wp \in \Lc_0^\theta: \partial \cyl_{u t_1} \overset{\wp}{\longleftrightarrow} \partial \cyl_{t_2} \},
\quad \bar{\tau} := \sup \{ v > t_2/t_1: \exists \wp \in \Lc_0^\theta: \partial \cyl_{t_1} \overset{\wp}{\longleftrightarrow} \cyl_{v t_1} \}.
\]
The levels $\underline{\tau} t_1$ and $\overline{\tau} t_1$ can be reached either by the same loop or by two distinct loops. In both cases, one can use Lemma \ref{L:cross_loop} to show that there exists $C >0$ such that for all $u \in (0,1)$ and $v \geq t_2/t_1$,
\begin{equation}
\label{E:pf_at2}
\Prob{ \underline{\tau} \leq u, \overline{\tau} \geq v } \leq C (u/v).
\end{equation}
With these notations, the left hand side of \eqref{E:L_at} is equal to
\begin{align}
\label{E:pf_at1}
\int_0^1 \int_{t_2/t_1}^\infty \Prob{ \underline{\tau} \in \d u, \bar{\tau} \in \d v }
\frac{\Prob{ T_\dagger > T \vert \underline{\tau} = u, \overline{\tau} = v }}{\Prob{T_\dagger > T}}.
\end{align}
If $vt_1 < T$, the above ratio on the right hand side is at most
\[
\Prob{ T_\dagger > ut_1, \cyl_{vt_1} \overset{\Lc_0^\theta}{\longleftrightarrow} \cyl_T }/\Prob{T_\dagger > T}.
\]
If $vt_1 \geq T$, this ratio can be bounded by
\[
\Prob{ T_\dagger > ut_1 }/\Prob{T_\dagger > T}.
\]
In both cases and by FKG inequality, the ratio in \eqref{E:pf_at1} is bounded from above by the inverse of the product of the probabilities of three events: 1) that there is a cluster joining $\partial \cyl_{ut_1/2}$ to $\partial \cyl_{2vt_1}$, 2) there is a non contractible loop in $[ut_1/2, ut_1] \times \mathbb{S}^1$, 3) there is a non contractible loop in $[vt_1, 2vt_1] \times \mathbb{S}^1$. By Lemma \ref{L:surround} and Theorem \ref{T:large_crossing}, the ratio in \eqref{E:pf_at1} is therefore bounded by $C_\eps (v/u)^{1-\theta+\eps}$. Plugging this estimate in \eqref{E:pf_at1} and performing two integrations by part, we obtain that the left hand side of \eqref{E:L_at} is bounded by
\begin{align*}
C \Prob{ \underline{\tau} \leq 1, \overline{\tau} \geq \frac{t_2}{t_1} } \left( \frac{t_2}{t_1} \right)^{1-\theta+\eps}
+ C \int_0^1 \Prob{ \underline{\tau} \leq u, \overline{\tau} \geq \frac{t_2}{t_1} } \left( \frac{t_2}{t_1} \right)^{1-\theta+\eps} \frac{1}{u^{2-\theta+\eps}} \d u \\
+ C \int_{t_2/t_1}^\infty \Prob{ \underline{\tau} \leq 1, \overline{\tau} \geq v } \frac{v^{-\theta+\eps}}{u^{1-\theta+\eps}} \d v
+ C \int_0^1 \int_{t_2/t_1}^\infty \Prob{ \underline{\tau} \leq u, \overline{\tau} \geq v } \frac{v^{-\theta+\eps}}{u^{2-\theta+\eps}} \d u \d v.
\end{align*}
\eqref{E:pf_at2} then concludes the proof.
\end{proof}

Given $n\in\N\setminus\{ 0\}$,
an integer time $T\geq c_{2}n$,
and an initial condition
$(K_{i}^{T}(0))_{i\geq 0}\in \espace^{\ast}$ for the conditioned Markov chain,
with $\chrT((K_{i}^{T}(0))_{i\geq 0})\leq n$,
we have that
\begin{eqnarray*}
\mathbb{P}
\big(
(K_{i}^{T}(n))_{i\geq 0}\in\fav_{n}
\big)
&\geq &
\mathbb{P}
\big(
\rho(K_{0}^{T}(n))\geq 2n
\big)
-
\mathbb{P}
\big(
\exists \wp\in\Lc^{\theta}_{0},
\partial\cyl_{n}\stackrel{\wp}{\longleftrightarrow}
\partial\cyl_{(c_{1}-1)n}
\big \vert
T_{\dagger}>T
\big)
\\
&\geq &
p_{0} - C_{\varepsilon}(c_{1}-1)^{-(\theta-\varepsilon)},
\end{eqnarray*}
where $p_{0}$ is given by Lemma \ref{Lem F n clusters},
$\varepsilon\in (0,\theta)$ and $C_{\varepsilon}$
is given by Lemma \ref{Lem exponent conditioned loops}.
We chose $\varepsilon = \theta/2$,
and we take an integer $c_{1}\geq 5$ such that
\begin{equation}
\label{Eq condition c1}
p_{0} - C_{\theta/2}(c_{1}-1)^{-\frac{\theta}{2}}>0.
\end{equation}
With this choice of $c_{1}$ we are sure that Condition \ref{cond2}
of Theorem \ref{Thm abstract coupling} is satisfied.
Let us also check that Condition \ref{cond3} of Theorem \ref{Thm abstract coupling}
is satisfied.
Indeed,
\begin{displaymath}
\mathbb{P}\big(
\chrT((K_{i}^{T}(c_{1}n))_{i\geq 0})\geq u n\big)
\leq
\mathbb{P}\big(
\exists\wp\in\Lc^{\theta}_{0},
\partial\cyl_{c_{1}n}\stackrel{\wp}{\longleftrightarrow}
\partial\cyl_{(c_{1}+u)n}
\big\vert T_{\dagger}>T
\big)
\leq 
C_{\varepsilon}\Big(\dfrac{c_{1}}{c_{1}+u}\Big)^{\theta-\varepsilon}
,
\end{displaymath}
where $\varepsilon\in (0,\theta)$ and $C_{\varepsilon}$
is given by Lemma \ref{Lem exponent conditioned loops}.
It remains to check Condition \ref{cond1} of Theorem \ref{Thm abstract coupling}.

For $t_{2}>t_{1}>0$, we will denote
\begin{displaymath}
\Lc^{\theta}_{\cap [t_{1},t_{2}]}
=\{
\wp\in\Lc^{\theta}_{0} : 
\Range(\wp)\cap [t_{1},t_{2}] \times \mathbb{S}^{1} 
\neq\emptyset
\}.
\end{displaymath}

\begin{lemma}
\label{Lem another crossing}
There is $\hat{p}_{0}\in (0,1]$ such that for every $n\in\N\setminus\{ 0\}$,
and for every $T\geq c_{2}n$,
\begin{displaymath}
\mathbb{P}
\Big(
\partial\cyl_{n}\stackrel{\Lc^{\theta}_{\cap [(c_{1}-1)n, T-n]}}{\longleftrightarrow}
\partial\cyl_{T-n}
\Big\vert
\partial\cyl_{(c_{1}-1)n}\stackrel{\Lc^{\theta}_{\cap [(c_{1}-1)n, T-n]}}{\longleftrightarrow}
\partial\cyl_{T-n}
\Big)
\geq \hat{p}_{0}.
\end{displaymath}
\end{lemma}

\begin{proof}
This is similar to Lemma \ref{Lem F n clusters}.
Let $E_{1,n}$ be the event
\begin{displaymath}
E_{1,n} = \{\exists\wp\in\Lc^{\theta}_{0},
\partial\cyl_{n}\stackrel{\wp}{\longleftrightarrow}
\partial\cyl_{c_{1}n}\}.
\end{displaymath}
Let $E_{2,n}$ be the event that there is a loop in
$\Lc^{\theta}_{0}$,
contained in the cylinder
$((c_{1}-1)n, c_{1}n)\times\mathbb{S}^{1}$ that surrounds the cylinder.
Then
\begin{displaymath}
E_{1,n}\cap E_{2,n}\cap
\Big\{\partial\cyl_{(c_{1}-1)n}\stackrel{\Lc^{\theta}_{\cap [(c_{1}-1)n, T-n]}}{\longleftrightarrow}
\partial\cyl_{T-n}\Big\}
\subset
\Big\{\partial\cyl_{n}\stackrel{\Lc^{\theta}_{\cap [(c_{1}-1)n, T-n]}}{\longleftrightarrow}
\partial\cyl_{T-n}\Big\}.
\end{displaymath}
The three events $E_{1,n}$, $E_{2,n}$ and
$\big\{\partial\cyl_{(c_{1}-1)n}\stackrel{\Lc^{\theta}_{\cap [(c_{1}-1)n, T-n]}}{\longleftrightarrow}
\partial\cyl_{T-n}\big\}$
are increasing for the Poisson point process $\Lc^{\theta}_{0}$,
so one can apply the FKG inequality.
Therefore,
\begin{displaymath}
\mathbb{P}
\big(
\partial\cyl_{n}\stackrel{\Lc^{\theta}_{\cap [(c_{1}-1)n, T-n]}}{\longleftrightarrow}
\partial\cyl_{T-n}
\big\vert
\partial\cyl_{(c_{1}-1)n}\stackrel{\Lc^{\theta}_{\cap [(c_{1}-1)n, T-n]}}{\longleftrightarrow}
\partial\cyl_{T-n}
\big)
\geq 
\mathbb{P}(E_{1,n}\cap E_{2,n})
\geq\mathbb{P}(E_{1,n})\mathbb{P}(E_{2,n}).
\end{displaymath}
By Lemma \ref{L:cross_loop},
$\mathbb{P}(E_{1,n})\geq 1-\hat{p}_{1}^{\theta}$
for some $\hat{p}_{1}\in (0,1)$.
Moreover, $\mathbb{P}(E_{1,n})\geq 1-\hat{p}_{2}^{\theta n}$
for some $\hat{p}_{2}\in (0,1)$.
This gives the desired lower bound uniform in $n$ and $T$.
\end{proof}

In the sequel, given a random variable $X$ and an event A,
we will denote by $\Law(X\vert A)$
the law of $X$ conditioned on the event $A$.

\begin{lemma}
\label{Lem overlap}
There is $q\in (0,1]$ such that
for every $n\in\N\setminus\{ 0\}$,
for every $T\geq c_{2} n$ and
and for every two initial conditions
$(K_{i}(0))_{i\geq 0}$
and $(\widehat{K}_{i}(0))_{i\geq 0}$ in $\fav_{n}$,
we have
\begin{displaymath}
1-\dfrac{1}{2} d_{\rm TV}(L_{n,T}, \widehat{L}_{n,T})\geq q,
\end{displaymath}
where
\begin{displaymath}
L_{n,T} = 
\Law
\big(
(K_{i}((c_{1}-1)n))_{i\geq 0}
\big\vert
T_{\dagger}>T-n
\big),
\qquad
\widehat{L}_{n,T} = 
\Law
\big(
(\widehat{K}_{i}((c_{1}-1)n))_{i\geq 0}
\big\vert
T_{\dagger}>T-n
\big),
\end{displaymath}
and where $(K_{i}(m))_{i\geq 0, m\geq 0}$,
resp. $(\widehat{K}_{i}(m))_{i\geq 0, m\geq 0}$
are the Markov chains of clusters starting from the corresponding initial conditions.
\end{lemma}

\begin{proof}
Consider the following events for the loop soup $\Lc^{\theta}_{\cap (0,T-n]}$.
Let $E_{1,n,T}$ be the event
\begin{displaymath}
E_{1,n,T} = 
\Big\{
\partial\cyl_{n}\stackrel{\Lc^{\theta}_{\cap [(c_{1}-1)n, T-n]}}{\longleftrightarrow}
\partial\cyl_{T-n}
\Big\}.
\end{displaymath}
Let $E_{2,n}$ be the event that there is a loop of
$\Lc^{\theta}_{\cap (0,T-n]}$ contained in the cylinder
$(n,2n)\times\mathbb{S}^{1}$ that surrounds the cylinder.
Let $E_{3,n}$ be the event that there is a loop of
$\Lc^{\theta}_{\cap (0,T-n]}$ contained in the cylinder
$((c_{1}-2)n,(c_{1}-1)n)\times\mathbb{S}^{1}$ that surrounds the cylinder.
See Figure \ref{fig2}.
Note that the events $E_{1,n,T}$, $E_{2,n}$ and $E_{3,n}$
are independent since they involve disjoint types of loops.

\begin{figure}
   \centering
    \def\svgwidth{0.8\columnwidth}
   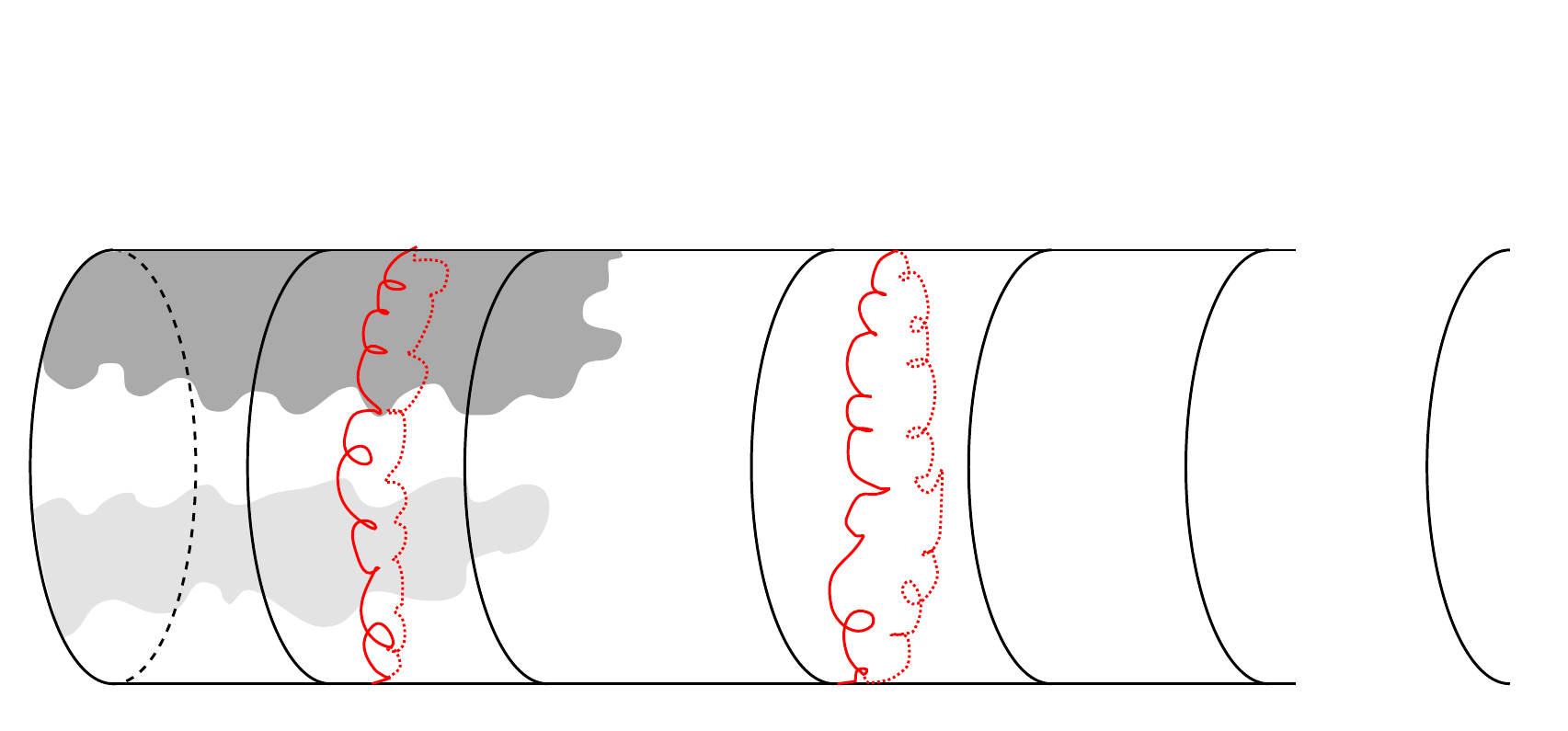
   \caption{Schematic representation of the events considered in the proof of Lemma \ref{Lem overlap}. Only the marked compact $K_0(0)$ and $\widehat{K}_0(0)$ of the initial conditions $(K_i(0))_{i \geq 0}$ and $(\widehat{K}_i(0))_{i \geq 0}$ have been depicted (light and dark grey respectively). The existence of the blue cluster corresponds to the event $E_{1,n,T}$ whereas the existence of the two red loops corresponds to the events $E_{2,n,T}$ and $E_{3,n,T}$.}\label{fig2}
\end{figure}

Now we consider the Markov chains $(K_{i}(m))_{i\geq 0, m\geq 0}$
and $(\widehat{K}_{i}(m))_{i\geq 0, m\geq 0}$ on the same probability space,
obtained from the same Brownian loop soup $\Lc^{\theta}_{0}$.
Note that in this coupling, neither of the two Markov processes is conditioned
on $T_{\dagger}>T-n$. Not yet.

Since $(K_{i}(0))_{i\geq 0}\in\fav_{n}$ and $(\widehat{K}_{i}(0))_{i\geq 0}\in\fav_{n}$, we have
\begin{displaymath}
E_{1,n,T}\cap E_{2,n}
\subset\{T_{\dagger}((K_{i})_{i\geq 0})> T-n\}
\quad \text{and} \quad
E_{1,n,T}\cap E_{2,n}
\subset\{T_{\dagger}((\widehat{K}_{i})_{i\geq 0})> T-n\};
\end{displaymath}
see Figure \ref{fig2}.
Further, on the event $E_{1,n,T}\cap E_{2,n}\cap E_{3,n}$,
we have $(K_{i}((c_{1}-1)n))_{i\geq 0}=(\widehat{K}_{i}((c_{1}-1)n))_{i\geq 0}$ a.s. 
Indeed, if two Brownian loops $\wp$ and $\wp'$ intersecting
$\partial\cyl_{(c_{1}-1)n}$ are connected by a finite chain that uses
$(K_{i}(0))_{i\geq 0}$
and $(\widehat{K}_{i}(0))_{i\geq 0}$,
then this chain has to intersect the non-contractible loops
in $((c_{1}-2)n,(c_{1}-1)n)\times\mathbb{S}^{1}$,
that exist on the event $E_{3,n}$,
and thus the connection between $\wp$ and $\wp'$
can also be done without using $(K_{i}(0))_{i\geq 0}$
or $(\widehat{K}_{i}(0))_{i\geq 0}$,
by using just the Brownian loop; see Figure \ref{fig2}.

Thus, we get some overlap in total variation between
the non-conditioned laws
$\Law\big((K_{i}((c_{1}-1)n))_{i\geq 0}\big)$ and
$\Law\big((\widehat{K}_{i}((c_{1}-1)n))_{i\geq 0}\big)$.
Moreover, this overlap happens on the event
\begin{displaymath}
\{T_{\dagger}((K_{i})_{i\geq 0})> T-n\}\cap
\{T_{\dagger}((\widehat{K}_{i})_{i\geq 0})> T-n\}.
\end{displaymath}
Therefore, we get a non-trivial overlap for the conditioned laws too:
\begin{displaymath}
1-\dfrac{1}{2} d_{\rm TV}(L_{n,T}, \widehat{L}_{n,T})\geq 
\dfrac{\mathbb{P}(E_{1,n,T}\cap E_{2,n}\cap E_{3,n})}
{\mathbb{P}(T_{\dagger}((K_{i})_{i\geq 0})> T-n)\vee
\mathbb{P}(T_{\dagger}((\widehat{K}_{i})_{i\geq 0})> T-n)}.
\end{displaymath}
Further,
\begin{displaymath}
\{T_{\dagger}((K_{i})_{i\geq 0})> T-n\}
\cup
\{T_{\dagger}((\widehat{K}_{i})_{i\geq 0})> T-n\}
\subset E_{4,n,T},
\end{displaymath}
where $E_{4,n,T}$ is the event
\begin{displaymath}
E_{4,n,T} = 
\Big\{
\partial\cyl_{(c_{1}-1)n}\stackrel{\Lc^{\theta}_{\cap [(c_{1}-1)n, T-n]}}{\longleftrightarrow}
\partial\cyl_{T-n}
\Big\}.
\end{displaymath}
Thus,
\begin{displaymath}
\mathbb{P}(T_{\dagger}((K_{i})_{i\geq 0})> T-n)\vee
\mathbb{P}(T_{\dagger}((\widehat{K}_{i})_{i\geq 0})> T-n)
\leq\mathbb{P}(E_{4,n,T}).
\end{displaymath}
Thus,
\begin{displaymath}
1-\dfrac{1}{2} d_{\rm TV}(L_{n,T}, \widehat{L}_{n,T})\geq 
\dfrac{\mathbb{P}(E_{1,n,T}\cap E_{2,n}\cap E_{3,n})}
{\mathbb{P}(E_{4,n,T})}
\geq
\dfrac{\mathbb{P}(E_{1,n,T})}{\mathbb{P}(E_{4,n,T})}
\mathbb{P}(E_{2,n})\mathbb{P}(E_{3,n}).
\end{displaymath}
By Lemma \ref{Lem another crossing},
\begin{displaymath}
\dfrac{\mathbb{P}(E_{1,n,T})}{\mathbb{P}(E_{4,n,T})}
\geq\hat{p}_{0}>0,
\end{displaymath}
where the lower bound is uniform in $n$ and $T$.
Further,
$\mathbb{P}(E_{2,n})=\mathbb{P}(E_{3,n})= 1 - \hat{p}_{2}^{\theta n}$,
for some $\hat{p}_{2}\in (0,1)$.
This concludes the proof of the lemma.
\end{proof}

Lemma \ref{Lem overlap} is exactly Condition \ref{cond1} of 
Theorem \ref{Thm abstract coupling}.
So it applies to our Markov chains of clusters in a Brownian loop soup and concludes the proof of Theorem \ref{T:coupling_rough}.

\subsection{Convergence of the ratio of survival probabilities}\label{SS:ratio}

We finish this section by proving Theorem \ref{T:ratio_general}. We first state and prove a corollary of Theorem \ref{T:coupling_rough}.
We use the same notations as the ones introduced above Theorem \ref{T:coupling_rough}.

\begin{corollary}\label{C:ratio}
For every initial conditions $K_0$ and $\hat{K}_0$ not identical to $\partial \cyl_0$, $\P ( T^\dagger > T ) / \P ( \hat{T}^\dagger > T )$ converges as $T \to \infty$.
\end{corollary}

\begin{remark}\label{rmk:ratio_bounded}
By FKG inequality, it is clear that the ratio $\P ( T^\dagger > T ) / \P ( \hat{T}^\dagger > T )$ remains uniformly bounded away from zero and from infinity. Indeed, recalling the notation $\rho(K_0)$ \eqref{E:rho_K}, let $E$ be the event that there is a loop intersecting $\hat{K}_0$ that disconnects $\partial \cyl_{\rho(K_0)}$ and $\partial \cyl_{\rho(K_0)+1}$. Since $E \cap \{ T^\dagger > T \} \subset \{ \hat{T}^\dagger > T \}$, by FKG inequality we get that
\[
\P ( T^\dagger > T ) / \P ( \hat{T}^\dagger > T ) \leq 1 / \Prob{E}.
\]
The lower bound is similar. On the other hand, showing that this ratio actually converges is much more involved and relies on the coupling results we developed.
\end{remark}

\begin{proof}
Let $T \gg t \gg 1$ be large integers and let $E_{t,T}$ be the event that under the coupling of Theorem \ref{T:coupling_rough}, the coalescence time $\Tc^T$ is smaller than $t$. By Theorem \ref{T:coupling_rough}, $\Prob{E_{t,T}} \to 1$ as $T \to \infty$ followed by $t \to \infty$.
In this proof, we will denote by $(\Fc_s)_{s \geq 0}$ the filtration $\Fc_s = \sigma( K_0(u), \hat{K}_0(u), u \in \{0, \dots, s\} )$, $s \geq 0$.
We have
\begin{align*}
\P ( T^\dagger >T & \vert T^\dagger>t ) = \Expect{ P (T^\dagger >T \vert T^\dagger > t, \Fc_t) \mathbf{1}_{E_{t,T}} } + \Expect{ \P(T^\dagger >T \vert T^\dagger > t, \Fc_t) \mathbf{1}_{E_{t,T}^c} } \\
& = \Expect{ \P(\hat{T}^\dagger >T \vert \hat{T}^\dagger > t, \Fc_t) \mathbf{1}_{E_{t,T}} } + \Expect{ \P(T^\dagger >T \vert T^\dagger > t, \Fc_t) \mathbf{1}_{E_{t,T}^c} } \\
& = \P(\hat{T}^\dagger >T \vert \hat{T}^\dagger > t) + \Expect{ \left( \P(T^\dagger >T \vert T^\dagger > t, \Fc_t) - \P(\hat{T}^\dagger >T \vert \hat{T}^\dagger > t, \Fc_t) \right) \mathbf{1}_{E_{t,T}^c} }.
\end{align*}
The rest of the proof is dedicated to showing that the second right hand side term is much smaller than the first one.
Let $p,q >1$ be such that $1/p + 1/q = 1$. By H\"older's inequality and the above inequality, $\abs{ \P ( T^\dagger >T \vert T^\dagger>t ) - \P ( \hat{T}^\dagger >T \vert \hat{T}^\dagger > t ) }$ is bounded by
\begin{align}
\label{E:pf_ratio1}
\Prob{E_{t,T}^c}^{1/p} \Expect{\Prob{T^\dagger >T \vert T^\dagger > t, \Fc_t}^q + \P(\hat{T}^\dagger >T \vert \hat{T}^\dagger > t, \Fc_t)^q }^{1/q}.
\end{align}
Let us denote by
$
\tau := \sup \{ s > 0: \exists \wp \in \Lc_{\cap (0,t]}^\theta: \Range(\wp) \cap \{s\} \times \mathbb{S}^1 \neq \varnothing \}.
$
Conditionally on $\tau, \Fc_t, T^\dagger > t$, in order to have $T^\dagger > T$, there must be a cluster of loops included in $(t,\infty) \times \mathbb{S}^1$ that crosses $[ \tau, T] \times \mathbb{S}^1$. By scaling, this probability is equal to $\P \Big( \tau - t \overset{\Lc_0^\theta}{\longleftrightarrow} T-t \vert \tau \Big)$. We obtain that
\begin{align}
\label{E:pf_ratio2}
\Expect{\Prob{T^\dagger >T \vert T^\dagger > t, \Fc_t}^q} = \Prob{ \tau > T} + \int_t^T \Prob{ \tau \in \d s } \Prob{ s - t \overset{\Lc_0^\theta}{\longleftrightarrow} T-t }^q.
\end{align}
Let $E_1, E_2$ and $E_3$ be the following events:
\begin{itemize}
\item $E_1$: there is a cluster of $\Lc_0^\theta$ crossing $[t-1,s-t+1] \times \mathbb{S}^1$;
\item $E_2$: there is a non contractible loop inside $[t-1,t] \times \mathbb{S}^1$;
\item $E_3$: there is a non contractible loop inside $[s-t,s-t+1] \times \mathbb{S}^1$.
\end{itemize}
We have $E_1 \cap E_2 \cap E_3 \cap \{ T^\dagger > t \} \cap \{ s - t \overset{\Lc_0^\theta}{\longleftrightarrow} T-t \} \subset \{ T^\dagger > T-t \} $. The five events on the left hand side are increasing events. By FKG inequality we deduce that the product of the five individual probabilities is at most $\Prob{T^\dagger > T-t}$. By Lemma \ref{L:surround}, $\Prob{E_2} = \Prob{E_3} \geq c$ and by Theorem \ref{T:large_crossing}, $\Prob{E_1} \geq c_\eps (s/t)^{-1+\theta-\eps}$. This leads to
\[
\P \Big( s - t \overset{\Lc_0^\theta}{\longleftrightarrow} T-t \Big) \leq C (s/t)^{1-\theta+\eps} \Prob{T^\dagger > T-t} / \Prob{T^\dagger > t}.
\]
Plugging this estimate back in \eqref{E:pf_ratio2} and integrating by parts (differentiating $s^{q(1-\theta+\eps)}$ and integrating $\Prob{ \tau \in \d s}$) gives the following upper bound for the left hand side of \eqref{E:pf_ratio2}:
\begin{align*}
C (1/t)^{q(1-\theta+\eps)} \Prob{T^\dagger > T-t}^q / \Prob{T^\dagger > t}^q \int_t^T \Prob{ \tau \geq s} s^{q(1-\theta+\eps) -1} \d s.
\end{align*}
By Lemma \ref{L:cross_loop}, $\Prob{ \tau \geq s} \leq C t/s$. With the change of variable $s = xt$, we have obtained that 
\begin{align*}
\Expect{\Prob{T^\dagger >T \vert T^\dagger > t, \Fc_t}^q} 
\leq C \Prob{T^\dagger > T-t}^q / \Prob{T^\dagger > t}^q \int_1^\infty  x^{q(1-\theta+\eps) -2} \d x.
\end{align*}
We choose $1 < q < 1/(1-\theta+\eps)$ to make sure that the integral above converges. We have overall obtained that $\Prob{T^\dagger>t}^q \Expect{\Prob{T^\dagger >T \vert T^\dagger > t, \Fc_t}^q} \leq C \Prob{T^\dagger > T-t}^q$. Since $\P(\hat{T}^\dagger > t) / \Prob{T^\dagger >t} \in [1/C,C]$, we also obtain that
\[
\Prob{T^\dagger>t}^q \Expect{\Prob{\hat{T}^\dagger >T \vert \hat{T}^\dagger > t, \Fc_t}^q} \leq C \Prob{T^\dagger > T-t}^q.
\]
Multiplying \eqref{E:pf_ratio1} by $\Prob{ T^\dagger>t} / \P(\hat{T}^\dagger > T)$, we deduce that
\begin{align*}
\abs{ \frac{\Prob{ T^\dagger >T}}{\P(\hat{T}^\dagger>T)} - \frac{\Prob{T^\dagger>t}}{\P(\hat{T}^\dagger > t)} }
\leq C \Prob{E_{t,T}^c}^{1/p} \frac{\Prob{T^\dagger > T-t}}{\P(\hat{T}^\dagger >T)}.
\end{align*}
Using FKG inequality and the fact that the probability of crossing $[T-t, T] \times \mathbb{S}^1$ tends to 1 as $T \to \infty$ ($t$ fixed), it can be checked that $\limsup_{T \to \infty} \Prob{T^\dagger > T-t} / \P (\hat{T}^\dagger >T )$ is bounded uniformly in $t$; see Remark \ref{rmk:ratio_bounded} for a very similar statement. By then sending $t \to \infty$, we obtain that
\[
\limsup_{t \to \infty} \limsup_{T \to \infty} \abs{ \frac{\Prob{ T^\dagger >T}}{\P(\hat{T}^\dagger>T)} - \frac{\Prob{T^\dagger>t}}{\P(\hat{T}^\dagger > t)} } = 0.
\]
This shows that $\Prob{ T^\dagger >T}/\P(\hat{T}^\dagger>T)$ converges as $T \to \infty$.
\end{proof}

We are now ready to prove Theorem \ref{T:ratio_general}.

\begin{proof}[Proof of Theorem \ref{T:ratio_general}]
The proof of the convergence of the first ratio in \eqref{E:ratio_D} follows quickly from Corollary \ref{C:ratio} and from conformal invariance of the loop soup. Indeed, let $f : D \to \D$ be a conformal map with $f(x) = 0$. By conformal invariance of the loop soup, the first ratio in \eqref{E:ratio_D} is equal to
\[
\P \Big( f(K) \overset{\Lc_\D^\theta}{\longleftrightarrow} f(D(x,r)) \Big) \Big/ \P \Big( f(\hat{K}) \overset{\Lc_\D^\theta}{\longleftrightarrow} f(D(x,r)) \Big).
\]
There exists a constant $C>1$ such that if $r$ is small enough $C^{-1} r \D \subset f(D(x,r)) \subset C r \D$. Because the probability of connecting $K$ to $C^{-1} r \D$ is asymptotically equivalent to that of connecting $K$ to $C r \D$, we deduce that the first ratio in \eqref{E:ratio_D} is equal to
\[
(1+o(1)) \P \Big( f(K) \overset{\Lc_\D^\theta}{\longleftrightarrow} r \D \Big) \Big/ \P \Big( f(\hat{K}) \overset{\Lc_\D^\theta}{\longleftrightarrow} r \D \Big).
\]
By Corollary \ref{C:ratio}, the above ratio converges as $r \to 0$ ($r = e^{-T}$ in the notations of Corollary \ref{C:ratio}).

We now move to the proof of the convergence of the second ratio in \eqref{E:ratio_D} and the identification of the limit. Let us denote by $\ell$ the limit of the first ratio in \eqref{E:ratio_D}.
Let $\delta >0$ and let $\eps \in (0, \d(z, \partial D))$ be small enough so that for all $r \in (0,\eps)$,
\[
\P \Big( K \overset{\Lc_D^\theta}{\longleftrightarrow} D(z,r) \Big)
\Big/
\P \Big( \hat{K} \overset{\Lc_D^\theta}{\longleftrightarrow} D(z,r) \Big)
\leq (1+\delta) \ell.
\]
As in \eqref{E:thick_to_disc}, we can replace $\Xi_a^z$ by a disc centred at $z$ with random radius $\norme{\Xi_a^z -z}_\infty$:
\[
\P \Big( K \overset{\Lc_D^\theta \cup \Xi_a^z}{\longleftrightarrow} z \Big)
= (1+o(1)) \P \Big( K \overset{\Lc_D^\theta}{\longleftrightarrow} D(z,\norme{\Xi_a^z -z}_\infty) \Big).
\]
By Lemma \ref{L:diameter_thick}, the probability that $\norme{\Xi_a^z -z}_\infty > \eps$ is at most $C a = C \gamma^2$ which is much smaller than the probability we are looking at (it behaves like $\gamma^{2(1-\theta)+o(1)}$, see Theorem \ref{T:convergenceL2}). By Lemma \ref{L:diameter_thick}, we get that
\begin{align*}
& \P \Big( K \overset{\Lc_D^\theta \cup \Xi_a^z}{\longleftrightarrow} z \Big)
= (1+o(1)) \int_0^\eps \d r \frac{a r^{a-1}}{\CR(z,D)^a} \P \Big( K \overset{\Lc_D^\theta}{\longleftrightarrow} D(z,r) \Big) \\
& \leq (1+o(1)) (1+\delta) \ell \int_0^\eps \d r \frac{a r^{a-1}}{\CR(z,D)^a} \P \Big( \hat{K} \overset{\Lc_D^\theta}{\longleftrightarrow} D(z,r) \Big)
 \leq (1+o(1))(1+\delta) \ell \P \Big( \hat{K} \overset{\Lc_D^\theta \cup \Xi_a^z}{\longleftrightarrow} z \Big)
\end{align*}
where in the last inequality we simply reversed the above procedure.
This proves that
\[
\limsup_{a \to 0} \P \Big( K \overset{\Lc_D^\theta \cup \Xi_a^z}{\longleftrightarrow} z \Big) \Big/ \P \Big( \hat{K} \overset{\Lc_D^\theta \cup \Xi_a^z}{\longleftrightarrow} z \Big)
\leq \ell.
\]
The lower bound is proved similarly. This concludes the proof.
\end{proof}

%
%

%% file: subfiles/figures/drawing1.pdf_tex
\begingroup%
  \makeatletter%
  \providecommand\color[2][]{%
    \errmessage{(Inkscape) Color is used for the text in Inkscape, but the package 'color.sty' is not loaded}%
    \renewcommand\color[2][]{}%
  }%
  \providecommand\transparent[1]{%
    \errmessage{(Inkscape) Transparency is used (non-zero) for the text in Inkscape, but the package 'transparent.sty' is not loaded}%
    \renewcommand\transparent[1]{}%
  }%
  \providecommand\rotatebox[2]{#2}%
  \newcommand*\fsize{\dimexpr\f@size pt\relax}%
  \newcommand*\lineheight[1]{\fontsize{\fsize}{#1\fsize}\selectfont}%
  \ifx\svgwidth\undefined%
    \setlength{\unitlength}{564.33762851bp}%
    \ifx\svgscale\undefined%
      \relax%
    \else%
      \setlength{\unitlength}{\unitlength * \real{\svgscale}}%
    \fi%
  \else%
    \setlength{\unitlength}{\svgwidth}%
  \fi%
  \global\let\svgwidth\undefined%
  \global\let\svgscale\undefined%
  \makeatother%
  \begin{picture}(1,0.53454246)%
    \lineheight{1}%
    \setlength\tabcolsep{0pt}%
    \put(0,0){\includegraphics[width=\unitlength,page=1]{drawing1.pdf}}%
    \put(0.08137569,0.22733275){\makebox(0,0)[lt]{\lineheight{1.25}\smash{\begin{tabular}[t]{l}$K_0$\end{tabular}}}}%
    \put(0.43178562,0.00962135){\makebox(0,0)[lt]{\lineheight{1.25}\smash{\begin{tabular}[t]{l}$t_1$\end{tabular}}}}%
    \put(0.78936224,0.00962135){\makebox(0,0)[lt]{\lineheight{1.25}\smash{\begin{tabular}[t]{l}$t_2$\end{tabular}}}}%
  \end{picture}%
\endgroup%

%% file: subfiles/figures/drawing2.pdf_tex
\begingroup%
  \makeatletter%
  \providecommand\color[2][]{%
    \errmessage{(Inkscape) Color is used for the text in Inkscape, but the package 'color.sty' is not loaded}%
    \renewcommand\color[2][]{}%
  }%
  \providecommand\transparent[1]{%
    \errmessage{(Inkscape) Transparency is used (non-zero) for the text in Inkscape, but the package 'transparent.sty' is not loaded}%
    \renewcommand\transparent[1]{}%
  }%
  \providecommand\rotatebox[2]{#2}%
  \newcommand*\fsize{\dimexpr\f@size pt\relax}%
  \newcommand*\lineheight[1]{\fontsize{\fsize}{#1\fsize}\selectfont}%
  \ifx\svgwidth\undefined%
    \setlength{\unitlength}{818.05844313bp}%
    \ifx\svgscale\undefined%
      \relax%
    \else%
      \setlength{\unitlength}{\unitlength * \real{\svgscale}}%
    \fi%
  \else%
    \setlength{\unitlength}{\svgwidth}%
  \fi%
  \global\let\svgwidth\undefined%
  \global\let\svgscale\undefined%
  \makeatother%
  \begin{picture}(1,0.46609829)%
    \lineheight{1}%
    \setlength\tabcolsep{0pt}%
    \put(0,0){\includegraphics[width=\unitlength,page=1]{drawing2.pdf}}%
    \put(0.19001896,0.00442645){\makebox(0,0)[lt]{\lineheight{1.25}\smash{\begin{tabular}[t]{l}$n$\end{tabular}}}}%
    \put(0.31814897,0.00442645){\makebox(0,0)[lt]{\lineheight{1.25}\smash{\begin{tabular}[t]{l}$2n$\end{tabular}}}}%
    \put(0.44845975,0.00442645){\makebox(0,0)[lt]{\lineheight{1.25}\smash{\begin{tabular}[t]{l}$(c_1-2)n$\end{tabular}}}}%
    \put(0.59883091,0.00442645){\makebox(0,0)[lt]{\lineheight{1.25}\smash{\begin{tabular}[t]{l}$(c_1-1)n$\end{tabular}}}}%
    \put(0.92543827,0.00442645){\makebox(0,0)[lt]{\lineheight{1.25}\smash{\begin{tabular}[t]{l}$T-n$\end{tabular}}}}%
    \put(0.10205191,0.27182739){\makebox(0,0)[lt]{\lineheight{1.25}\smash{\begin{tabular}[t]{l}$\widehat{K}_0(0)$\end{tabular}}}}%
    \put(0.03831965,0.11090306){\makebox(0,0)[lt]{\lineheight{1.25}\smash{\begin{tabular}[t]{l}$K_0(0)$\end{tabular}}}}%
    \put(0.77305462,0.00442645){\makebox(0,0)[lt]{\lineheight{1.25}\smash{\begin{tabular}[t]{l}$c_1 n$\end{tabular}}}}%
    \put(0,0){\includegraphics[width=\unitlength,page=2]{drawing2.pdf}}%
  \end{picture}%
\endgroup%

%% file: subfiles/field_new.tex
We start by stating the intermediate results needed in order to prove Theorems \ref{T:convergenceL2}, \ref{T:cluster2point} and \ref{T:cluster_xy_doob}, and then use them to prove these theorems. The rest of the section then consists of the proofs of these intermediate results.

We start with a lemma that shows that computing the second moment of $h_\gamma$ boils down to computing the crossing probabilities between $x$ and $y$ in a soup whose loops come from
1) an unconditioned Brownian loop soup $\Lc_D^\theta$,
2) a thick loop $\Xi_a^x$ at $x$ and
3) a thick loop $\Xi_a^y$ at $y$.

\begin{lemma}\label{L:second_moment_to_crossing_proba}
Let $\eps >0$ and $F : D \times D \times \mathfrak{L} \to \R$ be a bounded measurable admissible function such that for all $x,y \in D$, $F(x,y, \cdot)$ does not depend on loops with diameter smaller than $\eps$. For all $\gamma, \gamma' \in (0,\sqrt{2})$,
\begin{align}
\label{E:L_second_moment_to_crossing_proba}
& \Expect{ \int_{D \times D} F(x,y,\Lc_D^\theta) h_{\gamma}(x) h_{\gamma'}(y) \d x \d y } =
C(\gamma,\gamma',\eps) \frac{\gamma^2 {\gamma'}^2}{Z_{\gamma} Z_{\gamma'}}
\\
& + \frac{1}{Z_{\gamma} Z_{\gamma'}} \int_{D \times D} \E \Big[ F(x,y,\Lc_D^\theta \cup \{\Xi_{a_i}^x \}_{i \geq 1} \cup \{\Xi_{a_i'}^y \}_{i \geq 1}) \mathbf{1} \Big\{ x \overset{\Lc_D^\theta \cup \{\Xi_{a_i}^x \}_{i \geq 1} \cup \{\Xi_{a_i'}^y \}_{i \geq 1}}{\longleftrightarrow} y \Big\} \Big] \d x \d y
\nonumber
\end{align}
where the three collections of loops $\Lc_D^\theta$, $\{\Xi_{a_i}^x \}_{i \geq 1}$ and $\{ \Xi_{a_i'}^y \}_{i \geq 1}$ are independent and distributed as in Lemma \ref{L:moment_measure}.
The constant $C(\gamma, \gamma', \eps)$ is bounded uniformly w.r.t. $\gamma, \gamma'$ in compact subsets of $[0,\sqrt{2})$ by a quantity which may depend on $\eps$.
\end{lemma}

The next two results are then dedicated to estimating the crossing probability appearing in  Lemma \ref{L:second_moment_to_crossing_proba}. The first result has a very similar flavour as \cite[Proposition 4.1]{JLQ23a}. On the other hand, the second result is novel and crucially relies on Theorem \ref{T:ratio_general}.

\begin{lemma}\label{L:crossing_upper}
For any $\eta >0$, there exists $C>0$ such that for all $\gamma \in (0,1]$ and $x, y \in D$,
\begin{equation}
\label{E:P_crossing_bound}
\frac{1}{Z_\gamma^2} \Prob{ x \overset{\Lc_D^\theta \cup \Xi_a^x \cup \Xi_a^y}{\longleftrightarrow} y } \leq C |\log |x-y||^{2(1-\theta) + \eta}.
\end{equation}
\end{lemma}

\begin{proposition}\label{P:crossing}
For any distinct points $x, y \in D$, the limit
\begin{equation}
\label{E:P_crossing_limit}
C_\theta(x,y) := \lim_{\gamma \to 0}
\frac{1}{Z_\gamma^2} \P \Big( x \overset{\Lc_D^\theta \cup \Xi_a^x \cup \Xi_a^y}{\longleftrightarrow} y \Big)
\end{equation}
exists and
\begin{equation}
\label{E:P_crossing_limit2}
C_\theta(x,y) \leq \liminf_{\gamma_1, \gamma_2 \to 0}
\frac{1}{Z_{\gamma_1} Z_{\gamma_2}} \P \Big( x \overset{\Lc_D^\theta \cup \Xi_{a_1}^x \cup \Xi_{a_2}^y}{\longleftrightarrow} y \Big).
\end{equation}
Moreover,
\begin{equation}
\label{E:P_crossing_limit3}
C_\theta(x,y) = \left( \log \frac{1}{|x-y|} \right)^{2(1-\theta)+o(1)}
\quad \quad \text{as} \quad x-y \to 0.
\end{equation}
More generally, let $\eps >0$ and $F: \mathfrak{L} \to \R$ be a bounded measurable function that does not depend on loops with diameter smaller than $\eps>0$. Then the following two limits exist and are identical
\begin{equation}
\label{E:P_law}
\lim_{r \to 0}
\frac{1}{Z_r^2}
\Expect{ F(\Lc_D^\theta) \mathbf{1}\Big\{ D(x,r) \overset{\Lc_D^\theta}{\longleftrightarrow} D(y,r) \Big\} }
=
\lim_{a \to 0}
\frac{1}{Z_\gamma^2}
\Expect{ F(\Lc_D^\theta) \mathbf{1}\Big\{x \overset{\Lc_D^\theta \cup \Xi_{a}^x \cup \Xi_{a}^y}{\longleftrightarrow} y \Big\} }.
\end{equation}
\end{proposition}

\begin{notation}\label{N:covariance}
If we want to keep track of the domain $D$, we will denote by $C_{\theta,D}(x,y)$ the limit of \eqref{E:P_crossing_limit}.
\end{notation}

We can now show Theorems \ref{T:convergenceL2}, \ref{T:cluster2point} and \ref{T:cluster_xy_doob}.

\begin{proof}[Proof of Theorem \ref{T:cluster2point}]
It follows directly from \eqref{E:P_law}. Indeed,
let $x, y \in D$ be two distinct points.
By a density-type argument, to prove Theorem \ref{T:cluster2point}, it is enough to show that for any $\eps >0$ and any bounded measurable function $F: \mathfrak{L} \to \R$ which does not depend on loops with diameter smaller than $\eps$,
$
\EXPECT{x \leftrightarrow y,r}{F(\Lc_\D^\theta)}
$
converges as $r \to 0$. This is a mere restatement of the convergence \eqref{E:P_law}.
\end{proof}

\begin{proof}[Proof of Theorem \ref{T:convergenceL2}]
Let $f:D \to \R$ be a bounded measurable function. By decomposing $f$ into its positive and negative parts, we can assume without loss of generality that $f$ is nonnegative. We want to show that $((h_\gamma,f), \gamma \in (0,\sqrt{2})$ is Cauchy in $L^2$, i.e. that
$
\Expect{ | (h_{\gamma_1},f) - (h_{\gamma_2},f) |^2 } \to 0
$
as $\gamma_1, \gamma_2 \to 0$.
By Lemma \ref{L:second_moment_to_crossing_proba}, this second moment is equal to
\[
o(1) + \int_{D \times D} f(x) f(y) \Bigg( \frac{\P \Big( x \overset{\Lc_D^\theta \cup \Xi_{a_1}^x \cup \Xi_{a_1}^y}{\longleftrightarrow} y \Big)}{Z_{\gamma_1}^2} + \frac{\P \Big( x \overset{\Lc_D^\theta \cup \Xi_{a_2}^x \cup \Xi_{a_2}^y}{\longleftrightarrow} y \Big)}{Z_{\gamma_2}^2} - 2 \frac{\P \Big( x \overset{\Lc_D^\theta \cup \Xi_{a_1}^x \cup \Xi_{a_2}^y}{\longleftrightarrow} y \Big)}{Z_{\gamma_1} Z_{\gamma_2}} \Bigg).
\]
Proposition \ref{P:crossing} together with dominated convergence theorem and Lemma \ref{L:crossing_upper} (and Fatou's lemma for the crossed term involving both $a_1$ and $a_2$) then shows that
$
\limsup_{\gamma_1, \gamma_2 \to 0} \Expect{ | (h_{\gamma_1},f) - (h_{\gamma_2},f) |^2 } \leq 0
$.
This concludes the $L^2$ convergence of $(h_\gamma,f)$.

We now explain how one can lift this convergence to a convergence in the Sobolev space $H^{-\eps}(\C)$, where $\eps \in (0,1)$ is arbitrary small. Recall that $H^{-\eps}(\C)$ is the set of tempered distribution $f$ such that
\begin{equation}
\label{E:Sobolev}
\norme{f}_{H^{-\eps}(\C)}^2 := \int_\C (1+|\xi|^2)^{-\eps} |\hat{f}(\xi)|^2 \d \xi < \infty.
\end{equation}
Let $\gamma_1, \gamma_2 \in (0,\sqrt{2})$. Denoting $C_{\gamma_1,\gamma_2}(x,y) = \Expect{ (h_{\gamma_1}(x) - h_{\gamma_2}(x))(h_{\gamma_1}(y) - h_{\gamma_2}(y)) }$, we have
\begin{align}
\label{E:Sobolev2}
\E \norme{h_{\gamma_1}-h_{\gamma_2}}_{H^{-\eps}(\C)}^2 & = \int_{\R^2} \frac{\E \abs{\hat{h}_{\gamma_1}(\xi) - \hat{h}_{\gamma_2}(\xi)}^2}{(1+\abs{\xi}^2)^\eps} \d \xi
\leq \int_{\R^2} \frac{\E \abs{\hat{h}_{\gamma_1}(\xi) - \hat{h}_{\gamma_2}(\xi)}^2}{\abs{\xi}^{2\eps}} \d \xi \\
& = \int_{D \times D} \d x \d y ~C_{\gamma_1,\gamma_2}(x,y) \int \d \xi ~\frac{e^{2\pi i \xi \cdot (x-y)}}{\abs{\xi}^{2\eps}}
= C_\eps \int_{D \times D} \d x \d y ~\frac{C_{\gamma_1,\gamma_2}(x,y)}{|x-y|^{2-2\eps}}
\nonumber
\end{align}
where we used the fact that the Fourier transform of $|\cdot|^{-2\eps}$ is $C_\eps |\cdot|^{-2+2\eps}$ for some constant $C_\eps$ depending on $\eps$.
We can then conclude as before using Proposition \ref{P:crossing} that $h_\gamma$ is Cauchy in $L^2(H^{-\eps}(\C), \P)$.

Finally, the proof that the normalising constant $Z_\gamma$ \eqref{E:Zgamma} satisfies $Z_\gamma = \gamma^{2(1-\theta)+o(1)}$ as $\gamma \to 0$ can be found in Section \ref{SS:thick_loop}.
\end{proof}

\begin{proof}[Proof of Theorem \ref{T:cluster_xy_doob}]
By a simple density-type argument, it is enough to check the identity statement in Theorem \ref{T:cluster_xy_doob} for a function $F: \mathfrak{L} \to \R$ depending only on loops with diameter larger than a given threshold. For such functions, it then follows from Lemma \ref{L:second_moment_to_crossing_proba}, Theorem \ref{T:cluster2point} and \eqref{E:P_law}.
\end{proof}

\subsection{Proof of Lemma \ref{L:second_moment_to_crossing_proba}}

\begin{proof}[Proof of Lemma \ref{L:second_moment_to_crossing_proba}]
By definition of $h_\gamma$, the left hand side of \eqref{E:L_second_moment_to_crossing_proba} is equal to
\begin{align*}
\frac{1}{Z_\gamma Z_{\gamma'}} \int_{D \times D}
\Expect{ F(x,y,\Lc_D^\theta) (\Mc_a^+(\d x) \Mc_{a'}^+(\d y) - \Mc_a^+(\d x) \d y - \d x \Mc_{a'}^+(\d y) + \d x \d y) }.
\end{align*}
By \eqref{E:first_moment_measure},
\[
\Expect{ F(x,y,\Lc_D^\theta) \Mc_a^+(\d x)} = \Expect{F(x,y,\Lc_D^\theta \cup \{\Xi_{a_i}^x \}_{i \geq 1})} \d x
\]
where the loops $\{\Xi_{a_i}^x \}_{i \geq 1}$ are as defined above Lemma \ref{L:moment_measure}. A similar result holds at $y$.
To handle the term $\Expect{ F(x,y,\Lc_D^\theta) (\Mc_a^+(\d x) \Mc_{a'}^+(\d y)}$, we first control the contribution of soups containing at least one loop which visits both $x$ and $y$. By triangle inequality and then by
\eqref{E:second_moment_measure1} and \eqref{E:second_moment_measure2} applied to $F \equiv 1$, we have
\begin{align}
\nonumber
& \abs{ \Expect{F(x,y,\Lc_D^\theta) \Mc_a^+(\d x) \Mc_{a'}^+(\d y) \indic{\exists \wp \in \Lc_D^\theta: x,y \in \wp}} }
\leq \norme{F}_\infty \Expect{ \Mc_a(\d x) \Mc_{a'}(\d y) \indic{\exists \wp \in \Lc_D^\theta: x,y \in \wp}} \\
& = 4 \norme{F}_\infty \left( \left( 2 \pi \sqrt{aa'} G_D(x,y) \right)^{1-\theta} \Gamma(\theta) I_{\theta-1} \left(4\pi \sqrt{aa'} G_D(x,y) \right) - 1 \right) \d x \d y
= O(\gamma^2 {\gamma'}^2).
\label{E:proof_loop_2points}
\end{align}
On the other hand, by the Girsanov-type transform \eqref{E:second_moment_measure2}, the contribution with no such loop is equal to
\begin{align*}
& \Expect{F(x,y,\Lc_D^\theta) \Mc_a^+(\d x) \Mc_a^+(\d y) \indic{\nexists \wp \in \Lc_D^\theta: x,y \in \wp}} \\
& = 4 \E \Big[ F(x,y,\Lc_D^\theta \cup \{\Xi_{a_i}^x \}_{i \geq 1} \cup \{\Xi_{a_i'}^y \}_{i \geq 1}) \mathbf{1} \{ \Cf^+ (\Lc_D^\theta \cup \{\Xi_{a_i}^x \}_{i \geq 1} \cup \{\Xi_{a_i'}^y \}_{i \geq 1}) \} \Big] \d x \d y.
\end{align*}
We then notice that, conditionally on $\Lc_D^\theta \cup \{\Xi_{a_i}^x \}_{i \geq 1} \cup \{\Xi_{a_i'}^y \}_{i \geq 1}$, the probability that both $x$ and $y$ belong to positive clusters of $\Lc_D^\theta \cup \{\Xi_{a_i}^x \}_{i \geq 1} \cup \{\Xi_{a_i'}^y \}_{i \geq 1}$ is equal to
\begin{align*}
&\frac14 \Big(1- \mathbf{1} \Big\{ x \overset{\Lc_D^\theta \cup \{\Xi_{a_i}^x \}_{i \geq 1} \cup \{\Xi_{a_i'}^y \}_{i \geq 1}}{\longleftrightarrow} y \Big\} \Big) + \frac12 \mathbf{1} \Big\{ x \overset{\Lc_D^\theta \cup \{\Xi_{a_i}^x \}_{i \geq 1} \cup \{\Xi_{a_i'}^y \}_{i \geq 1}}{\longleftrightarrow} y \Big\}\\
= &\frac{1}{4} + \frac{1}{4} \mathbf{1} \Big\{ x \overset{\Lc_D^\theta \cup \{\Xi_{a_i}^x \}_{i \geq 1} \cup \{\Xi_{a_i'}^y \}_{i \geq 1}}{\longleftrightarrow} y \Big\}.
\end{align*}
Putting everything together, we have obtained that the left hand side of \eqref{E:L_second_moment_to_crossing_proba} is equal to
\begin{align*}
& \frac{1}{Z_\gamma Z_{\gamma'}} \int_{D \times D} \Expect{ F(x,y,\Lc_D^\theta \cup \{\Xi_{a_i}^x \}_{i \geq 1} \cup \{\Xi_{a_i'}^y \}_{i \geq 1}) \mathbf{1} \Big\{ x \overset{\Lc_D^\theta \cup \{\Xi_{a_i}^x \}_{i \geq 1} \cup \{\Xi_{a_i'}^y \}_{i \geq 1}}{\longleftrightarrow} y \Big\} } \d x \d y \\
& + O(1) \frac{\gamma^2 {\gamma'}^2}{Z_\gamma Z_{\gamma'}} + \frac{1}{Z_\gamma Z_{\gamma'}} \int_{D \times D} \E [ F(x,y,\Lc_D^\theta \cup \{\Xi_{a_i}^x \}_{i \geq 1} \cup \{\Xi_{a_i'}^y \}_{i \geq 1}) - F(x,y,\Lc_D^\theta \cup \{\Xi_{a_i}^x \}_{i \geq 1}) \\
& ~~~~~~~~~~~~~~~~~~~~~~~~~~~~~~~~~~~~~~~~ - F(x,y,\Lc_D^\theta \cup \{\Xi_{a_i'}^y \}_{i \geq 1}) + F(x,y,\Lc_D^\theta) ] \d x \d y.
\end{align*}
To bound the last integral above, we need to recall that $F$ does not depend on loops with diameter smaller than a given threshold $\eps$. Hence, if $\wedge_{i \geq 1} \Xi_{a_i}^x$ (concatenation of $\Xi_{a_i}^x, i \geq 1$)  or $\wedge_{i \geq 1} \Xi_{a_i'}^y$ has diameter smaller than $\eps$, then
\[
F(x,y,\Lc_D^\theta \cup \{\Xi_{a_i}^x \}_{i \geq 1} \cup \{\Xi_{a_i'}^y \}_{i \geq 1}) - F(x,y,\Lc_D^\theta \cup \{\Xi_{a_i}^x \}_{i \geq 1}) - F(x,y,\Lc_D^\theta \cup \{\Xi_{a_i'}^y \}_{i \geq 1}) + F(x,y,\Lc_D^\theta) = 0.
\]
The probability that both $\wedge_{i \geq 1} \Xi_{a_i}^x$ and $\wedge_{i \geq 1} \Xi_{a_i'}^y$ have diameter at least $\eps$ is of order $\gamma^2 {\gamma'}^2$ (Lemma \ref{L:diameter_thick}) concluding the proof.
\end{proof}

\subsection{Proof of Lemma \ref{L:crossing_upper} and Proposition \ref{P:crossing}}

We start by proving Lemma \ref{L:crossing_upper}.

\begin{proof}[Proof of Lemma \ref{L:crossing_upper}]
Let $\eta >0$ be small and $x,y \in D$. If $|x-y| \leq e^{-a^{-1+\eta}}$, we simply bound the probability in  \eqref{E:P_crossing_bound} by 1. Using that $Z_\gamma = \gamma^{2(1-\theta)+o(1)} = a^{1-\theta + o(1)}$, we obtain that the left hand side of \eqref{E:P_crossing_bound} is at most
\[
a^{-2(1-\theta) + o(1)} \leq C |\log |x-y||^{2(1-\theta)/(1-2\eta)}.
\]
We therefore only need to consider the case where $x$ and $y$ are not too close to each other, i.e. $|x-y| \geq e^{-a^{-1+\eta}}$. 
For $z \in \{x,y\}$, we will denote by $D_z = D(z, |x-y|/2)$.
By Lemma \ref{L:diameter_thick}, the minimal disc centred at $x$ containing $\Xi_a^x$ has typically a radius of order $e^{- O(1) /a}$. It is therefore very unlikely that $\Xi_a^x$ exits $D_x$. We are first going to show that this is still the case conditionally on the existence of a cluster joining $x$ and $y$:
\begin{equation}
\label{E:exitDx}
\frac{1}{Z_\gamma^2} \Prob{ x \overset{\Lc_D^\theta \cup \Xi_a^x \cup \Xi_a^y}{\longleftrightarrow} y, \exists z \in \{x,y\}: \Xi_a^z \not\subset D_z } \to 0 
\quad \quad \text{as~} \gamma \to 0
\end{equation}
and also that the left hand side can be bounded by $C |\log |x-y||^{\frac{2(1-\theta)}{1-2\eta}}$, uniformly in $\gamma$.
Thanks to a union bound, we can focus on the event that $\Xi_a^x \not\subset D_x$.
With some abuse of notation, we will view $\Xi_a^x$ both as the collection of excursions in a Poisson point process with intensity $a \mu_D^{x,x}$ and as the loop formed by the concatenation of all these excursions. 
Let $E_x^1$ (resp. $E_x^2$) be the event that there is a unique excursion (resp. at least two distinct excursions) of $\Xi_a^x$ that exits $D_x$.
By Lemma \ref{L:diameter_thick},
$
\mu_D^{x,x}( \{ \wp \not\subset D_x \} ) = \log \frac{2 \CR(x,D)}{|x-y|}.
$
So
\[
\Prob{E_x^2} \leq
\E \Big[ \sum_{\wp \neq \wp' \in \Xi_a^x} \indic{\wp \not\subset D_x, \wp' \not\subset D_x } \Big]
= a^2 \mu_D^{z,z}( \{ \wp \not\subset D_x \} )^2 \leq C a^2 |\log |x-y||^2.
\]
In that case, we simply bound
\[
\frac{1}{Z_\gamma^2} \Prob{ x \overset{\Lc_D^\theta \cup \Xi_a^x \cup \Xi_a^y}{\longleftrightarrow} y, E_x^2 }
\leq a^{-2(1-\theta)+o(1)} \Prob{E_x^2} \leq C a^{2\theta+o(1)} |\log |x-y||^2
\leq C |\log |x-y||^{\frac{2(1-\theta)}{1-2\eta}}.
\]
The probability of the event $E_x^1$ is at most $C a |\log |x-y||$. On that event, let $R_x$ be the minimal distance between $y$ and the unique excursion of $\Xi_a^x$ that exists $D_x$. Let also $R_y$ be the maximal $r>0$ such that there exists a cluster of loops of $\Lc_D^\theta \cup \Xi_a^y$ that connects $y$ to $\partial D(y,r)$.
On the event $E_1^x$, in order to have a cluster joining $x$ and $y$, we need to have $R_x \leq R_y$.
Since $R_y$ and $R_x$ are independent, we get that
\begin{align}
\label{E:pf_q6}
\frac{1}{Z_\gamma^2} \Prob{ x \overset{\Lc_D^\theta \cup \Xi_a^x \cup \Xi_a^y}{\longleftrightarrow} y, E_x^1 }
& \leq \frac{\Prob{E_x^1}}{Z_\gamma^2} \Prob{R_x \leq R_y \vert E_x^1} \\
& \leq \frac{C a |\log |x-y||}{Z_\gamma^2} \int_0^{|x-y|} \Prob{R_y \in \d r} \Prob{R_x \leq r \vert E_x^1}.
\nonumber
\end{align}
For all $r \in (0,|x-y|)$, the probability that $R_x$ is smaller than $r$ can be compared with the probability that a Brownian path starting on the circle $\partial D_x$ hits $D(y,r)$ before exiting the domain $D$. Hence for all $r \in (0,|x-y|)$,
\[
\Prob{ R_x \leq |x-y| \Big\vert E_x^1} \leq C \frac{|\log|x-y||}{|\log r|}.
\]
Injecting this estimate in \eqref{E:pf_q6} and then integrating by parts leads to
\[
\frac{1}{Z_\gamma^2} \Prob{ x \overset{\Lc_D^\theta \cup \Xi_a^x \cup \Xi_a^y}{\longleftrightarrow} y, E_x^1 }
\leq \frac{C a |\log |x-y||^2}{Z_\gamma^2} \int_0^{|x-y|} \frac{\Prob{R_y \geq r}}{r |\log r|^2} \d r.
\]
Now, by Lemma \ref{L:Rgamma} (and conformal invariance to get back to the unit disc), 
\[
\frac{1}{Z_\gamma} \Prob{ R_y \geq r }
\leq C |\log r|^{1-\theta+\eta}, \quad \quad r \in (0,|x-y|).
\]
We thus have
\[
\frac{1}{Z_\gamma^2} \Prob{ x \overset{\Lc_D^\theta \cup \Xi_a^x \cup \Xi_a^y}{\longleftrightarrow} y, E_x^1 }
\leq \frac{C a |\log |x-y||^2}{Z_\gamma} \int_0^{|x-y|} \frac{\d r}{r |\log r|^{1+\theta-\eta}}
\leq C a^{\theta+o(1)} |\log |x-y||^{2-\theta+\eta}.
\]
In the last inequality we used that $a/Z_\gamma = a^{\theta+o(1)}$.
This concludes the proof of \eqref{E:exitDx}. This also shows that the left hand side of \eqref{E:exitDx} is bounded by $C |\log |x-y||^{\frac{2(1-\theta)}{1-2\eta}}$ because $|x-y| \geq e^{-a^{-1+\eta}}$.

We can now work on the event that $\Xi_a^z \subset D_z$ for $z=x,y$.
Let $\Lc_{D,x,y}^\theta$ be the subset of $\Lc_D^\theta$ consisting of the loops that are not included in $D_x \cup D_y$. Let
\[
r_z = \inf\{ r>0: \exists \wp \in \Lc_{D,x,y}^\theta: \wp \cap D(z,r) \neq \varnothing \},
\quad \quad z=x,y.
\]
By the restriction property of $\Xi_a^x$ (see Lemma \ref{L:restriction_thick_loop}),
\begin{align*}
& \frac{1}{Z_\gamma^2} \Prob{ x \overset{\Lc_D^\theta \cup \Xi_a^x \cup \Xi_a^y}{\longleftrightarrow} y, \forall z=x,y, \Xi_a^z \subset D_z}
\leq \frac{1}{Z_\gamma^2} \Prob{ \forall z=x,y, z \overset{\Lc_{D_z}^\theta \cup \Xi_a^z}{\longleftrightarrow} \partial D(z,r_z), \Xi_a^z \subset D_z } \\
& = \frac{1}{Z_\gamma^2} \frac{\CR(x,D_x)^a \CR(y,D_y)^a}{\CR(x,D)^a \CR(y,D)^a} \Prob{ \forall z=x,y, z \overset{\Lc_{D_z}^\theta \cup \Xi_a^{z,D_z}}{\longleftrightarrow} \partial D(z,r_z)}.
\end{align*}
In what follows, we will simply bound the ratio of conformal radii by 1 (which is also a good approximation since $a \to 0$).
By scale invariance of $\Lc_{D_z}^\theta$ and $\Xi_a^{z,D_z}$ and by Lemma \ref{L:Rgamma},
\[
\frac{1}{Z_\gamma} \Prob{ z \overset{\Lc_{D_z}^\theta \cup \Xi_a^{z,D_z}}{\longleftrightarrow} \partial D(z,r_z) \Big\vert r_z } \leq C \left( 1 + \log \frac{|x-y|}{2r_z} \right)^{1-\theta+\eta},
\quad \quad z=x,y.
\]
Hence,
\begin{align*}
\frac{1}{Z_\gamma^2} \Prob{ x \overset{\Lc_D^\theta \cup \Xi_a^x \cup \Xi_a^y}{\longleftrightarrow} y, \forall z=x,y, \Xi_a^z \subset D_z}
\leq C \Expect{ \prod_{z=x,y} \left( \log \frac{|x-y|}{2r_z} \right)^{1-\theta+\eta} }.
\end{align*}
In \cite{JLQ23a}, we show that the above quantity is upper bounded by $C|\log |x-y||^{2(1-\theta)+2\eta}$ (see (4.16) in Remark 4.5 therein) concluding the proof of Lemma \ref{L:crossing_upper}.
\end{proof}

We can now prove Proposition \ref{P:crossing}.

\begin{proof}[Proof of Proposition \ref{P:crossing}]
Let $x$ and $y$ be two distinct points in $D$. Compare to the upper bound, the two points $x$ and $y$ are fixed and we wish to take the limit as $\gamma_1, \gamma_2 \to 0$. It is therefore important for the proof to keep in mind that $x$ and $y$ are at macroscopic distance to each other and that we can assume $\gamma_1$ and $\gamma_2$ to be as small as desired. We start by proving \eqref{E:P_crossing_limit}. We will explain at the end of the proof what needs to be changed in order to get \eqref{E:P_crossing_limit2} and \eqref{E:P_law}.

Let $\eta >0$ be small and let $D_x = D(x,\eta|x-y|/2)$ and $D_y = D(y,\eta|x-y|/2)$. By taking $\eta >0$ small enough, we can ensure that $D_x$ and $D_y$ are subsets of $D$. We will denote $\Lc_{D,x,y}^\theta = \Lc_D^\theta \setminus (\Lc_{D_x}^\theta \cup \Lc_{D_y}^\theta)$ the collection of loops which are not included in $D_x$ or $D_y$.
We now introduce the following events:
\begin{itemize}
\item $E_1$: there is a loop in $\Lc_{D,x,y}^\theta$ that disconnects $D_x$ from $D_y$;
\item $E_2$: for $z = x,y$, $\Xi_a^z \subset D_z$.
\end{itemize}
We first claim that we can work on $E_1 \cap E_2$. Indeed, by Lemma \ref{L:surround} $\Prob{E_1^c} \to 0$ as $\eta \to 0$ and, by FKG inequality,
\[
\limsup_{\gamma \to 0} \frac{1}{Z_{\gamma}^2}
\Prob{ x \overset{\Lc_D^\theta \cup \Xi_{a}^x \cup \Xi_{a}^y}{\longleftrightarrow} y, E_1^c } \to 0 \quad \quad \text{as} \quad \eta \to 0.
\]
Moreover, in the proof of the upper bound, we showed that
\[
\lim_{\gamma \to 0} \frac{1}{Z_{\gamma}^2}
\Prob{ x \overset{\Lc_D^\theta \cup \Xi_{a}^x \cup \Xi_{a}^y}{\longleftrightarrow} y, E_2^c } = 0.
\]
See \eqref{E:exitDx}.
We can therefore work on the event $E_1 \cap E_2$. On the event $E_1$, there is at most one cluster $\Cc_{x,y}$ of loops in $\Lc_{D,x,y}^\theta$ that intersects both $D_x$ and $D_y$ (we set $\Cc_{x,y} = \varnothing$ if there is no such cluster).
By the restriction property \eqref{E:restriction_thick_loop} of $\Xi_a^z$, we have
\begin{align*}
& \Prob{ x \overset{\Lc_D^\theta \cup \Xi_a^x \cup \Xi_a^y}{\longleftrightarrow} y, E_1\cap E_2}
 = \frac{\CR(x,D_x)^a \CR(y,D_y)^a}{\CR(x,D)^a \CR(y,D)^a} \Expect{ \mathbf{1}_{E_1} \prod_{z=x,y} \Prob{z \overset{\Lc_{D_z}^\theta \cup \Xi_a^{z,D_z}}{\longleftrightarrow} \Cc_{x,y} \Big\vert \Cc_{x,y} }  }.
\end{align*}
The ratio of conformal radii converges to 1. Moreover, for $z = x$ or $y$, we can rewrite
\[
\frac{1}{Z_\gamma}
\P \Big( z \overset{\Lc_{D_z}^\theta \cup \Xi_a^{z,D_z}}{\longleftrightarrow} \Cc_{x,y} \vert \Cc_{x,y} \Big)
=
\P \Big( z \overset{\Lc_{D_z}^\theta \cup \Xi_a^{z,D_z}}{\longleftrightarrow} \Cc_{x,y} \vert \Cc_{x,y} \Big)
\Big/ \P \Big( z \overset{\Lc_{D_z}^\theta \cup \Xi_a^{z,D_z}}{\longleftrightarrow} \partial D(z, e^{-1}\eta |x-y|/2) \Big)
\]
and we can replace
$\Cc_{x,y}$ by the closure of $(\Cc_{x,y} \cap D_z) \cup \partial D_z$ which is a.s. a connected compact subset of $\overline{D_z} \setminus \{z\}$. We are now in the setting of Theorem \ref{T:ratio_general} which shows that the above ratio converges as $a \to 0$, a.s. with respect to $\Cc_{x,y}$.
The upper bound \eqref{E:P_crossing_bound} provides the bound necessary to apply dominated convergence in order to exchange the expectation w.r.t. $\Cc_{x,y}$ and the limit. We have proved that
\begin{equation}
\label{E:proof_P_crossing4}
\lim_{\gamma \to 0} \frac{1}{Z_\gamma^2} \Prob{ x \overset{\Lc_D^\theta \cup \Xi_a^x \cup \Xi_a^y}{\longleftrightarrow} y, E_1 \cap E_2}
\end{equation}
exists and that
$
\frac{1}{Z_\gamma^2} \P \Big( x \overset{\Lc_D^\theta \cup \Xi_a^x \cup \Xi_a^y}{\longleftrightarrow} y \Big)
$
converges as $\gamma \to 0$ towards the nondecreasing limit of \eqref{E:proof_P_crossing4} as $\eta \to 0$. This shows \eqref{E:P_crossing_limit}.

The proof of \eqref{E:P_crossing_limit2} concerning the mixed case $(a_1,a_2)$ follows along the same lines. The main difference comes from the fact that we only need a lower bound, so we can add for free the extra events $E_1$ and $E_2$.
The proof of \eqref{E:P_law} is obtained similarly. Indeed, because the function $F$ does not depend on loops that have diameter smaller than a given threshold $\eps$, $F$ does not depend on loops included in $D_x$ and $D_y$ provided $\eta$ is small enough. We then repeat the above procedure and use Theorem \ref{T:ratio_general}.

Finally, the upper bound of \eqref{E:P_crossing_limit3} follows from Lemma \ref{L:crossing_upper}. The lower bound is much easier to prove and follows from Theorem \ref{T:large_crossing} and FKG inequality. We omit the details.
\end{proof}


%% file: subfiles/minkowski_new.tex
This section is devoted to the proofs of Theorems \ref{T:measure_cluster} and \ref{T:h_and_minkowski}. We start by proving Theorem \ref{T:measure_cluster}. Recall the notations introduced above Theorem \ref{T:measure_cluster}.

\begin{proof}[Proof of Theorem \ref{T:measure_cluster}]
Let $k \geq 1$.
Let $B \subset D$ be some Borel set. It is enough to show that $(\mu_{k,r}(B))_{r >0}$ and $(\mu_{k,\gamma}(B))_{\gamma >0}$ are Cauchy in $L^2$ and that $\mu_{k,r}(B) - \mu_{k,\gamma}(B)$ converges to 0 in $L^2$ as $\gamma, r \to 0$. We already have all the ingredients to prove this. Indeed,
\begin{align*}
& \limsup_{r,r' \to 0}
\Expect{ (\mu_{k,r}(B) - \mu_{k,r'}(B))^2 }
= \limsup_{r,r' \to 0} \frac{1}{Z_r Z_{r'}} \int_{B \times B} \Big( \Prob{ \Cc_k \cap D(x,r) \neq \varnothing, \Cc_k \cap D(y,r) \neq \varnothing } \\
&~~~ + \Prob{ \Cc_k \cap D(x,r') \neq \varnothing, \Cc_k \cap D(y,r') \neq \varnothing } - 2 \Prob{ \Cc_k \cap D(x,r) \neq \varnothing, \Cc_k \cap D(y,r') \neq \varnothing }\Big) \d x \d y.
\end{align*}
The analogue of Lemma \ref{L:crossing_upper} with small discs instead of thick loops (and Fatou's lemma to deal with the mixed term $r$-$r'$) provides the necessary domination to exchange limit and integral. To show that the integrand converges pointwise to zero, we then use the same approach as in the proof of Proposition \ref{P:crossing}. Eventually, the proof of Theorem \ref{T:measure_cluster} boils down to Theorem \ref{T:ratio_general}. We omit the details.

The above line of argument also shows that the second moment of $\mu_k(D)$ is positive. In particular, the probability that $\mu_k(D)>0$ is positive. We are going to show that this probability is actually equal to 1.
Without loss of generality assume that $D = [0,1]^2$. 
Let $\wp$ be some fixed macroscopic loop of $\Lc_D^\theta$, e.g. the loop with the largest diameter. Let $n \geq 1$ be a large integer and divide $D$ into disjoint squares $Q_i$, $i = 1, \dots, 4^n$, of side length $2^{-n}$. For each $i=1, \dots, 4^n$, let $\Lc_i$ be the set of loops included in $Q_i$. Conditionally on $\{ \text{diam}(\wp) > \sqrt{2} 2^{-n} \}$, the collections $\Lc_i$, $i=1, \dots, 4^n$, are independent from $\wp$ and independent loop soups in $D_i$.
For $i=1, \dots, 4^n$, let $Q_i'$ be the square with side length $2^{-n-1}$ with the same centre as $Q_i$ and let $E_i$ be the event that there is a cluster $\Cc_i$ of $\Lc_i$ with non-zero Minkowski content and such that $\Cc_i$ disconnects $Q_i'$ from $\partial Q_i$. By scale invariance, $p = \P(E_i)$ is positive and independent of $n$ and $i$.
Let $I = \{ i=1, \dots, 4^n: \wp \cap Q_i' \neq \varnothing \}$. On the event $E_i \cap \{i \in I\}$, the cluster of $\wp$ contains $\Cc_i$. Hence, the probability that the Minkowski content of the cluster of $\wp$ vanishes is at most
\[
\Prob{ \text{diam}(\wp) \leq \sqrt{2} 2^{-n} } + \Expect{(1-p)^{\# I} }
\xrightarrow[n \to \infty]{} 0.
\]
This concludes the proof.
\end{proof}

The main item that remains to be checked is that, as stated in Theorem \ref{T:h_and_minkowski}, $\mu_k$ corresponds to the restriction of $|h_\theta|$ to $\Cc_k$. The rest of this section is dedicated to showing this.

As usually, we will denote by $\overline{\Cc}_{k}$ the topological closure of
$\Cc_{k}$, seen as a compact subset of $D$.
We will denote
\begin{displaymath}
D_{k} = D\setminus\bigcup_{1\leq j\leq k}\overline{\Cc}_{j}.
\end{displaymath}
Then $D_{k}$ is an open subset of $D$ that has infinitely many connected components,
among which finitely many (at most $k$) are multiply connected and the rest are simply connected, the total number of holes being $k$.
We will denote by
\begin{displaymath}
\Lc^{\theta}_{D,k} = \{\wp\in \Lc^{\theta}_{D}: \wp \subset D_{k}\}.
\end{displaymath}
We will denote by $\sigma_{k}$ the sign $\sigma_{\Cc_{k}}$ of the cluster
$\Cc_{k}$.
We will denote by $\Fc_{k}$ the $\sigma$-algebra generated by the loops contained in
$\Cc_{1},\dots,\Cc_{k}$ (including their time parametrization)
and the signs $\sigma_{1},\dots, \sigma_{k}$.

\begin{lemma}
\label{Lem cond loop soup D k}
Fix $k\geq 1$.
Let $O_{i}$, $i\geq 0$ be the connected components of $D_{k}$.
Conditionally on $\Fc_{k}$, we have:
\begin{itemize}
\item
The collections of loops $\{ \wp \in \Lc_D^\theta: \wp \subset O_i \}, i \geq 0$, are independent;
\item
For all $i \geq 0$ such that $O_{i}$ is simply connected, 
$\{ \wp \in \Lc_D^\theta: \wp \subset O_i \}$ has the law of a Brownian loop soup in $O_i$ with intensity $\theta$;
\item
For all $i \geq 0$ such that $O_{i}$ is not simply connected, 
$\{ \wp \in \Lc_D^\theta: \wp \subset O_i \}$ has the law of a Brownian loop soup in $O_i$ with intensity $\theta$ conditioned on not having clusters surrounding an inner hole of $O_i$.
\end{itemize}
\end{lemma}

\begin{proof}
This result is folklore and has a very similar flavour as Lemma \ref{L:condition_cluster} and follows from the same proof. We refer to Section \ref{S:discrete} for more details.
\end{proof}

In particular, the collection of loops $\Lc^{\theta}_{D,k}$ is a Brownian loop soup in
$D_{k}$ conditioned on an event that has a positive probability.
So for $\gamma\in (0,2)$, one can define a multiplicative chaos measure associated to
$\Lc^{\theta}_{D,k}$. 
The fact that $D_k$ is not simply connected is non essential: see Remark \ref{rmk:multiplicative_chaos}.
The multiplicative chaos measure associated to $\Lc^{\theta}_{D,k}$
is nothing else than
\begin{displaymath}
\indic{x\in D_{k}}e^{\gamma^{2}\pi(G_{D}(x,x)-G_{D_{k}}(x,x))}\Mc_\gamma(dx),
\end{displaymath}
where $G_{D}$ is the Dirichlet Green's function on $D$ 
and $G_{D_{k}}$ is the Dirichlet Green's function on $D_{k}$.
Note that $G_{D}(x,x)-G_{D_{k}}(x,x)<+\infty$.
We get a density and not just a restriction to $D_{k}$ because our convention
$\E[\Mc_\gamma(dx)] = 2 dx$ makes the normalization for the multiplicative chaos domain-dependent: how we normalize around a point $x$ depends also on the global shape of the domain and not just the Brownian loops we see around $x$. With the normalisation used in \cite{ABJL21}, we would not get this extra density: see \eqref{E:difference_normalisation} for the relation between these two normalisations.
One can further consider the measure $\Mc_\gamma^{+}\mathbf{1}_{D_{k}}$
and the field
\begin{displaymath}
h_{\theta, k}
= \lim_{\gamma\to 0^{+}}
\dfrac{1}{Z_{\gamma}}\indic{x\in D_{k}}(e^{\gamma^{2}\pi(G_{D}(x,x)-G_{D_{k}}(x,x))}
\Mc_\gamma^{+}(dx) - dx),
\end{displaymath}
which is a well defined random element of $H^{-\varepsilon}(\C)$
for $\varepsilon>0$ (simply by working in the domain $D_k$).

Further, one can consider the conditional expectation
$\mathbb{E}[h_{\theta}\vert \Fc_{k}]$.
Since $\mathbb{E}[\Vert h_{\theta}\Vert^{2}_{H^{-\varepsilon}(\C)}] <+\infty$
for $\varepsilon>0$,
the conditional expectation $\mathbb{E}[h_{\theta}\vert \Fc_{k}]$
is a well defined random element of $H^{-\varepsilon}(\C)$,
with
\begin{displaymath}
\Vert \mathbb{E}[h_{\theta}\vert \Fc_{k}]\Vert^{2}_{H^{-\varepsilon}(\C)}
\leq 
\mathbb{E}[\Vert h_{\theta}\Vert^{2}_{H^{-\varepsilon}(\C)}\vert \Fc_{k}]
< +\infty ~~\text{a.s.}
\end{displaymath}
See \cite[Section 2.6]{hytonen2016analysis} for a reference on conditional expectation of random variables in a Banach space.
We define $\nu_{1} = \sigma_{1}\mathbb{E}[h_{\theta}\vert \Fc_{1}]$,
and for $k\geq 2$,
\begin{displaymath}
\nu_{k} = \sigma_{k}(\mathbb{E}[h_{\theta}\vert \Fc_{k}]
-\mathbb{E}[h_{\theta}\vert \Fc_{k-1}]),
\quad \quad \text{so~that} \quad \quad
\mathbb{E}[h_{\theta}\vert \Fc_{k}] = 
\sum_{j=1}^{k}\sigma_{j}\nu_{j}.
\end{displaymath}

\begin{lemma}
The following holds.
\begin{enumerate}
\item For every $k\geq 1$,
conditionally on $\Fc_{k}$, $h_{\theta,k}$ has the same law as $-h_{\theta,k}$,
and in particular
$\mathbb{E}[h_{\theta, k}\vert \Fc_{k}] = 0$.
\item A.s., for every $k\geq 1$
and every $f\in\Cc^{\infty}_{c}(\C)$, such that
$\operatorname{Supp}(f) \cap (D \setminus D_k) = \varnothing$,
$(h_{\theta},f) = (h_{\theta,k},f)$.
\item For every $k\geq 1$, 
$h_{\theta} = \mathbb{E}[h_{\theta}\vert \Fc_{k}] + h_{\theta, k}$ a.s.
\item For every $\varepsilon>0$, as $k\to +\infty$,
$\Vert h_{\theta}-\mathbb{E}[h_{\theta}\vert \Fc_{k}]\Vert_{H^{-\varepsilon}(\C)}
\to 0$ a.s. and in $L^2$.
\item For every $k\geq 2$,
$h_{\theta, k-1} = \sigma_{k}\nu_{k} + h_{\theta, k}$ a.s. and
\begin{displaymath}
\nu_{k} = \sigma_{k}\mathbb{E}[h_{\theta, k-1}\vert \Fc_{k}]~~\text{a.s.}
\end{displaymath}
\item $\nu_{1}$ and $\sigma_{1}$ are independent,
and for every $k\geq 2$,
$\nu_{k}$ and $\sigma_{k}$ are independent conditionally on $\Fc_{k-1}$.
\item A.s., for every $k\geq 1$
and every $f\in\Cc^{\infty}_{c}(\C)$ such that
$\operatorname{Supp}(f)\cap \overline{\Cc}_{k} = \varnothing$,
$(\nu_{k},f) = 0$,
that is to say $\nu_{k}$ is supported on $\overline{\Cc}_{k}$.
\end{enumerate}
\end{lemma}

\begin{proof}
1. The proof is the same as for $h_{\theta}$ (see Section \ref{SS:properties_field}).

2. Note that we also deal with test functions that are not compactly supported in $D$.
Let $k\geq 1$ and $n\geq 1$.
Let $\C_{k,n}$ be the open subset of $\C$
formed by points $x\in \C$ such that
$d(x,D\setminus D_{k})>1/n$.
By construction, for every $f\in\Cc^{\infty}_{c}(\C_{k,n})$,
\begin{displaymath}
(h_{\theta},f) = 
\lim_{\gamma\to 0^{+}}
\dfrac{1}{Z_{\gamma}}\int_{D_{k}}f(x)(\Mc_\gamma^{+}(dx) - dx),
\end{displaymath}
where the convergence holds in $L^2$ for a given Sobolev norm
$H^{-\varepsilon}(\C)$
(and thus for all test functions simultaneously).
Further, since $G_{D}(x,x)-G_{D_{k}}(x,x)$ is bounded on $\C_{k,n}$,
the convergence and the limit do no change if we replace $\Mc_\gamma^{+}$
by
\begin{displaymath}
e^{\gamma^{2}\pi(G_{D}(x,x)-G_{D_{k}}(x,x))}
\Mc_\gamma^{+}(dx).
\end{displaymath}

\medskip

3. By 1., we only need to show that $h_{\theta} - h_{\theta,k}$ is
measurable w.r.t. $\Fc_{k}$.
Actually will see that $h_{\theta} - h_{\theta,k}$ is
measurable w.r.t. the $\sigma$-algebra $\Fc_{k}$ augmented by the negligible events (i.e. with probability $0$),
which also implies the desired result.

Consider the sets $\C_{k,n}$ defined above. 
Let $\Fc_{k,n}$ be the $\sigma$-algebra generated by $\Fc_{k}$
as well as by the loops of $\Lc^{\theta}_{D,k}$ that are in the clusters at distance at most $1/n$ from $D\setminus D_{k}$, and the signs of these clusters.
By definition, $\Fc_{k}\subset \Fc_{k,n}$ and $\Fc_{k,n+1}\subset \Fc_{k,n}$.
Denote	
\begin{displaymath}
\Fc_{k,\infty} = \bigcap_{n\geq 1}\Fc_{k,n}.
\end{displaymath}
We have $\Fc_{k}\subset \Fc_{k,\infty}$.

For $k\geq 1$ and $n\geq 1$, let $\chi_{k,n}$ be a cutoff function
such that $\chi_{k,n}\in\Cc^{\infty}(\C)$, $\chi_{k,n}$ takes values in $[0,1]$,
$\chi_{k,n}$ equals $0$ on $\C_{k,n}$ and $1$ on $\C\setminus \C_{k,n+1}$.
We also want $\chi_{k,n}$ to be, as a random variable, measurable w.r.t.
$D_{k}$.
We omit the detailed construction of such a cutoff function, which is standard.
Fix $k\geq 1$ and $n\geq 1$. For $m\geq 1$ we have
$h_{\theta} = h_{\theta}\chi_{k,m}+h_{\theta}(1-\chi_{k,m})$,
and by 2., $h_{\theta}(1-\chi_{k,m}) = h_{\theta,k}(1-\chi_{k,m})$.
Moreover, for $m\geq n$, $h_{\theta}\chi_{k,m}$
is measurable w.r.t. $\Fc_{k,n}$.
Further, for $\varepsilon>0$,
\begin{displaymath}
\lim_{m\to +\infty}
\mathbb{E}[\Vert h_{\theta,k}\chi_{k,m}\Vert_{H^{-\varepsilon}(\C)}^{2}
\vert \Fc_{k}] = 0 ~~ \text{a.s.}
\end{displaymath}
by a computation similar to \eqref{E:Sobolev2}.
Therefore,
\begin{displaymath}
\lim_{m\to +\infty}
\mathbb{E}[\Vert(h_{\theta} - h_{\theta,k})- h_{\theta}\chi_{k,m}
\Vert_{H^{-\varepsilon}(\C)}^{2} \vert \Fc_{k}] = 0 ~~ \text{a.s.}
\end{displaymath}
So in particular, $h_{\theta} - h_{\theta,k}$ is the limit in probability of $h_{\theta}\chi_{k,m}$ as $m\to +\infty$.
Therefore, $h_{\theta} - h_{\theta,k}$ is measurable w.r.t. $\Fc_{k,n}$.
By letting $n\to +\infty$, we get that
$h_{\theta} - h_{\theta,k}$ is measurable w.r.t. $\Fc_{k,\infty}$.

To conclude, it is enough to check that the $\sigma$-algebra $\Fc_{k,\infty}$
is the same as $\Fc_{k}$ up to negligible events.
This is what is sketched below, by omitting some details.
\begin{itemize}[leftmargin=*]
\item Step 1: Let $\Fc'_{k,n}$ be the $\sigma$-algebra generated by $\Fc_{k}$
as well as by the loops of $\Lc^{\theta}_{D,k}$ that are in the clusters at distance at most $1/n$ from $D\setminus D_{k}$.
Contrary to $\Fc_{k,n}$, $\Fc'_{k,n}$ does not contain any information on the
sign of clusters of $\Lc^{\theta}_{D,k}$. Denote
\begin{displaymath}
\Fc'_{k,\infty} = \bigcap_{n\geq 1}\Fc'_{k,n}.
\end{displaymath}
By construction, $\Fc_{k}\subset \Fc'_{k,\infty}\subset \Fc_{k,\infty}$.
We further claim that $\Fc_{k,\infty}$ is the same as
$\Fc'_{k,\infty}$ up to negligible events.
Indeed, the signs of clusters are sampled independently
of the loops in $\Lc^{\theta}_{D,k}$ and are i.i.d.
Moreover, as $n$ increases, $\Fc_{k,n}$ contains information about less clusters and their signs, and in particular, $\Fc_{k,\infty}$ contains
information only on the tail of the sequence of signs.
So essentially this follows from the $0-1$ law of the tails of sequences of 
independent r.v.s.
\item Step 2: Let $\Fc''_{k,n}$ be the $\sigma$-algebra generated by $\Fc_{k}$
as well as by the loops of $\Lc^{\theta}_{D,k}$ that are at distance at most $1/n$ from $D\setminus D_{k}$.
So contrary to $\Fc'_{k,n}$, we look only at individual loops and do not consider how they are grouped into clusters.
By construction, $\Fc''_{k,n}\subset \Fc'_{k,n}$.
Denote
\begin{displaymath}
\Fc''_{k,\infty} = \bigcap_{n\geq 1}\Fc''_{k,n}.
\end{displaymath}
We claim that $\Fc'_{k,\infty}$ is the same as
$\Fc''_{k,\infty}$ up to negligible events. 
For $m>n\geq 1$, let $A_{k,m,n}$ be the event that all the clusters of
$\Lc^{\theta}_{D,k}$ that are at distance at most $1/m$ from $D\setminus D_{k}$,
are actually made of loops that stay in the $1/n$-neighborhood of
$D\setminus D_{k}$.
Then $A_{k,m,n}\in \Fc''_{k,n}$ and for every $A\in \Fc'_{k,m}$,
we have $A\cap A_{k,m,n}\in \Fc''_{k,n}$. Moreover,
\begin{displaymath}
\lim_{m\to +\infty}\mathbb{P}(A_{k,m,n}^{\rm c}) = 0.
\end{displaymath}
Thus, $\Fc'_{k,\infty}$ is contained in $\Fc''_{k,n}$ augmented by the negligible
events, whatever the value of $n$.
\item Step 3: We have that $\Fc_{k}\subset \Fc''_{k,\infty}$.
We claim that $\Fc''_{k,\infty}$ is the same as $\Fc_{k}$ up to negligible events.
If the conditional distribution of $\Lc^{\theta}_{D,k}$ given $\Fc_{k}$
were exactly Poisson, this would have been immediate by $0-1$ law.
However, by Lemma \ref{Lem cond loop soup D k},
this conditional distribution is Poisson conditioned on an event with positive probability, so the $0-1$ law still holds.
\end{itemize}

\medskip

4. This is simply due to the fact that the $\sigma$-algebra
$\bigvee_{k\geq 1}\Fc_{k}$ generated by the union of the $\Fc_{k}$
is the total $\sigma$-algebra generated by the loop soup
$\Lc^{\theta}_{D}$ and the signs clusters.

5. Since
\begin{displaymath}
\mathbb{E}[h_{\theta}\vert \Fc_{k-1}] + h_{\theta, k-1}
=\mathbb{E}[h_{\theta}\vert \Fc_{k-1}] + \sigma_{k}\nu_{k} + h_{\theta, k}
= h_{\theta},
\end{displaymath}
we get that $\sigma_{k}\nu_{k} = h_{\theta, k-1}-h_{\theta, k}$.
Since $\sigma_{k}\nu_{k}$ is $\Fc_{k}$-measurable
and $\mathbb{E}[h_{\theta, k}\vert \Fc_{k}] = 0$,
we conclude by taking the conditional expectation w.r.t. $\Fc_{k}$.

6. Since $-h_{\theta}$ has the same law as $h_{\theta}$, we have that
\begin{displaymath}
\nu_{1}\indic{\sigma_{1} = -1}
=
-\mathbb{E}[h_{\theta}\vert \Fc_{1}]\indic{\sigma_{1} = -1}
=\mathbb{E}[-h_{\theta}\vert \Fc_{1}]\indic{\sigma_{1} = -1}
\stackrel{\text{(d)}}{=}\mathbb{E}[h_{\theta}\vert \Fc_{1}]\indic{\sigma_{1} = 1}
=\nu_{1}\indic{\sigma_{1} = 1},
\end{displaymath}
which gives the independence between $\nu_{1}$ and $\sigma_{1}$.
For $k\geq 2$, we use the same reasoning, by relying on points 5. and 1.

7. From 3. and 2. follows that a.s. for every 
$f\in\Cc^{\infty}_{c}(\C\setminus\overline{\Cc}_{1})$,
$(\mathbb{E}[h_{\theta}\vert\Fc_{1}],f)=0$.
So this concludes in the case $k=1$.
Now take $k\geq 2$. By 5., for every 
$f\in\Cc^{\infty}_{c}(\C)$, 
\begin{displaymath}
\sigma_{k}(\nu_{k},f) = (h_{\theta, k-1}-h_{\theta, k},f).
\end{displaymath}
Then, similarly to 2., we get that
$(h_{\theta, k-1},f) = (h_{\theta, k},f)$
for every $f\in \Cc^{\infty}_{c}(\C\setminus\overline{\Cc}_{k})$.
\end{proof}

By convention, we set $D_{0}=D$.
Let $D_{k}^{\ast}$ be the open subset of $D_{k}$ obtained as the union of the
non-simply connected components of $D_{k}$, which are finitely many.


\begin{lemma}
\label{Lem remove multiply}
The measure
\begin{equation}
\label{Eq cond measure multiply}
\dfrac{1}{Z_{\gamma}}
(\indic{x\in D_{k}^{\ast}}dx-
\mathbb{E}[\mathbf{1}_{x\in D_{k}^{\ast}}e^{\gamma^{2}\pi (G_{D}(x,x) - G_{D_{k-1}}(x,x))}\Mc_{\gamma}^{+}(dx))\vert \Fc_{k}]) 
\end{equation}
is a.s. a positive finite measure that converges a.s. to $0$ for the weak topology of measures on
$\overline{D_{k}^{\ast}}$ as $\gamma \to 0$.
\end{lemma}

\begin{proof}
We start by proving that \eqref{Eq cond measure multiply} is a positive measure.
Let $\Lc^{\theta}_{D_{k}^{\ast}}$ be the collection of Brownian loops,
distributed conditionally on $D_{k}$,
as the Brownian loop soup in $D_{k}^{\ast}$ with intensity $\theta$
(that is without any topological conditioning).
Let $\Mc_{D_{k}^{\ast},\gamma}$ denote the multiplicative chaos
of $\Lc^{\theta}_{D_{k}^{\ast}}$, normalized by
$\mathbb{E}[\Mc_{D_{k}^{\ast},\gamma}\vert D_{k}]
= \indic{x\in D_{k}^{\ast}} 2 dx$.
Let $E_{k}$ be the event that no cluster of loops in $\Lc^{\theta}_{D_{k}^{\ast}}$
surrounds an inner hole of $D_{k}^{\ast}$,
which is an event with positive probability.
Then
\begin{displaymath}
\mathbb{E}[\mathbf{1}_{x\in D_{k}^{\ast}}
e^{\gamma^{2}\pi (G_{D}(x,x) - G_{D_{k-1}}(x,x))}\Mc_{\gamma}^{+}(dx))\vert \Fc_{k}]
=\dfrac{1}{2}
\mathbb{E}\Big[
e^{\gamma^{2}\pi(G_{D_{k}}(x,x)-G_{D_{k-1}}(x,x))}
\Mc_{D_{k}^{\ast},\gamma}(dx)
\Big\vert D_{k}^{\ast}, E_{k}\Big].
\end{displaymath}
Note that the Dirichlet Green's function on $D_{k}^{\ast}$
is the restriction of $G_{D_{k}}$ to
$D_{k}^{\ast}\times D_{k}^{\ast}$.
The event $E_{k}$ is a decreasing event.
Therefore, by FKG inequality for Poisson point processes, a.s.
\begin{multline*}
\mathbb{E}[\mathbf{1}_{x\in D_{k}^{\ast}}
e^{\gamma^{2}\pi (G_{D}(x,x) - G_{D_{k-1}}(x,x))}\Mc_{\gamma}^{+}(dx))\vert \Fc_{k}]
\\
\leq 
\dfrac{1}{2}
\mathbb{E}\Big[
e^{\gamma^{2}\pi(G_{D_{k}}(x,x)-G_{D_{k-1}}(x,x))}
\Mc_{D_{k}^{\ast},\gamma}(dx)
\Big\vert D_{k}^{\ast}\Big]
=
\indic{x\in D_{k}^{\ast}}
e^{\gamma^{2}\pi(G_{D_{k}}(x,x)-G_{D_{k-1}}(x,x))} dx .
\end{multline*}
Thus, for every $\gamma\in (0,2)$, 
\eqref{Eq cond measure multiply} is a positive finite measure.

The rest of the proof is then dedicated to showing that \eqref{Eq cond measure multiply} goes to zero as $\gamma \to 0$. Define $f(x) = e^{2\pi(G_{D_{k}}(x,x)-G_{D_{k-1}}(x,x))}$.
We can write the total mass of \eqref{Eq cond measure multiply} as
\begin{equation}
\label{Eq cond measure multiply1}
\dfrac{1}{Z_{\gamma}} \int_{D_{k}^{\ast}}
(1-e^{\gamma^{2}\pi(G_{D_{k}}(x,x)-G_{D_{k-1}}(x,x))}) dx
-
\dfrac{1}{2 Z_{\gamma} \mathbb{P}(E_{k})}
\mathbb{E}\Big[\Big((\Mc_{D_{k}^{\ast},\gamma},f^{\gamma^{2}/2})
-2\int_{D_{k}^{\ast}}f^{\gamma^{2}/2}\Big) \mathbf{1}_{E_k}
\Big\vert D_{k}^{\ast}\Big].
\end{equation}
We are going to show separately that each of these two terms vanishes as $\gamma \to 0$. By Cauchy--Schwarz and then by \eqref{E:second_moment_measure1}, the absolute value of the second term is at most
\begin{align*}
& \dfrac{1}{2 Z_{\gamma} \mathbb{P}(E_{k})^{1/2}}
\mathbb{E}\Big[\Big((\Mc_{D_{k}^{\ast},\gamma},f^{\gamma^{2}/2})
-2\int_{D_{k}^{\ast}}f^{\gamma^{2}/2}\Big)^{2}
\Big\vert D_{k}^{\ast}\Big]^{1/2} \\
& = \frac{2}{Z_\gamma \P(E_k)^{1/2}} \left( \int_{D_{k}^{\ast}\times D_{k}^{\ast}}
\big((\gamma^{2}\pi G_{D_{k}^{\ast}}(x,y))^{1-\theta}
\Gamma(\theta)I_{\theta -1}(2\gamma^{2}\pi G_{D_{k}^{\ast}}(x,y))
-1\big) f(x)^{\gamma^{2}/2} f(y)^{\gamma^{2}/2}\,dx\,dy \right)^{1/2} \\
& = O(\gamma^2 / Z_\gamma).
\end{align*}
The last estimate can be obtained by expanding $I_{\theta -1}$ (see \eqref{E:BesselI}).
Hence the second term of \eqref{Eq cond measure multiply1} vanishes as $\gamma \to 0$. 

%
%

To conclude the proof, it remains to show that
\begin{equation}
\label{Eq int to 0 dim}
\dfrac{1}{Z_{\gamma}} \int_{D_{k}^{\ast}}
(1-e^{\gamma^{2}\pi(G_{D_{k}}(x,x)-G_{D_{k-1}}(x,x))}) dx
\to 0
\quad \quad \text{as} \quad \gamma \to 0.
\end{equation}
For fixed $x\in D_{k}^{\ast}$, the density
$
1-e^{\gamma^{2}\pi(G_{D_{k}}(x,x)-G_{D_{k-1}}(x,x))})
$
is of order $\gamma^{2}$.
Since $Z_{\gamma}\gg \gamma^{2}$, the density converges to $0$,
actually uniformly on compact subsets of $D_{k}^{\ast}$.
However, this is not enough for our purpose and we need to control what happens
near the boundary $\partial D_{k}^{\ast}$.

The boundary $\partial D_{k}^{\ast}$ has exactly $k+1$ connected components,
one of them being $\partial D$, and the $k$ others being the
outer boundaries of $\Cc_{1},\dots,\Cc_{k}$.
$\partial D$ and $\Cc_1, \dots, \Cc_{k-1}$ lie outside both $D_{k-1}$ and $D_k$. Thus $G_{D_{k-1}}(x,x)-G_{D_{k}}(x,x)$ stays bounded in a neighborhood of
$\partial D \cup \Cc_{1} \cup \dots \cup \Cc_{k-1}$. Therefore,
the uniform convergence of the density to $0$
is also true on compact neighborhoods of these $k$ boundary components of
$\partial D_{k}^{\ast}$.
On the other hand, $\Cc_k$ is contained in $D_{k-1}$ and lies outside of $D_k$. We will therefore have to treat this case separately.

Let us denote by $\partial_{o}\Cc_{k}$ the outer boundary of $\Cc_k$.
Near $\partial_{o}\Cc_{k}$, $G_{D_{k-1}}(x,x)- G_{D_{k}}(x,x)$ explodes to $+\infty$.
But there is a constant $c(D_{k-1})\in (0,1)$, depending on $D_{k-1}$, 
such that for every $x\in D_{k}^{\ast}$,
\begin{displaymath}
e^{2\pi(G_{D_{k}}(x,x)-G_{D_{k-1}}(x,x))}\geq 
c(D_{k-1})\dfrac{d(x,\partial D_{k}^{\ast})}{d(x,\partial D_{k-1})}.
\end{displaymath}
This is in a way a generalization of Koebe quarter theorem to domains with holes; see \cite[Lemma 5.13]{ALS1}.
Fix $\varepsilon_{0}\in 
(0, d(\partial_{o}\Cc_{k},\partial D_{k}^{\ast}\setminus\partial_{o}\Cc_{k}))$.
Then for every $x\in D_{k}^{\ast}$ such that
$d(x,\partial_{o}\Cc_{k})<1\wedge(\varepsilon_{0}/2)$, we have
\begin{multline*}
\dfrac{1}{Z_{\gamma}}(1-e^{\gamma^{2}\pi(G_{D_{k}}(x,x)-G_{D_{k-1}}(x,x))})
\leq 
\dfrac{1}{Z_{\gamma}}\Big(
1
-c(D_{k-1})^{\gamma^{2}/2}
\dfrac{d(x,\partial D_{k}^{\ast})^{\gamma^{2}/2}}{d(x,\partial D_{k-1})^{\gamma^{2}/2}}
\Big)
\\\leq
\dfrac{1}{Z_{\gamma}}\Big(
1
-
c(D_{k-1})^{\gamma^{2}/2}
\dfrac{d(x,\partial_{o}\Cc_{k})^{\gamma^{2}/2}}{\diam(D)^{\gamma^{2}/2}}
\Big)
=
\dfrac{1}{Z_{\gamma}}(
1
-
d(x,\partial_{o}\Cc_{k})^{\gamma^{2}/2}
)
+
\dfrac{1}{Z_{\gamma}}\Big(
1- \dfrac{c(D_{k-1})^{\gamma^{2}/2}}{\diam(D)^{\gamma^{2}/2}}
\Big)
d(x,\partial_{o}\Cc_{k})^{\gamma^{2}/2}
\end{multline*}
Since
\begin{displaymath}
\lim_{\gamma\to 0} \dfrac{1}{Z_{\gamma}}\Big(
1- \dfrac{c(D_{k-1})^{\gamma^{2}/2}}{\diam(D)^{\gamma^{2}/2}}
\Big) = 0,
\end{displaymath}
we only need to consider the term
$(1-d(x,\partial_{o}\Cc_{k})^{\gamma^{2}/2})/Z_{\gamma}$.
Further,
\begin{equation}
\label{Eq Mink bound 1}
\dfrac{1}{Z_{\gamma}}\int_{d(x,\partial_{o}\Cc_{k})<1}
(1-d(x,\partial_{o}\Cc_{k})^{\gamma^{2}/2})\,dx
=
\dfrac{\gamma^{2}}{2Z_{\gamma}}\int_{0}^{1} dr\,
r^{\gamma^{2}/2-1}
\int_{d(x,\partial_{o}\Cc_{k})<r}dx
\end{equation}
Since $\partial_{o}\Cc_{k}$ is an SLE$_{\kappa(\theta)}$ type curve,
it is of dimension $1+\kappa(\theta)/8\leq 3/2<2$ \cite{BeffaradimSLE}.
In particular, for $\beta\in (0,1/2)$, there is a.s. a (random) constant
$c(\beta)>0$ such that for every $r\in (0,1)$,
\begin{displaymath}
\int_{d(x,\partial_{o}\Cc_{k})<r}dx \leq c(\beta)r^{\beta}.
\end{displaymath}
Thus, \eqref{Eq Mink bound 1} is bounded by
\begin{displaymath}
c(\beta)\dfrac{\gamma^{2}}{2Z_{\gamma}}\int_{0}^{1}
r^{\gamma^{2}/2 + \beta -1}\, dr,
\end{displaymath}
which tends to $0$ as $\gamma$ to $0$. This concludes.
\end{proof}

%
%

We now have all the ingredients to prove Theorem \ref{T:h_and_minkowski}.

\begin{proof}[Proof of Theorem \ref{T:h_and_minkowski}]
The only thing that remains to be proved is that the random generalised function 
$\nu_{k}$ agrees a.s. with the measure $\mu_k$ from Theorem \ref{T:measure_cluster}. Let $f\in\Cc^{\infty}_{c}(\C)$. It is enough to show that
\begin{equation}
\label{Eq nu k limit}
(\nu_{k},f) = \lim_{\gamma \to 0}
\dfrac{1}{Z_{\gamma}}\int_{d(x,\Cc_{k})<1}
f(x) (1-d(x,\Cc_{k})^{\gamma^{2}/2}) \,dx
,
\end{equation}
in probability.
We have that
\begin{displaymath}
(\nu_{k},f) = \sigma_{k}\mathbb{E}[(h_{\theta, k-1},f)\vert\Fc_{k}],
\end{displaymath}
where by convention, $h_{\theta, 0} = h_{\theta}$.
Since $\nu_{k}$ and $\sigma_{k}$ are independent conditionally on $\Fc_{k}$,
we get that
\begin{equation}
\label{Eq nu k f}
(\nu_{k},f) = -\mathbb{E}[(h_{\theta, k-1},f)\vert\Fc_{k},\{\sigma_{k} = -1\}]
\end{equation}
Since
\begin{displaymath}
(h_{\theta, k-1},f) = 
\lim_{\gamma\to 0}
\dfrac{1}{Z_{\gamma}}
\int_{D_{k-1}}
f(x)(
e^{\gamma^{2}\pi (G_{D}(x,x) - G_{D_{k-1}}(x,x))}
\Mc^{+}_{\gamma}(dx) 
- dx),
\end{displaymath}
where the convergence is a.s. in $L^2$ conditionally on
$\Fc_{k-1}$
(that is to say $\lim_{\gamma\to 0}\mathbb{E}[(\cdot)^{2}\vert \Fc_{k-1}] =0$ a.s.).
Therefore, in \eqref{Eq nu k f} one can interchange the conditional expectation and the limit in $\gamma$ and get
\begin{displaymath}
(\nu_{k},f) = 
\lim_{\gamma\to 0}
\dfrac{1}{Z_{\gamma}}
\Big(\int_{D_{k-1}} f(x)\,dx
-\mathbb{E}\Big[
\int_{D_{k-1}}
 f(x)e^{\gamma^{2}\pi (G_{D}(x,x) - G_{D_{k-1}}(x,x))}
\Mc^{+}_{\gamma}(dx) 
\Big\vert
\Fc_{k},\{\sigma_{k} = -1\}
\Big]
\Big),
\end{displaymath}
where the convergence is a.s. in $L^2$ conditionally on $\Fc_{k-1}$.
By Theorem \ref{T:large_crossing}, $\overline{\Cc}_{k}$ has a.s. zero Lebesgue measure (the expectation of the Lebesgue measure of its $r$-neighbourhood vanishes as $r \to 0$). So $\int_{D_{k-1}} f(x)\,dx = \int_{D_k} f(x)\,dx$ a.s.
Moreover,
since on the event $\{\sigma_{k} = -1\}$,
$\indic{D_{k-1}}\Mc^{+}_{\gamma} = \indic{D_{k}}\Mc^{+}_{\gamma}$,
and since $\sigma_{k}$ is independent from $\indic{D_{k}}\Mc^{+}_{\gamma}$
conditionally on $\Fc_{k} $,
we get that
\begin{displaymath}
(\nu_{k},f) = 
\lim_{\gamma\to 0}
\dfrac{1}{Z_{\gamma}}
\Big(\int_{D_{k}} f(x)\,dx
-\mathbb{E}\Big[
\int_{D_{k}}
 f(x)e^{\gamma^{2}\pi (G_{D}(x,x) - G_{D_{k-1}}(x,x))}
\Mc^{+}_{\gamma}(dx) 
\Big\vert
\Fc_{k}
\Big]
\Big).
\end{displaymath}

Now we partition $D_{k} = D_{k}^{0} \cup D_{k}^{\ast}$,
where $D_{k}^{0}$ is the union of the simply connected components of
$D_{k}$ (there are infinitely many of them) and $D_{k}^{\ast}$
is the union of the multiply connected components of
$D_{k}$ (there are finitely many of them).
Then Lemma \ref{Lem remove multiply} ensures that
\begin{displaymath}
(\nu_{k},f) = 
\lim_{\gamma\to 0}
\dfrac{1}{Z_{\gamma}}
\Big(\int_{D_{k}^{0}} f(x)\,dx
-\mathbb{E}\Big[
\int_{D_{k}^{0}}
 f(x)e^{\gamma^{2}\pi (G_{D}(x,x) - G_{D_{k-1}}(x,x))}
\Mc^{+}_{\gamma}(dx) 
\Big\vert
\Fc_{k}
\Big]
\Big).
\end{displaymath}
Since conditionally on $\Fc_{k}$,
the loops in $D_{k}^{0}$ form a Poisson point process without conditioning,
we get that
\begin{displaymath}
\mathbb{E}\Big[
\int_{D_{k}^{0}}
 f(x)e^{\gamma^{2}\pi (G_{D}(x,x) - G_{D_{k-1}}(x,x))}
\Mc^{+}_{\gamma}(dx) 
\Big\vert
\Fc_{k}
\Big]
=
\int_{D_{k}^{0}}
f(x)e^{\gamma^{2}\pi (G_{D_{k}}(x,x) - G_{D_{k-1}}(x,x))}\,dx.
\end{displaymath}
The function $G_{D_{k-1}}(x,x)-G_{D_{k}}(x,x)$ is bounded for $x$
away from $\Cc_{k}$ and explodes when $x$ approaches $\partial\Cc_{k}$.
Further, \cite[Lemma 5.13]{ALS1} provides a comparison with the Euclidean
distance:
there is a constant $c(D_{k-1})\in (0,1)$
depending on $D_{k-1}$
such that for every $x\in D_{k}^{0}$,
\begin{displaymath}
c(D_{k-1})\dfrac{d(x,\partial D_{k}^{0})}{d(x,\partial D_{k-1})}
\leq e^{2\pi (G_{D_{k}}(x,x) - G_{D_{k-1}}(x,x))}
\leq 4 \dfrac{d(x,\partial D_{k}^{0})}{d(x,\partial D_{k-1})}.
\end{displaymath}
Then using these bounds,
similarly to the proof of \eqref{Eq int to 0 dim},
we get that
\begin{displaymath}
(\nu_{k},f) = 
\lim_{\gamma\to 0}
\dfrac{1}{Z_{\gamma}}
\int_{\substack{x\in D_{k}^{0}\\d(x,\Cc_{k})<1}}f(x)(1-d(x,\Cc_{k})^{\gamma^{2}/2})\,dx .
\end{displaymath}
Since the dimension of the outer boundary of
$\Cc_{k}$ is strictly smaller than $2$,
we get that (see the proof of \eqref{Eq int to 0 dim})
\begin{displaymath}
(\nu_{k},f) = 
\lim_{\gamma\to 0}
\dfrac{1}{Z_{\gamma}}
\int_{\substack{d(x,\Cc_{k})<1}}f(x)(1-d(x,\Cc_{k})^{\gamma^{2}/2})\,dx ,
\end{displaymath}
which concludes.
\end{proof}

%% file: subfiles/boundary_excursions.tex
In this section, we will obtain estimates about the excursions induced by the Brownian loops that touch the outer boundary of a loop soup cluster.

Let $\Lc$ be a loop soup with intensity $\theta\in (0,1/2]$ in the unit disk $\D$, and let $\Cc_0$ be the outermost cluster in $\Lc$ that encircles the origin. Let $O$ be the domain encircled by the outermost boundary of $\Cc_0$.
We will denote by $\Lc_O$ the subset of $\Lc$ consisting in loops included in $\overline{O}$.
The set of loops in $\Lc_O$ that touch $\partial O$ induce a collection $\Ec(\Lc_O)$ of excursions in $O$ with endpoints in $\partial O$. If one maps $O$ onto $\Hb$ by some conformal map $f$, then we say that $f(\Lc_O)$ is a \emph{loop soup in $\Hb$ with wired boundary conditions}. It was shown in \cite{QianWerner19Clusters} that the law of $f(\Lc_O)$ is conformally invariant, and does not depend on the choice of $f$. Moreover, the following properties were proved in \cite{QianWerner19Clusters, qian2018}.
\begin{theorem}[\cite{QianWerner19Clusters, qian2018}]
\label{thm:qw19}
The law of $f(\Lc_O)$ is independent from $\partial O$ and the distribution of $\Lc$ outside of $O$.
Inside $f(\Lc_O)$ (a loop soup in $\Hb$ at intensity  $\theta\in (0,1/2]$ with wired boundary conditions), the set of loops that touch $\Rb$ are independent from the set of loops that stay in $\Hb$. The latter set is distributed as a loop soup in $\Hb$ (with free boundary conditions). Let $\Ec$ be the collection of excursions induced by the loops in  $f(\Lc_O)$ that touch $\Rb$. Then at $\theta=1/2$, $\Ec$ is distributed as a Poisson point process of excursions of intensity $1/4$.
\end{theorem}

As one knows precisely the law of $\Ec$ for $\theta=1/2$, much less is known for the subcritical intensities $\theta\in(0,1/2)$. However, \cite{MR3901648} proved that the excursions attached to the boundary of a cluster satisfy a conformal restriction property that we describe now (we give some background on conformal restriction measures in Section \ref{SS:conformal_restriction}).
One can explore the outer boundaries of the outermost clusters in $\Lc$ progressively, instead of discovering an entire cluster at once, using the CLE exploration \cite{SheffieldWernerCLE}. This exploration process allows one to define a \emph{loop soup wired on a portion of its boundary} (see \cite{MR3901648} for a precise definition). Such a configuration was shown to be conformally invariant (with two marked points on the boundary). As a consequence, it is sufficient to consider a loop soup in the upper half plane $\Hb$ which is wired on $\Rb^-$ and free on $\Rb^+$.

\begin{theorem}[\cite{MR3901648}]\label{thm:restriction}
Let $\Lc_{\Rb^-}$ be a loop soup in $\Hb$ with intensity $\theta\in(0,1/2]$ which is wired on $\Rb^-$ and free on $\Rb^+$. The set of loops in $\Lc_{\Rb^-}$ that touch $\Rb^-$  are independent from the set of loops that stay in $\Hb$. The latter set is distributed as a loop soup in $\Hb$ (with free boundary conditions). Let $\Ec_{\Rb^-}$ be the collection of excursions induced by the loops in  $\Lc_{\Rb^-}$ that touch $\Rb$. Then the filling of $\Ec_{\Rb^-}$ (namely the complement in $\C$ of the unbounded connected component of $\C\setminus K$, where $K$ is the closure of the union of all the excursions in $\Ec_{\Rb^-}$) satisfies one-sided restriction in $\Hb$ with exponent $\alpha(\theta):=(6-\kappa)/(2\kappa)$, where $\kappa\in(8/3,4]$ is related to $\theta$ via \eqref{E:kappa}.
\end{theorem}

Throughout this section, we use the following notations:

\begin{notation}\label{N:exc}
For any interval $I\subset \Rb$, denote by $\Lc_I$ a loop soup in $\Hb$ wired on $I$ and free elsewhere. We will use $\Ec_I$ to denote the boundary-touching excursions induced by $\Lc_I$.

For $z\in \Hb$, let $f_z(u)= (u-z)/(u-\bar z)$ be the conformal map from $\Hb$ onto $\D$ that sends $z$ to $0$. For all $r\in(0,1)$, let $U(z,r):=f_z^{-1}(D(0,r))$.
\end{notation}

Our main result on the boundary excursions in a cluster with any intensity $\theta\in(0,1/2]$ is encapsulated in Proposition \ref{prop:be_main} below. This proposition will be crucial to define and study the field $h_\theta$ with wired boundary condition.

\begin{proposition}\label{prop:be_main}
Let $\Lc_\Rb$ be a loop soup in $\Hb$ with intensity $\theta\in (0,1/2]$ with wired boundary conditions. Let $\Ec_\Rb$ be the Brownian excursions in $\Hb$ attached on $\Rb$, induced by $\Lc_\Rb$.
For all $\eta>0$, there exists $C>0$, such that the following holds. For all $z\in\Hb$ and $r\in(0,1-\eta)$, we have
\begin{align}\label{propeq:one-point}
\Pb[\Ec_\Rb \cap U(z,r)\not=\emptyset] \le C |\log r|^{-1}.
\end{align}
For all $z_1, z_2\in\Hb$ and $r_1, r_2 \in (0,1-\eta)$, we have
\begin{align}\label{propeq:two-point}
\Pb[\Ec_\Rb \cap U(z_1,r_1)\not=\emptyset, \Ec_\Rb \cap U(z_2, r_2)\not=\emptyset]
\le C (1+ G_\Hb(z_1,z_2)) |\log r_1|^{-1} |\log r_2|^{-1}.
\end{align}
\end{proposition}

In the special case of $\theta=1/2$,
Proposition~\ref{prop:be_main} can be obtained directly using the Poisson point process description of $\Ec_\Rb$.
We would like to emphasise that, contrary to the restriction property stated in Theorem \ref{thm:restriction}, Proposition \ref{prop:be_main} is not concerned with the outer boundary of $\Ec_\R$ and requires a deeper understanding of these excursions. 
Our approach to prove Proposition \ref{prop:be_main} is to extract information from the restriction property (Theorem~\ref{thm:restriction}), and transform this property about the outer boundary of $\Ec_\Rb$ to a property about the endpoints of the excursions in $\Ec_\Rb$.
More concretely, we deduce the following identities which are of independent interest.



\begin{proposition}\label{prop:integral}
For each excursion $e$, let $a(e)$ and $b(e)$ be the left and right endpoints of $e$. For all $r>0$, we have
\begin{align*}
\Eb\left[\sum_{e\in\Ec_{\Rb^-}} \frac{(b(e) -a(e))^2 r^2}{(r-b(e))^2(r-a(e))^2} \right]=\alpha(\theta)
\end{align*} 
where $\alpha(\theta)$ is the restriction exponent from Theorem \ref{thm:restriction}.
\end{proposition}

A similar but more complicated ``second-moment'' result is also stated in Proposition~\ref{prop:sum_mm2}. Note that Propositions~\ref{prop:integral} and~\ref{prop:sum_mm2} are exact identities, not only bounds. Nevertheless, a complete description of the laws of $\Ec_{\Rb^-}$ and $\Ec_{\Rb}$ remains an open question.

\begin{remark}
In Section~\ref{subsec:bexr-}, we prove Propositions~\ref{prop:integral} and~\ref{prop:sum_mm2} using only two properties of the set $\Ec_{\Rb^-}$: it satisfies conformal restriction (Theroem~\ref{thm:restriction}) and is a locally finite point process of Brownian excursions (Lemma~\ref{lem:p_loc}). Therefore, Propositions~\ref{prop:integral} and~\ref{prop:sum_mm2} also hold for any locally finite point process $\Ec$ of Brownian excursions in $\Hb$ with endpoints on $\Rb^-$, such that the filling of $\Ec$ satisfies one-sided conformal restriction. Then in Sections~\ref{subsec:wired} and~\ref{subsec:proof}, we prove Proposition~\ref{prop:be_main} using the additional property that we can obtain $\Ec_\Rb$ by exploring the outer boundary of $\Ec_{\Rb^-}$ in a Markovian way.
\end{remark}

\subsection{Properties and estimates about restriction measures}\label{SS:conformal_restriction}
We will recall some properties of chordal restriction measures introduced in \cite{MR1992830}, and deduce some estimates which will be useful later. 

Fix a simply connected domain $D$ which is not the whole plane. Fix two points $a, b\in \partial D$. Let $\Omega$ be the collection of all simply connected compact sets $K\subset \overline D$, such that $K\cap \partial D=\{a, b\}$. Let $\Qc$ be the collection of all compact sets $A\subset\overline D$ such that $D\setminus A$ is simply connected and $a, b \not\in A$. For all $A\in \Qc$, let $f_A$ be a conformal map from $D\setminus A$ onto $D$ that leaves $a, b$ fixed.
A probability law on $K\in \Omega$ is said to satisfy chordal conformal restriction if
\begin{enumerate}
\item The law of $K$ is invariant under conformal maps from $D$ to itself that fixes $a, b$.
\item For any $A\in \Qc$, conditionally on $K \cap A=\emptyset$, $f_A(K)$ is distributed as $K$.
\end{enumerate}
It was shown in \cite{MR1992830} that the family of chordal restriction measures is characterized by one parameter $\alpha\ge 5/8$. A chordal restriction measure in $D$ with parameter $\alpha$ satisfies
\begin{align}\label{eq:restriction1}
\Pb(K\cap A=\emptyset) =(f_A'(a)f_A'(b))^\alpha.
\end{align}
The simplest example of chordal restriction measure is provided by:

\begin{proposition}[\cite{MR1992830}, Proposition 4.1]\label{P:Brownian_restriction}
Let $D$ be a simply connected domain different from $\C$. Let $a, b \in \partial D$ be two distinct boundary points. Then the law of the filling of a Brownian excursion from $a$ to $b$ is the chordal restriction measure in $D$ from $a$ to $b$ with exponent $\alpha=1$.
\end{proposition}

In the following, we obtain an estimate about a chordal restriction measure in $\Hb$.
\begin{lemma}\label{lem:est_be}
Let $a, b, r\in \Rb$ be three distinct boundary points. Let $K$ be a chordal restriction measure in $\Hb$ between $a$ and $b$ with exponent $\alpha$. We have
\begin{align*}
\Pb[K \cap D(r,\eps)\not=\emptyset]=\alpha \frac{(a-b)^2}{(r-a)^2(r-b)^2} \eps^2 + O(1) \eps^3,
\end{align*}
where the $O(1)$ term is uniformly bounded for all $\eps\le 1$ by some function of $a,b,r$.
In particular, if $K$ is a Brownian excursion in $\Hb$ between $a$ and $b$, then $\Pb[K \cap D(r,\eps)\not=\emptyset]$ is given by the above formula with $\alpha=1$.
\end{lemma}

\begin{proof}
Let $\eps_0 = \eps_0(a,b,r) >0$ be small enough so that $\overline{D(r,\eps_0)}$ does not contain $a$ or $b$. Since the result is clear if $\eps \in [\eps_0,1)$, we can assume that $\eps < \eps_0$.
We define the following conformal map $f_{r,\eps}$ from $\Hb\setminus D(r,\eps)$ onto $\Hb$, and compute its derivative
\begin{align}\label{eq:f}
f_{r,\eps}(z)= z+\frac{\eps^2}{z-r},\quad f_{r,\eps}'(z)= 1- \frac{\eps^2}{(z-r)^2}.
\end{align} 
If we rescale $f_{r,\eps}$ by $(b-a) (f_{r,\eps}(b)-f_{r,\eps}(a))^{-1}$ and then do a proper translation, then we obtain a conformal map  from $\Hb\setminus D(r,\eps)$ onto $\Hb$ that fixes $a$ and $b$. 
By \eqref{eq:restriction1}, we have
\begin{align*}
\Pb[K \cap D(r,\eps)=\emptyset]=\left(f_{r,\eps}'(a)f_{r,\eps}'(b) \frac{(b-a)^2}{(f_{r,\eps}(b)-f_{r,\eps}(a))^2}\right)^\alpha 
=\left(1- \frac{(a-b)^2}{(r-a)^2(r-b)^2} \eps^2 + O(1) \eps^3 \right)^\alpha,
\end{align*}
where the $O(1)$ term is bounded by some constant $C(a,b,r)$ for all $\eps\le \eps_0$.
Since $\Pb[K \cap D(r,\eps) \not=\emptyset] =1-\Pb[B\cap D(1,\eps)=\emptyset]$, the lemma follows. The statement concerning Brownian excursion follows from Proposition \ref{P:Brownian_restriction}.
\end{proof}

Next, we need to obtain the following two-point estimate for chordal restriction measures.
\begin{lemma}\label{lem:Kr1r2}
Let $a, b, r_1, r_2 \in \Rb$ be four distinct points. Let $K$ be a chordal restriction measure in $\Hb$ between $a$ and $b$ with exponent $\alpha$. Then
\begin{equation}\label{eq:Kr1r2}
\begin{split}
\Pb[K\cap D(r_1,\eps)\not=\emptyset, K\cap D(r_2,\eps)\not=\emptyset]=\frac{\alpha(\alpha-1)(a-b)^4 \eps^4}{(r_1-a)^2 (r_2-a)^2(r_1-b)^2 (r_2-b)^2}\\
+\frac{\alpha(a-b)^2 \eps^4}{(r_1-b)^2(r_2-a)^2(r_2-r_1)^2}+\frac{\alpha(a-b)^2 \eps^4}{(r_1-a)^2(r_2-b)^2(r_2-r_1)^2}+O(1) \eps^5,
\end{split}
\end{equation}
where the $O(1)$ term is uniformly bounded for all $\eps\le 1$ by some function of $a,b, r_1, r_2$.
\end{lemma}
There can be several ways to deduce Lemma~\ref{lem:Kr1r2}. One can apply \eqref{eq:restriction1} to the conformal maps that map out either one or both of the balls $D(r_1,\eps)$ and $D(r_2,\eps)$. However, a direct computation seems long and complicated.
We opt for a proof of Lemma~\ref{lem:Kr1r2} that relates to the trichordal restriction measures introduced in \cite{MR3827221}, which are measures on random sets in  a simply connecte domain $D$ with three marked points $a, b, c \in\partial D$. Let $\Omega_3$ be the collection of all simply connected compact set $K\subset \overline D$, such that $K\cap \partial D=\{a, b, c\}$. 
Let $\Qc_3$ be the collection of all compact sets $A\subset\overline D$ such that $D\setminus A$ is simply connected and $a, b, c\not\in A$. For all $A\in \Qc_3$, let $\wt f_A$ be the unique conformal map from $D\setminus A$ onto $D$ that leaves $a,b,c$ fixed. A probability law on $K\in \Omega$ is said to satisfy trichordal conformal restriction if for any $A\in \Qc_3$, conditionally on $K \cap A=\emptyset$, $\wt f_A(K)$ is distributed as $K$. 
It was shown in \cite{MR3827221} that the family of trichordal restriction measures is characterized by three parameter $\alpha, \beta,\gamma$. A trichordal restriction measure in $D$ with parameters $\alpha, \beta, \gamma$ is characterized by
\begin{align}
\label{E:trich}
\Pb(K\cap A=\emptyset) =\wt f_A'(a)^\alpha \wt f_A'(b)^\beta \wt f_A'(c)^\gamma.
\end{align}
The connection to trichordal restriction measures comes from the following lemma.
\begin{lemma}\label{lem:cond_tri}
Let $a, b, c\in \Rb$ be three distinct points. Suppose that $K$ is a chordal restriction measure in $\Hb$ with marked points $a,b$ and exponent $\alpha$, then the law of $K$ conditionally on $K\cap D(c, \eps) \not=\emptyset$ converges as $\eps\to 0$ to that of a trichordal restriction measure in $\Hb$ with marked points $a, b, c$ and exponents $\alpha, \alpha, 2$.
Moreover, if $a<b<c<r$ and $A=D(r, \delta)$, then there exist $\eps_0, \delta_0>0$ and a constant $C(a,b,c,r)$ such that for all $\eps\le \eps_0, \delta\le \delta_0$,
\begin{align}\label{eq:cond_tri}
\left|\Pb[K \cap A =\emptyset \mid K \cap D(c, \eps)\not=\emptyset]-\wt f_A'(a)^\alpha  \wt f_A'(b)^\alpha \wt f_A'(c)^2\right| \le C(a,b,c,r) \eps \delta^2.
\end{align}
\end{lemma}
\begin{proof}
Here $D=\Hb$. For $A \in \Qc_3$, we have
\begin{equation}\label{eq:cond_tri2}
\begin{split}
&\Pb[K \cap A =\emptyset \mid K \cap D(c, \eps)\not=\emptyset]=\Pb[K \cap D(c, \eps)\not=\emptyset \mid K \cap A =\emptyset ] \frac{\Pb[K \cap A =\emptyset ]}{\Pb[K \cap D(c, \eps)\not=\emptyset]}\\
& = \Pb[K \cap \wt f_A(D(c, \eps)) \not=\emptyset] \frac{\Pb[K \cap A =\emptyset ]}{\Pb[K \cap D(c, \eps)\not=\emptyset]} =  \frac{\Pb[K \cap \wt f_A(D(c, \eps)) \not=\emptyset] }{\Pb[K \cap D(c, \eps)\not=\emptyset]} \wt f_A'(a)^\alpha \wt f_A'(b)^\alpha.
\end{split}
\end{equation}
Since $\wt f_A$ is analytic in a neighbourhood of $c$, there exist $\eps_0>0$ and $C=\wt f_A''(c)$, such that for all $\eps\le \eps_0$,
\begin{align*}
B\big(c, \wt f_A'(c)\eps - C \eps^2 \big) \subset \wt f_A(D(c, \eps)) \subset B\big(c, \wt f_A'(c)\eps + C \eps^2 \big).
\end{align*}
By Lemma~\ref{lem:est_be}, we can deduce
\begin{align}\label{eq:cond_tri3}
 \frac{\Pb[K \cap \wt f_A(D(c, \eps)) \not=\emptyset] }{\Pb[K \cap D(c, \eps)\not=\emptyset]}=\wt f_A'(c)^2 +O(1)\eps,
\end{align}
where $O(1)$ is uniformly bounded for $\eps\le \eps_0$ by a function of $a,b,c, A$. 
A trichordal restriction measure $\wt K$ in $\Hb$ with marked points $a, b, c$ and exponents $\alpha, \alpha, 2$ is characterized by the formula $\Pb[\wt K\cap A=\emptyset]=f_A'(a)^\alpha  f_A'(b)^\alpha f_A'(c)^2$ for all $A \in\Qc_3$. Plugging \eqref{eq:cond_tri3} back into \eqref{eq:cond_tri2} proves the weak convergence of the law of $K$ conditionally on $K \cap D(c, \eps)\not=\emptyset$.
Now, suppose $A=D(r,\delta)$, then $C=C(a,b,c,r) \delta^2$ for all $\delta\le (r-c)/2:=\delta_0$. One can compute using Lemma~\ref{lem:est_be} that the $O(1)$ term in \eqref{eq:cond_tri3} is bounded by $2C(a,b,c,r)\delta^2$ for all $\eps\le \eps_0$ and $\delta\le \delta_0$.
\end{proof}

\begin{lemma}\label{lem:trichordal}
Let $K$ be a trichordal restriction measure in $\Hb$ with marked points $x_1,x_2,x_3\in \Rb$ and respective exponents $\alpha_1, \alpha_2, \alpha_3$. Let $A \in \Qc_3$ and $f :\Hb\setminus A \to \Hb$ be any conformal map. Then
\begin{align}\label{eq:tri_form}
&\Pb[K\cap A=\emptyset] =\prod_{i=1}^3 \left[f'(x_i) f'(x_{i+1}) \frac{(x_i -x_{i+1})^2}{(f(x_i) -f(x_{i+1}))^2}\right]^{\frac{\alpha_i+\alpha_{i+1}-\alpha_{i+2}}{2}},
\end{align}
where we let $x_4=x_1$, $\alpha_4=\alpha_1$, $\alpha_5=\alpha_2$.
\end{lemma}
\begin{proof}
Let $f_A$ be the unique conformal map from  $\Hb\setminus A$ onto $\Hb$ that fixes $x_1, x_2, x_3$. By \eqref{E:trich}, we have
\begin{align*}
\Pb[K\cap A=\emptyset]=f_A'(x_1)^{\alpha_1} f_A'(x_2)^{\alpha_2} f_A'(x_3)^{\alpha_3}= \prod_{i=1}^3  \left(f_A'(x_i) f_A'(x_{i+1})\right)^{\frac{\alpha_i+\alpha_{i+1}-\alpha_{i+2}}{2}}.
\end{align*}
Let $\varphi : \Hb \to \Hb$ be the unique conformal map that maps $f(x_i)$ to $x_i$, $i=1,2,3$, so that $f_A = \varphi \circ f$. With explicit computations, we obtain that for $i=1,2,3$,
\[
\varphi'(f(x_i)) = \frac{(x_i-x_{i+1})(x_i-x_{i+2})(f(x_{i+1})-f(x_{i+2}))}{(x_{i+1}-x_{i+2})(f(x_i)-f(x_{i+1})(f(x_i)-f(x_{i+2}))}.
\]
We thus get that, for $i=1,2,3$,
\begin{align*}
f_A'(x_i) f_A'(x_{i+1})= f'(x_i) f'(x_{i+1}) \varphi'(f(x_i)) \varphi'(f(x_{i+1})) =
f'(x_i) f'(x_{i+1}) \frac{(x_i -x_{i+1})^2}{(f(x_i) -f(x_{i+1}))^2}.
\end{align*}
This completes the proof.
\end{proof}


We can now return to the proof of Lemma \ref{lem:Kr1r2}.

\begin{proof}[Proof of Lemma~\ref{lem:Kr1r2}]

As in the proof of Lemma \ref{lem:est_be}, we can assume that $\eps_1$ and $\eps_2$ are small enough so that $\overline{D(r_1,\eps_1)} \cup \overline{D(r_2,\eps_2)}$ does not contain $a$ or $b$.
Note that $\Pb[K\cap D(r_1,\eps_1)\not=\emptyset, K\cap D(r_2,\eps_2)\not=\emptyset]$ is equal to
\begin{align}\label{eq:2point1}
\Pb[K\cap D(r_1,\eps_1)\not=\emptyset] \Pb[ K \cap D(r_2,\eps_2)\not=\emptyset \mid K\cap D(r_1,\eps_1)\not=\emptyset].
\end{align}
By Lemma~\ref{lem:est_be}, we have
\begin{align}\label{eq:2point2}
\Pb[K\cap D(r_1,\eps_1)\not=\emptyset]=\alpha \frac{(a-b)^2}{(r_1-a)^2(r_1-b)^2} \eps_1^2 + O(1)\eps_1^3.
\end{align}
By Lemma~\ref{lem:cond_tri}, we know that the law of $K$ conditioned on $K\cap D(r_1,\eps_1)\not=\emptyset$ converges as $\eps_1\to 0$ to the law of a trichordal restriction measure $\wt K$ in $\Hb$ with marked points $a, b, r_1$ and respective exponents $\alpha, \alpha, 2$. More quantitatively, there exist $\eps_0, \delta_0 \ge 0$, such that
\begin{align}\label{eq:2point3}
 \Pb[ K \cap D(r_2,\eps_2)\not=\emptyset \mid K\cap D(r_1,\eps_1)\not=\emptyset] =\Pb[\wt K \cap D(r_2,\eps_2)\not=\emptyset] + O(1)\eps_1\eps_2^2,
\end{align}
where $O(1)$ is bounded by $C(a,b,r_1, r_2)$ for all  for all $\eps_1\le \eps_0$ and $\eps_2\le \delta_0$. Let $f_{r_2,\eps}: \Hb \setminus D(r_2, \eps) \to \Hb$ be the map defined in \eqref{eq:f}. By \eqref{eq:f}, we have
\begin{align*}
&f'_{r_2, \eps}(a) f'_{r_2,\eps}(b) \frac{(a-b)^2}{(f_{r_2,\eps}(a) -f_{r_2,\eps}(b))^2}=1- \frac{(a-b)^2}{(r_2-a)^2(r_2-b)^2} \eps^2 + O(1)\eps^4,\\
&f'_{r_2, \eps}(a) f'_{r_2,\eps}(r_1) \frac{(a-r_1)^2}{(f_{r_2,\eps}(a) -f_{r_2,\eps}(r_1))^2}=1- \frac{(a-r_1)^2}{(r_2-a)^2(r_2-r_1)^2} \eps^2 + O(1)\eps^4,\\
&f'_{r_2, \eps}(r_1) f'_{r_2,\eps}(b) \frac{(r_1-b)^2}{(f_{r_2,\eps}(r_1) -f_{r_2,\eps}(b))^2}=1- \frac{(r_1-b)^2}{(r_2-r_1)^2(r_2-b)^2} \eps^2 + O(1)\eps^4.
\end{align*}
Plugging them into Lemma~\ref{lem:trichordal}, we can compute $\Pb[\wt K \cap D(r_2,\eps_2)=\emptyset]$, and then deduce that
\begin{align}\label{eq:2point4}
\Pb[\wt K \cap D(r_2,\eps_2)\not=\emptyset]= \frac{(\alpha-1)(a-b)^2}{(r_2-a)^2(r_2-b)^2} \eps_2^2 + \frac{(a-r_1)^2}{(r_2-a)^2(r_2-r_1)^2} \eps_2^2+ \frac{(r_1-b)^2}{(r_2-r_1)^2(r_2-b)^2} \eps_2^2 + O(1)\eps_2^4.
\end{align}
Combining \eqref{eq:2point1}, \eqref{eq:2point2}, \eqref{eq:2point3}, \eqref{eq:2point4} and letting $\eps_1=\eps_2$ proves  \eqref{eq:Kr1r2} where the $O(1)$ term therein is uniformly bounded for $\eps\le \min(\eps_0, \delta_0)$. It is clear that the $O(1)$ term is also bounded for $\eps\in [\min(\eps_0, \delta_0), 1]$. This completes the proof of the lemma.
\end{proof}

We immediately get the following estimate on Brownian excursions. 
\begin{corollary}\label{cor:Br1r2}
Let $a, b, r_1, r_2 \in \Rb$ be four distinct points. Let $K$ be a Brownian excursion in $\Hb$ between $a$ and $b$. Then
\begin{align*}
&\Pb[K\cap D(r_1,\eps)\not=\emptyset, K\cap D(r_2,\eps)\not=\emptyset]\\
=&\frac{(a-b)^2}{(r_1-a)^2(r_2-b)^2 (r_1-r_2)^2} \eps^4 + \frac{(a-b)^2}{(r_2-a)^2(r_1-b)^2 (r_1-r_2)^2} \eps^4 + O(1)\eps^5,
\end{align*}
where the $O(1)$ term is uniformly bounded for all $\eps\le 1$ by some function of $a,b, r_1, r_2$.
\end{corollary}
\begin{proof}
The formula follows from \eqref{eq:Kr1r2} applied to $\alpha=1$. In the case of a Brownian motion, one can alternatively compute $\Pb[K\cap D(r_1,\eps)\not=\emptyset, K\cap D(r_2,\eps)\not=\emptyset]$ using the heat kernel. The latter computation will show that $\Pb[K\cap D(r_1,\eps)\not=\emptyset, K\cap D(r_2,\eps)\not=\emptyset]$ is analytic in $\eps$ near $0$. Therefore the $o(\eps^4)$ term in  \eqref{eq:Kr1r2}  is in fact $O(\eps^5)$.
\end{proof}

\subsection{Estimates about Brownian excursions}
In this subsection, we aim to obtain estimates about the probability that a Brownian excursion gets close to one interior point (Lemma~\ref{lem:B_1point}) and two interior points (Lemma~\ref{lem:b_2p}). We derive these standard estimates for ease of reference.
First recall some definitions from \cite[Section 5]{MR2129588}.

\paragraph{Green's function and Poisson kernel}
For any domain $D$, let $G_D$ be the Green's function in $D$ as defined in Section \ref{SS:preliminaries_paths}. In our normalisation,
\begin{equation}
\label{E:Green}
\forall x,y \in \D,
G_\D(x,y) = \frac{1}{2\pi} \log \frac{|1-x\bar y|}{|x-y|}
\quad \text{and} \quad \forall x,y \in \Hb, G_\Hb(x,y) = \frac{1}{2\pi} \log \frac{|x-\bar y|}{|x-y|}.
\end{equation}
Suppose that $\partial D$ is smooth. Then for $x\in D$ and $z\in \partial D$, the Poisson kernel is given by
$
H_D(x,z)=\lim_{\eps\to0} \eps^{-1} G(x, z+\eps \mathbf{n}_z),
$
where $\mathbf{n}_z$ is the inward unit normal vector of $\partial D$ at $z$.
We have
\begin{equation}
\label{E:Poisson}
\forall x \in \D, \forall z \in \partial \D,
H_\D(x,z) = \frac{1}{2\pi} \frac{1-|x|^2}{|x-z|^2}
\quad \text{and} \quad
\forall x \in \Hb, \forall z \in \partial \Hb,
H_\Hb(x,z) = \frac{\Im(x)}{\pi|x-z|^2}.
\end{equation}
For $z, w\in \partial D$, let the boundary Poisson kernel be
$
H_D(z,w)=\lim_{\eps\to0} \eps^{-1} H_D(z+\eps \mathbf{n}_z, w).
$
Under this normalisation, we have
\begin{equation}
\label{E:Poisson_boundary}
\forall z,w\in \partial \D, H_\D(z,w) = \frac{1}{\pi|z-w|^2}
\quad \text{and} \quad
\forall z,w \in \partial \Hb, H_\Hb(z,w) = \frac{1}{\pi(z-w)^2}.
\end{equation}

Finally recall the definitions introduced in Notation \ref{N:exc} of the conformal map $f_z : \Hb \to \D$ mapping $z$ to $0$ and $U(z,r) = f_z^{-1}(D(0,r))$.

\begin{lemma}[Getting close to one interior point]\label{lem:B_1point}
For all $\eta\in (0,1)$, there exists $C=C(\eta)>0$ such that the following holds.
Let $a,b \in \R$ be two distinct boundary points and let $B$ be a Brownian excursion in $\Hb$ between $a$ and $b$. For all $r \in (0,1-\eta)$ and $z=x+yi\in\Hb$,
\begin{align}\label{eq:lemB_1point1}
\Pb[B \cap U(z, r)\not=\emptyset] \le C \frac{(a-b)^2 y^2}{|a-z|^2 |b-z|^2} \left|\log r\right|^{-1}.
\end{align}
\end{lemma}
\begin{proof}
Consider the conformal map $f_z : \Hb \to \D$ that sends $z$ to 0.
Let $\beta_0 \in [0,2\pi)$ be the angle between $f_z(a)$ and $f_z(b)$.
By conformal invariance, the left hand side of \eqref{eq:lemB_1point1} is equal to the probability $p(\beta_0, r)$ that a Brownian excursion in $\D$ between $1$ and $e^{i\beta_0}$ visits $D(0,r)$.
Let $\mu$ be the infinite measure on excursions $B$ in $\D$ started from $1$, defined by $\mu=\int_0^{2\pi} H_\D(1, e^{i\beta})\mu_\D^\#(1, e^{i\beta}) d\beta$. Let $T$ be the total time length of $B$.
We have
\begin{align*}
& p(\beta, r) H_\D(1, e^{i\beta}) d\beta=\mu(B \cap D(0,r)\not=\emptyset, \arg B_T\in (\beta, \beta+d\beta))\\
& = \Pb\left(\arg B_T\in (\beta, \beta+d\beta) \mid B \cap D(0,r)\not=\emptyset \right) \mu(B \cap D(0,r)\not=\emptyset).
\end{align*}
Moreover,
\begin{align*}
\mu(B \cap D(0,r)\not=\emptyset)
= \lim_{\eps \to 0} \frac{1}{\eps} \PROB{1-\eps}{ \tau_{D(0,r)} < \tau_{\partial \D}} = |\log r|^{-1},
\end{align*}
where the probability above stands for the probability that a Brownian motion starting at $1-\eps$ hits $D(0,r)$ before $\partial \D$.
For fixed $\eta\in(0,1)$ and $r\in(0,\eta)$, for an excursion $B$ started from $1$ and conditioned to hit $D(0,r)$, 
there exists a constant $C(\eta)>0$, such that 
\[
\Pb\left(\arg B_T\in (\beta, \beta+d\beta) \mid B \cap D(0,r)\not=\emptyset \right) \le C(\eta)  d\beta.
\]
Therefore
\begin{align*}
p(\beta, r) \le C(\eta) H_\D(1, e^{i\beta})^{-1} |\log r|^{-1}\quad\text{ for all } r\le 1-\eta.
\end{align*}
Finally, by \eqref{E:Poisson_boundary},
\[
H_\D(1,e^{i\beta_0}) = \frac{1}{\pi |f_z(a)-f_z(b)|^2} = \frac{|a-z|^2 |b-z|^2}{4\pi y^2 (a-b)^2}.
\]
This concludes the proof.
\end{proof}

\begin{lemma}[Getting close to two interior points]\label{lem:b_2p}
For all $\eta\in (0,1)$, there exist $\nu=\nu(\eta)>0$ and $C=C(\eta)>0$ such that the following holds.
Let $a,b \in \R$ be two distinct boundary points and let $B$ be a Brownian excursion in $\Hb$ between $a$ and $b$.
For all $z_j=x_j+y_j i\in\Hb$, $j=1,2$, and $r_1, r_2 \in(0, 1-\eta)$ such that $r_j \Im(z_j) \leq \nu |z_1-z_2|$, $j=1,2$, we have 
\begin{align*}
&\Pb[B \cap U(z_1,r_1)\not=  \emptyset, B \cap U(z_2,r_2)\not=  \emptyset] \\
& \leq C \left( \frac{(a-b)^2 y_1 y_2}{|a-z_1|^2|b-z_2|^2}  +   \frac{(a-b)^2 y_1 y_2}{|a-z_2|^2|b-z_1|^2} \right) (1+ G_\Hb(z_1,z_2)) \left|\log r_1 \right|^{-1}  \left|\log r_2\right|^{-1}.
\end{align*}
\end{lemma}

The condition that $r_j \Im(z_j) \leq \nu |z_1-z_2|$, $j=1,2$ ensures that $U(z_1,r_1)$ and $U(z_2,r_2)$ do not intersect and are at distance of order $|z_1-z_2|$ to each other.

\begin{proof}
We start by recording two facts for ease of future reference:

1.
For all $w \in \partial \D$ and $r \in (0,1)$,
\begin{equation}
\label{E:integral_Poisson}
\int_{\partial D(0,r)} H_{\D \setminus D(0,r)}(w,z) \d z = |\log r|^{-1}.
\end{equation}
Indeed, the left hand side is equal to
\begin{align*}
    \lim_{\eps \to 0} \frac1\eps \int_{\partial D(0,r)} H_{\D \setminus D(0,r)}(w+\eps \mathbf{n}_w,z) \d z
    = \lim_{\eps \to 0} \frac1\eps \PROB{w+\eps \mathbf{n}_w}{\tau_{D(0,r)} < \tau_{\partial \D}}
    = \lim_{\eps \to 0} \frac{\log(1-\eps)}{\eps \log r} = \frac{1}{|\log r|},
\end{align*}
where the probability appearing in the above display stands for the probability that a Brownian motion starting from $w+\eps \mathbf{n}_w$ hits $D(0,r)$ before hitting $\partial \D$.

2. There exists $C_0>0$ such that for all $z_3 \in \partial U(z_1,r_1)$ and $z_4 \in \partial U(z_2,r_2)$,
\begin{equation}\label{E:boundGreen}
G_\Hb(z_3,z_4) \le C_0 (1+ G_\Hb(z_1,z_2)).
\end{equation}
This follows from the fact that there exists $\nu=\nu(\eta)>0$ such that for all $z_{2j} \in U(z_j,r_j)$, $|z_{2j}-z_j| \leq (10\nu)^{-1} r_j \Im(z_j)$, $j=1,2$. The assumption that $r_j \Im(z_j) \leq \nu |z_1-z_2|$, $j=1,2$, then ensures that the points of $U(z_1,r_1)$ stay much closer to $z_1$ than to any other point of $U(z_2,r_2)$. A direct computation using \eqref{E:Green} then concludes the proof of \eqref{E:boundGreen}.

\medskip
Having collected these two initial facts, we can now proceed with the proof of Lemma \ref{lem:b_2p}.
The event $B \cap U(z_1, r_1)\not=  \emptyset, B \cap U(z_2, r_2)\not=  \emptyset$ is the disjoint union of the two events $E_1$ and $E_2$:
\begin{itemize}
\item Let $E_1$ be the event that $B$ first  visits $U(z_1, r_1)$ before $U(z_2, r_2)$.
\item Let $E_2$ be the event that $B$ first  visits $U(z_2, r_2)$ before $U(z_1, r_1)$.
\end{itemize}
On $E_1$, we can decompose $B$ into an excursion from $a$ to $z_3\in \partial U(z_1, r_1)$ in $\Hb\setminus (U(z_1, r_1) \cup U(z_2, r_2))$, an excursion from $z_3$ to $z_4 \in \partial U(z_2, r_2)$ in $\Hb$ and an excursion from $z_4$ to $b$ in $\Hb\setminus U(z_2, r_2)$. This results in the following identity
\begin{align}
\label{E:proof_2p}
H_\Hb(a,b) \Pb[E_1]
= \int_{\partial U(z_1,r_1)} \d z_3 \int_{\partial U(z_2,r_2)} \d z_4
H_{\Hb\setminus (U(z_1, r_1) \cup U(z_2, r_2))}(a, z_3) G_\Hb(z_3,z_4) H_{\Hb\setminus U(z_2, r_2)}(b, z_4).
\end{align}
Increasing the domain can only increase the pointwise values of the Poisson kernel. We can thus bound $H_{\Hb\setminus (U(z_1, r_1) \cup U(z_2, r_2))}(a, z_3) \le H_{\Hb\setminus U(z_1, r_1)}(a, z_3)$ for all $z_3 \in \partial U(z_1,r_1)$. Using further \eqref{E:boundGreen}, we deduce that
\begin{align*}
& H_\Hb(a,b) \Pb[E_1]
\leq C_0 (1+ G_\Hb(z_1,z_2)) \int_{\partial U(z_1,r_1)} \d z_3 H_{\Hb\setminus U(z_1, r_1)}(a, z_3) \times \int_{\partial U(z_2,r_2)} \d z_4
 H_{\Hb\setminus U(z_2, r_2)}(b, z_4).
\end{align*}
With a change of variable and conformal covariance of the boundary Poisson kernel, this is further equal to
\begin{align*}
C_0 (1+ G_\Hb(z_1,z_2)) f_{z_1}'(a) f_{z_2}'(b) \int_{\partial D(0,r_1)} H_{\D \setminus D(0,r_1)}(f_{z_1}(a),w_3) \d w_3 \int_{\partial D(0,r_2)} H_{\D \setminus D(0,r_2)}(f_{z_2}(b),w_4) \d w_4.
\end{align*}
By \eqref{E:integral_Poisson}, the product of the last two integrals equals $|\log r_1|^{-1} |\log r_2|^{-1}.$ Finally, using that (see \eqref{E:Poisson_boundary})
\begin{align*}
H_\Hb(a,b) = \frac{1}{\pi (a-b)^2}, \quad
 f_{z_1}'(a)= \frac{2 y_1}{|a- z_1|^2}, \quad  f_{z_2}'(b)= \frac{2 y_2}{|b- z_2|^2},
\end{align*}
we obtain that
\begin{align*}
 \Pb[E_1] \le \pi^{-1} C_1 (1+ G_\Hb(z_1,z_2)) \frac{(a-b)^2 y_1 y_2}{|a-z_1|^2|b-z_2|^2} \left|\log r_1 \right|^{-1}  \left|\log r_2 \right|^{-1}.
\end{align*}
To bound $\P[E_2]$ simply exchange the roles of $z_1$ and $z_2$. This concludes the proof.
\end{proof}

\subsection{Boundary excursions in a loop soup with mixed boundary conditions}\label{subsec:bexr-}
In this subsection, we aim to prove Propositions~\ref{prop:integral} and~\ref{prop:sum_mm2} which provide quantitative controls on the number of excursions in $\Ec_{\Rb^-}$. Proposition~\ref{prop:integral} can be understood as a first moment estimate, and Propositions~\ref{prop:sum_mm2} can be understood as a second moment estimate. 

The proofs of Propositions~\ref{prop:integral} and~\ref{prop:sum_mm2} rely essentially on the restriction property of $\Ec$ (Theorem~\ref{thm:restriction}) and the independence of the excursions conditionally on their endpoints (Lemma~\ref{lem:p_loc}).
Lemma~\ref{lem:p_loc} is a variation of \cite[Lemma 9]{qian2018} which states that $\Ec_\Rb$ is a  locally finite point process of Brownian excursions. Changing $\Ec_\Rb$ to $\Ec_{\Rb^-}$ only modifies the setup of the proof, not the essential arguments, so we omit the proof of Lemma~\ref{lem:p_loc}.
\begin{lemma}\label{lem:p_loc}
The collection $\Ec_{\Rb^-}$ is a locally finite point process of Brownian excursions in $\Hb$.
\end{lemma}

The ``point process of Brownian excursions'' property means that, conditionally on the endpoints of the excursions, $\Ec_{\Rb^-}$ is distributed as the collection of independent Brownian excursions between the given endpoints.
Throughout, for each excursion $e$, we assign an i.i.d.\ orientation with probability $1/2$. Let  $a(e)$ and $b(e)$ be respectively the starting and ending endpoints of an excursion $e$. 
The ``locally finite'' property means that for each $\delta>0$, the collection of excursions $e\in \Ec_{\Rb^-}$ with $|a(e)-b(e)|\ge \delta$ is a.s.\ finite.
 

Let us first prove Proposition~\ref{prop:integral}.
\begin{proof}[Proof of Proposition~\ref{prop:integral}]
We write $\Ec$ for $\Ec_{\Rb^-}$ for simplicity.
We define for each $e\in \Ec$ and $r>0$,
\begin{align*}
p_r(e):=\frac{(b(e) -a(e))^2}{(r-b(e))^2(r-a(e))^2}.
\end{align*}
By Lemma~\ref{lem:p_loc} and Lemma~\ref{lem:est_be}, we know that
\begin{align}\label{eq:pebreps}
\Pb\left[e\cap  D(r,\eps)\not=\emptyset \mid a(e), b(e)\right]= p_r(e) \eps^2+O(1)\eps^4,
\end{align}
where the $O(1)$ term depends on $a(e), b(e)$ and $r$, but is uniformly bounded for $\eps\le 1$. Note that for all $a(e)<b(e)<0$ and $r>0$, we have $p_r(e) \le r^{-2}$.
For each $M>0$, let $\Ec^M$ denote the collection of the $M$ excursions in $\Ec$ with the largest $p_r(e)$. Then $\Ec=\cup_{M>0}\Ec^M$.
Let us first show that 
\begin{align}\label{eq:Ec_delta}
\Pb\left[\Ec^M \cap D(r,\eps)\not=\emptyset \mid \{(a(e), b(e))\}_{e\in\Ec^M}\right]= \sum_{e\in\Ec^M} p_r(e)\eps^2 +O(1) \eps^4,
\end{align}
where the  $O(1)$ term depends on $r, M$ and the random set $\{(a(e), b(e))\}_{e\in\Ec^M}$, but is uniformly bounded for $\eps\le 1$.
The $\le$ direction in \eqref{eq:Ec_delta} directly follows from Lemma~\ref{lem:p_loc} and \eqref{eq:pebreps}. 
To prove the $\ge$ direction, note that by the inclusion-exclusion principle, the left-hand side of  \eqref{eq:Ec_delta} is at least
\begin{align*}
&\sum_{e\in\Ec^M} \Pb[e\cap D(r,\eps)\not=\emptyset] -\sum_{e_1\not=e_2\in \Ec^M} \Pb[e_1\cap D(r,\eps)\not=\emptyset, e_2\cap D(r,\eps)\not=\emptyset]\\
=&\sum_{e\in\Ec^M} p_r(e) \eps^2 -\sum_{e_1 \not=e_2 \in\Ec^M} p_r(e_1) p_r(e_2) \eps^4  +O(1)\eps^4=\sum_{e\in\Ec^M} p_r(e) \eps^2+O(1)\eps^4.
\end{align*}
In the last equality, we have used the bound $\sum_{e_1 \not=e_2 \in\Ec^M} p_r(e_1) p_r(e_2) \le M^2 r^{-4}$.
This completes the proof of \eqref{eq:Ec_delta}.
Note that the $O(1)$ term in \eqref{eq:Ec_delta} has finite expectation, since $\Pb\left[\Ec^M \cap D(r,\eps)\not=\emptyset\right]\in [0,1]$.
Taking expectations on both sides of \eqref{eq:Ec_delta}, we get
\begin{align}\label{eq:ecdelta}
\Pb\left[\Ec^M \cap D(r,\eps)\not=\emptyset\right]= \Eb\bigg[\sum_{e\in\Ec^M} p_r(e) \bigg]\eps^2 +O(1)\eps^4,
\end{align}
where the $O(1)$ term depends on $M$ and $r$, but is uniformly bounded for $\eps\le 1$. 

On the other hand, by Theorem~\ref{thm:restriction}, we know that $\Ec$ is distributed as a restriction measure of exponent $\alpha(\theta)$. Applying Lemma~\ref{lem:est_be} to $a=\infty$ and $b=0$, we get
\begin{align}\label{eq:Ec_restriction}
\Pb\left[\Ec \cap D(r,\eps)\not=\emptyset\right]= \alpha(\theta) \eps^2/r^2 +O(1)\eps^4.
\end{align}
Since $\Pb\left[\Ec^M \cap D(r,\eps)\not=\emptyset\right]\le \Pb\left[\Ec \cap D(r,\eps)\not=\emptyset\right]$, together with \eqref{eq:ecdelta}, we deduce that
\begin{align*}
 \Eb\bigg[\sum_{e\in\Ec^M} p_r(e) \bigg] \le \alpha(\theta) /r^2.
\end{align*}
Letting $M\to \infty$, by monotone convergence, we have
\begin{align}\label{eq:sum_bound}
\Eb\bigg[\sum_{e\in\Ec} p_r(e) \bigg] \le \alpha(\theta) /r^2.
\end{align}
Let $M\to \infty$ in both sides of \eqref{eq:ecdelta}. The left-hand side of \eqref{eq:ecdelta} tends to $\Pb\left[\Ec \cap D(r,\eps)\not=\emptyset\right]$ which is less than $1$. In the right-hand side of \eqref{eq:ecdelta}, the coefficient in front of $\eps^2$ is increasing as $M\to\infty$, and tends to the same expected sum over $\Ec$ instead of $\Ec^M$, which is less than $\alpha(\theta)/r^2$ by \eqref{eq:sum_bound}. Therefore, the $O(1)$ term in  \eqref{eq:ecdelta} should also be bounded as $M\to\infty$.
This implies
\begin{align*}
\Pb\left[\Ec \cap D(r,\eps)\not=\emptyset\right] = \Eb\bigg[\sum_{e\in\Ec}p_r(e)\bigg]\eps^2 +O(1)\eps^4,
\end{align*}
where the $O(1)$ term depends only on $r$. Comparing with \eqref{eq:Ec_restriction} yields the lemma.
\end{proof}

Let us now deduce a ``second moment'' estimate.
\begin{proposition}\label{prop:sum_mm2}
For all $0<r_1<r_2$, we have
\begin{align}
\notag
&\Eb\left[\sum_{e \in\Ec_{\Rb^-}} \frac{(b(e) -a(e))^2}{(r_1-b(e))^2 (r_2-a(e))^2 (r_1-r_2)^2} \right] + \Eb\left[\sum_{e \in\Ec_{\Rb^-}} \frac{(b(e) -a(e))^2 }{(r_1-a(e))^2 (r_2-b(e))^2 (r_1-r_2)^2} \right]\\
\label{eq:e2}
&+\Eb\left[\sum_{e_1 \not=e_2 \in\Ec_{\Rb^-}} \frac{(b(e_1) -a(e_1))^2 }{(r_1-b(e_1))^2(r_1-a(e_1))^2}  \frac{(b(e_2) -a(e_2))^2 }{(r_2-b(e_2))^2(r_2-a(e_2))^2} \right] \\
\notag
& = \frac{ \alpha(\theta)}{r_1^2(r_2-r_1)^2}+\frac{\alpha(\theta) }{r_2^2(r_2-r_1)^2} + \frac{\alpha(\theta)(\alpha(\theta)-1) }{r_1^2 r_2^2}.
\end{align} 
\end{proposition}

\begin{remark}
When $\theta=1/2$, the endpoints of $\Ec_{\R^-}$ are distributed according to a Poisson point process with intensity $1/4$ (see Theorem \ref{thm:qw19} for $\Ec_\R$ instead of $\Ec_{\R^-}$). At this value of $\theta$, the restriction exponent $\alpha(\theta)$ equals $1/4$ and thus also corresponds to the intensity of the point process. With a simple Poisson point process computation, one can deduce from Proposition~\ref{prop:integral} that the third expectation on the left hand side of \eqref{eq:e2} is equal to the third term on the right hand side. There is \textit{a priori} no reason to believe that this still holds for other values of $\theta$. However, we will only need to upper bound the terms on the left hand side of \eqref{eq:e2}.
\end{remark}

\begin{proof}
We write $\Ec$ for $\Ec_{\Rb^-}$ for simplicity.
By Theorem~\ref{thm:restriction} and Lemma~\ref{lem:Kr1r2} applied to $a=\infty$ and $b=0$, we have
\begin{align}\label{eq:K0infty}
\Pb\left[\Ec\cap D(r_1, \eps)\not=\emptyset, \Ec\cap D(r_2, \eps)\not=\emptyset\right]=\frac{\alpha(\theta)(\alpha(\theta)-1) \eps^4}{r_1^2 r_2^2} +\frac{ \alpha(\theta)\eps^4}{r_1^2(r_2-r_1)^2}+\frac{\alpha(\theta) \eps^4}{r_2^2(r_2-r_1)^2}+O(1)\eps^5.
\end{align}
The event $\{\Ec\cap D(r_1, \eps)\not=\emptyset, \Ec\cap D(r_2, \eps)\not=\emptyset\}$ is equal to the union of the two following events:
\begin{itemize}
\item Let $E_1$ be the event that $ D(r_1, \eps)$ and $ D(r_2, \eps)$ are visited by a same excursion in $\Ec$.
\item Let $E_2$ be the event that $ D(r_1, \eps)$ and $ D(r_2, \eps)$ are visited by two different excursions in $\Ec$.
\end{itemize}
We first evaluate $\Pb[E_1]$. Let
\begin{align*}
s_{r_1, r_2}(e):= \frac{(b(e) -a(e))^2}{(r_1-b(e))^2 (r_2-a(e))^2 (r_1-r_2)^2} +\frac{(b(e) -a(e))^2 }{(r_1-a(e))^2 (r_2-b(e))^2 (r_1-r_2)^2}.
\end{align*}
By Lemma~\ref{lem:p_loc} and Corollary~\ref{cor:Br1r2}, we have
\begin{align*}
\Pb[e \cap D(r_1, \eps)\not=\emptyset, e \cap D(r_2, \eps)\not=\emptyset ]=s_{r_1, r_2}(e)\eps^4+O(1)\eps^5,
\end{align*}
where the $O(1)$ term depends on $a(e), b(e)$ and $r_1, r_2$, but is uniformly bounded for $\eps\le 1$. Note that for all $a(e)<b(e)<0$ and $r_1, r_2>0$, we have $p_r(e) \le (r_1-r_2)^{-2} (r_1^{-2}+r_2^{-2})$.

For each $M>0$, let $\Ec^{M}$ be the set of the $M$ excursions in $\Ec$ with the largest $s_{r_1, r_2}(e)$. Then $\Ec=\cup_{M>0} \Ec^{M}$.
Let $E_1^M$ and $E_2^M$ be defined respectively as $E_1$ and $E_2$ where we replace $\Ec$ by $\Ec^{M}$ in their definitions.
By Lemma~\ref{lem:p_loc}, Corollary~\ref{cor:Br1r2}, and a use of the inclusion-exclusion principle as in the proof of \ref{prop:integral}, we can deduce
\begin{align*}
\Pb \left[E_1^M  \mid \{(a(e), b(e))\right]=\sum_{e \in\Ec^{M}} s_{r_1, r_2}(e) \eps^4 +O(1)\eps^5,
\end{align*}
where the $O(1)$ term depends on $r_1, r_2$ and the random set $\{(a(e), b(e))\}_{e\in\Ec^{M}}$, but is uniformly bounded for $\eps\le 1$. Moreover, this $O(1)$ term has finite expectation, since $\Pb[E_1^M]\in[0,1]$. A similar reasoning as in the proof of Proposition~\ref{prop:integral} will allow us to take the expectations on both sides of the previous equality and then let $M\to\infty$. We will get that 
\begin{align*}
\Pb[E_1]=\Eb\left[\sum_{e \in\Ec} s_{r_1, r_2}(e)   \right]\eps^4 +O(1)\eps^5.
\end{align*}
Let us now evaluate $\Pb[E_2]$.
By Lemmas~\ref{lem:p_loc} and~\ref{lem:est_be}, we have
\begin{align*}
\Pb \left[E_2^M  \mid \{(a(e), b(e))\right]
=\sum_{e_1 \not=e_2 \in\Ec} \frac{(b(e_1) -a(e_1))^2 \eps^2}{(r_1-b(e_1))^2(r_1-a(e_1))^2}  \frac{(b(e_2) -a(e_2))^2 \eps^2}{(r_2-b(e_2))^2(r_2-a(e_2))^2}+O(\eps^5).
\end{align*}
The same reasoning will again allow us to conclude that $\Pb[E_2]$ is equal to the line \eqref{eq:e2} times $\eps^4$ plus $O(1)\eps^5$.

Finally, one can see that $\Pb[E_1\cap E_2]$ is of smaller order than $\eps^4$, because on $E_1\cap E_2$, there needs to be one excursion $e_1\in\Ec$ that visits both $ D(r_1, \eps)$ and $ D(r_2, \eps)$ and a different excursion  $e_2\in\Ec$ that visits $ D(r_1, \eps)$ or $ D(r_2, \eps)$. This implies that
\begin{align*}
\Pb\left[\Ec\cap D(r_1, \eps)\not=\emptyset, \Ec\cap D(r_2, \eps)\not=\emptyset\right]=\Pb[E_1] +\Pb[E_2] +o(\eps^4).
\end{align*}
Combining with \eqref{eq:K0infty} and the previous evaluations of $\Pb[E_1]$ and $\Pb[E_2]$, we obtain the lemma.
\end{proof}

\subsection{Boundary excursions in a loop soup with wired boundary conditions}\label{subsec:wired}
In this section, we will use the results from Section~\ref{subsec:bexr-} to obtain quantitative controls on excursions in $\Ec_\Rb$, i.e., from a loop soup in $\Hb$ which is wired on the entire boundary. The main step is to prove Lemma~\ref{lem:1point_inf} below. After that, we will deduce Lemmas~\ref{lem:X_eps} and~\ref{lem:estimate_interval} as consequences.

\begin{lemma}\label{lem:1point_inf}
We have
\begin{align}
\label{eq:bound}
&S_1:=\Eb\left[ \sum_{e\in \Ec_{\Rb}, a(e), b(e) \in [0,1]}  (a(e) -b(e))^2\right]<\infty.\\
\label{eq:bound1}
&S_2:=\Eb\left[\sum_{e_1, e_2 \in\Ec_{\Rb};\, a(e_1), b(e_1), a(e_2), b(e_2)\in [0, 1]} (b(e_1) -a(e_1))^2 (b(e_2) -a(e_2))^2 \right] <\infty.
\end{align}
\end{lemma}
\begin{proof}
We use $\Ec$ to denote $\Ec_{\Rb}$ for simplicity. We apply the conformal map $z \in \Hb \mapsto 1/(1-z) \in \Hb$ that maps $\R^-$ to $[0,1]$ to Proposition~\ref{prop:integral} (with $r=1$), and get
\begin{align*}
\Eb\left[\sum_{e\in \Ec_{[0,1]}}  (a(e) -b(e))^2\right]=\alpha(\theta).
\end{align*}
By scaling invariance and translation invariance, for all $R>0$, we have
\begin{align*}
\Eb\left[\sum_{e\in \Ec_{[-R,R]}}  (a(e) -b(e))^2\right]=4R^2\alpha(\theta).
\end{align*}
Therefore
\begin{align*}
\Eb\left[\sum_{e\in \Ec_{[-R,R]}, a(e), b(e)\in[0,1]}  (a(e) -b(e))^2\right]\le 4R^2\alpha(\theta).
\end{align*}
In a loop soup $\Lc_{[-R,R]}$ in $\Hb$ wired on $[-R, R]$ and free elsewhere, let $\eta$ be the curve from $-R$ to $R$ which is the outer boundary of the cluster glued at $[-R, R]$. Then $\eta$ is distributed as a SLE$_\kappa$ from $-R$ to $R$. Let $g_t$ be the conformal map from $\Hb\setminus \eta([0,t])$ onto $\Hb$ that fixes $0,1, \infty$.  Let $E$ be the event that $\eta \cap D(0, 10)=\emptyset$. We can take $R$ big enough so that $\Pb(E)\ge 1/2$. Let us first prove that on the event $E$,  
\begin{align}\label{eq:conf1}
\sum_{e\in \Ec_{[-R,R]}, a(e), b(e)\in[0,1]}  (g_t(a(e)) -g_t(b(e)))^2 \le C \sum_{e\in \Ec_{[-R,R]}, a(e), b(e)\in[0,1]}  (a(e) -b(e))^2.
\end{align}
There is an infinite measure $\mu$ on Brownian excursions in $\Hb$ between a point $x\in \Rb$ and $\infty$, so that the endpoint $x$ is distributed according to the Lebesgue measure \cite{MR1992830}. For a given $g_t$, there is a constant $C$ (which depends only on $g_t$) such that for all $-R < a<b <R$,  $g_t(b)-g_t(a)$ is equal to $C$ times the $\mu$-measure of excursions in $\Hb$ between $\infty$ and an endpoint in $[a,b]$ which do not intersect $\eta([0,t])$. On the event $E$, an excursion in $\Hb\setminus \eta([0,t])$ from $\infty$ to $[a,b]$ can further be decomposed as the concatenation of an excursion $\xi_1$ between $z_0 \in \partial D(0,10)$ in $\Hb\setminus \eta([0,t])$ and an excursion $\xi_2$ between $z_0$ and a point in $[a,b]$ in $D(0,10)$. The distribution of $z_0$ depends on $\eta([0,t])$, and we denote it by $\Pb_0$. However, given $z_0$, the distribution of $\xi_2$ is independent from $\eta([0,t])$. For all $0\le a<b\le 1$, the ratio
\begin{align}\label{eq:ratio}
\frac{g_t(b) -g_t(a)}{g_t(1) -g_t(0)} = \Eb_0 \left[ \frac{\int_a^b H_{D(0,10)}(z_0, x) dx}{\int_0^1 H_{D(0,10)}(z_0, x) dx} \right].
\end{align}
Since $\partial D(0,10)$ has positive distance from $[0,1]$, there exist $C_1>c_1>0$ such that for all $z_0 \in \partial D(0,10)$ and $x\in[0,1]$, we have $c_1\le H_{D(0,10)}(z_0, x) \le C_1$. Therefore \eqref{eq:ratio} is at most $(C_1/c_1)(b-a)$. This implies that $g_t(b) -g_t(a)\le C(b-a)$ for all $0\le a<b\le 1$ where $C$ is some universal constant, hence \eqref{eq:conf1} follows.


As $t\to \infty$, $g_t$ converges to the conformal map $g_\infty$ from the bounded connected component of $\Hb\setminus \eta([0,\infty])$ onto $\Hb$ that sends $0,1,R$ to $0,1, \infty$. The left-hand side of \eqref{eq:conf1} is at least
\begin{align*}
\sum_{e\in \Ec_{[-R,R]}, a(e), b(e)\in[0,1], e\cap \eta=\emptyset}  (g_t(a(e)) -g_t(b(e)))^2,
\end{align*}
which converges a.s.\ to 
\begin{align}\label{eq:absquare}
\sum_{e\in \Ec, a(e), b(e)\in[0,1]}  (a(e) -b(e))^2,
\end{align}
where $\Ec$ is the set of excursions induced from $g_\infty(\Lc_{[-R,R]})$, a loop soup in $\Hb$ with wired boundary conditions.
By Theorem~\ref{thm:qw19}, the law of \eqref{eq:absquare} is independent from $\eta$, hence remains unchanged if we condition on $E$. Taking the conditional expectation on $E$ of both sides of \eqref{eq:conf1}, letting $t\to\infty$, and using the dominated convergence where the dominator is given by the right-hand side of  \eqref{eq:conf1}, we get
\begin{align}\label{eq:inf1}
&\Eb\left[\sum_{e\in \Ec, a(e), b(e)\in[0,1]}  (a(e) -b(e))^2 \right] \le C \Eb\left[ \sum_{e\in \Ec_{[-R,R]}, a(e), b(e)\in[0,1]}  (a(e) -b(e))^2 \mid E \right]\\
& \le C \Eb\left[\sum_{e\in \Ec_{[-R,R]}, a(e), b(e)\in[0,1]}  (a(e) -b(e))^2\right] \Pb(E)^{-1} \le 8C R^2\alpha(\theta).
\end{align}
This completes the proof of \eqref{eq:bound}.

The proof of \eqref{eq:bound1} follows a similar line.
From Proposition~\ref{prop:sum_mm2}, we can deduce that for $0<r_1<r_2<1$,
\begin{align}
\label{eq:exp1}
&\Eb\left[\sum_{e_1 \not=e_2 \in\Ec_{\Rb^-}} \frac{(b(e_1) -a(e_1))^2 }{(1-b(e_1))^2(1-a(e_1))^2}  \frac{(b(e_2) -a(e_2))^2 }{(1-b(e_2))^2(1-a(e_2))^2} \right] \\
\notag
& \le \Eb\left[\sum_{e_1 \not=e_2 \in\Ec_{\Rb^-}} \frac{(b(e_1) -a(e_1))^2 }{(r_1-b(e_1))^2(r_1-a(e_1))^2}  \frac{(b(e_2) -a(e_2))^2 }{(r_2-b(e_2))^2(r_2-a(e_2))^2} \right] <\infty.
\end{align}
Under the conformal map $f(z)=1/(1-z)$, $f(\Ec_{\Rb^-})$ is distributed as $\Ec_{[0,1]}$, and \eqref{eq:exp1} becomes
\begin{align*}
\Eb\left[\sum_{e_1 \not=e_2 \in\Ec_{[0,1]}} (b(e_1) -a(e_1))^2 (b(e_2) -a(e_2))^2 \right]<\infty.
\end{align*}
Following the same strategy as before, we deduce that
\begin{align*}
\Eb\left[\sum_{e_1\not= e_2 \in\Ec_{\infty};\, a(e_1), b(e_1), a(e_2), b(e_2)\in [0, 1]} (b(e_1) -a(e_1))^2 (b(e_2) -a(e_2))^2 \right] <\infty.
\end{align*}
To deal with the case $e_1 = e_2$, we simply bound $(b(e)-a(e))^4 \leq (b(e)-a(e))^2$ and rely on \eqref{eq:bound}.
This completes the proof of \eqref{eq:bound1}.
\end{proof}

\begin{lemma}\label{lem:X_eps}
For all $\eps\in [0,1]$, let
\begin{align*}
X(\eps):=\sum_{e \in\Ec_{\Rb};\, a(e) \in [0,\eps); b(e) \in [0,1]} (b(e) -a(e))^2.
\end{align*} 
Then there exists a universal constant $C>0$ such that
$\Eb\left[X(\eps) \right] \le C \eps$ and $\Eb\left[X(\eps)^2 \right]\le C \eps$.
\end{lemma}
\begin{proof}
Without loss of generality, assume that $\eps = 1/N$ for some integer $N \geq 1$.
Note that $[0,1)$ is the disjoint union of the sets $I_n$ for $0\le n \le N-1$, where $I_n:=[n /N, (n+1)/N )$. Let
\begin{align*}
X_n(\eps):=\sum_{e \in\Ec_{\Rb};\, a(e) \in I_n; \, b(e) \in [n/N, n/N+1]} (b(e) -a(e))^2.
\end{align*}
By translation invariance, $X_n(\eps)$ has the same distribution as $X(\eps)$.
Hence 
\begin{align*}
\Expect{\sum_{e \in\Ec_{\Rb};\, a(e), b(e) \in [0,2]} (b(e) -a(e))^2} \ge \sum_{n=0}^{N-1} \Eb\left[X_n(\eps) \right]\ge N \Eb[X(\eps)].
\end{align*}
By \eqref{eq:bound}, this implies that $\Eb[X(\eps)]\le CN^{-1} = C \eps$.
The upper bound on $\E[X(\eps)^2]$ is obtained similarly using \eqref{eq:bound1} and
\begin{align*}
\Expect{\left(\sum_{e \in\Ec_{\Rb};\, a(e), b(e) \in [0,2]} (b(e) -a(e))^2\right)^2} \ge \Eb\left[\left(\sum_{n=0}^{N-1} X_n(\eps) \right)^2 \right]\ge\sum_{n=0}^{N-1} \Eb[X_n(\eps)^2]\ge N \Eb[X(\eps)^2].
\end{align*}
\end{proof}

\begin{lemma}\label{lem:estimate_interval}
For all intervals $I_1, I_2\subset \Rb$ with respective lengths $r_1$ and $r_2$, there is a universal constant $C>0$, such that
\begin{align*}
\Eb\left[\sum_{e\in\Ec_{\Rb},\, a(e)\in I_1, b(e)\in I_2} (b(e)-a(e))^2 \right]\le Cr_1r_2.
\end{align*}
\end{lemma}

\begin{proof}
By decomposing $I_1 = (I_1 \setminus I_2) \cup (I_1 \cap I_2)$, it is enough to treat two cases: $I_1 \subset I_2$ and $I_1 \cap I_2 = \varnothing$.

\textbf{Case 1.} $I_1 \subset I_2$. We can decompose $I_2$ into the union of two subintervals sharing an endpoint with $I_1$. We can therefore assume that $I_2$ shares an endpoint with $I_1$. By scaling and translation, this case reduces to Lemma \ref{lem:X_eps}.

\textbf{Case 2.} $I_1 \cap I_2 = \varnothing$.
By scaling and translation, it is enough to consider the case where $I_1=[0,x]$ and $I_2=[y,1]$, where $0<x\le y<1$. If $I_1$ or $I_2$ has length at least $1/4$, then the result is a direct consequence of Lemma \ref{lem:X_eps}. We can thus suppose that both $I_1$ and $I_2$ have length less than $1/4$. We apply the conformal map 
$f:z \in \Hb \mapsto \frac{(1-x)z}{(1-2 x)z+x} \in \Hb$ that fixes $0$ and $1$ and sends $x$ to $1/2$. Then
\begin{align}\label{eq:sum_case4}
\Eb\left[\sum_{a(e)\in [0,x], b(e)\in [y,1]} (b(e)-a(e))^2 \right] =\Eb\left[\sum_{a(e)\in [0, f(x)], b(e)\in [f(y),1]} (f^{-1}(b(e)) - f^{-1}(a(e)))^2 \right].
\end{align}
Note that
\begin{align*}
y-x+2x(1-y)\ge y-x\ge 1/2
\quad \text{and} \quad
1-f(y) =\frac{x(1-y)}{y-x+2x(1-y)} \le 2 x(1-y) \le 1/8.
\end{align*}
This implies that for all $a\in [0, x]$ and $b\in [y, 1]$, we have $f(b)-f(a)\ge 3/8 \ge 3/8 (b-a)$. Conversely, for all $a\in [0, f(x)]$ and $b\in[f(y),1]$, we have $f^{-1}(b)- f^{-1}(a)\le 8/3 (b-a)$. Therefore \eqref{eq:sum_case4} is at most
\begin{align*}
&\frac{64}{9} \Eb\left[\sum_{a(e)\in [0, f(x)], b(e)\in [f(y),1]} (b(e) - a(e))^2 \right]
\le \frac{64}{9} \Eb\left[\sum_{a(e)\in [0, 1], b(e)\in [f(y),1]} (b(e) - a(e))^2 \right]
\end{align*}
which is bounded by $C(1-f(y)) \leq 2Cx(1-y)$ by Lemma \ref{lem:X_eps}. This concludes Case 2 and the proof of Lemma \ref{lem:estimate_interval}.
\end{proof}


\subsection{Proof of Proposition~\ref{prop:be_main}}\label{subsec:proof}
We have now gathered all the ingredients to prove  Proposition~\ref{prop:be_main}.
In this section, we will only deal with $\Ec_\Rb$, so will write $\Ec$ for $\Ec_\Rb$. 
Let us first prove the one-point estimate \eqref{propeq:one-point}.
\begin{proof}[Proof of \eqref{propeq:one-point}]
By conformal invariance, it is enough to prove \eqref{propeq:one-point} for $z=i$. Fix $\eta>0$. By Lemma~\ref{lem:B_1point}, there exists $C>0$, such that for all $r \in (0,1-\eta)$,
\begin{align}\label{eq:probeb}
\Pb[\Ec\cap U(i,r)\not=\emptyset] \le C \Eb\left[\sum_{e\in\Ec} \frac{(a(e)-b(e))^2 }{|a(e)-i|^2 |b(e)-i|^2} \right] |\log r|^{-1}.
\end{align}
Let $I_{-1} = [-1,1]$ and for each $n\ge 0$, let $I_n:=[-2^{n+1}, -2^n) \cup (2^n, 2^{n+1}]$ so that $\R = \bigcup_{n \geq -1} I_n$. Note that for all $n,m \geq -1$, if $a(e) \in I_n$ and $b(e) \in I_m$, then $|a(e)-i| \geq 2^n$ and $|b(e)-i| \geq 2^m$. We can thus bound the expectation in \eqref{eq:probeb} by
\begin{align*}
\sum_{n,m \geq -1} 2^{-2n-2m} \Eb\left[\sum_{e\in\Ec, a(e)\in I_n,  b(e)\in I_m} (a(e)-b(e))^2 \right].
\end{align*}
By Lemma \ref{lem:estimate_interval}, the expectation above is bounded by $C2^{n+m}$. This shows that the expectation in \eqref{eq:probeb} is bounded by some constant concluding the proof.
\end{proof}

Let us now prove the two-point estimate \eqref{propeq:two-point}.
\begin{proof}[Proof of \eqref{propeq:two-point}]
Suppose $z_1 =x_1 + i y_1$ and $z_2 =x_2 +i y_2$.
Let $\nu = \nu(\eta)$ be as in Lemma \ref{lem:b_2p}. Let us first assume that the condition $r_i y_i \leq \nu |z_1-z_2|$, $i=1,2$, from Lemma \ref{lem:b_2p} is violated. For instance, assume that $r_1 y_1 > \nu |z_1-z_2|$. In that case, we simply bound
\begin{align*}
\Pb[\Ec \cap U(z_1,r_1)\not=\emptyset, \Ec_\Rb \cap U(z_2, r_2)\not=\emptyset]
\leq \Pb[\Ec_\Rb \cap U(z_2, r_2)\not=\emptyset]
\leq C |\log r_2|^{-1}
\end{align*}
by \eqref{propeq:one-point}. Since $r_1 y_1 > \nu |z_1-z_2|$, we further have
\[
|\log r_2|^{-1} \leq C \log \frac{y_1}{|z_1-z_2|} |\log r_1|^{-1} |\log r_2|^{-1} \leq C G_\Hb(z_1,z_2) |\log r_1|^{-1} |\log r_2|^{-1}.
\]
This concludes the proof of \eqref{propeq:two-point} in this case.

Let us now assume that $z_1$ and $z_2$ are not too close to each other so that $r_i y_i \leq \nu |z_1-z_2|$, $i=1,2$. This will allow us to apply Lemma \ref{lem:b_2p}.
The event $\Ec\cap U(z_1, r_1)\not=\emptyset, \Ec\cap U(z_2, r_2)\not=\emptyset$ is the union of the following two events: 
\begin{itemize}
\item Let $E_1$ be the event that $U(z_1, r_1)$ and $U(z_2, r_2)$ are both visited by a same excursion in $\Ec$.
\item Let $E_2$ be the event that $U(z_1, r_1)$ and $U(z_2, r_2)$ are visited by two different excursions in $\Ec$.
\end{itemize}
We are going to show that
\[
\P[E_1] \leq C (1+G_\Hb(z_1,z_2)) |\log r_1|^{-1}|\log r_2|^{-1}
\quad \text{and} \quad
\P[E_2] \leq C |\log r_1|^{-1}|\log r_2|^{-1}
\]
which will conclude the proof of \eqref{propeq:two-point}.

\noindent \textbf{Contribution of $E_1$.}
By Lemma~\ref{lem:b_2p}, $\P[E_1]$ is at most
\begin{align*}
C  \Eb \left[\sum_{e\in \Ec} \left( \frac{(a(e)-b(e))^2 y_1 y_2}{|a(e)-z_1|^2|b(e)-z_2|^2}  +   \frac{(a(e)-b(e))^2 y_1 y_2}{|a(e)-z_2|^2|b(e)-z_1|^2} \right)   \right] (1+G_\Hb(z_1,z_2)) |\log r_1|^{-1}|\log r_2|^{-1}.
\end{align*}
By symmetry, it is enough to prove that there exists a universal constant $C>0$ such that
\begin{align}
\label{E:contributionE1}
\Eb \left[\sum_{e\in \Ec} \frac{(a(e)-b(e))^2 y_1 y_2}{|a(e)-z_1|^2|b(e)-z_2|^2} \right]\le C.
\end{align}
For $j=1,2$ and $n \geq 0$, let
\begin{equation}\label{E:Ijn}
I_{j,-1} := [x_j-y_j,x_j+y_j]
\quad \text{and} \quad
I_{j,n}:=[x_j-y_j 2^{n+1}, x_j- y_j 2^n) \cup (x_j+y_j 2^n, x_j+y_j 2^{n+1}].
\end{equation}
As in the proof of \eqref{propeq:one-point}, we can bound the left hand side of \eqref{E:contributionE1} by
\begin{align*}
y_1 y_2 \sum_{n,m \geq -1} (y_1 2^n)^{-2} (y_2 2^n)^{-2} \Eb \left[\sum_{e\in \Ec; a(e)\in I_{1,n}, b(e)\in I_{2,m}} (a(e)-b(e))^2 \right]
\leq C \sum_{n,m \geq -1} 2^{-n-m}
\end{align*}
using Lemma \ref{lem:estimate_interval} in the last inequality.

\noindent \textbf{Contribution of $E_2$.} 
By Lemma~\ref{lem:B_1point}, $\Pb[E_2]$ is upper bounded by a constant times
\begin{align}\label{eq:2psum}
\Eb\left[\sum_{e_1 \in \Ec} \sum_{e_2 \in \Ec}  \frac{(b(e_1) -a(e_1))^2 y_1^2}{|z_1-b(e_1)|^2 |z_1-a(e_1)|^2} \frac{(b(e_2) -a(e_2))^2 y_2^2}{|z_2-b(e_2)|^2 |z_2-a(e_2)|^2}\right] |\log r_1|^{-1} |\log r_2|^{-1}.
\end{align}
Let $I_{j,n}$, $j=1,2$, $n \geq -1$, be as in \eqref{E:Ijn}. We can bound the above expectation by
\begin{align*}
y_1^{-2} y_2^{-2} \sum_{\substack{n_1,m_1 \geq -1\\n_2,m_2 \geq -1}} 2^{-2(n_1-m_1-n_2-m_2)}
\Eb\left[ \sum_{\substack{e_1 \in \Ec; a(e_1)\in I_{1,n_1}, b(e_1) \in I_{1,m_1}\\ e_2 \in \Ec_2; a(e_2)\in I_{2,n_2}, b(e_2) \in I_{2,m_2}}}   (b(e_1) -a(e_1))^2 (b(e_2) -a(e_2))^2  \right].
\end{align*}
By Cauchy-Schwarz, the expectation inside the sum is at most $\sqrt{X_1X_2}$ where
\[
X_j := \Eb\left[ \left( \sum_{e \in \Ec; a(e)\in I_{j,n_j}, b(e) \in I_{j,m_j}}   (b(e) -a(e))^2 \right)^2  \right], \quad j=1,2.
\]
It is enough to show that $X_j \leq C y_j^4 2^{3n_j + 3m_j}$ to conclude the proof.
If $n=m$, then by scaling, $X_j$ is at most $C y_j^4 2^{4n_j}S_2$, where $S_2$ is defined in \eqref{eq:bound1}. 
By symmetry, we can now suppose that $m_j>n_j$. By scaling, we have
\begin{align*}
X_j \leq y_j^4 2^{4m_j} \Eb\left[ \left( \sum_{e \in \Ec; a(e)\in [-2^{n_j-m_j},2^{n_j-m_j}], b(e) \in [-2,2]} (b(e) -a(e))^2 \right)^2  \right]
\leq C y_j^4 2^{4m_j} 2^{n_j-m_j}
\end{align*}
by Lemma \ref{lem:X_eps}. This shows that $X_j \leq C y_j^4 2^{3m_j+n_j} \leq C y_j^4 2^{3m_j+3n_j}$ as desired.
\end{proof}

The previous proofs already indicate that the probability of $\Ec_\Rb$ intersecting a neighborhood of $z\in\Hb$ is mainly contributed by the excursions in $\Ec_\Rb$ which are ``close'' to $z$. We formulate this precisely in the following lemmas, which will be useful in the proof of Proposition \ref{P:independance_u}; see Section \ref{SS:quasi_independence}.

\begin{lemma}\label{L:1pointconc}
Fix $R>0$. Let $\Ec_0$ be the set of excursions in $\Ec_\Rb$ with endpoints in $[-R, R]$. 
For all $\eta>0$ and $r \in (0, 1-\eta)$, we have
\begin{align*}
\Pb[\Ec_\Rb \cap U(i,r) \not=\emptyset]-\Pb[\Ec_0 \cap U(i,r) \not=\emptyset]=o(1) |\log r|^{-1}.
\end{align*}
where the $o(1)$ term goes to $0$ as $R\to \infty$, uniformly for all $r \in(0,1-\eta)$. 
\end{lemma}
\begin{proof}
By Lemma~\ref{lem:B_1point}, there exists $C>0$, such that for all $r \in(0,1-\eta)$,
\begin{align}\label{eq:PCR}
\Pb[\Ec_\Rb \cap U(i,r)\not=\emptyset, \Ec_0 \cap U(i,r)=\emptyset]
\le C\Eb\left[\sum_{e\in\Ec_\Rb\setminus \Ec_0}  \frac{(a(e)-b(e))^2 }{|a(e)-i|^2 |b(e)-i|^2} \right] |\log r|^{-1}.
\end{align}
By monotonicity, it suffices to prove the result for $R=2^N$ and $N\to \infty$. Using the intervals $I_{-1} = [-1,1]$ and $I_n=[-2^{n+1}, -2^n) \cup (2^n 2^{n+1}]$ as in the proof of \eqref{propeq:one-point} and with the help of Lemma \ref{lem:estimate_interval}, the expectation in the right-hand side of \eqref{eq:PCR} is upper bounded by
\begin{align*}
&\sum_{n\ge N,m\ge -1} 2^{-2(n+m)}\Eb\left[\sum_{a(e) \in I_m , b(e)\in I_n}  (a(e)-b(e))^2 \right]
\le C \sum_{n\ge N,m\ge -1} 2^{-n-m} \leq C' 2^{-N} \to 0
\end{align*}
as $N \to \infty$.
This completes the proof.
\end{proof}

\begin{lemma}\label{lem:quasi_indep}
Fix $\eps>0$ and $\nu\in(0,1)$.
Let $\Ec_1$ (resp. $\Ec_2$) be the set of excursions in $\Ec_\Rb$ with both endpoints in $[-\eps^\nu, \eps^\nu]$ (resp. $[1-\eps^\nu, 1+ \eps^\nu]$). 
Fix $\eta\in (0,1)$. For all $z_1\in D(0, \eps)\cap\Hb$, $z_2\in D(1, \eps)\cap \Hb$ and $r_1, r_2\in(0, 1-\eta)$, we have
\begin{align}\label{coreq:quasi_indep}
&\Pb[\Ec_\Rb \cap U(z_1, r_1) \not=\emptyset, \Ec_\Rb \cap U(z_2, r_2) \not=\emptyset]-\Pb[\Ec_1 \cap U(z_1,r_1) \not=\emptyset, \Ec_2 \cap U(z_2, r_2)\not=\emptyset]\\
\notag
 &=o (1)  |\log r_1|^{-1}|\log r_2|^{-1},
\end{align}
where the $o(1)$ term goes to $0$ as $\eps\to 0$, uniformly for all $r_1, r_2\in(0, 1-\eta)$.
\end{lemma}

\begin{proof}
As we explained in details, the proof of Lemma \ref{L:1pointconc} followed from our approach to proving \eqref{propeq:one-point}. Similarly,
the proof of Lemma \ref{lem:quasi_indep} follows from the method we developed in the proof of \eqref{propeq:two-point}. We omit the details.
\end{proof}

%% file: subfiles/properties.tex
In this section, we prove Theorems \ref{T:prop}, \ref{T:level_line} and \ref{T:mixed}, starting with Theorem \ref{T:prop}.
Note that the spatial Markov property in Theorem~\ref{T:level_line}, as well as Theorem  \ref{T:mixed}, follow directly from the corresponding spatial Markov property of the loop soup, proved respectively in \cite{QianWerner19Clusters} and \cite{MR3901648}, as we explained before the statement of these theorems.

\subsection{Proof of Theorem \ref{T:prop}}\label{SS:properties_field}

\begin{proof}[Proof of Theorem \ref{T:prop}]
Point \ref{I:covariance}: See Lemma \ref{L:second_moment_to_crossing_proba} and Proposition \ref{P:crossing}.

Point \ref{I:conformal_invariance}: The conformal invariance of the field follows directly from the conformal covariance of the approximation procedure used to define it, together with the fact that the conformal anomaly disappears when $\gamma \to 0$ (derivative of the conformal map to the power $\gamma^2/2$). See \cite[Theorem 1.1]{ABJL21}.

Point \ref{I:symmetry}:
By definition,
$
h_\theta + h_\theta^- = \lim_{\gamma \to 0} \frac{1}{Z_\gamma} ( \Mc_\gamma - \E[\Mc_\gamma] ).
$
Because the second term in the expansion of $\Mc_\gamma$ is $\gamma^2$ times the occupation measure of $\Lc_D^\theta$ (see Theorem \ref{T:expansion_unsigned}) and because $Z_\gamma \geq \gamma^{2(1-\theta)+o(1)}$, this limit equals zero.

Point \ref{I:nondegenerate}: Let $f : D \to \R$ be a smooth test function, not identically zero. By Theorem \ref{T:h_and_minkowski},
\[
(h_\theta,f) = \sum_{k \geq 1} \sigma_k (\mu_k, f) \quad \quad \text{a.s.}
\]
and by Theorem \ref{T:measure_cluster}, there is a.s. an infinite number of $k$ such that $(\mu_k, f) \neq 0$. Lemma \ref{L:sum_elementary} below then concludes the proof.
\end{proof}

\begin{lemma}\label{L:sum_elementary}
Let $(a_n)_{n \geq 1}$ be a sequence of non zero real numbers and let $(\sigma_n)_{n \geq 1}$ be i.i.d. spins taking values in $\{\pm 1\}$ with equal probability 1/2. Assume that $\sum_n \sigma_n a_n$ converges almost surely. Then the distribution of $\sum_n \sigma_n a_n$ is non-atomic.
\end{lemma}

\begin{proof}
Let $x \in \R$. We are going to show that $\sum_n \sigma_n a_n \neq x$ a.s. which will conclude the proof. The conditional probability $\Prob{ \sum_n \sigma_n a_n = x \vert \sigma_1, \dots, \sigma_{k-1} }$ converges a.s. to $\indic{\sum_n \sigma_n a_n = x}$ as $k \to \infty$. But, by flipping the $k$-th spin and using the fact that $a_k \neq 0$, we see that
\begin{align*}
\P \Big( \sum_n \sigma_n a_n = x \big\vert \{\sigma_i\}_{i=1}^{k-1} \Big)
& \leq \P \Big( \sum_{n \neq k} \sigma_n a_n -\sigma_k a_k \neq x \big\vert \{\sigma_i\}_{i=1}^{k-1}  \Big)  = \P \Big( \sum_n \sigma_n a_n \neq x \big\vert \{\sigma_i\}_{i=1}^{k-1}  \Big).
\end{align*}
Rearranging, this shows that  $\Prob{ \sum_n \sigma_n a_n = x \vert \sigma_1, \dots, \sigma_{k-1} } \leq 1/2$ a.s. Therefore, $\indic{\sum_n \sigma_n a_n = x}$ is at most $1/2$ a.s., meaning that  $\sum_n \sigma_n a_n \neq x$ a.s.
\end{proof}

\subsection{First moment of \texorpdfstring{$h_\theta$}{h theta} with wired boundary conditions}\label{SS:wired_first_moment}

We now move to the proof of \eqref{E:T_level_line_expectation} in Theorem \ref{T:level_line}.

\begin{proof}[Proof of \eqref{E:T_level_line_expectation} in Theorem \ref{T:level_line}]
Let $f$ be a smooth test function compactly included in $\D$.
We will denote by $\Mc_\gamma^{\Lc_\D^\theta \cup \Ec_{\partial \D}^\theta}$ and $\Mc_\gamma^{\Lc_\D^\theta}$ the multiplicative chaos associated to $\Lc_\D^\theta \cup \Ec_{\partial \D}^\theta$ and $\Lc_\D^\theta$ respectively. The definition of the latter has been recalled in Section \ref{S:multiplicative_chaos}. The former is more complicated since it involves the interaction of $\Lc_\D^\theta$ and $\Ec_{\partial \D}^\theta$. This has been extensively studied in \cite{jegoRW} and \cite{ABJL21}. We will only need to use the simple fact that
\begin{equation}
\label{E:level_line1}
\Mc_\gamma^{\Lc_\D^\theta \cup \Ec_{\partial \D}^\theta}(\d z) \indic{z \notin R(\Ec_{\partial \D}^\theta)} = \Mc_\gamma^{\Lc_\D^\theta}(\d z) \indic{z \notin R(\Ec_{\partial \D}^\theta)}
\quad \quad \text{a.s.},
\end{equation}
where we denoted by $R(\Ec_{\partial \D}^\theta)$ the range of $\Ec_{\partial \D}^\theta$, i.e. the set of points visited by some trajectory in $\Ec_{\partial \D}^\theta$.
This property is simply saying that a point not visited by $\Ec_{\partial \D}^\theta$ is thick for $\Lc_\D^\theta \cup \Ec_{\partial \D}^\theta$ if and only if it is thick for $\Lc_\D^\theta$.

We have
\begin{align*}
\Expect{  (h_\theta^\mathrm{wired},f) \vert \sigma_{\partial \D} }
= \lim_{\gamma \to 0}
\int_\D f(z) \frac{1}{Z_\gamma} \Expect{ \Mc_\gamma^{\Lc_\D^\theta \cup \Ec_{\partial \D}^\theta} (\d z) \indic{z \in \Cf^+(\Lc_\D^\theta \cup \Ec_{\partial \D}^\theta)} - \d z \big\vert \sigma_{\partial \D} }
\end{align*}
where the convergence holds in $L^2$. This follows simply from Cauchy--Schwarz and Theorem \ref{T:convergenceL2}.
Let us consider first the case $\sigma_{\partial \D} = -1$. If $z$ belongs to the range of $\Ec_{\partial \D}^\theta$, its sign necessarily agrees with $\sigma_{\partial \D}$. Together with \eqref{E:level_line1}, this gives
\begin{align*}
\Expect{ \Mc_\gamma^{\Lc_\D^\theta \cup \Ec_{\partial \D}^\theta} (\d z) \indic{z \in \Cf^+(\Lc_\D^\theta \cup \Ec_{\partial \D}^\theta)} \big\vert \sigma_{\partial \D} = -1 } 
= \Expect{ \Mc_\gamma^{\Lc_\D^\theta} (\d z) \indic{z \in \Cf^+(\Lc_\D^\theta \cup \Ec_{\partial \D}^\theta)} \indic{z \notin R(\Ec_{\partial \D}^\theta)} \big\vert \sigma_{\partial \D} = -1 }.
\end{align*}
By Lemma \ref{L:moment_measure}, this is further equal to
\begin{align*}
2 \Prob{z \in \Cf^+(\Lc_\D^\theta \cup \Xi_a^{z,\D} \cup \Ec_{\partial \D}^\theta), z \notin R(\Ec_{\partial \D}^\theta) \vert \sigma_{\partial \D} = -1 } \d z.
\end{align*}
Since a Lebesgue-typical point is a.s. not visited by $\Ec_{\partial \D}^\theta$, the above expression further reduces to
\[
\Big( 1 - \P \Big( z \overset{ \Lc_\D^\theta \cup \Xi_a^{z,\D}}{\longleftrightarrow} \Ec_{\partial \D}^\theta \Big) \Big) \d z
= \Big( 1 - \P \Big( 0 \overset{ \Lc_\D^\theta \cup \Xi_a^{0,\D}}{\longleftrightarrow} \Ec_{\partial \D}^\theta \Big) \Big) \d z.
\]
The last equality is obtained by noticing that the probability is conformally invariant.
Altogether, we have obtained that
\[
\Expect{  (h_\theta^\mathrm{wired},f) \vert \sigma_{\partial \D} = -1 } = - \Big( \lim_{\gamma \to 0} \frac{1}{Z_\gamma} \P \Big( 0 \overset{ \Lc_\D^\theta \cup \Xi_a^{0,\D}}{\longleftrightarrow} \Ec_{\partial \D}^\theta \Big) \Big) \int_\D f(z) \d z.
\]
This proves Theorem \ref{T:level_line} in the case $\sigma_{\partial \D} = -1$ with
\begin{equation}
\label{E:c_theta}
c(\theta) = \lim_{\gamma \to 0} \Prob{ 0 \overset{ \Lc_\D^\theta \cup \Xi_a^{0,\D}}{\longleftrightarrow} \Ec_{\partial \D}^\theta } \big/ Z_\gamma.
\end{equation}
The convergence of the right hand side of \eqref{E:c_theta} follows from the above line of arguments. It could also be obtained directly from Theorem \ref{T:ratio_general} and from the estimate \eqref{propeq:one-point} in Proposition \ref{prop:be_main} on $\Ec_{\partial \D}^\theta$.

To obtain the result when $\sigma_{\partial \D} = +1$, we use Theorem \ref{T:prop}, Point \ref{I:symmetry}. With the notations therein, and because $(h_\theta^\mathrm{wired})^- = - h_\theta^\mathrm{wired}$ a.s.,
\begin{align*}
\Expect{  (h_\theta^\mathrm{wired},f) \vert \sigma_{\partial \D} = 1 }
= - \Expect{ ((h_\theta^\mathrm{wired})^-, f) \vert - \sigma_{\partial \D} = -1 }.
\end{align*}
From the case of negative sign on $\partial \D$ that we have just treated, we know that the expectation on the right hand side equals $- c(\theta) \int_\D f(z) \d z$ a.s. This concludes the proof of the case $\sigma_{\partial \D} = +1$.
\end{proof}

\subsection{Second moment of \texorpdfstring{$h_\theta$}{h theta} with wired boundary conditions}\label{SS:wired_second_moment}

The goal of this section is to prove \eqref{E:T_level_line_bc} in Theorem \ref{T:level_line}. We start with a preliminary result that relates the second moment of $h_\theta^\mathrm{wired}$ to connectivity probabilities (see Lemma \ref{L:second_moment_to_crossing_proba} for the analogous result in a domain with zero-boundary condition).
As in Notation \ref{N:covariance} and conditionally given $\Ec_{\partial \D}^\theta$, we will denote by $C_{\theta, \D \setminus \Ec_{\partial \D}^\theta}$ the covariance of $h_\theta$ built out of a loop soup in the domain $\D \setminus \Ec_{\partial \D}^\theta$ with zero-boundary condition. Since the loops that are in different connected components of $\D \setminus \Ec_{\partial \D}^\theta$ are independent of each other, if $x$ and $y$ belong to different connecting component of $\D \setminus \Ec_{\partial \D}^\theta$, then $C_{\theta, \D \setminus \Ec_{\partial \D}^\theta}(x,y) = 0$.

\begin{lemma}\label{L:train}
Let $f : \D \to [0,\infty)$ be a smooth nonnegative function with compact support in $\D$. Then $\Expect{(h_\theta^\mathrm{wired},f)^2}$ is equal to
\begin{equation}
\label{E:L_height_second}
\int_{\D \times \D} \left( \Expect{ C_{\theta, \D \setminus \Ec_{\partial \D}^\theta}(x,y) } + \lim_{\gamma \to 0} \frac{1}{Z_\gamma^2} \Prob{ \Ec_{\partial \D}^\theta \overset{\Lc_\D^\theta}{\longleftrightarrow} \Xi_a^{x,\D}, \Ec_{\partial \D}^\theta \overset{\Lc_\D^\theta}{\longleftrightarrow} \Xi_a^{y,\D} } \right) f(x) f(y) \d x \d y.
\end{equation}
\end{lemma}

\begin{proof}
By symmetry,
\begin{align}
\nonumber
& \Expect{(h_\theta^\mathrm{wired},f)^2} = \Expect{ (h_\theta^\mathrm{wired},f)^2 \vert \sigma_{\partial \D} = -1 } 
= \lim_{\gamma \to 0} \frac{1}{Z_\gamma^2} \int_{\D \times \D} f(x) f(y) \times \\
& \times \Expect{ \left. \left( \Mc_\gamma^{\Lc_\D^\theta \cup \Ec_{\partial \D}^\theta} (\d x) \indic{x \in \Cf^+(\Lc_\D^\theta \cup \Ec_{\partial \D}^\theta)} - \d x \right) \left( \Mc_\gamma^{\Lc_\D^\theta \cup \Ec_{\partial \D}^\theta} (\d y) \indic{y \in \Cf^+(\Lc_\D^\theta \cup \Ec_{\partial \D}^\theta)} - \d y \right) \right\vert \sigma_{\partial \D} = -1 }.
\label{E:proof_height4}
\end{align}
We expand the product in the expectation on the right hand side as a sum of four terms. In Section \ref{SS:wired_first_moment} we already dealt with the crossed terms:
\[
\Expect{ \left. \Mc_\gamma^{\Lc_\D^\theta \cup \Ec_{\partial \D}^\theta} (\d z) \indic{z \in \Cf^+(\Lc_\D^\theta \cup \Ec_{\partial \D}^\theta)} \right\vert \sigma_{\partial \D} = -1 } = (1- p_{\gamma}) \d z
\quad \text{where} \quad
p_{\gamma} := \Prob{ z \overset{ \Lc_\D^\theta \cup \Xi_a^{z,\D}}{\longleftrightarrow} \mathcal{E}_{\partial \D}^\theta }.
\]
Let us recall that, by conformal invariance, $p_\gamma$ does not depend on $z$.
We now focus on
\begin{equation}
\label{E:proof_height1}
\Expect{ \left. \Mc_\gamma^{\Lc_\D^\theta \cup \Ec_{\partial \D}^\theta} (\d x) \Mc_\gamma^{\Lc_\D^\theta \cup \Ec_{\partial \D}^\theta} (\d y) \indic{x,y \in \Cf^+(\Lc_\D^\theta \cup \Ec_{\partial \D}^\theta)} \right\vert \sigma_{\partial \D} = -1 }.
\end{equation}
As in Section \ref{SS:wired_first_moment}, because $\sigma_{\partial \D} = -1$, $x$ and $y$ cannot belong to $\Ec_{\partial \D}^\theta$ allowing us to replace $\Mc_\gamma^{\Lc_\D^\theta \cup \Ec_{\partial \D}^\theta}$ by $\Mc_\gamma^{\Lc_\D^\theta}$. Moreover, by Lemma \ref{L:moment_measure}, the contribution of soups $\Lc_\D^\theta$ with a loop visiting both $x$ and $y$ vanishes as $\gamma \to 0$ (this is the same computation as in \eqref{E:proof_loop_2points}). \eqref{E:proof_height1} is therefore equal to
\[
\Expect{ \left. \Mc_\gamma^{\Lc_\D^\theta} (\d x) \Mc_\gamma^{\Lc_\D^\theta} (\d y) \indic{x,y \in \Cf^+(\Lc_\D^\theta \cup \Ec_{\partial \D}^\theta), \nexists \wp \in \Lc_\D^\theta: \{x,y\} \subset \wp} \indic{x,y \notin R(\Ec_{\partial \D}^\theta)} \right\vert \sigma_{\partial \D} = -1 } + o(Z_\gamma^2) \d x  \d y.
\]
Now, by Lemma \ref{L:moment_measure}, this is further equal to
\begin{align*}
& 4 \Prob{ x,y \in \Cf^+( \Lc_\D^\theta \cup \Xi_a^{x,\D} \cup \Xi_a^{y,\D} \cup \Ec_{\partial \D}^\theta ), x,y \notin R(\Ec_{\partial \D}^\theta) \vert \sigma_{\partial \D} = -1 } \d x  \d y + o(Z_\gamma^2)\d x  \d y.
\end{align*}
Since Lebesgue-typical points are almost surely not visited by $\Ec_{\partial \D}^\theta$, the event $\{ x,y \notin R(\Ec_{\partial \D}^\theta)\}$ has full probability. Then density of the first measure above is then equal to
\begin{align*}
& 2 \Prob{ \Xi_a^{x,\D} \overset{\Lc_\D^\theta}{\longleftrightarrow} \Xi_a^{y,\D}, \{ \Ec_{\partial \D}^\theta \overset{\Lc_\D^\theta}{\longleftrightarrow} \Xi_a^{x,\D} \cup \Xi_a^{y,\D} \}^c} 
+ \Prob{ \{ \Xi_a^{x,\D} \overset{\Lc_\D^\theta}{\longleftrightarrow} \Xi_a^{y,\D} \}^c, \{ \Ec_{\partial \D}^\theta \overset{\Lc_\D^\theta}{\longleftrightarrow} \Xi_a^{x,\D} \cup \Xi_a^{y,\D} \}^c } \\
& = \Prob{ \Xi_a^{x,\D} \overset{\Lc_\D^\theta}{\longleftrightarrow} \Xi_a^{y,\D}, \{ \Ec_{\partial \D}^\theta \overset{\Lc_\D^\theta}{\longleftrightarrow} \Xi_a^{x,\D} \cup \Xi_a^{y,\D} \}^c }
+ \Prob{ \{ \Ec_{\partial \D}^\theta \overset{\Lc_\D^\theta}{\longleftrightarrow} \Xi_a^{x,\D} \cup \Xi_a^{y,\D} \}^c } \\
& = \Prob{ \Xi_a^{x,\D} \overset{\Lc_\D^\theta}{\longleftrightarrow} \Xi_a^{y,\D}, \{ \Ec_{\partial \D}^\theta \overset{\Lc_\D^\theta}{\longleftrightarrow} \Xi_a^{x,\D} \cup \Xi_a^{y,\D} \}^c }
+ 1 - 2p_\gamma + \Prob{ \Ec_{\partial \D}^\theta \overset{\Lc_\D^\theta}{\longleftrightarrow} \Xi_a^{x,\D}, \Ec_{\partial \D}^\theta \overset{\Lc_\D^\theta}{\longleftrightarrow} \Xi_a^{y,\D} }.
\end{align*}
By the restriction properties of $\Lc_\D^\theta$ and $\Xi_a^{z,\D}$, $z=x,y$, one can show that
\[
\lim_{\gamma \to 0} \frac{1}{Z_\gamma^2} \Prob{ \Xi_a^{x,\D} \overset{\Lc_\D^\theta}{\longleftrightarrow} \Xi_a^{y,\D}, \{ \Ec_{\partial \D}^\theta \overset{\Lc_\D^\theta}{\longleftrightarrow} \Xi_a^{x,\D} \cup \Xi_a^{y,\D} \}^c } = \Expect{ C_{\theta, \D \setminus \Ec_{\partial \D}^\theta}(x,y) }
\]
where the expectation is w.r.t. $\Ec_{\partial \D}^\theta$.
The term $1-2p_\gamma$ cancels out with the three other terms in the expansion of the product inside the expectation in \eqref{E:proof_height4}.
Putting things together leads to \eqref{E:L_height_second}.
\end{proof}

We now state another key intermediate result (Proposition \ref{P:independance_u} below) for the proof of \eqref{E:T_level_line_bc} that we will prove in Section \ref{SS:quasi_independence} below.
By Theorem \ref{T:ratio_general}, we can define
\begin{equation}\label{E:def_u}
u(z) := \lim_{\gamma \to 0} \frac{1}{Z_\gamma} \Prob{ \left. \Ec_{\partial \D}^\theta \overset{\Lc_\D^\theta}{\longleftrightarrow} \Xi_a^{z,\D} \right\vert \Ec_{\partial \D}^\theta}, \quad z \in \D.
\end{equation}
By conformal invariance, the law of $u(z)$ does not depend on $z$. The constant $c(\theta)$ appearing in Theorem \ref{T:level_line} corresponds exactly to the expectation of $u(z)$; see \eqref{E:c_theta}. Proposition \ref{P:independance_u} below shows that $u(x_\eps)$ and $u(y_\eps)$ become asymptotically uncorrelated when $(x_\eps)_\eps$ and $(y_\eps)_\eps$ are targeting two distinct fixed boundary points.

\begin{proposition}\label{P:independance_u}
Let $x, y \in \partial \D$ be two distinct boundary points.
Let $(x_\eps)_\eps$ and $(y_\eps)_\eps$ be two sequences of points in $\D$ such that for all $\eps >0$, $|x-x_\eps| < \eps$ and $|y-y_\eps| < \eps$.
Then
\[
\max( \Expect{u(x_\eps) u(y_\eps)} - c(\theta)^2, 0) \to 0 \quad \quad \text{as} \quad \quad \eps \to 0.
\]
Moreover, for all $\delta >0$ fixed, the convergence is uniform in $x,y,(x_\eps)_\eps,(y_\eps)_\eps,$ provided that $|x - y| \geq \delta$.
\end{proposition}

The proof of Proposition \ref{P:independance_u} is written in Section \ref{SS:quasi_independence}. When $\theta = 1/2$, one can use that $\Ec_{\partial \D}$ is a Poisson point process of Brownian excursions to prove Proposition \ref{P:independance_u} fairly painlessly. No such description is available for general intensities $\theta \leq 1/2$ and we will rely on estimates and techniques developed in Section \ref{sec:exc_bdy}. 
We can now prove:

\begin{proof}[Proof of \eqref{E:T_level_line_bc} in Theorem \ref{T:level_line}]
Let $(f_\eps)_{0<\eps<1}$ be a sequence of functions satisfying Assumption \ref{assumption_feps}. 
By \eqref{E:T_level_line_expectation}, for any $\eps >0$, $\Expect{(h_\theta^\mathrm{wired},f_\eps) \vert \sigma_{\partial \D}} = \sigma_{\partial \D} c(\theta)$ a.s. To show that $(h_\theta^\mathrm{wired},f_\eps) - \sigma_{\partial \D} c(\theta) \to 0$ in $L^2$, it is therefore enough to show that
\[
\limsup_{\eps \to 0} \Expect{(h_\theta^\mathrm{wired},f_\eps)^2} \leq c(\theta)^2.
\]
By Lemma \ref{L:train}, $\limsup_{\eps \to 0} \Expect{(h_\theta^\mathrm{wired},f_\eps)^2}$ is at most
\begin{equation}
\label{E:july2}
\limsup_{\eps \to 0} \int_{\D \times \D} \lim_{\gamma \to 0} \frac{1}{Z_\gamma^2} \Prob{ \Ec_{\partial \D}^\theta \overset{\Lc_\D^\theta}{\longleftrightarrow} \Xi_a^{x,\D}, \Ec_{\partial \D}^\theta \overset{\Lc_\D^\theta}{\longleftrightarrow} \Xi_a^{y,\D} } f_\eps(x) f_\eps(y) \d x \d y
\end{equation}
plus
\begin{equation}
\label{E:july3}
\limsup_{\eps \to 0} \int_{\D \times \D} \Expect{ C_{\theta,\D \setminus \Ec_{\partial \D}^\theta}(x,y) } f_\eps(x) f_\eps(y) \d x \d y.
\end{equation}
Let us first show that \eqref{E:july3} vanishes. To this end, notice that $C_{\theta,\D \setminus \Ec_{\partial \D}^\theta}(x,y)$ vanishes as soon as $x$ and $y$ belong to different connected components of $\D \setminus \Ec_{\partial \D}^\theta$ and notice also that $C_{\theta,D}(x,y)$ is nondecreasing with the domain $D$ (this follows from the expression \eqref{E:P_crossing_limit}). We can thus bound
$
\Expect{ C_{\theta,\D \setminus \Ec_{\partial \D}^\theta}(x,y) }
$
by $C_{\theta,\D}(x,y)$
times the probability $p_{xy}$ that $x$ and $y$ belong to the same connected component of $\D \setminus \Ec_\D^\theta$. Let $\delta >0$.  If $|x-y| < \delta$, we simply bound $p_{xy} \leq 1$. Otherwise, $p_{xy}$ vanishes uniformly in $x,y$ with $\d(x,\partial \D), \d(y, \partial \D) \to 0$ and $|x-y| \geq \delta$. \eqref{E:july3} is therefore at most
\begin{align*}
\limsup_{\eps \to 0} \int_{\D \times \D} \indic{|x-y|< \delta} C_{\theta,\D}(x,y) f_\eps(x) f_\eps(y) \d x \d y
+ \limsup_{\eps \to 0} o(1) \int_{\D \times \D} \indic{|x-y| \geq \delta} C_{\theta,\D}(x,y) f_\eps(x) f_\eps(y) \d x \d y
\end{align*}
where $o(1)$ may depend on $\delta$ and vanishes as $\eps \to 0$, $\delta >0$ fixed. By \eqref{E:assumption_feps}, the second integral above is bounded uniformly in $\eps$, so the second term vanishes. We can now let $\delta \to 0$ and observe that the first term goes to 0 by \eqref{E:assumption_feps2}. This shows that \eqref{E:july3} vanishes.

The rest of the proof consists in showing that \eqref{E:july2} is at most $c(\theta)^2$. Let $\eps>0, \gamma >0, x,y \in \D$. Let us first show that
\begin{equation}
\label{E:july4}
\frac{1}{Z_\gamma^2} \Prob{ \Ec_{\partial \D}^\theta \overset{\Lc_\D^\theta}{\longleftrightarrow} \Xi_a^{x,\D}, \Ec_{\partial \D}^\theta \overset{\Lc_\D^\theta}{\longleftrightarrow} \Xi_a^{y,\D} }
\leq C \max(1,-\log|x-y|).
\end{equation}
Let
\begin{equation}
\label{E:Rz}
R_z := \sup \{ R >0: \partial D(z, R) \overset{\Lc_\D^\theta}{\longleftrightarrow} \Xi_a^{z,\D} \}, \quad \quad z \in \{x,y\}.
\end{equation}
We have
\begin{equation}
\label{E:benoit2}
\Prob{ \Ec_{\partial \D}^\theta \overset{\Lc_\D^\theta}{\longleftrightarrow} \Xi_a^{x,\D}, \Ec_{\partial \D}^\theta \overset{\Lc_\D^\theta}{\longleftrightarrow} \Xi_a^{y,\D} }
\leq \int_{(0,2)^2} \Prob{R_z \in \d r_z, z=x,y} \Prob{\Ec_{\partial \D}^\theta \cap D(z,r_z), z =x,y}.
\end{equation}
By Proposition \ref{propeq:two-point},
\[
\Prob{\Ec_{\partial \D}^\theta \cap D(z,r_z), z =x,y}
\leq \frac{C \max(1, -\log|x-y|)}{\max(1,\log(\d(x,\partial \D)/r_x)) \max(1,\log(\d(y,\partial \D)/r_y))}.
\]
Plugging this in \eqref{E:benoit2} and integrating by parts with respect to both $r_x$ and $r_y$ leads to the following upper bound for the left hand side of \eqref{E:benoit2}:
\begin{align*}
C \max(1, -\log|x-y|) \int_{0}^{e^{-1} \d(y,\partial \D)} \d r_y \int_{0}^{e^{-1} \d(x,\partial \D)} \d r_x \frac{\Prob{R_z \geq r_z, z =x,y}}{r_x r_y (\log \d(x,\partial \D)/r_x)^2 (\log(\d(y,\partial \D)/r_y))^2}.
\end{align*}
The fact that we integrate over the intervals $(0,e^{-1} \d(z,\partial \D))$, $z=x,y$, stems from the fact that the derivative of $r_z \mapsto 1/\max(1,\log(\d(z,\partial \D)/r_z)$ vanishes outside of this interval. 
A slight variant of the argument of the proof of Lemma \ref{L:crossing_upper} shows that for $\eta$ arbitrary small (for our purposes, $\eta = \theta/2$ would do), there exists $C=C(\eta)>0$ such that
\begin{equation}
\label{E:Rz2}
\Prob{R_z \geq r_z, z =x,y} \leq C Z_\gamma^2 \prod_{z=x,y} \left( \log \frac{\d(z,\partial \D)}{r_z} \right)^{1-\theta+\eta},
\quad \quad r_z \in (0, e^{-1} \d(z,\partial \D)), \quad z=x,y.
\end{equation}
We omit these details. Putting things together and noticing that
\[
\int_{0}^{e^{-1} \d(z,\partial \D)} \frac{\d r_z}{r_z (\log \d(z,\partial \D)/r_z)^{1-\theta+\eta}} \leq C,
\quad \quad \text{if}~ \eta <\theta,
\quad \quad z =x,y,
\]
proves \eqref{E:july4}.

We now continue the proof that \eqref{E:july2} is at most $c(\theta)^2$. 
By \eqref{E:july4} and \eqref{E:assumption_feps2}, the contribution to \eqref{E:july3} of points $x,y$ with $|x-y| < \delta$ vanishes as $\delta \to 0$. More precisely, \eqref{E:july3} is at most
\[
\limsup_{\delta \to 0} \limsup_{\eps \to 0} \int_{\D \times \D} \indic{|x-y| \geq \delta} \lim_{\gamma \to 0} \frac{1}{Z_\gamma^2} \Prob{ \Ec_{\partial \D}^\theta \overset{\Lc_\D^\theta}{\longleftrightarrow} \Xi_a^{x,\D}, \Ec_{\partial \D}^\theta \overset{\Lc_\D^\theta}{\longleftrightarrow} \Xi_a^{y,\D} } f_\eps(x) f_\eps(y) \d x \d y.
\]
Let $\delta >0$ be fixed. Recalling that $\int_\D f_\eps = 1$ for all $\eps$, in order to conclude that \eqref{E:july3} is at most $c(\theta)^2$, it is enough to show that
\begin{equation}
\label{E:july5}
\max \left(
\lim_{\gamma \to 0} \frac{1}{Z_\gamma^2} \Prob{ \Ec_{\partial \D}^\theta \overset{\Lc_\D^\theta}{\longleftrightarrow} \Xi_a^{x,\D}, \Ec_{\partial \D}^\theta \overset{\Lc_\D^\theta}{\longleftrightarrow} \Xi_a^{y,\D} } - c(\theta)^2, 0 \right)
\end{equation}
converges uniformly to $0$ as $\d(x,\partial \D), \d(y,\partial \D) \to 0$ while $|x-y|$ stays at least $\delta$. By extracting subsequences if necessary, we can moreover assume that $x$ and $y$ converge to distinct boundary points.
So, let $x_0, y_0 \in \partial \D$ be two distinct boundary points and let $(x_\eps)_\eps$ and $(y_\eps)_\eps$ be two sequences of points in $\D$ that converge to $x_0$ and $y_0$ respectively.
We first notice that the loops in $\Lc_\D^\theta$ that will contribute to the probability in \eqref{E:july5} are localised near $x_\eps$ and $y_\eps$. More precisely, consider $A>0$ large and denote by $\Lc_{\D,z}^\theta = \{ \wp \in \Lc_\D^\theta : \wp \subset D(z, A \d(z,\partial \D)) \}$, for $z \in D$.
Then
\begin{equation}
\label{E:july1}
\Prob{ \Ec_{\partial \D}^\theta \overset{\Lc_\D^\theta}{\longleftrightarrow} \Xi_a^{x_\eps,\D}, \Ec_{\partial \D}^\theta \overset{\Lc_\D^\theta}{\longleftrightarrow} \Xi_a^{y_\eps,\D} }
= (1+o(1)) \Prob{ \Ec_{\partial \D}^\theta \overset{\Lc_{\D,x_\eps}^\theta}{\longleftrightarrow} \Xi_a^{x_\eps,\D}, \Ec_{\partial \D}^\theta \overset{\Lc_{\D,y_\eps}^\theta}{\longleftrightarrow} \Xi_a^{y_\eps,\D} }
\end{equation}
where $o(1) \to 0$ as $A \to \infty$, uniformly in $\gamma$ and $\eps>0$. This follows from the same type of considerations as in the proof of \eqref{E:july4}, and eventually boils down to the fact that the random variables $R_z$, $z \in \{x,y\}$, defined in \eqref{E:Rz}, are of order $\d(z,\partial \D)$. See in particular \eqref{E:Rz2}. But now, for $A>0$ large but fixed, if $\eps$ is small enough the discs $D(x_\eps, A \d(x_\eps,\partial \D))$ and $D(y_\eps, A \d(y_\eps,\partial \D))$ are disjoint implying that $\Lc_{\D,x_\eps}^\theta$ and $\Lc_{\D,y_\eps}^\theta$ are independent. Hence the probability on the right hand side of \eqref{E:july1} is equal to
\[
\Expect{ \Prob{\left. \Ec_{\partial \D}^\theta \overset{\Lc_{\D,x_\eps}^\theta}{\longleftrightarrow} \Xi_a^{x_\eps,\D} \right\vert \Ec_{\partial \D}^\theta } \Prob{ \left. \Ec_{\partial \D}^\theta \overset{\Lc_{\D,y_\eps}^\theta}{\longleftrightarrow} \Xi_a^{y_\eps,\D} \right\vert \Ec_{\partial \D}^\theta}
}.
\]
As in \eqref{E:july1} and by sending $A \to \infty$, we can replace $\Lc_{\D,x_\eps}^\theta$ and $\Lc_{\D,y_\eps}^\theta$ by $\Lc_{\D}^\theta$ (alternatively, adding more loops can only increase the probability of connection). This shows that
\begin{align*}
& \limsup_{\eps \to 0}
\lim_{\gamma \to 0} \frac{1}{Z_\gamma^2} \Prob{ \Ec_{\partial \D}^\theta \overset{\Lc_\D^\theta}{\longleftrightarrow} \Xi_a^{x_\eps,\D}, \Ec_{\partial \D}^\theta \overset{\Lc_\D^\theta}{\longleftrightarrow} \Xi_a^{y_\eps,\D} } \\
& \leq
\limsup_{\eps \to 0}
\lim_{\gamma \to 0} \frac{1}{Z_\gamma^2} \Expect{ \Prob{\left. \Ec_{\partial \D}^\theta \overset{\Lc_\D^\theta}{\longleftrightarrow} \Xi_a^{x_\eps,\D} \right\vert \Ec_{\partial \D}^\theta } \Prob{ \left. \Ec_{\partial \D}^\theta \overset{\Lc_\D^\theta}{\longleftrightarrow} \Xi_a^{y_\eps,\D} \right\vert \Ec_{\partial \D}^\theta}
}
= \limsup_{\eps \to 0} \Expect{u(x_\eps)u(y_\eps)}
\end{align*}
where $u(z), z \in \D,$ is defined in \eqref{E:def_u}. In the last equality above we applied dominated convergence theorem (the domination is obtained from the estimate \eqref{E:Rz2} on $R_z$ and by Proposition \ref{propeq:two-point}).
By Proposition \ref{P:independance_u}, $\limsup_{\eps \to 0} \Expect{u(x_\eps)u(y_\eps)} \leq c(\theta)^2$. This concludes the proof that \eqref{E:july3} is at most $c(\theta)^2$ which concludes the proof of \eqref{E:T_level_line_bc}.
\end{proof}

\subsection{Quasi independence of excursions near two distinct points}\label{SS:quasi_independence}

This section is devoted to the proof of Proposition \ref{P:independance_u}.

\begin{proof}[Proof of Proposition \ref{P:independance_u}]
By conformal invariance, we may work in the upper half plane $\Hb$ instead of the unit disc $\D$ and assume that our boundary points are $0$ and $1$.
Let $\eps \in (0,1/4)$, $z_1 \in \Hb \cap D(0,\eps)$, $z_2 \in \Hb \cap D(1,\eps)$.
Let $\Ec_1$ (resp. $\Ec_2$) be the set of excursions in $\Ec_\Rb$ with both endpoints in $[-\eps^{2/3}, \eps^{2/3}]$ (resp. $[1-\eps^{2/3}, 1+\eps^{2/3}]$).
For $j=1,2$, define
\begin{equation}\label{E:uj}
u_j:=\lim_{\gamma\to 0} \frac{1}{Z_\gamma} \Pb(\Ec_j  \overset{\Lc_\Hb^\theta \cup \Xi_a^{z_j,\Hb}}{\longleftrightarrow}  z_j \mid \Ec_j).
\end{equation}
The difference with $u(z_j)$ \eqref{E:def_u} is that the intersection can only come from an excursion from $\Ec_j$, instead of the whole set $\Ec_\R$ of excursions.
We are first going to show that
\begin{equation}
\label{E:pf_q3}
\limsup_{\eps \to 0}
\Expect{u(z_1) u(z_2)} \leq \limsup_{\eps \to 0} \Expect{u_1 u_2}.
\end{equation}
Let $K_{\gamma,1}$ and $K_{\gamma,2}$ be two \emph{independent} random compacts of $\Hb$ distributed as (the topological closure of) the cluster of $z_j$ in $\Lc_\Hb^\theta \cup \Xi_a^{z_j,\Hb}$, $j=1,2$.
We can bound for $j=1,2$,
\begin{align*}
u(z_j) \leq u_j + \lim_{\gamma \to 0} \frac{1}{Z_\gamma} \P (\Ec_\R \setminus \Ec_j \cap K_{\gamma,j} \neq \varnothing \mid \Ec_\R ).
\end{align*}
Multiplying these two inequalities for $j=1,2$ leads to
\begin{align*}
& \Expect{u(z_1) u(z_2)} - \Expect{u_1 u_2} \leq \sum_{\substack{i=1,2\\j=3-i}} \Expect{\lim_{\gamma \to 0} \frac{1}{Z_\gamma^2} \Prob{\left. \Ec_{\R} \setminus \Ec_i \cap K_{\gamma,i} \neq \varnothing, \Ec_j \cap K_{\gamma,j} \neq \varnothing \right\vert \Ec_\R }  } \\
& + \Expect{\lim_{\gamma \to 0} \frac{1}{Z_\gamma^2}
\Prob{\left. \forall j=1,2, \Ec_{\R} \setminus \Ec_j \cap K_{\gamma,j} \neq \varnothing \right\vert \Ec_\R } } \\
& \leq 2 \lim_{\gamma \to 0} \frac{1}{Z_\gamma^2} \sum_{\substack{i=1,2\\j=3-i}} \Prob{ \Ec_{\R} \setminus \Ec_i  \cap K_{\gamma,i} \neq \varnothing, \Ec_\R \cap K_{\gamma,j} \neq \varnothing }.
\end{align*}
In the last inequality we exchanged expectation and limit with the help of Fatou's lemma (or alternatively dominated convergence theorem).
For $j=1,2$, let $f_j : \D \to \Hb$ be a conformal map that sends $0$ to $z_j$ and let
\[
C_j(r) := f_j( r \partial \D ), \quad r \in [0,1],
\quad \text{and} \quad
R_{\gamma,j} := \sup \{ r>0: C_j(r) \cap K_{\gamma,j} \neq \varnothing \}.
\]
The definition of $C_j(r)$ does not depend on the specific choice of $f_j$. Note also that, by conformal invariance, $R_{\gamma,1}$ and $R_{\gamma,2}$ have the same law as $R_\gamma$ defined in \eqref{E:Rgamma}.
Let $\eta >0$ and let us first consider the event that $R_{\gamma,1}$ and $R_{\gamma,2}$ are at most $1-\eta$.
We have
\begin{align}\label{E:pf_q1}
    & \lim_{\gamma \to 0} \frac{1}{Z_\gamma^2} \Prob{ \Ec_{\R} \setminus \Ec_1  \cap K_{\gamma,1} \neq \varnothing, \Ec_\R \cap K_{\gamma,2} \neq \varnothing, R_{\gamma,1} \leq 1-\eta, R_{\gamma,2} \leq 1-\eta } \\
    & \leq \lim_{\gamma \to 0} \frac{1}{Z_\gamma^2} \int_{(0,1-\eta)^2} \Prob{R_{\gamma,1} \in \d r_1} \Prob{R_{\gamma,2} \in \d r_2}
\Prob{\Ec_\R \setminus \Ec_1 \cap C_1(r_1) \neq \varnothing, \Ec_\R \cap C_2(r_2) \neq \varnothing }. \nonumber
\end{align}
For $\eta >0$ fixed, Lemma \ref{lem:quasi_indep} shows that
\[
\Prob{\Ec_\R \setminus \Ec_1 \cap C_1(r_1) \neq \varnothing, \Ec_\R \cap C_2(r_2) \neq \varnothing }
\leq o_{\eps \to 0}(1) |\log r_1|^{-1} |\log r_2|^{-1}
\]
where $o_{\eps \to 0}(1) \to 0$ as $\eps \to 0$, uniformly in $r_1, r_2 \in (0,1-\eta)$. Plugging this estimate into \eqref{E:pf_q1} and integrating by parts shows that the left hand side of \eqref{E:pf_q1} is at most
\[
\frac{o_{\eps \to 0}(1)}{Z_\gamma^2} \int_{(0,1-\eta)^2} \frac{ \Prob{R_1 \geq r_1} \Prob{R_2 \geq r_2} }{r_1 r_2 |\log r_1 \log r_2|^2} \d r_1 \d r_2.
\]
Now, by \eqref{E:L_Rgamma1} in Lemma \ref{L:Rgamma} (applied to $\eta=\theta/2$), there exists $C>0$ such that for all $r_j \in (0,1-\eta)$, $\Prob{R_j \geq r_j} \leq C Z_\gamma |\log r_j|^{1-\theta/2}$. This shows that \eqref{E:pf_q1} vanishes as $\eps \to 0$.
On the event that $R_{\gamma,1}$ or $R_{\gamma,2}$ is at least $1-\eta$, we have with a union bound,
\begin{align*}
    & \lim_{\gamma \to 0} \frac{1}{Z_\gamma^2} \Prob{ \Ec_{\R} \setminus \Ec_1  \cap K_{\gamma,1} \neq \varnothing, \Ec_\R \cap K_{\gamma,2} \neq \varnothing, R_{\gamma,1} \text{~or~} R_{\gamma,2} \geq 1-\eta } \\
    & \leq \lim_{\gamma \to 0} \frac{1}{Z_\gamma^2} \sum_{\substack{i=1,2\\j=3-i}} \Prob{R_{\gamma,i} \geq 1-\eta} \Prob{\Ec_\R \cap K_{\gamma,j} \neq \varnothing}
    = 2c(\theta) \lim_{\gamma \to 0} \frac{1}{Z_\gamma} \Prob{R_\gamma \geq 1-\eta}
\end{align*}
where $R_\gamma$ is defined in \eqref{E:Rgamma}. The last equality follows from conformal invariance and \eqref{E:c_theta}. In particular,
the right hand side is independent of $\eps$ and goes to zero as $\eta \to 0$ (\eqref{E:L_Rgamma2} in Lemma \ref{L:Rgamma}). This concludes the proof of \eqref{E:pf_q3}.

The proof of Proposition \ref{P:independance_u} will then follow from (recall that $u_1$ and $u_2$ are defined in \eqref{E:uj}):

\begin{lemma}\label{lem:indep_u}
$\limsup_{\eps \to 0} \Expect{u_1 u_2} \leq c(\theta)^2$.
\end{lemma}

\begin{proof}[Proof of Lemma~\ref{lem:indep_u}]
Instead of considering directly a loop soup $\Lc_\Rb$ which is wired everywhere, we will obtain $\Lc_\Rb$ and $\Ec_\Rb$ by exploring a loop soup which is partially wired, and then use the independence between different parts of the loop soup derived for this partial exploration process \cite{MR3901648,QianWerner19Clusters}. 

Let $\Lc_0$ be a loop soup in $\Hb$ wired on $[-\eps^{1/2},\eps^{1/2}]$ and free elsewhere.
We will denote by $\Lc_0 \cap \Hb$ the subset of $\Lc_0$ consisting of all the loops that are contained in $\Hb$.
By Theorems~\ref{thm:qw19} and~\ref{thm:restriction}, we can decompose $\Lc_0$  into several independent parts:
\begin{enumerate}
\item The excursions $\wh\Ec_1$ attached to $[-\eps^{1/2}, \eps^{1/2}]$.
\item The CLE coming from the outer boundaries of the outermost clusters in $\Lc_0 \cap \Hb$.
\item For each loop $\ell$ in the CLE in the previous item, let $O_\ell$ be the domain encircled by $\ell$. Let $f_\ell$ be a conformal map from $\Hb$ onto $O_\ell$. For each $\ell$, we sample an independent loop soup $\Lc_\ell$ with wired boundary conditions in $\Hb$, and put $f_\ell(\Lc_\ell)$ inside $O_\ell$.
\end{enumerate}
We will denote by $\Fc$ the $\sigma$-algebra generated by the items 1 and 2.

To obtain a loop soup $\Lc_\R$ in $\Hb$ wired on $\R$ from $\Lc_0$, one can proceed as follows.
Let $\eta$ be the outer boundary of the cluster in $\Lc_0$ which is wired on $[-\eps^{1/2},\eps^{1/2}]$. Let $O$ be the domain enclosed by $\eta$ and  $[-\eps^{1/2},\eps^{1/2}]$. Let $g$ be the conformal map from $O$ onto $\Hb$ that sends $-\eps^{2/3}, \eps^{2/3}, \eps^{1/2}$ to $-\eps^{2/3}, \eps^{2/3}, \infty$. Let $\Lc_\Rb$ be the image under $g$ of $\Lc_0$ restricted to $\overline O$. Then $\Lc_\R$ is distributed as a loop soup in $\Hb$ with wired boundary conditions. Let $\Ec$ be the boundary excursions induced by $\Lc_\Rb$ and let $\Ec_1$ (resp. $\Ec_2$) be the set of excursions in $\Ec$ with both endpoints in $[-\eps^{2/3}, \eps^{2/3}]$ (resp. $[1-\eps^{2/3}, 1+\eps^{2/3}]$).
Let $\wt z_1:=g^{-1}(z_1), \wt z_2:=g^{-1}(z_2).$
Let $\wt \Ec:=g^{-1}(\Ec)$, $\wt \Ec_1:=g^{-1}(\Ec_1)$, $\wt \Ec_2:=g^{-1}(\Ec_2)$. See Figure \ref{fig:exploration} for a schematic representation of some of these notations. Let
\begin{align*}
\wt u_j:=\lim_{\gamma\to 0} \frac{1}{Z_\gamma} \Pb(\wt \Ec_j  \overset{\Lc_O^\theta \cup \Xi_a^{\tilde z_j,O}}{\longleftrightarrow} \wt z_j \mid \wt \Ec_j, O), \quad \quad j=1,2.
\end{align*}
By conformal invariance, $(u_1, u_2)$ and $(\wt u_1, \wt u_2)$ have the same law. Therefore, to prove Lemma~\ref{lem:indep_u}, it is equivalent to prove that
$
\limsup_{\eps \to 0}
\Eb[\wt u_1 \wt u_2] \leq c(\theta)^2.
$

Let $\xi$ be the CLE loop that encircles $\wt z_2$. The wired loop soup $f_\xi(\Lc_\xi)$ (see item 3 above) inside the domain $O_\xi$ delimited by $\xi$ is composed of excursions $\wt \Ec_\xi$ attached to $\xi$ and loops $\Lc_{O_\xi}$ contained in $O_\xi$. Define
\begin{equation}\label{E:pf_v2}
v_2:=\lim_{\gamma\to 0} \frac{1}{Z_\gamma} \Pb( \wt \Ec_\xi  \overset{\Lc_{O_\xi}^\theta \cup \Xi_a^{\tilde z_2,O_\xi}}{\longleftrightarrow}  \wt z_2 \mid \wt \Ec_\xi, O_\xi).
\end{equation}
By conformal invariance and the independence of $\Lc_\xi$ and $\Fc$, $\Expect{v_2 \vert \Fc} = c(\theta)$.
We are going to show that
\begin{equation}
\label{E:pf_indep_goal}
\P(\wt u_2 \leq v_2) \to 1
\quad \text{as} \quad
\eps \to 0.
\end{equation}
Let us explain how we will conclude from this.
Notice that $\wt u_1$ is measurable w.r.t. $\Fc$. Indeed, $\wt \Ec_1$, the domain $O$ and $\wt z_1$ are all measurable w.r.t. the excursions $\wh \Ec_1$ and the CLE loops. Therefore, we have on the event $E := \{ \wt u_2 \leq v_2 \}$,
\[
\Expect{\wt u_1 \wt u_2 \mathbf{1}_E} \leq \Expect{\wt u_1 v_2 \mathbf{1}_E}
\leq \Expect{\wt u_1 v_2} = \Expect{\wt u_1 \Expect{v_2 \vert \Fc}} = c(\theta) \Expect{\wt u_1} \leq c(\theta)^2
\]
since, by adding more excursions,
\[
\Expect{\wt u_1} \leq \lim_{\gamma \to 0} \frac{1}{Z_\gamma} \E[\Pb(\wt \Ec  \overset{\Lc_O^\theta \cup \Xi_a^{\tilde z_j,O}}{\longleftrightarrow} \wt z_j \mid \wt \Ec, O)] = c(\theta).
\]
On the complementary event $E^c$ (whose probability vanishes as $\eps \to 0$ by \eqref{E:pf_indep_goal}), we can bound for all $M>0$,
\begin{align*}
    \limsup_{\eps \to 0} \Expect{\wt u_1 \wt u_2 \mathbf{1}_{E^c}} \leq \limsup_{\eps \to 0} \Expect{\wt u_1 \wt u_2 \indic{\tilde u_1 \tilde u_2 > M}} + M \limsup_{\eps \to 0} \P(E^c) = \limsup_{\eps \to 0} \Expect{\wt u_1 \wt u_2 \indic{\tilde u_1 \tilde u_2 > M}}.
\end{align*}
We will show in Lemma \ref{L:uniform_integrability_u} below that the quantity on the right hand side vanishes as $M \to \infty$. Altogether, this will prove Lemma \ref{lem:indep_u}.

It remains to prove \eqref{E:pf_indep_goal}. We proceed in several steps.

\begin{figure}\label{fig:quasi_indep}
\centering
\includegraphics[width=.9\textwidth]{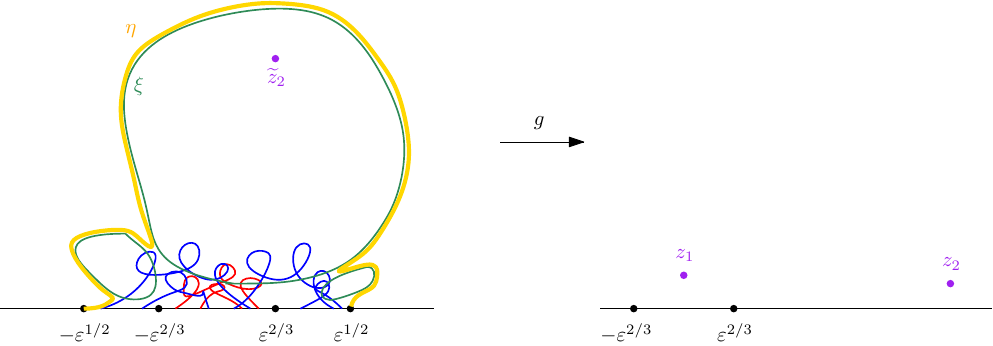}
\caption{\label{fig:exploration} Proof of Lemma~\ref{lem:indep_u}. The excursions in $\wh \Ec_0$ are in red. The set $\wh \Ec_1$ is the union of the red and blue excursions.} 
\end{figure}

$\bullet~$ \textbf{Step 1.} The goal of this step is to show that $\Pb[ g([-\eps^{1/2}, \eps^{1/2}]) \subseteq [-2\eps^{1/2}, 2 \eps^{1/2}] ] = 1 +o(1)$.
Let $\eps^{-1/2}O$ be the domain $O$ rescaled by $\eps^{-1/2}$. 
Let $f$ be the conformal map from $\eps^{-1/2}O$ onto $\Hb$ that sends $-1/2, 1/2, 1$ to $-1,1,\infty$.
Then $g(z)= h\circ f (\eps^{-1/2}z)$, where $h$ is a composition of scaling and translation that sends $f(\pm \eps^{1/6})$  to $\pm \eps^{2/3}$.
Note that $f$ is analytic, so for all $z\in D(0, \eps^{7/12}) \cap \Hb$,
\begin{align*}
g'(z) = \eps^{-1/2} f'(\eps^{-1/2}z) \frac{2\eps^{2/3}}{f(\eps^{1/6})- f(-\eps^{1/6})}
=(f'(0) + O(\eps^{1/12}) ) \frac{2\eps^{1/6}}{f'(0) 2 \eps^{1/6} + O(\eps^{1/3})}=1+O(\eps^{1/12}).
\end{align*}
Consequently, there exists $x_0 \in[\eps^{2/3}, \eps^{1/2}]$ such that
\begin{align*}
g(\eps^{1/2}) = g(\eps^{2/3}) + (\eps^{1/2} -\eps^{2/3}) g'(x_0) =\eps^{1/2} + O(1) \eps^{7/12},
\end{align*}
where $O(1)$ is bounded by some random variable $A_0$ whose law only depends on the law of $f$, which does not depend on $\eps$. 
Let $F_1$ be the event that $A_0\le \eps^{-1/24}$, so that $\Pb(F_1)=1+o(1)$. On $F_1$, we have $g(\eps^{1/2}) \le 2\eps^{1/2}$. Similarly, we can define $F_2$ with $\Pb(F_2)=1+o(1)$ so that $g(-\eps^{1/2}) \ge -2\eps^{1/2}$.
On the event $F_1 \cap F_2$, we have $g([-\eps^{1/2}, \eps^{1/2}]) \subseteq [-2\eps^{1/2}, 2 \eps^{1/2}]$. This concludes Step 1.

$\bullet~$ \textbf{Step 2.} Recall that $\wh \Ec_1$ denotes the set of excursions in $\Ec_0$ with both endpoints in $[-\eps^{1/2}, \eps^{1/2}]$.
The goal of this step is to prove that
\begin{align}\label{eq:PF}
\Pb[g(\wh \Ec_1) \cap D(1, \eps^{7/12}) =\emptyset] =1+o(1).
\end{align}
Let $\Ec_4$ be the set of excursions in $\Ec$ that have both endpoints in $[-2\eps^{1/2}, 2 \eps^{1/2}]$. Then $g(\wh \Ec_1) \subset \Ec_4$ on the event that $g([-\eps^{1/2}, \eps^{1/2}]) \subseteq [-2\eps^{1/2}, 2 \eps^{1/2}]$. Thanks to Step 1, we deduce that
\begin{align*}
\Pb[g(\wh \Ec_1) \cap D(1, \eps^{7/12}) \neq \emptyset]\leq \Pb[\Ec_4 \cap D(1, \eps^{7/12}) \neq \emptyset] + o(1)
\end{align*}
By Lemma~\ref{lem:est_be}, we can further bound
\begin{align*}
\Pb[\Ec_4 \cap D(1, \eps^{7/12}) \neq \emptyset] & \leq \Eb \left[ \sum_{e\in \Ec; a(e), b(e) \in [-2\eps^{1/2}, 2 \eps^{1/2}]} \frac{(b(e)-a(e))^2}{(1-a(e))^2 (1-b(e))^2}\right] \eps^{7/6} + o(1) \\
& \leq C \Eb \left[ \sum_{e\in \Ec; a(e), b(e) \in [-2\eps^{1/2}, 2 \eps^{1/2}]} (b(e)-a(e))^2 \right] \eps^{7/6} + o(1).
\end{align*}
With the help of Lemma~\ref{lem:estimate_interval}, we can bound the last expectation above by $C \eps^{1/2} \eps^{1/2}$ proving \eqref{eq:PF}.

$\bullet~$ \textbf{Step 3.} In this step we show that
\[
\P[\Ec_2 \subset D(1, \eps^{7/12})] = 1+o(1).
\]
We have
\begin{align*}
\Pb[\Ec_2 \not\subset D(1, \eps^{7/12})] \le C_1\Eb\left[\sum_{e\in \Ec_2} (b(e) -a(e))^2 \eps^{-7/6}\right] \le C_2 \eps^{4/3} \eps^{-7/6}= C \eps^{1/6},
\end{align*}
where the term $C_1 (b(e) -a(e))^2 \eps^{-7/6}$ represents the probability that an excursion with endpoints $a(e)$ and $b(e)$ has diameter at least $\eps^{7/12}/2$, and the last inequality follows from Lemma~\ref{lem:estimate_interval}.

$\bullet~$ \textbf{Step 4. Conclusion.}
On the event $F:=\{g(\wh \Ec_1) \cap D(1, \eps^{7/12}) =\emptyset\} \cap \{\Ec_2 \subset D(1, \eps^{7/12})\}$, $g(\wh \Ec_1) \cap \Ec_2 = \varnothing$ and the loops in $\Lc_O$ that can contribute to $\wt u_2$ can only come from $\Lc_{O_\xi}$ ($\wh \Ec_1$ cannot contribute). Hence, on $F$,
\[
\wt u_2 = \lim_{\gamma\to 0} \frac{1}{Z_\gamma} \Pb( \wt \Ec_2  \overset{\Lc_{O_\xi}^\theta \cup \Xi_a^{\tilde z_2,O_\xi}}{\longleftrightarrow}  \wt z_2 \mid \wt \Ec_2, O_\xi).
\]
Now, because on the event $F$, the excursions in $\wt\Ec_2$ are contained in $\wt \Ec_\xi$, changing $\wt \Ec_2$ into $\wt \Ec_\xi$ can only increase the above probability, i.e. $\wt u_2 \leq v_2$. Since $\P[F] = 1+o(1)$ by Steps 2 and 3, this concludes the proof of \eqref{E:pf_indep_goal} and the proof of Lemma \ref{lem:indep_u}.
\end{proof}

To finish the proof of Proposition \ref{P:independance_u}, we need to deal with the uniform integrability of $u_1 u_2$:

\begin{lemma}[Uniform integrability of $u_1 u_2$]\label{L:uniform_integrability_u}
For all $p \in (0,1/(1-\theta))$, $\limsup_{\eps \to 0} \Expect{u_1^p u_2^p} < \infty$. In particular, 
\begin{equation}
\label{E:L_uniform_integrability_u}
\lim_{M \to \infty} \limsup_{\eps \to 0} \Expect{u_1 u_2 \indic{u_1 u_2 > M}} =0.
\end{equation}
\end{lemma}

\begin{proof}[Proof of Lemma \ref{L:uniform_integrability_u}]
Let $p \in (0,1/(1-\theta))$ and let $\eta >0$ be small enough so that $p(1-\theta+\eta)<1$.
For $j=1,2$, let $f_j : \Hb \to \D$ be a conformal map that sends $z_j$ to 0 and let $R_j := \d (0, f_j(\Ec_j))$. By conformal invariance and Lemma \ref{L:Rgamma},
\[
u_j = \lim_{\gamma\to 0} \frac{1}{Z_\gamma} \Pb(f_j(\Ec_j)  \overset{\Lc_\D^\theta \cup \Xi_a^{0,\D}}{\longleftrightarrow}  0 \mid \Ec_j)
\leq \lim_{\gamma\to 0} \frac{1}{Z_\gamma} \Pb( R_j \partial \D  \overset{\Lc_\D^\theta \cup \Xi_a^{0,\D}}{\longleftrightarrow} 0 \mid R_j) \leq C (1 + |\log R_j|)^{1-\theta+\eta}.
\]
Using that for nonnegative random variables $X$ and $Y$, $\Expect{XY} = \int_{(0,\infty)^2} \Prob{X \geq x, Y \geq y} \d x \d y$, we deduce that $\Expect{u_1^p u_2^p}$ is at most
\begin{align*}
 C \E \Big[\prod_{j=1,2} (1 + |\log R_j|)^{p(1-\theta+\eta)} \Big]
\leq C \int_{(0,\infty)^2} \Prob{(1+|\log R_j|)^{p(1-\theta+\eta)} \geq t_j, j=1,2 } \d t_1 \d t_2.
\end{align*}
By \eqref{propeq:two-point} in Proposition \ref{prop:be_main}, the probability in the integral is equal to
\[
\Prob{\forall j=1,2, R_j \leq \exp \left(1-t_j^{\frac{1}{p(1-\theta+\eta)}} \right)} \leq C ((t_1+1)(t_2+1))^{-\frac{1}{p(1-\theta+\eta)}}.
\]
Since $p(1-\theta+\eta)<1$, this implies that the integral converges which concludes the proof of the bound on $\Expect{u_1^p u_2^p}$. \eqref{E:L_uniform_integrability_u} then follows by H\"{o}lder's inequality and the fact that the interval $(0,1/(1-\theta))$ contains real numbers that are strictly larger than 1.
\end{proof}
This concludes the proof of Proposition \ref{P:independance_u}.
\end{proof}

%% file: subfiles/discrete.tex
The purpose of this section is to show that the field $h_\theta$ agrees with a multiple of the GFF when $\theta = 1/2$. See Theorem \ref{T:identification_field1/2}. To do so, we will show in Theorem \ref{T:one_cluster} a stronger version of the convergence of the discrete approximation $\Mc_{a,N}$ \eqref{E:def_discrete_measure} to $\Mc_a$: we will show that this convergence even holds when one restricts the measures to individual clusters. This result of independent interest holds for any subcritical intensity $\theta \in (0,1/2]$.

\begin{notation}
For all $x \in D$ and $k \geq 1$, we will denote by $\Cc(x,k)$ the $k$-th outermost cluster surrounding $x$. ($k=1$ corresponds to the outermost cluster)
\end{notation}

Recall that we denote by $\Cf$ the collection of all clusters $\Cc$ of $\Lc_D^\theta$.
Let $\Lc_{D_N}^\theta$ be a random walk loop soup approximating $\Lc_D^\theta$ as in Section \ref{S:multiplicative_chaos}.
We will use analogous notations $\Cf_N$, $\Cc_N(x,k)$, etc. in the discrete. In this section we will view clusters (or rather their closures) as elements of $(\Kc, \d_\Kc)$ \eqref{E:Hausdorff_distance}, i.e. as random compact subsets of $\overline{D}$. We recall that discrete clusters converge to continuous clusters:

\begin{theorem}[\cite{lupu2018convergence}, Section 4.3 in \cite{ALS2}]\label{T:discrete_clusters}
Let $n \geq 1$, $x_i \in D, k_i \geq 1$, $i=1, \dots, n$. The following convergence
\[
( \{ \overline{\Cc_N(x_i,k_i)} \}_{i=1}^n, \Lc_{D_N}^\theta )
\xrightarrow[N \to \infty]{} ( \{ \overline{\Cc(x_i,k_i)} \}_{i=1}^n, \Lc_D^\theta )
\]
holds in distribution relatively to the topologies induced by $\d_\Kc$ and $\d_\mathfrak{L}$ respectively. 
\end{theorem}

Let $\Mc_{a,N}$ be the discrete approximation \eqref{E:def_discrete_measure} of $\Mc_a$. The main result of this section is:

\begin{theorem}\label{T:one_cluster}
Let $\theta \leq 1/2$ and $a \in (0,2)$.
For all $n \geq 1$, $x_i \in D, k_i \geq 1$, $i=1 \dots n$, the following convergence
\[
\left( \{ \Mc_a^N \mathbf{1}_{\Cc_N(x_i,k_i)} \}_{i=1}^n, \{ \overline{\Cc_N(x_i,k_i)} \}_{i=1}^n, \Lc_{D_N}^\theta \right) \xrightarrow[N \to \infty]{} \left( \{ \Mc_a \mathbf{1}_{\Cc(x_i,k_i)} \}_{i=1}^n, \{ \overline{\Cc(x_i,k_i)} \}_{i=1}^n, \Lc_D^\theta \right)
\]
holds in distribution relatively to the weak topologies of measures and the topologies induced by $\d_\Kc$ and $\d_\mathfrak{L}$.
\end{theorem}

As in the continuum, let $\Mc_{a,N}^+$ be the measure $\Mc_{a,N}$ restricted to positive clusters, where we assign independent spins to each cluster. As a consequence of Theorem \ref{T:one_cluster}, we will get:

\begin{corollary}\label{C:discrete_ma+}
Let $\theta \in (0,1/2]$ and $a \in (0,2)$. The following convergence
\[
(\Lc_{D_N}^\theta, \Mc_{a,N}^+) \to (\Lc_D^\theta, \Mc_a^+) \quad \quad \text{as} \quad N \to \infty,
\]
holds in distribution relatively to the topology induced by $d_{\mathfrak{L}}$ for $\Lc_{D_N}^\theta$ and the weak topology on $\C$ for $\Mc_{a,N}^+$.
\end{corollary}

The exact same result would still hold if we were considering the loop soup on the cable graph where the connection with the Gaussian free field is the strongest \cite{LupuIsomorphism}. This allows us to prove Theorem \ref{T:identification_field1/2}:

\begin{proof}[Proof of Theorem \ref{T:identification_field1/2}, assuming Corollary \ref{C:discrete_ma+}]
Once Corollary \ref{C:discrete_ma+} is established, the proof of Theorem \ref{T:identification_field1/2} follows along the same lines as the proof of \cite[Theorem 1.5]{ABJL21}. The main difference is that we need to use Lupu's version \cite{LupuIsomorphism} of Le Jan's isomorphism which recovers the signs of the discrete GFF: on the cable graph, $(\sigma_x \sqrt{2 \ell_x})_x$ has the same law as the cable graph GFF \cite{LupuIsomorphism}. In particular, this identifies $\Mc_{a,N}^+$ with the uniform measure on the set of $a$-thick points of the cable graph GFF which is known to converge to the exponential of the continous GFF \cite{BiskupLouidor}. Together with Corollary \ref{C:discrete_ma+} (or rather, the cable graph analogue), this allows us to prove Theorem \ref{T:identification_field1/2}. We refer to the proof of \cite[Theorem 1.5]{ABJL21} for more details.
\end{proof}

The rest of Section \ref{S:discrete} is dedicated to the proof of Theorem \ref{T:one_cluster} and Corollary \ref{C:discrete_ma+}. We will first state and prove some preliminary lemmas before returning to the proof of the main results of this section.

\subsection{Preliminary lemmas}

Let us denote by
\begin{equation}
\label{E:not_Qf}
\Qf_p := \{ x + [-2^{-p},2^{-p}]^2, x \in 2^{-p} \Z^2 \}, \quad \quad p \geq 1,
\end{equation}
a collection of dyadic squares of side length $2^{-p+1}$. For all $x \in D$ and $k \geq 1$, we enlarge the cluster $\Cc(x,k)$ with the help of these dyadic squares and we denote by
\begin{equation}
\label{E:not_enlargedC}
\Cc^{(p)}(x,k) := \bigcup_{Q \in \Qf_p, Q \cap \Cc(x,p) \neq \varnothing} Q.
\end{equation}

\begin{lemma}\label{L:condition_cluster}
Let $x \in D$, $k \geq 1$. Let $O_i, i \geq 0$, be the connected components of $D \setminus \Cc(x,k)$ with $O_0$ being the unique component whose boundary include $\partial D$. Conditionally on $\Cc(x,k)$, we have:
\begin{itemize}
\item
The collections of loops $\{ \wp \in \Lc_D^\theta, \wp \subset O_i \}, i \geq 0$, are independent;
\item
For all $i \geq 1$, $\{ \wp \in \Lc_D^\theta, \wp \subset O_i \}$ has the law of a Brownian loop soup in $O_i$ with intensity $\theta$;
\item
$\{ \wp \in \Lc_D^\theta, \wp \subset O_0 \}$ has the law of a Brownian loop soup in $O_0$ with intensity $\theta$ conditionally to have exactly $k-1$ clusters surrounding the inner boundary of $O_0$.
\end{itemize}
\end{lemma}

\begin{proof}
This lemma can be proven by working on the event that the enlargement $\Cc^{(p)}(x,k)$ of $\Cc(x,k)$ is equal to a specific union $C_p$ of dyadic squares. One then uses that the loops that do not intersect $C_p$ are independent from $\Cc(x,k)$ (except that $\Cc(x,k)$ is exactly the $k$-th outermost cluster surrounding $x$) and then let $p \to \infty$. We omit the details.
\end{proof}

Next, we show that the $\Mc_a$-measure of one cluster is well approximated by the $\Mc_a$-measure of its $\eps$-neighbourhood.

\begin{lemma}\label{L:neighbourhood}
Let $n, p \geq 1$, and for $i=1 \dots n$, let $x_i \in D$, $k_i \geq 1$ and let $C_{i,p}$ be a subset of $D$ which is a union of dyadic squares in $\Qf_p$. Assume that the sets $C_{i,p}, i =1 \dots n$, are pairwise disjoint and let $E_p$ be the event that for all $i=1 \dots n$, $\Cc^{(p)}(x_i,k_i) = C_{i,p}$. Then 
\begin{equation}
\label{E:L_neighbourhood}
\E \Big[ \sum_{i=1}^n \Mc_a(C_{i,p} \setminus \Cc(x_i,k_i)) \Big\vert \Cc(x_i,k_i), i=1 \dots n \Big] \mathbf{1}_{E_p} \leq C 2^{-p a}  \mathbf{1}_{E_p}
\end{equation}
for some deterministic constant $C>0$ which only depends on the domain $D$.
\end{lemma}

\begin{proof}[Proof of Lemma \ref{L:neighbourhood}]
To ease notations, we will prove this lemma for the $n$ outermost clusters surrounding a given point $x \in D$. With the notations of the lemma, this corresponds to the particular case $x_i = x$ and $k_i = i$ for $i=1 \dots n$. The general case follows from the same proof.

For $j=1 \dots n$, let $O_j$ be the connected component of $D \setminus \bigcup_{i=1}^n \Cc(x,i)$ whose boundary intersects both $\Cc(x,j-1)$ and $\Cc(x,j)$, where by convention we write $\Cc(x,0) = \partial D$. Let $O_j, j \geq n+1$, denote the other connected components of $D \setminus \bigcup_{i=1}^n \Cc(x,i)$. $O_1, \dots, O_n$ are annular-like, whereas $O_{n+1}, O_{n+2}, \dots$ are disc-like. Conditionally on $\Cc(x,i), i =1 \dots n$, the law of the loops in each of these connected components is described in Lemma \ref{L:condition_cluster}.
For all $j \geq 1$, let $O_j^p$ denote the set of points of $O_j$ which are at distance at most $\sqrt{2} \cdot 2^{-p+1}$ to the boundary of $O_j$. On the event $E_p$,
\[
\bigcup_{i=1}^n C_{i,p} \setminus \Cc(x,i) \subset \bigcup_{j \geq 1} O_j^p,
\]
which implies that
\begin{align*}
& \E \Big[ \sum_{i=1}^n \Mc_a(C_{i,p} \setminus \Cc(x,i)) \Big\vert \Cc(x,i), i=1 \dots n \Big] \mathbf{1}_{E_p}  \leq \sum_{j \geq 1} \Expect{ \Mc_a( O_j^p ) \Big\vert \Cc(x,i), i=1 \dots n } \mathbf{1}_{E_p}.
\end{align*}
For all $j \geq n+1$, conditionally on $\Cc(x,i), i=1 \dots n$, the law of the loops inside $O_j$ is that of a Brownian loop soup in $O_j$. Therefore, by \cite[Theorem 1.1]{ABJL21}, we have for all $j \geq n+1$,
\begin{align*}
& \Expect{ \Mc_a(O_j^p) \Big\vert \Cc(x,i), i =1 \dots n } \leq C \int_{O_j^p} \CR(z, O_j)^{a} \d z
\leq C 2^{-pa} |O_j|
\end{align*}
because if a point $z$ is a distant at most $\sqrt{2} \cdot 2^{-p+1}$ to the boundary of $O_j$, then $\CR(z,O_j)$ is at most $4 \sqrt{2} \cdot 2^{-p+1}$ (Koebe quarter theorem).
For $j =1 \dots n$, the conditional law of the loops in $O_j$ is that of a loop soup in $O_j$ conditioned to not have any cluster surrounding the inner boundary of $O_j$. By FKG inequality, such a loop soup is stochastically dominated by an unconditioned loop soup in the same domain. So the same reasoning applies, except that the domain $O_j$ is not simply connected so one has to replace $\log \CR(z,O_j)$ by the harmonic extension of $\log |z - \cdot|$ from $\partial O_j$ to $O_j$ (see Remark \ref{rmk:multiplicative_chaos}). Nevertheless, the same conclusion holds: for all $j=1 \dots n$,
\[
\Expect{ \Mc_a(O_j^p) \Big\vert \Cc(x,i), i =1 \dots n }
\leq C 2^{-pa} |O_j|.
\]
Summing over $j$ leads to
\begin{align*}
& \E \Big[ \sum_{i=1}^n \Mc_a(C_{i,p} \setminus \Cc(x,i)) \Big\vert \Cc(x,i), i=1 \dots n \Big] \mathbf{1}_{E_p}
\leq C 2^{-p\gamma^2/2} \sum_{j \geq 1} |O_j| \mathbf{1}_{E_p} \leq C |D| 2^{-pa} \mathbf{1}_{E_p}.
\end{align*}
This concludes the proof of Lemma \ref{L:neighbourhood}.
\end{proof}

We will also need to ensure that most of the $\Mc_a$-mass comes from the ``large'' clusters:
\[
\Cf_{q} := \{ \Cc(x,k), x \in 2^{-q} \Z^2 \cap D, k \leq 2^q \},
\quad
\Cf_{N,q} := \{ \Cc_N(x,k), x \in 2^{-q} \Z^2 \cap D, k \leq 2^q \},
\quad q \geq 1.
\]

\begin{lemma}\label{L:large_clusters}
As $q \to \infty$,
\[
\E \Big[ \Mc_a ( D \setminus \bigcup_{\Cc \in \Cf_{q}} \Cc ) \Big] \to 0
\quad \text{and} \quad
\limsup_{N \to \infty} \E \Big[ \Mc_a^N ( D \setminus \bigcup_{\Cc \in \Cf_{N,q}} \Cc ) \Big] \to 0.
\]
\end{lemma}

\begin{proof}[Proof of Lemma \ref{L:large_clusters}]
We prove the claim in the discrete setting. The continuous case is similar.
The proof of this lemma is very close in spirit to the the proof of Lemma \ref{L:neighbourhood}. Let $q \geq 1$ and let $p \leq q$ be much small than $q$. For a cluster $\Cc \in \Cf_N$, we will denote by $\Cc^{(p)}$ its $2^{-p}$-enlargement with dyadic squares as in \eqref{E:not_enlargedC}.
Let $E_{N,q,p}$ be the event that $\bigcup_{\Cc \in \Cf_{N,q}} \Cc^{(p)}$ covers the whole domain $D$. Since the clusters $\Cc \in \Cf_N$ are dense in $D$, for any fixed $p$,
\begin{equation}
\label{E:pf_L_large1}
\lim_{q \to \infty} \liminf_{N \to \infty} \Prob{E_{N,q,p}} \to 1.
\end{equation}
On this event, we can write
\[
\limsup_{N \to \infty}
\E \Big[ \Mc_a^N \Big( D \setminus \bigcup_{\Cc \in \Cf_{N,q}} \Cc ) \mathbf{1}_{E_{N,q,p}} \Big]
\leq
\limsup_{N \to \infty}
\E \Big[ \Mc_a^N \Big( \bigcup_{\Cc \in \Cf_{N,q}} \Big( \Cc^{(p)} \setminus \Cc \Big) \Big) \Big]
\]
and we can use the exact same strategy as in the proof of Lemma \ref{L:neighbourhood} to show that the right hand side term is at most $C 2^{-p a}$ for some constant $C>0$ depending only on the domain $D$. More precisely, we use a discrete version of Lemma \ref{L:condition_cluster} together with the expression given in \cite[Propostion 10.1]{ABJL21} for the first moment of $\Mc_a^N$ that converges to the continuum expression.

On the complementary event, we simply bound
\[
\E \Big[ \Mc_a^N ( D \setminus \bigcup_{\Cc \in \Cf_{N,q}} \Cc ) \mathbf{1}_{E_{N,q,p}^c} \Big]
\leq \Expect{ \Mc_a^N(D) \mathbf{1}_{E_{N,q,p}^c} }.
\]
The uniform integrability of $(\Mc_a^N(D))_{N \geq 1}$ (see \cite{ABJL21}), together with \eqref{E:pf_L_large1}, ensures that the right hand side term goes to zero as $N \to \infty$ and then $q \to \infty$. Let us give some details. Let $M >0$ be large. We can bound
\begin{align*}
\Expect{ \Mc_a^N(D) \mathbf{1}_{E_{N,q,p}^c} }
\leq \Expect{ \Mc_a^N(D) \mathbf{1}_{\Mc_a^N(D) > M} } + M \Prob{E_{N,q,p}^c}.
\end{align*}
By \eqref{E:pf_L_large1}, we deduce that
\[
\limsup_{q \to \infty} \limsup_{N \to \infty} \Expect{ \Mc_a^N(D) \mathbf{1}_{E_{N,q,p}^c} }
\leq \limsup_{N \to \infty} \Expect{ \Mc_a^N(D) \mathbf{1}_{\Mc_a^N(D) > M} }.
\]
By the uniform integrability of $(\Mc_a^N(D))_{N \geq 1}$, the right hand side term goes to zero as $M \to \infty$. The left hand side term being independent of $M$, it has to vanish as well.
Wrapping things up, we have
\[
\limsup_{q \to \infty} \limsup_{N \to \infty} \E \Big[ \Mc_a^N ( D \setminus \bigcup_{\Cc \in \Cf_{N,q}} \Cc ) \Big] \leq C 2^{-p a}.
\]
We obtain the desired result by letting $p \to \infty$.
\end{proof}

\subsection{Proof of Theorem \ref{T:one_cluster}}

\begin{proof}[Proof of Theorem \ref{T:one_cluster}]
Let $n \geq 1$, $x_1, \dots, x_n \in D$ and $k_1, \dots, k_n \geq 1$.
First of all, we claim that the convergence in distribution
\begin{equation}
\label{E:pf_thm_cluster1}
\left( \Mc_a^N, \{ \overline{\Cc_N(x_i,k_i)} \}_{i=1}^n, \Lc_{D_N}^\theta \right) \xrightarrow[N \to \infty]{} \left( \Mc_a, \{ \overline{\Cc(x_i,k_i)} \}_{i=1}^n, \Lc_D^\theta \right)
\end{equation}
holds. Indeed it was proven in \cite{ABJL21} that 
\[
\left( \Mc_a^N, \Lc_{D_N}^\theta \right) \xrightarrow[N \to \infty]{} \left( \Mc_a, \Lc_D^\theta \right)
\]
and in \cite{Lupu18} (see Theorem \ref{T:discrete_clusters}) that
\[
\left( \{ \overline{\Cc_N(x_i,k_i)} \}_{i=1}^n, \Lc_{D_N}^\theta \right) \xrightarrow[N \to \infty]{} \left( \{ \overline{\Cc(x_i,k_i)} \}_{i=1}^n, \Lc_D^\theta \right).
\]
The convergence \eqref{E:pf_thm_cluster1} then follows from the fact that $\Mc_a$ and $\{ \Cc(x_i,k_i) \}_{i=1}^n$ are both measurable with respect to $\Lc_D^\theta$.
Now, we want to establish the joint convergence of
\begin{equation}
\label{E:pf_thm_cluster4}
\left( \{ \Mc_a^N \mathbf{1}_{\Cc_N(x_i,k_i)} \}_{i=1}^n, \{ \Cc_N(x_i,k_i) \}_{i=1}^n, \Lc_{D_N}^\theta \right), \quad N \geq 1.
\end{equation}
The convergence \eqref{E:pf_thm_cluster1} together with the trivial bound $\mathbf{1}_{\Cc_N(x_i,k_i)} \leq 1$ shows that this sequence is tight. To conclude the proof, it is enough to identify the law of the subsequential limits. We will do this in three steps.

\textbf{Step 1: stochastic domination of the subsequential limits.}
Let $\left( \{ \Mc_i \}_{i=1}^n, \{ \Cc_i \}_{i=1}^n, \Lc \right)$ be any subsequential limit of \eqref{E:pf_thm_cluster4}. In the following and to ease the notations, we will omit the subsequence and we will assume that the convergence holds as $N \to \infty$.
For $i=1 \dots n$, let $f_i : D \to [0,\infty)$ be bounded nonnegative continuous functions, let $\varphi : \R^n \to [0,\infty)$ be a nondecreasing, bounded, 1-Lipschitz function and let $\psi_1, \psi_2$ be two nonnegative continuous bounded functions. The goal of this first step is to establish the following stochastic domination:
\begin{equation}
\label{E:proof_discrete_stoch_domi}
\varphi \left( \left( \Mc_i, f_i \right)_{i=1}^n \right) \psi_1 \left( \left( \Cc_i \right)_{i=1}^n \right) \psi_2 \left( \Lc \right)
\preceq
\varphi \left( \left( \Mc_a \mathbf{1}_{\Cc(x_i,k_i)}, f_i \right)_{i=1}^n \right) \psi_1 \left( \left( \Cc(x_i,k_i) \right)_{i=1}^n \right) \psi_2 \left( \Lc_D^\theta \right).
\end{equation}

Let $p \geq 1$ be large and recall the definition \eqref{E:not_Qf} of $\Qf_p$.
Let $\Cc_N^{(p)}(x_i,k_i)$ be the discrete analogue of the enlarged version $\Cc^{(p)}(x_i,k_i)$ of $\Cc(x_i,k_i)$ \eqref{E:not_enlargedC}.
For $i=1, \dots, n$, let $C_{i,p}$ be any subset of $D$ which can be written as a union of dyadic squares $Q \in \Qf_p$ and let $E_{N,p}$ (resp. $E_p$) be the event that for all $i =1 \dots n$, $\Cc_N^{(p)}(x_i,k_i) = C_{i,p}$ (resp. $\Cc^{(p)}(x_i,k_i) = C_{i,p}$). Because we are looking at ``large'' clusters, we will be able to restrict ourselves to the event that these clusters are well separated. For now, this means that we can assume that the sets $C_{i,p}, i =1 \dots n$ are pairwise disjoint. As a consequence of this assumption, we will later have to deal with the probability that the clusters $\Cc(x_i,k_i)$, $i=1 \dots n$, are at distance at least $10.2^{-p}$ to each other, but this probability tends to 1 as $p \to \infty$.
Because the $f_i$'s are nonnegative, on the event $E_{N,p}$, we can bound
\[
\left( \Mc_a^N \mathbf{1}_{\Cc_N(x_i,k_i)}, f_i \right) \leq \left( \Mc_a^N \mathbf{1}_{C_{i,p}}, f_i \right), \quad i=1 \dots n.
\]
By the joint convergence \eqref{E:pf_thm_cluster1}, we deduce that
\begin{align}
\label{E:proof_discrete_stoch_domi2}
&
\limsup_{N \to \infty} \Expect{ \mathbf{1}_{E_{N,p}} \varphi \left( (\Mc_a^N \mathbf{1}_{\Cc_N(x_i,k_i)},f_i)_{i=1}^n \right) \psi_1 \left( \left( \Cc_N(x_i,k_i) \right)_{i=1}^n \right) \psi_2 \left( \Lc_{D_N}^\theta \right) } \\
& \leq
\Expect{ \mathbf{1}_{E_p} \varphi \left( (\Mc_a \mathbf{1}_{C_{i,p}},f_i)_{i=1}^n \right) \psi_1 \left( \left( \Cc(x_i,k_i) \right)_{i=1}^n \right) \psi_2 \left( \Lc_D^\theta \right) }. \nonumber
\end{align}
On the event $E_p$, we can bound
\[
(\Mc_a \mathbf{1}_{C_{i,p}},f_i) \leq (\Mc_a \mathbf{1}_{\Cc(x_i,k_i)},f_i) + \norme{f_i}_\infty \Mc_a(C_{i,p} \setminus \Cc(x_i,k_i)).
\]
Recalling that $\varphi$ is 1-Lipschitz, this leads to the following upper bound for the right hand side of \eqref{E:proof_discrete_stoch_domi2}:
\begin{align*}
& \Expect{ \mathbf{1}_{E_p} \varphi \left( (\Mc_a \mathbf{1}_{\Cc(x_i,k_i)},f_i)_{i=1}^n \right) \psi_1 \left( \left( \Cc(x_i,k_i) \right)_{i=1}^n \right) \psi_2 \left( \Lc_D^\theta \right) } \\
& + \norme{\psi_1}_\infty \norme{\psi_2}_\infty \sum_{i=1}^n \norme{f_i}_\infty \Expect{ \mathbf{1}_{E_p} \Mc_a(C_{i,p} \setminus \Cc(x_i,k_i)) }.
\end{align*}
By Lemma \ref{L:neighbourhood}, the second term is at most $C 2^{-pa} \Prob{E_p}$.
Summing over all pairwise disjoint $C_{i,p}, i = 1\dots n$, we obtain that
\begin{align*}
& \Expect{ \varphi \left( (\Mc_i,f_i)_{i=1}^n \right) \psi_1 \left( \left( \Cc_i \right)_{i=1}^n \right) \psi_2 \left( \Lc \right) }
\leq \Expect{ \varphi \left( (\Mc_a \mathbf{1}_{\Cc(x_i,k_i)},f_i)_{i=1}^n \right) \psi_1 \left( \left( \Cc(x_i,k_i) \right)_{i=1}^n \right) \psi_2 \left( \Lc_D^\theta \right) } \\
& ~~~~~~~~ + C 2^{-p a} + \norme{\varphi}_\infty \norme{\psi_1}_\infty \norme{\psi_2}_\infty \Prob{ \min_{i \neq j} \d(\Cc(x_i,k_i), \Cc(x_j,k_j)) \leq 10.2^{-p} }.
\end{align*}
After letting $p \to \infty$, this gives
\[
\Expect{ \varphi \left( (\Mc_i,f_i)_{i=1}^n \right) \psi_1 \left( \left( \Cc_i \right)_{i=1}^n \right) \psi_2 \left( \Lc \right) }
\leq \Expect{ \varphi \left( (\Mc_a \mathbf{1}_{\Cc(x_i,k_i)},f_i)_{i=1}^n \right) \psi_1 \left( \left( \Cc(x_i,k_i) \right)_{i=1}^n \right) \psi_2 \left( \Lc_D^\theta \right) }.
\]
Since this is true for any suitable $\varphi, \psi_1, \psi_2$, this shows the desired stochastic domination \eqref{E:proof_discrete_stoch_domi}.

\medskip

\textbf{Step 2: the expectations agree.}
We specify the result of the first step to a single test function $f: D \to [0,\infty)$.
Recall the notation of Lemma \ref{L:large_clusters}. As a consequence of Step 1, any subsequential limit
\[
\left( (\tilde{\Mc}_a \mathbf{1}_{\Cc}, f)_{\Cc \in \tilde{\Cf}_{q_2} \setminus \tilde{\Cf}_{q_1}}, 0 \leq q_1 < q_2 \right)
\quad \text{of} \quad
\left( (\Mc_a^N \mathbf{1}_{\Cc_N}, f)_{\Cc_N \in \Cf_{N,q_2} \setminus \Cf_{N,q_1}}, 0 \leq q_1 < q_2 \right)_{N \geq 1}
\]
is stochastically dominated by its continuum counterpart (in the sense that each finite dimensional marginals are stochastically dominated). In particular, for all $q_2 > q_1 \geq 0$,
\begin{equation}
\label{E:pf_thm_cluster5}
\E \Big[ \sum_{\Cc \in \tilde{\Cf}_{q_2} \setminus \tilde{\Cf}_{q_1}} \left( \tilde{\Mc}_a \mathbf{1}_{\Cc}, f \right) \Big]
\leq \E \Big[ \sum_{\Cc \in \Cf_{q_2} \setminus \Cf_{q_1}} \left( \Mc_a \mathbf{1}_{\Cc}, f \right) \Big].
\end{equation}
Assume that this inequality is strict for some $q_2 > q_1 \geq 0$ and let $\eps >0$ be the difference of the right hand side and the left hand side. Applying \eqref{E:pf_thm_cluster5} between 0 and $q_1$ and between $q_2$ and some large $q \geq q_2$, we obtain that
\[
\E \Big[ \sum_{\Cc \in \tilde{\Cf}_{q}} \left( \tilde{\Mc}_a \mathbf{1}_{\Cc}, f \right) \Big]
\leq \E \Big[ \sum_{\Cc \in \Cf_{q}} \left( \Mc_a \mathbf{1}_{\Cc}, f \right) \Big] - \eps.
\]
By Lemma \ref{L:large_clusters}, letting $q \to \infty$ in the previous inequality yields
\[
\lim_{N \to \infty} \Expect{ \left( \Mc_a^N, f \right) } \leq \Expect{ \left( \Mc_a, f \right) } - \eps.
\]
But we know that $\Expect{ \left( \Mc_a^N, f \right) }$ converges to $\Expect{ \left( \Mc_a, f \right) }$. Therefore, it implies that for all $q_2 > q_1 \geq 1$, the inequality \eqref{E:pf_thm_cluster5} is an equality. Together with the stochastic domination, it proves that
\[
\left( (\tilde{\Mc}_a \mathbf{1}_{\Cc}, f)_{\Cc \in \tilde{\Cf}_{q_2} \setminus \tilde{\Cf}_{q_1}}, 0 \leq q_1 < q_2 \right)
\overset{\mathrm{(d)}}{=} \left( (\Mc_a \mathbf{1}_{\Cc}, f)_{\Cc \in \Cf_{q_2} \setminus \Cf_{q_1}}, 0 \leq q_1 < q_2 \right).
\]

\medskip

\textbf{Step 3: conclusion.}
Putting Steps 1 and 2 together, we see that for any $x \in D$, $k \geq 1$ and $f:D \to [0,\infty)$,
\[
\lim_{N \to \infty} \Expect{ \left( \Mc_a^N \mathbf{1}_{\Cc_N(x,k)}, f \right) } = \Expect{ \left( \Mc_a \mathbf{1}_{\Cc(x,k)}, f \right) }.
\]
Together with the stochastic domination \eqref{E:proof_discrete_stoch_domi} proven in Step 1, this proves Theorem \ref{T:one_cluster}.
\end{proof}

We conclude this section with:

\begin{proof}[Proof of Corollary \ref{C:discrete_ma+}]
This follows directly from Theorem \ref{T:one_cluster} and from Lemma \ref{L:large_clusters} that shows that most of the mass is carried by large clusters.
\end{proof}

%% file: subfiles/Wick.tex
\subsection{Main results}

The Wick powers of the GFF $h$, denoted by $:\!h^n\!:, n \geq 2$, give a precise meaning to $h^n$, $n \geq 2$. Because the GFF is only a generalised function, taking such powers is \textit{a priori} not a well defined operation. A non trivial fact is that it is actually possible to make sense of these powers via a renormalisation procedure. Contrary to the renormalisation of the exponential of the field, the renormalisation is not multiplicative but additive: one subtracts to $h(x)^n$ a polynomial in $h(x)$ with diverging coefficients and degree strictly smaller than $n$. See Section \ref{S:Wick_recall} for more details.

The Wick renormalisation of the powers of the GFF $h$ provides another approach to the multiplicative chaos of the field. Indeed, by naively expanding the exponential around 0, one could hope to define the exponential of $\gamma$ times the GFF $h$ by
\begin{equation}
\label{E:expansion_Liouville}
:\!e^{\gamma h(x)} \d x\!:
\quad = \quad
\sum_{n \geq 0} \frac{\gamma^n}{n!} :\!h^n\!:.
\end{equation}
It is well known that this procedure works in the $L^2$-phase (see e.g. \cite[Section 2.3]{lacoin2017semiclassical}) in the sense that the above sum converges to a random measure that agrees with the exponential of the field. See Remark \ref{R:complex_L2} for why \eqref{E:expansion_Liouville} does not hold outside of the $L^2$-phase. Note that it is \textit{a priori} not obvious that the resulting sum is nonnegative; eventually this boils down to \eqref{E:Hermite_exp}.

In the Brownian loop soup setting, Le Jan \cite{LeJan2011Loops} gave a construction of the renormalised powers $:\!L_x^n\!:$ of the local time $:\!L_x\!:$ for any intensity $\theta$. This is closely related to the (self-)intersection local time of Brownian motion (see the lecture notes \cite{LeGall92} and the references therein). 
In Section \ref{S:expansion_unsigned}, we will give the precise relation between the multiplicative chaos $\Mc_\gamma$ of the loop soup and these intersection local times. When $\theta=1/2$, this question reduces to the aforementioned GFF setting and we directly obtain that
\begin{equation}
\label{E:expansion_unsigned_critical}
\Mc_{\gamma,\theta=1/2}( \d x) = 2\sum_{n \geq 0} \frac{2^n \gamma^{2n}}{(2n)!} (2\pi)^n :\!L_x^n\!: \d x \quad \quad \text{a.s.}
\end{equation}
In Theorem \ref{T:expansion_unsigned}, we show that such a relation also holds for any intensity $\theta>0$. However, the coefficients in the expansion are different and depend on $\theta$:
\begin{equation}
\label{E:expansion_unsigned_general}
\Mc_\gamma(\d x) = 2 \sum_{n \geq 0} \frac{\gamma^{2n}}{2^n} \frac{\Gamma(\theta)}{n! \Gamma(n + \theta)} (2\pi)^n :L_x^n: \d x \quad \quad \text{a.s.}
\end{equation}
See Theorem \ref{T:expansion_unsigned} for a precise statement.

\medskip

As already alluded to, when $\theta =1/2$, the Wick powers of the local time reduce to the even powers of the GFF. To define the odd powers, one has to take into account the sign of the field.
In Sections \ref{S:1D} and \ref{S:expansion_signed} we address the question of defining \emph{all} the powers of the field $h_\theta$ for any intensity $\theta \in (0,1/2)$.
We showed in \cite{JLQ23a} that the 1D and 2D loop soups are intimately related.
We will therefore start in Section \ref{S:1D} by building our intuition with the one-dimensional setting. This case is particularly appealing since for any $\theta$, there is an isomorphism relating the local time of the loop soup to a squared Bessel process of dimension $d=2\theta$, see \cite{Lupu18} and Proposition \ref{P:1Disomorphism} below.
In dimension 1, the analogue of the field $h_\theta$ and the local time are both well defined pointwise, so there is no need of normalising the powers to define them. However, we will show in Lemmas \ref{L:1D_martingale} and \ref{L:1D_martingale2} that there is a unique way of normalising the powers so that the resulting process is a martingale. When these powers are integer powers of the local time, this normalisation agrees with Le Jan's two-dimensional normalisation procedure (generalised Laguerre polynomials). 

In Section \ref{S:expansion_signed}, we will come back to the 2D setting and state a precise conjecture about a renormalisation procedure that would define all the Wick powers of $h_\theta$; see Conjecture \ref{Conjecture1}. This procedure in particular reveals a surprising notion of \emph{duality} between the intensities $\theta$ and $\theta^* = 2-\theta$.
We then elaborate on this conjecture and show that it would imply the following expansion for the measure $\Mc_\gamma^+$: almost surely,
\[
\Mc_\gamma^+(\d x) = \Big( \sum_{k \geq 0} \frac{\gamma^{2k}}{2^k} \frac{\Gamma(\theta)}{k! \Gamma(k+\theta)} (2\pi)^k :L_x^k:
+ c_\mathrm{conj} \gamma^{2(1-\theta)} \frac{\gamma^{2k}}{2^k} \frac{\Gamma(\theta^*)}{k! \Gamma(k+\theta^*)} (2\pi)^k :h_\theta(x) L_x^k: \Big) \d x.
\]
See Section \ref{S:expansion_signed} for the definition of the different terms appearing in the above display.

\subsection{Renormalised powers of the GFF and of the occupation field of the loop soup}
\label{S:Wick_recall}

The purpose of this section is to recall the definition of the Wick powers of the GFF and of the local time of the loop soup. We start by working in the discrete setting where everything is well defined pointwise. Let $D_N \subset \frac1N \Z^2$ be a discrete approximation of $D$ as in \eqref{eq:DN}. Let $\Lc_{D_N}^\theta$ be a random walk loop soup in $D_N$ and denote by  $\ell_x$ the local time at $x$ (see \eqref{E:discrete_local_time}). Let $\varphi_N$ be a discrete GFF in $D_N$ whose covariance is given by the discrete Green function $G_N$.

To renormalise the powers of the GFF, the relevant polynomials are the Hermite polynomials $H_n, n \geq 0,$ which are the only monic\footnote{Monic means that the leading coefficient equals 1.} polynomials that are orthogonal w.r.t. the Gaussian measure $e^{-x^2/2} \d x$ and such that the degree of $H_n$ is $n$ for all $n \geq 0$. They can be explicitly written as
\begin{equation}
H_n (X) = \sum_{i=0}^{\floor{n/2}} (-1)^i \frac{n!}{i! (n-2i)!} \frac{X^{n-2i}}{2^i}.
\end{equation}
The $n$-th Wick power of the discrete GFF $\varphi_N$ is then defined by
\begin{equation}
\label{E:def_Wick_GFF}
:\varphi_N(x)^n:
\quad = \quad
G_N(x,x)^{n/2} H_n \left( \varphi_N(x)/ \sqrt{G_N(x,x)} \right).
\end{equation}

Regarding the local time of the loop soup with intensity $\theta$, the relevant polynomials are the generalised Laguerre polynomials $\mathbf{L}_n^{(\theta-1)}$, $n \geq 0,$ which are the only monic polynomials that are orthogonal w.r.t. the Gamma$(\theta)$-measure $x^{\theta-1} e^{-x} \indic{x >0} \d x$ (which is natural since for any given point $x$, $\ell_x$ is Gamma$(\theta)$-distributed) and such that the degree of $\mathbf{L}_n^{(\theta-1)}$ is $n$ for all $n$. Their explicit expression is given by
\begin{equation}
\label{E:Laguerre}
\mathbf{L}_n^{(\theta-1)}(X) = \sum_{i=0}^n (-1)^{n-i} \frac{\Gamma(n+\theta)}{(n-i)! \Gamma(\theta+i)} \frac{n!}{i!} X^i.
\end{equation}
The $n$-th Wick power of the local time $\ell_x$ is defined by
\[
:\ell_x^n: \quad = \quad G_N(x,x)^n \mathbf{L}_n^{(\theta-1)}(\ell_x/G_N(x,x)).
\]

The following result states that these renormalisations in the discrete possess nondegenerate scaling limits. We will view $\varphi_N, :\varphi_N^n:, \ell$ and $:\ell^n:$, $n \geq 1$, as random elements of $\R^{C_0(D)}$ by setting for all $f \in C_0(D)$,
\[
\scalar{ \varphi_N, f } = \frac{1}{N^2} \sum_{x \in D_N} \varphi_N(x) f(x)
\]
and similarly for the other objects. We will endow $\R^{C_0(D)}$ with the product topology.

\begin{theorem}\label{T:Wick_convergence}
For all $n \geq 1$, $(\varphi_N, :\varphi_N^2:, \dots, :\varphi_N^n:)$ converges in distribution as $N \to \infty$ to some nondegenerate $(\varphi,:\varphi^2:, \dots, :\varphi^n:)$ where $:\varphi^2:, \dots, :\varphi^n:$ are measurable w.r.t. the Gaussian free field $\varphi$.

For all $n \geq 1$, $(:\ell:, :\ell^2:, \dots, :\ell^n:)$ converges in distribution as $N \to \infty$ to some $(:L:,:L^2:, \dots, :L^n:)$ where $:L^2:, \dots, :L^n:$ are measurable w.r.t. the renormalised occupation field $:L:$. Moreover the convergence is joint with the loop soup.
\end{theorem}

\begin{proof}
This is a folklore result. A proof of the joint convergence of $( \varphi_N, :\varphi_N^2: )$ can be found in \cite[Appendix A]{ABJL21}. This proof generalises to any power. It also applies to the local time case since it relies on the computation of moments that are well understood in the loop soup setting \cite{LeJan2011Loops}.
\end{proof}

\subsection{Expansion of \texorpdfstring{$\Mc_\gamma$}{M gamma}}
\label{S:expansion_unsigned}

The main result of this section is the expansion \eqref{E:expansion_unsigned_general} of the measure $\Mc_\gamma$ that we now state precisely:

\begin{theorem}[Expansion of $\Mc_\gamma$]
\label{T:expansion_unsigned}
Let $\theta >0$ and $\gamma \in (0,\sqrt{2})$. For any bounded test function $f: D \to \R$,
\begin{equation}
\label{E:expansion_function}
2 \sum_{n =0}^{n_0} \frac{\gamma^{2n}}{2^n} \frac{\Gamma(\theta)}{n! \Gamma(n + \theta)} \int (2\pi)^n :L_x^n: f(x) \d x
\xrightarrow[n_0 \to \infty]{L^2} \scalar{ \Mc_\gamma, f }.
\end{equation}
\end{theorem}

Let us emphasise that, contrary to most of the current article, the above result is not restricted to the subcritical and critical regimes, but holds for any intensity $\theta >0$.

\begin{remark}\label{R:complex_L2}
By (the proof of) Theorem \ref{T:expansion_unsigned}, for any bounded test function $f$ and complex values of $\gamma$ with $|\gamma| < \sqrt{2}$, the left hand side of \eqref{E:expansion_function} converges in $L^2$. This gives a direct way of generalising the definition of $\Mc_\gamma$ to complex values of $\gamma$ by defining $\scalar{\Mc_\gamma,f}$ as the limit of the left hand side of \eqref{E:expansion_function}. As in the Gaussian case, when $\gamma \notin \R$, the resulting object should not be a (complex) measure but only a generalised function.
In the Gaussian setting, it is actually possible to define a complex multiplicative chaos for complex values of $\gamma$ in the eye-shaped domain 
\[
\gamma \in \mathrm{convex~hull} \left( \pm 2 \cup D(0,\sqrt{2}) \right).
\]
We refer to the introduction of \cite{lacoin2022universality} for more background on complex Gaussian multiplicative chaos.
In our non-Gaussian setting, making sense of $\Mc_\gamma$ for non real values of $\gamma$ outside of the $L^2$-regime $\{ |\gamma| < \sqrt{2} \}$ is a much more challenging question that we do not address in this article.

Note that since it should not be possible to define $\Mc_\gamma$ outside of the (closure of the -- to avoid talking about boundary cases that are not fully understood in the complex case) above eye-shaped domain, the radius of convergence of the left hand side of \eqref{E:expansion_unsigned_general} should be $\sqrt{2}$. In other words, the fact that Theorem \ref{T:expansion_unsigned} does not consider the regime $\gamma \in [\sqrt{2}, 2)$ is not a limitation of the proof.
\end{remark}

\begin{remark}
Extrapolating the above result to the case of one Brownian trajectory ($\theta \to 0$), it is likely that the Brownian multiplicative chaos measure associated to one trajectory (considered in \cite{bass1994, AidekonHuShi2018, jegoBMC, jegoRW, jegoCritical}) can be written as the sum of the $n$-th self intersection local time of the path (see the lecture notes \cite{LeGall92} and the references therein).
\end{remark}

\medskip

We now move on to the proof of Theorem \ref{T:expansion_unsigned}.
In the GFF case, the relation \eqref{E:expansion_Liouville} eventually boils down to the identity:
\begin{equation}
\label{E:Hermite_exp}
\sum_{n \geq 0} \frac{\gamma^n t^{n/2}}{n!} H_n(u/\sqrt{t}) = e^{\gamma u - \gamma^2 t/2}, \quad \quad t,u \in \R.
\end{equation}
Essentially, this identity guarantees that \eqref{E:expansion_Liouville} holds even at the discrete level.
In order to extend \eqref{E:expansion_unsigned_critical} to general intensities, one has to find a way of summing the Laguerre polynomials in such a way that the resulting sum is positive and behaves asymptotically like the exponential in \eqref{E:Hermite_exp}. This algebraic property is the content of the next lemma (see \eqref{E:expansion_unsigned_discrete} for the discrete version of \eqref{E:expansion_unsigned_general}). This identity does not seem to have been noticed before. Recall the definition \eqref{E:BesselI} of the modified Bessel function of the first kind $I_{\theta-1}$. Recall also its asymptotic behaviour (see \cite[9.7.1]{special}):
\begin{equation}
\label{E:BesselI_asymp}
I_{\theta-1}(u) \sim \frac{1}{\sqrt{2\pi u}} e^u, \quad \quad \text{as} \quad u \to \infty.
\end{equation}

\begin{lemma}\label{L:Laguerre_Bessel}
For all $t, u, \gamma >0$,
\begin{equation}
\label{E:Laguerre_Bessel}
\sum_{n \geq 0} \left( \frac{\gamma^2 t}{2} \right)^n \frac{1}{n! \Gamma(\theta+n)} \mathbf{L}_n^{(\theta-1)} \left( \frac{u^2}{2t} \right) = e^{-\gamma^2 t/2} \left( \frac{\gamma^2 ut}{4} \right)^{\frac{1-\theta}{2}} I_{\theta-1} \left( \gamma u \right).
\end{equation}
\end{lemma}

\begin{proof}
This is a direct computation using the explicit expressions \eqref{E:Laguerre} and \eqref{E:BesselI} of the Laguerre polynomials and of $I_{\theta-1}$.
\end{proof}

Specifying Lemma \ref{L:Laguerre_Bessel} to $u = \sqrt{2 \times 2\pi \ell_x}$ and $t = 2\pi G_N(x,x)$, we obtain
\begin{equation}
\label{E:expansion_unsigned_discrete}
\Big( \frac{\gamma^2}{2} 2\pi \ell_x \Big)^{\frac{1-\theta}{2}} e^{-\frac{\gamma^2}{2} 2\pi G_N(x,x)} I_{\theta-1} \left( \gamma \sqrt{2 \times 2 \pi \ell_x} \right) = \sum_{n \geq 0} \frac{\gamma^{2n}}{2^n} \frac{1}{n! \Gamma(n + \theta)} (2\pi)^n :\ell_x^n: \quad \quad \text{a.s.}
\end{equation}
This is the discrete version of Theorem \ref{T:expansion_unsigned} that we are now ready to prove.

\begin{proof}[Proof of Theorem \ref{T:expansion_unsigned}]
Let $f: D \to \R$ be a bounded test function.
We start by showing that the left hand side of \eqref{E:expansion_function} converges in $L^2$. For $n_0 \geq 1$, we write
\[
S_{n_0} := \sum_{n=0}^{n_0} \frac{\gamma^{2n}}{2^n} \frac{\Gamma(\theta)}{n! \Gamma(n + \theta)} \int (2\pi)^n :L_x^n: f(x) \d x.
\]
Let $n_0 < n_1$ be two integers.
By the equation (4.6) in \cite{LeJan2011Loops} (beware that our normalisation of $:L_x^n:$ is $n!$ times Le Jan's normalisation), for all $n, m \geq 1$ and $z, w \in D$, we have
\begin{equation}
\label{E:covariance_Wick}
\Expect{ :L_z^n: : L_w^m: } = \indic{n=m} \frac{\Gamma(\theta+n)n!}{\Gamma(\theta)} G_D(z,w)^{2n}.
\end{equation}
Hence,
\begin{align*}
& \Expect{ \left( S_{n_1} - S_{n_0} \right)^2 } = \frac{1}{\Gamma(\theta)} \sum_{n = n_0}^{n_1} \frac{\gamma^{4n}}{2^{2n}} \frac{1}{n! \Gamma(n+\theta)} \int_{D \times D} (2\pi)^{2n} G_D(z,w)^{2n} ~\d z~\d w.
\end{align*}
We can moreover bound for all $n \geq 0$,
\[
\int_{D \times D} (2\pi)^{2n} G_D(z,w)^{2n} ~\d z~\d w
\leq C \int_0^1 r (\log r)^{2n} \d r = C 2^{-2n} n (2n-1)!
\]
which implies that
\begin{align*}
& \Expect{ \left( S_{n_1} - S_{n_0} \right)^2 }
\leq C \sum_{n = n_0}^{n_1} \frac{\gamma^{4n}}{2^{4n}} \frac{(2n)!}{n! \Gamma(n+\theta)}
\leq C \sum_{n \geq n_0} \sqrt{n} \frac{n!}{\Gamma(n+\theta)} \frac{\gamma^{4n}}{2^{2n}}.
\end{align*}
Since $\gamma < \sqrt{2}$, the above sum is finite and vanishes as $n_0 \to \infty$. This shows that $(S_{n_0}, n_0 \geq 1)$ is a Cauchy sequence in $L^2$ and in particular converges to a limiting random variable $S_\infty$. The main task actually consists in showing that $S_\infty$ agrees almost surely with $\frac12 \scalar{\Mc_\gamma, f}$.

We explain this identification now.
We want to take the limit of the relation \eqref{E:expansion_unsigned_discrete}. We start by considering the limit of the right hand side of \eqref{E:expansion_unsigned_discrete}. An analogous computation as the one we have just done in the continuum shows that
\[
\lim_{n_0 \to \infty} \limsup_{N \to \infty} \E \Big[ \Big( \sum_{n \geq n_0} \frac{\gamma^{2n}}{2^n} \frac{1}{n! \Gamma(n + \theta)} \frac{1}{N^2} \sum_{z \in D_N} f(z) (2\pi)^n :\ell_z^n: \Big)^2 \Big] = 0.
\]
This follows from the exact same computation that we just performed since \eqref{E:covariance_Wick} also holds in the discrete with the continuum Green function $G_D$ replaced by its discrete version $G_{D_N}$. These estimates about the tails of the sums in the discrete and in the continuum, in combination with Theorem \ref{T:Wick_convergence}, show that
\[
\sum_{n \geq 0} \frac{\gamma^{2n}}{2^n} \frac{1}{n! \Gamma(n + \theta)} \frac{1}{N^2} \sum_{z \in D_N} f(z) (2\pi)^n :\ell_z^n:
\xrightarrow[N \to \infty]{\mathrm{(d)}} S_\infty.
\]
Moreover, this convergence is joint with $\Lc_{D_N}^\theta \to \Lc_D^\theta$.

We now move to the convergence of the left hand side of \eqref{E:expansion_unsigned_discrete}.
Recall that $a = \gamma^2/2$.
By \cite[Theorem 1.12]{ABJL21},
\begin{equation}
\label{E:uniform_measure_discrete}
\frac{(\log N)^{1-\theta}}{N^{2-a}} \sum_{z \in D_N} \CR(z,D)^{-a} f(z) \indic{ 2 \pi \ell_z \geq a (\log N)^2 }
\end{equation}
converges in distribution as $N \to \infty$ to $c_* \scalar{ \Mc_\gamma, f }$ where $c_* = \frac{(2\sqrt{2})^a e^{a \gamma_{\text{EM}}}}{2 a^{1-\theta} \Gamma(\theta)}$ and $\gamma_{\text{EM}}$ is the Euler--Mascheroni constant. Moreover, the convergence is joint with $\Lc_{D_N}^\theta \to \Lc_D^\theta$.
We now argue that
\begin{align}
\label{E:compare0}
& \frac{1}{N^2} \sum_{z \in D_N} f(z) \left( \frac{\gamma^2}{2} 2\pi \ell_z \right)^{\frac{1-\theta}{2}} e^{-\frac{\gamma^2}{2} 2\pi G_N(z,z)} I_{\theta-1} \left( \gamma \sqrt{2 \times 2 \pi \ell_z} \right) \\
& - \frac{1}{c_* \Gamma(\theta)} \frac{(\log N)^{1-\theta}}{N^{2-a}} \sum_{z \in D_N} \CR(z,D)^{-a} f(z) \indic{ 2 \pi \ell_z \geq a (\log N)^2 }
\nonumber
\end{align}
goes to zero in $L^2$ as $N \to \infty$ (when $\gamma \in [\sqrt{2}, 2)$ this convergence should hold in $L^1$, but we do not need to consider this case here).
This type of arguments is now fairly routine in the multiplicative chaos literature. Indeed, it amounts to comparing two different types of approximations of the measure. One approximation is based on the uniform measure on the set of thick points. The other one basically consists in taking the exponential of the square root of the local time. See for instance \cite[Theorems 1.1 and 1.2]{jegoBMC} where a similar result is shown in the context of the local times of small circles of a given planar Brownian motion. Let us emphasise that the situation we are dealing with in this article is simpler than the one in \cite{jegoBMC} since we are only interested in the $L^2$-regime.

We are therefore not going to prove that \eqref{E:compare0} goes to zero in $L^2$. Instead, we are going to give some heuristics.
Using the asymptotic behaviour \eqref{E:BesselI_asymp} of $I_{\theta-1}$, we see that the first term in \eqref{E:compare0} is very close to
\begin{align*}
& \frac{1}{\sqrt{2\pi \gamma} 2^{1/4}} \frac{1}{N^2} \sum_{z \in D_N} f(z) (2\pi \ell_z)^{-1/4} \left( \frac{\gamma^2}{2} 2\pi \ell_z \right)^{\frac{1-\theta}{2}} e^{-\frac{\gamma^2}{2} 2\pi G_N(z,z)} e^{ \gamma \sqrt{2 \times 2 \pi \ell_z} }.
\end{align*}
By a Cameron-Martin type result, we also see that the shift $e^{ \gamma \sqrt{2 \times 2 \pi \ell_z} }$ makes $2\pi \ell_z / (\log N)^2$ very much concentrated around the value $a$. Hence, the power of $2\pi \ell_z$ in front of the exponential can be replaced by the same power of $a (\log N)^2$ and the first term of \eqref{E:compare0} is very close to
\[
\frac{1}{2\sqrt{\pi}} a^{1/2-\theta} \frac{(\log N)^{1/2-\theta}}{N^2} \sum_{z \in D_N} f(z) e^{-\frac{\gamma^2}{2} 2\pi G_N(z,z)} e^{ \gamma \sqrt{2 \times 2 \pi \ell_z} }.
\]
We are now back to the more familiar situation of comparing the uniform measure on the set of thick points and the exponential of the square root of the local time.

Overall, this shows that
\[
\frac{1}{N^2} \sum_{z \in D_N} f(z) \left( \frac{\gamma^2}{2} 2\pi \ell_z \right)^{\frac{1-\theta}{2}} e^{-\frac{\gamma^2}{2} 2\pi G_N(z,z)} I_{\theta-1} \left( \gamma \sqrt{2 \times 2 \pi \ell_z} \right)
\xrightarrow[N \to \infty]{\mathrm{(d)}} \frac{1}{\Gamma(\theta)} \scalar{\Mc_\gamma, f}
\]
jointly with the convergence $\Lc_{D_N}^\theta \to \Lc_D^\theta$.
Wrapping things up, and recalling that $2S_\infty$ denotes the $L^2$ limit of the left hand side of \eqref{E:expansion_function}, we have shown that
\[
\left( \scalar{ \Mc_\gamma, f }, \Lc_D^\theta \right)
\overset{\mathrm{(d)}}{=} \left( 2 S_\infty, \Lc_D^\theta \right).
\]
Because $\scalar{ \Mc_\gamma, f }$ and $2S_\infty$ are measurable w.r.t. $\Lc_D^\theta$, they must agree almost surely. This concludes the proof.
\end{proof}

\subsection{One-dimensional Brownian loop soup and Wick powers of Bessel processes}\label{S:1D}

In this section, we study the analogue of the field $h_\theta$ in the one-dimensional case.
This is partly motivated by the recent connection between the 1D and 2D loop soup discovered in \cite{JLQ23a}.
As we will recall in Proposition \ref{P:1Disomorphism} below, the local time of the 1D loop soup is distributed as a (reflected) squared Bessel process of dimension $d=2\theta$.
When $\theta=1/2$, obtaining a Brownian motion (1D analogue of the GFF) from the loop soup then boils down to taking the square root of a (reflected) 1D squared Bessel process and then flipping each excursion independently of each other with equal probability.
What is the analogue of this procedure when $\theta \neq 1/2$? 
We will see that, surprisingly, taking the square root is not the right thing to do (although this might be very natural since the square root of a squared Bessel process is just a Bessel process) and that one has to take the power $1-\theta = 1-d/2$ instead. This power is exactly the one appearing in the Radon--Nikodym derivative of the Bessel process and its dual Bessel process of dimension $d^* = 4 -d = 2\theta^*$ where $\theta^* = 2-\theta$; see \eqref{E:Bessel_duality}. The purpose of this section is also to hint at the duality between the intensities $\theta$ and $2-\theta$ that might still hold in some sense in dimension two. See Section \ref{S:expansion_signed} for more details concerning the 2D case.

\subsubsection{Loop soup on $\R^+$}

Let us start by defining the loop soup on the half line $\R^+$. We follow \cite{Lupu18}.
For $x,y \in \R^+$, let $p(t,x,y)$ and $p_{\R^+}(t,x,y)$ be the heat kernels on $\R$ and $\R^+$:
\[
p(t,x,y) = \frac{1}{\sqrt{2\pi t}} e^{-(y-x)^2/(2t)}, \quad p_{\R^+}(t,x,y) = p(t,x,y) - p(t,x,-y)
\]
and let $\P_{\R^+}^{t,x,y}$ be the law of a Brownian bridge from $x$ to $y$ of duration $t$ in the half line $\R^+$. 
The Green function and the loop measure are respectively given by
\[
G_{\R^+}(x,y) = \int_0^\infty p_{\R^+}(t,x,y) \d t = 2 x \wedge y
\quad \text{and} \quad
\loopmeasure_{\R^+} = \int_0^\infty \frac{p_{\R^+}(t,x,x)}{t} \P_{\R^+}^{t,x,x} \d t.
\]
For $\theta >0$, the loop soup $\Lc_{\R^+}^\theta$ in $\R^+$ with parameter $\theta$ is distributed according to a Poisson Point Process with intensity $\theta \loopmeasure_{\R^+}$. Compared to the 2D case, the local time of $\Lc_{\R^+}^\theta$ is well defined pointwise and finite and we will denote it
\[
L_x(\Lc_{\R^+}^\theta) = \sum_{\wp \in \Lc_{\R^+}^\theta} L_x(\wp), \quad \quad x \in \R^+.
\]

\begin{proposition}[\cite{Lupu18}, Proposition 4.6]\label{P:1Disomorphism}
Let $\theta>0$.
$(L_x(\Lc_{\R^+}^\theta))_{x \geq 0}$ has the law of a squared Bessel process of dimension $d=2\theta$, reflected at 0 when $\theta <1$.
\end{proposition}

Let $\Ps^d$ be the law under which $(R_x,x \geq 0)$ is a squared Bessel process of dimension $d=2\theta$, reflected at 0 when $\theta <1$.
The clusters of $\Lc_{\R^+}^\theta$ correspond to the excursions away from 0 of $R$. Therefore Proposition \ref{P:1Disomorphism} implies that when $\theta <1$ there are infinitely many clusters and when $\theta \geq 1$ there is only one cluster. Notice that the critical point is different from the two-dimensional case. However, we show in \cite[Theorem 1.9]{JLQ23a} that there is a phase transition at $\theta = 1$ in the percolative behaviour of large loops on the half-infinite cylinder $(0,\infty) \times \mathbb{S}^1$ (by conformal invariance, this translates to a phase transition in the percolative behaviour of large loops targeting a given point of some bounded 2D domain). See \cite[Section 1.3]{JLQ23a} for more details.

\subsubsection{Characterisation of Laguerre polynomials in terms of Bessel processes}

As already seen, the generalised Laguerre polynomials $\mathbf{L}_n^{(\theta-1)}$ have been used by Le Jan to renormalise the powers of the local time of the loop soup in dimension two. In this section, we show that these polynomials are also special in dimension one in the sense that they are the only polynomials leading to martingales. See Lemma \ref{L:1D_martingale} for a precise statement. 

For any $\theta>0$, the Wick powers of the local time $L_x(\Lc_{\R^+}^\theta)$ are defined by
\begin{equation}
\label{E:1D_power}
:L_x(\Lc_{\R^+}^\theta)^n:
\quad = \quad (2x)^n \mathbf{L}_n^{(\theta-1)} \Big( \frac{L_x(\Lc_{\R^+}^\theta)}{2x} \Big), \quad \quad x \geq 0, n \geq 1.
\end{equation}
In the special case $\theta = 1/2$ this corresponds to the even Wick powers of 1D Brownian motion.

\begin{lemma}\label{L:1D_martingale}
Let $\theta >0$ and $n \geq 1$. The process $ (:L_x(\Lc_{\R^+}^\theta)^n:)_{x \geq 0}$ is a martingale.
Moreover, $\mathbf{L}_n^{(\theta-1)}$ is the only monic polynomial such that \eqref{E:1D_power} is a martingale.
\end{lemma}

Let us stress that it is a well known fact that the Wick powers of 1D Brownian motion lead to martingales. Martingales associated to diffusions are very well studied. However, our goal here is to highlight the strong link between Laguerre polynomials and Bessel processes.

\begin{proof}
Because in this lemma the space $\R^+$ is thought of as time, we will change the notations $x$ and $y$ for the space variables by $s$ and $t$. We will denote by $(\Fc_t)_{t \geq 0}$ the natural filtration of the Bessel process.
Let $X_t = \sqrt{R_t}$ be a Bessel process of dimension $d=2\theta$. Recall that $X_t$ satisfies the stochastic differential equation
$
\d X_t = \frac{d-1}{2 X_t} \d t + \d B_t
$
where $(B_t)_{t \geq 0}$ is a 1D standard Brownian motion. A computation using It\^{o} formula leads to the following SDE for $: L_t(\Lc_{\R^+}^\theta)^n :$ (we write $\mathbf{L}_n$ instead of $\mathbf{L}_n^{(\theta-1)}$)
\begin{align*}
\d \left\{ (2t)^n \mathbf{L}_n^{(\theta-1)} \Big( \frac{X_t^2}{2t} \Big) \right\}
& = 2^nt^{n-1} X_t \mathbf{L}_n' \Big( \frac{X_t^2}{2t} \Big) \d B_t \\
& + 2^nt^{n-1} \left\{ n \mathbf{L}_n \Big( \frac{X_t^2}{2t} \Big) + \Big( \theta - \frac{X_t^2}{2t} \Big) \mathbf{L}_n' \Big( \frac{X_t^2}{2t} \Big) + \frac{X_t^2}{2t} \mathbf{L}_n'' \Big( \frac{X_t^2}{2t} \Big) \right\} \d t.
\end{align*}
Because $\mathbf{L}_n = \mathbf{L}_n^{(\theta-1)}$ satisfies $n \mathbf{L}_n(u) + (\theta - u) \mathbf{L}_n'(u) + u \mathbf{L}_n''(u) = 0$ (see e.g. (22.1.3) and (22.8.6) in \cite{special}), we see that the finite variation part vanishes. In other words, $(: L_t(\Lc_{\R^+}^\theta)^n :)_{t \geq 0}$ is a local martingale. A routine argument concludes that it is a martingale as desired.

The fact that $\mathbf{L}_n = \mathbf{L}_n^{(\theta-1)}$ is the only monic polynomial such that \eqref{E:1D_power} is a martingale follows from the fact that it is the only monic polynomial satisfying $n \mathbf{L}_n(u) + (\theta - u) \mathbf{L}_n'(u) + u \mathbf{L}_n''(u) = 0$.
\end{proof}

\subsubsection{Analogue of $h_\theta$ on $\R^+$ and renormalisation of its powers}

Assume that $\theta < 1$ (subcritical regime for 1D loop soups).
In this section we define the natural analogue $h_{\theta, \R^+}$ of $h_\theta$ on $\R^+$ and give a procedure to renormalise its ``odd'' powers. Let $R$ be a $2\theta$-dimensional squared Bessel process. Let $\Ec$ be the set of excursions of $R$ and, conditionally on $\Ec$, let $\sigma_e, e \in \Ec$ be i.i.d. spins taking values in $\{ \pm 1\}$ with equal probability. Define
\[
h_{\theta, \R^+}(x) := \sigma_{e_x} R_x^{1-\theta}, \quad \quad x \geq 0,
\]
where $e_x$ is the excursion containing $x$.

The following result computes the two-point function of the field $h_{\theta,\R^+}$ and shows that tilting the probability measure by $h_{\theta,\R^+}(x) h_{\theta,\R^+}(y)$ amounts to changing the dimension of the Bessel process to its dual $d^* = 4-d$ on the interval $[x,y]$. Equivalently, this changes the intensity of the loop soup to $\theta^* = 2-\theta$ on the interval $[x,y]$.

\begin{lemma}\label{L:2point1D}
Let $\theta \in (0,1)$.
For all $0 \leq x \leq y$, and all bounded measurable function $F$,
\begin{equation}
\label{E:two-point1D}
\Es^d \left[ h_{\theta, \R^+}(x) h_{\theta, \R^+}(y) F( (R_z)_{z \in [x,y]}) \right] = \frac{\Gamma(\theta^*)}{\Gamma(\theta)} G_{\R^+}(x,y)^{2(1-\theta)} \Es^{d^*} \left[ F( (R_z)_{z \in [x,y]}) \right]
\end{equation}
where $d^* = 4-d = 2 \theta^*$ and $\theta^* = 2-\theta$.
\end{lemma}

\begin{proof}
By independence of the signs on different clusters and then by Markov property, we have
\begin{align*}
\Es^d \left[ h_{\theta, \R^+}(x) h_{\theta, \R^+}(y) F( (R_z)_{z \in [x,y]}) \right]
& = \Es^d \left[ R_x^{1-\theta} R_y^{1-\theta} \indic{\forall z \in [x,y], R_z>0} F( (R_z)_{z \in [x,y]}) \right] \\
& = \Es^d \left[ R_x^{1-\theta} \Es^d_{R_x} \left[ R_{y-x}^{1-\theta} \indic{ \forall z \in [0,y-x], R_z >0 } F( (R_z)_{z \in [0,y-x]}) \right] \right].
\end{align*}
As already mentioned, the key observation is that the power $1-\theta$ is exactly the power one needs to take in order to obtain the dual Bessel process of dimension $d^* = 4-d$. More precisely, for any $t>0$ and starting point $r>0$, restricted to the event that $\{ \forall s \in [0,t], R_s >0 \}$, we have
\begin{equation}
\label{E:Bessel_duality}
\frac{\d \P_r^{d^*}}{\d \P_r^{d}} {}_{\big\vert_{\Fc_t}} = \left( \frac{R_t}{r} \right)^{1-d/2}
\end{equation}
where $(\Fc_s)_{s \geq 0}$ is the natural filtration associated to $(R_s)_{s \geq 0}$.
See \cite[Proposition 2.2]{LawlerBessel}.
This implies that the left hand side of \eqref{E:two-point1D} is equal to
\begin{align*}
\Es^d \left[ R_x^{2(1-\theta)} \Es^{d^*}_{R_x} \left[ \indic{ \forall z \in [0,y-x], R_z >0 } F( (R_z)_{z \in [0,y-x]}) \right] \right] = \Es^d \left[ R_x^{2(1-\theta)} \Es_{R_x}^{d^*} \left[ F( (R_z)_{z \in [0,y-x]}) \right] \right]
\end{align*}
since the dual Bessel process does not reach zero a.s. ($d^* \geq 2$).
Under $\Ps^d$, $R_x$ is Gamma$(\theta,2x)$-distributed ($\theta$ and $2x$ are respectively the shape and scale parameters). A computation with Gamma distributions shows that for any measurable function $f: [0,\infty) \to \R$,
\[
\Es^d \left[ R_x^{2(1-\theta)} f(R_x) \right]
= \Es^d \left[ R_x^{2(1-\theta)} \right] \Es^{d^*} \left[ f(R_x) \right]
= (2x)^{2(1-\theta)} \frac{\Gamma(2-\theta)}{\Gamma(\theta)} \Es^{d^*} \left[ f(R_x) \right].
\]
Wrapping things up and by Markov property, we have obtained that the left hand side of \eqref{E:two-point1D} is equal to
\begin{align*}
(2x)^{2(1-\theta)} \frac{\Gamma(2-\theta)}{\Gamma(\theta)} \Es^{d^*} \left[ \Es_{R_x}^{d^*} \left[ F( (R_z)_{z \in [0,y-x]}) \right] \right] 
= (2x)^{2(1-\theta)} \frac{\Gamma(2-\theta)}{\Gamma(\theta)} \Es^{d^*} \left[ F( (R_z)_{z \in [x,y]}) \right].
\end{align*}
Lemma \ref{L:2point1D} then follows from the fact that $G_{\R^+}(x,y) = 2x$.
\end{proof}

When $\theta \in (0,1)$, Lemma \ref{L:2point1D} suggests to define the ``odd'' powers (i.e. powers that are not integer powers of the local time) of $h_{\theta,\R^+}$ by
\begin{equation}
:h_{\theta,\R^+}(x) L_x(\Lc_{\R^+}^\theta)^n:
\quad = \quad (2x)^n h_{\theta,\R^+}(x) \mathbf{L}_n^{(\theta^*-1)} \Big( \frac{L_x(\Lc_{\R^+}^\theta)}{2x} \Big), \quad \quad x \geq 0, n \geq 1.
\end{equation}
When $\theta = 1/2$, this definition agrees with the definition \eqref{E:def_Wick_GFF} of odd Wick powers of 1D Brownian motion. This follows from the identity \eqref{E:Laguerre_Hermite} between odd Hermite polynomials and the Laguerre polynomials $\mathbf{L}_n^{(\theta^*-1)}$ where $\theta^* = 3/2$.
In the following lemma, we show that these renormalised powers define martingales. In view of Lemma \ref{L:1D_martingale}, we believe that this normalisation has the potential to work in 2D as well. See Section \ref{S:expansion_signed} for much more about the two-dimensional case.

\begin{lemma}\label{L:1D_martingale2}
Let $\theta \in (0,1)$. The processes
\[
(h_{\theta,\R^+}(x))_{x \geq 0}, \quad \quad  (:h_{\theta,\R^+}(x) L_x(\Lc_{\R^+}^\theta)^n:)_{x \geq 0}, \quad \quad, n \geq 1,
\]
are martingales.
\end{lemma}

\begin{proof}
Let $n \geq 0$ (the case $n=0$ corresponding to $h_{\theta,\R^+}$). As before, denote by $(\Fc_t)_{t \geq 0}$ the natural filtration associated to $R$. 
By independence of the spins on different excursions and then by the relation \eqref{E:Bessel_duality} between Bessel processes with dimensions $d=2\theta$ and $d^* = 4-d=2\theta^*$, we obtain that for any $0 \leq s < t$,
\begin{align*}
\Es^d \left[ :h_{\theta,\R^+}(t) L_t(\Lc_{\R^+}^\theta)^n: \Big\vert \Fc_s \right]
& = \sigma_{e_s} \Es^d \left[ R_t^{1-\theta} \indic{\forall u \in [s,t], R_u >0} (2t)^n \mathbf{L}_n^{(\theta^*-1)} \Big( \frac{L_t(\Lc_{\R^+}^\theta)}{2t} \Big) \Big\vert \Fc_s \right] \\
& = \sigma_{e_s} R_s^{1-\theta} \Es^{d^*} \left[ (2t)^n \mathbf{L}_n^{(\theta^*-1)} \Big( \frac{L_t(\Lc_{\R^+}^\theta)}{2t} \Big) \Big\vert \Fc_s \right].
\end{align*}
If $n=0$, this immediately gives the desired martingale property. If $n\geq 1$, we use Lemma \ref{L:1D_martingale} to conclude that
\[
\Expect{ :h_{\theta,\R^+}(t) L_t(\Lc_{\R^+}^\theta)^n: \Big\vert \Fc_s }
= \sigma_{e_s} R_s^{1-\theta} (2s)^n \mathbf{L}_n^{(\theta^*-1)} \Big( \frac{L_s(\Lc_{\R^+}^\theta)}{2s} \Big)
= ~:\!h_{\theta,\R^+}(s) L_s(\Lc_{\R^+}^\theta)^n\!:.
\]
The last equality follows by definition. This concludes the proof.
\end{proof}

\subsection{Wick powers involving \texorpdfstring{$h_\theta$}{h theta} and expansion of \texorpdfstring{$\Mc_\gamma^+$}{Mgamma+}}
\label{S:expansion_signed}

In this section we elaborate on Conjecture \ref{Conj:intro} by giving a much more precise conjecture (Conjecture \ref{Conjecture1} below) that would allow us to define \emph{all} the Wick powers of $h_\theta$. We will then show that this conjecture would have very strong implications concerning the asymptotic behaviour of the crossing probabilities studied in \cite{JLQ23a} and would also provide the expansion of the signed measure $\Mc_\gamma^+$; see \eqref{E:conj_crossing_probability} and \eqref{E:conj_expansion}.

Let us go back to the discrete setting first. Recall the definition \eqref{E:def_htheta_discrete} of $h_{\theta,N}$, the (conjecturally) analogue of $h_\theta$ in the discrete. We work on the cable graph so that $h_{\theta,N}$ is simply a discrete GFF.
The following lemma shows that the duality between $\theta$ and $\theta^* = 2-\theta$ that holds in the one-dimensional setting (Lemma \ref{L:2point1D}) also holds \emph{locally} in the two-dimensional setting when $\theta=1/2$; see also Remark \ref{rmk:isomorphism1/2}.

\begin{lemma}\label{L:duality_criticality}
Assume that $\theta=1/2$ and let $F : [0,\infty)^2 \to \R$ be a bounded measurable function. For all $x, y \in D_N$,
\begin{equation}
\label{E:duality_critical}
\Expect{ h_{\theta,N}(x) h_{\theta,N}(y) F(\ell_x(\Lc_{D_N}^\theta), \ell_y(\Lc_{D_N}^\theta)) }
= G_{D_N}(x,y) \Expect{ F(\ell_x(\Lc_{D_N}^{\theta^*}), \ell_y(\Lc_{D_N}^{\theta^*})) }
\end{equation}
where $\theta^* = 2-\theta = 3/2$.
\end{lemma}

The key ingredient in the proof of Lemma \ref{L:duality_criticality} is that the odd Hermite polynomials are related to the Laguerre polynomials $\mathbf{L}_n^{(\theta^*-1)}$ with $\theta^* = 3/2$. See \eqref{E:Laguerre_Hermite} below.

\begin{proof}
One could use the BFS-Dynkin isomorphism \eqref{E:BFS_Dynkin} to show this lemma. We will use another approach here that allows us to highlight some extra features of the Wick powers.
By density-type arguments, it is enough to prove \eqref{E:duality_critical} when $F(s,t) = P(s/G_{D_N}(x,x)) \times Q(t/G_{D_N}(y,y))$ is a product of two polynomials $P(\cdot/G_{D_N}(x,x))$ and $Q(\cdot/G_{D_N}(y,y))$. We can moreover assume that these polynomials belong to the basis $\{ \mathbf{L}_n^{(\theta^*-1)}, n \geq 0 \}$, i.e. that $P = \mathbf{L}_n^{(\theta^*-1)}$ and $Q = \mathbf{L}_m^{(\theta^*-1)}$ for some $n, m \geq 0$. The right hand side of \eqref{E:duality_critical} can then be computed directly thanks to the study of the renormalised local time of the loop soup with intensity $\theta^*$ (see \cite{LeJan2011Loops})
\[
\E \Big[ \mathbf{L}_n^{(\theta^*-1)} \Big( \frac{\ell_x(\Lc_{D_N}^{\theta^*})}{G_{D_N}(x,x)} \Big) \mathbf{L}_m^{(\theta^*-1)} \Big( \frac{\ell_y(\Lc_{D_N}^{\theta^*})}{G_{D_N}(y,y)} \Big) \Big]
= \indic{n=m} \frac{\Gamma(\theta^*+n)n!}{\Gamma(\theta^*)} \frac{G_{D_N}(x,y)^{2n}}{G_{D_N}(x,x)G_{D_N}(y,y)}.
\]
Concerning the left hand side term in \eqref{E:duality_critical}, we use the fact that $\ell_x(\Lc_{D_N}^{\theta}) = h_{\theta,N}(x)^2/2$ and the following identity between odd Hermite polynomials and Laguerre polynomials \cite[(22.5.41)]{special}
\begin{equation}
\label{E:Laguerre_Hermite}
\frac{1}{2^n} G^{\frac{2n+1}{2}} H_{2n+1}\left( \frac{h}{\sqrt{G}} \right)
= h G^n \mathbf{L}_n^{(\theta^*-1)} \left( \frac{h^2}{2G} \right), \quad \quad n\geq 0, h \in \R, G>0.
\end{equation}
Combining these two ingredients gives
\begin{align*}
& \E \Big[ h_{\theta,N}(x) h_{\theta,N}(y) \mathbf{L}_n^{(\theta^*-1)} \Big( \frac{\ell_x(\Lc_{D_N}^{\theta})}{G_{D_N}(x,x)} \Big) \mathbf{L}_m^{(\theta^*-1)} \Big( \frac{\ell_y(\Lc_{D_N}^{\theta})}{G_{D_N}(y,y)} \Big) \Big] \\
& = \frac{1}{2^{n+m}} G_{D_N}(x,x)^{1/2} G_{D_N}(y,y)^{1/2} \E \Big[ H_{2n+1} \Big( \frac{h_{\theta,N}(x)}{G_{D_N}(x,x)^{1/2}} \Big) H_{2m+1} \Big( \frac{h_{\theta,N}(y)}{G_{D_N}(y,y)^{1/2}} \Big) \Big] \\
& = \indic{n =m} \frac{(2n+1)!}{2^{2n}} \frac{G_{D_N}(x,y)^{2n+1}}{G_{D_N}(x,x)^n G_{D_N}(y,y)^n}
\end{align*}
where the last equality comes from the study of the Wick powers of the GFF (see \cite{LeJan2011Loops}).
Recalling that $\theta^* = 3/2$, one can check that $\frac{\Gamma(\theta^*+n)n!}{\Gamma(\theta^*)} = \frac{(2n+1)!}{2^{2n}}$ and we obtain that both sides of \eqref{E:duality_critical} are equal when $F(s,t) = \mathbf{L}_n^{(\theta^*-1)}(s / G_{D_N}(x,x)) \mathbf{L}_m^{(\theta^*-1)}(t / G_{D_N}(y,y))$. This concludes the proof.
\end{proof}

Let $\theta \in (0,1/2]$.
Lemmas \ref{L:2point1D}, \ref{L:1D_martingale2} and \ref{L:duality_criticality} suggest to define the Wick renormalisation of $h_{\theta,N}(x) \ell_x^n$ by
\begin{equation}
\label{E:Wick_def_conjecture}
:h_{\theta,N}(x) \ell_x^n: \quad = \quad h_{\theta,N}(x) G(x,x)^n \mathbf{L}_n^{(\theta^*-1)}(\ell_x / G(x,x))
\end{equation}
where $\theta^* = 2-\theta$. The relation \eqref{E:Laguerre_Hermite} between odd Hermite polynomials and Laguerre polynomials $\mathbf{L}_n^{(1/2)}$ guarantees that this renormalisation of $h_{\theta,N}(x) \ell_x^n$ is consistent with the usual renormalisation when $\theta=1/2$:
\[
:h_{\theta=1/2,N}(x) \ell_x^n:
\quad = \quad \frac{1}{ 2^n} G(x,x)^{\frac{2n+1}{2}} H_{2n+1} \left( \frac{h_{\theta=1/2,N}(x)}{\sqrt{G(x,x)}} \right).
\]
We make the following conjecture:

\begin{conjecture}\label{Conjecture1}
Let $\theta \in (0,1/2)$.
The Wick powers $:h_{\theta,N}(x) \ell_x^n:$ defined in \eqref{E:Wick_def_conjecture} have nondegenerate scaling limits. More precisely, for any $n_0 \geq 1$,
\[
\left( h_{\theta,N}, (:h_{\theta,N} \ell^n:)_{n=1 \dots n_0}, \Lc_{D_N}^\theta \right)
\xrightarrow[N \to \infty]{\mathrm{(d)}} \left(h_\theta^{\mathrm{conj}}, (:h_\theta^{\mathrm{conj}} L^n:)_{n =1 \dots n_0}, \Lc_D^\theta \right)
\]
where the right hand side term is defined by this limit. The convergence holds in the space $(H^{-\eps})^{n_0} \times \mathfrak{L}$ for any $\eps >0$. Moreover, $h_\theta^{\mathrm{conj}}$ and  $(:h_\theta^{\mathrm{conj}} L^n:)_{n =1 \dots n_0}$ are measurable w.r.t. $\Lc_D^\theta$.
\end{conjecture}

Let us stress that this conjecture is known when $\theta = 1/2$; see Theorem \ref{T:Wick_convergence}.
This conjecture is particularly appealing to us since:

\begin{theorem}\label{T:expansion_signed_discrete}
Let $\theta \in (0,1/2]$ and $\gamma \in (0,\sqrt{2})$. Define
\begin{equation}
\label{E:def_expansion_discrete_signed}
m_{\gamma,N}(x) = 
\sum_{k \geq 0} \frac{\gamma^{2k}}{2^k} \frac{\Gamma(\theta)}{k! \Gamma(k+\theta)} (2\pi)^k:\ell_x^k:
+ \gamma^{2(1-\theta)} \sum_{k \geq 0} \frac{\gamma^{2k}}{2^k} \frac{\Gamma(\theta^*)}{k! \Gamma(k+\theta^*)} (2\pi)^k :h_{\theta,N}(x) \ell_x^k:
\end{equation}
where the Wick powers $:h_{\theta,N}(x) \ell_x^k:$ are defined in \eqref{E:Wick_def_conjecture}.
For any bounded smooth test function $f$,
\begin{equation}
\frac{1}{N^2} \sum_{x \in D_N} f(x) m_{\gamma,N}(x) \xrightarrow[N \to \infty]{\mathrm{(d)}} \scalar{\Mc_\gamma^+,f}
\end{equation}
jointly with $\Lc_{D_N}^\theta \to \Lc_D^\theta$.
\end{theorem}

Let us emphasise that Theorem \ref{T:expansion_signed_discrete} is not conditionally on Conjecture \ref{Conjecture1}. This result is another strong indication that Conjecture \ref{Conjecture1} should hold.

\begin{proof}[Proof of Theorem \ref{T:expansion_signed_discrete}]
Let $m_{\gamma,N}^\mathrm{even}(x)$ and $m_{\gamma,N}^\mathrm{odd}(x)$ be respectively the first and second term on the right hand side of \eqref{E:def_expansion_discrete_signed}.
Recall that, as a consequence of Lemma \ref{L:Laguerre_Bessel}, we obtained that (see \eqref{E:expansion_unsigned_discrete})
\begin{align*}
m_{\gamma,N}^\mathrm{even}(x)
= 
\Gamma(\theta) \left( \frac{\gamma^2}{2} 2\pi\ell_x \right)^{\frac{1}{2} - \frac{\theta}{2}} e^{-\frac{\gamma^2}{2} G(x,x)} I_{\theta -1} ( \gamma \sqrt{2\times 2\pi \ell_x} ).
\end{align*}
Similarly, Lemma \ref{L:Laguerre_Bessel} applied to $\theta^*$ instead of $\theta$ yields
\begin{align*}
m_{\gamma,N}^\mathrm{odd}(x)
& = \gamma^{2(1-\theta)} h_{\theta,N}(x) \Gamma(\theta^*) 
\left( \frac{\gamma^2}{2} 2\pi\ell_x \right)^{\frac{1}{2} - \frac{\theta^*}{2}} e^{-\frac{\gamma^2}{2} G(x,x)} I_{\theta^* -1} ( \gamma \sqrt{2\times 2\pi\ell_x} ) \\
& = 2^{1-\theta} \Gamma(\theta^*) \left( \frac{\gamma^2}{2} \right)^{\frac{1}{2} - \frac{\theta}{2}} h_{\theta,N}(x) (2\pi\ell_x)^{\frac{1}{2} - \frac{\theta^*}{2}} e^{-\frac{\gamma^2}{2} G(x,x)} I_{\theta^* -1} ( \gamma \sqrt{2\times 2\pi \ell_x} ).
\end{align*}
Recall the definition \eqref{E:def_htheta_discrete} of $h_{\theta,N}$: $h_{\theta,N}(x) = c_\theta \sigma_x (2\pi\ell_x)^{1-\theta}$ where $c_\theta = 2^{\theta-1} \Gamma(\theta) / \Gamma(\theta^*)$. We can rewrite
$
h_{\theta,N}(x) (2\pi\ell_x)^{\frac{1}{2} - \frac{\theta^*}{2}} = c_\theta \sigma_x (2\pi\ell_x)^{\frac{1}{2} - \frac{\theta}{2}}
$
and we overall find that $m_{\gamma,N}(x)$ is equal to
\begin{align*}
\left( \frac{\gamma^2}{2} 2\pi \ell_x \right)^{\frac{1}{2} - \frac{\theta}{2}} e^{-\frac{\gamma^2}{2} G(x,x)} \left\{ \Gamma(\theta) I_{\theta -1} ( \gamma \sqrt{2\times 2\pi \ell_x} ) + 2^{1-\theta} \Gamma(\theta^*) c_\theta \sigma_x I_{\theta^* -1} ( \gamma \sqrt{2\times 2\pi \ell_x} ) \right\}.
\end{align*}
The constant $c_\theta$ has been chosen so that the two constants in front of the modified Bessel functions match (which is particularly important since they have the same asymptotic behaviour at infinity)
and
\begin{align*}
m_{\gamma,N}(x) = \Gamma(\theta) \left( \frac{\gamma^2}{2} 2\pi \ell_x \right)^{\frac{1}{2} - \frac{\theta}{2}} e^{-\frac{\gamma^2}{2} G(x,x)} \left\{  I_{\theta -1} ( \gamma \sqrt{2\times 2\pi \ell_x} ) + \sigma_x I_{\theta^* -1} ( \gamma \sqrt{2\times 2\pi \ell_x} ) \right\}.
\end{align*}
In the special case $\theta = 1/2$, we have
\[
I_{-1/2}(z) = \sqrt{\frac{2}{\pi z}} \cosh(z)
\quad \text{and} \quad
I_{1/2}(z) = \sqrt{\frac{2}{\pi z}} \sinh(z)
\]
and we do recover that
\[
m_{\gamma,N}(x) = \Gamma(1/2) \frac{1}{\sqrt{\pi}} e^{-\frac{\gamma^2}{2} G(x,x)} e^{\sigma_x \gamma \sqrt{2\times 2\pi \ell_x} } = e^{-\frac{\gamma^2}{2} G(x,x)} e^{\sigma_x \gamma \sqrt{2\times 2\pi\ell_x} }
\]
which converges to $:e^{\gamma h}:$ as $N \to \infty$. In general ($\theta \in (0,1/2]$), if $\sigma_x= -1$, the leading order terms of the two Bessel functions will cancel each other. Since the normalisation factor has been tuned to normalise these leading order terms, the subleading order terms will vanish in the scaling limit and the measure will concentrate on points with a positive spin. One can therefore use a similar reasoning as in the proof of Theorem \ref{T:expansion_unsigned} and compare $m_{\gamma,N}(x)$ with $\Mc_{a,N}(\d x) \indic{\sigma_x = 1}$. In Corollary \ref{C:discrete_ma+} we showed that this latter measure converges to $\Mc_\gamma^+$, so this argument shows that $m_{\gamma,N}(x)$ also converges to $\Mc_\gamma^+$.
\end{proof}

\paragraph*{Consequences of Conjecture \ref{Conjecture1}}
Let us now assume that Conjecture \ref{Conjecture1} holds and discuss the implications it would have. In the course of proving Conjecture \ref{Conjecture1}, it is likely that one would obtain a control on the growth of the correlation of $:h_{\theta,N}(x) \ell_x^n:$ and $:h_{\theta,N}(y) \ell_y^m:$ of a similar type as \eqref{E:covariance_Wick}. This control would ensure that most of the mass of $m_{\gamma,N}$ is carried by, say, the first hundred terms in the two sums in \eqref{E:def_expansion_discrete_signed}. This would then allow us to take the scaling limit of these sums (exactly like in the proof of Theorem \ref{T:expansion_unsigned}) and we would obtain, in conjunction with Theorem \ref{T:expansion_signed_discrete}, that
\begin{equation}
\label{E:conj_expansion}
\Mc_\gamma^+ = \sum_{k \geq 0} \frac{\gamma^{2k}}{2^k} \frac{\Gamma(\theta)}{k! \Gamma(k+\theta)} (2\pi)^k :L_x^k:
+ \gamma^{2(1-\theta)} \frac{\gamma^{2k}}{2^k} \frac{\Gamma(\theta^*)}{k! \Gamma(k+\theta^*)} (2\pi)^k :h_\theta^\mathrm{conj}(x) L_x^k:.
\end{equation}
In particular, this shows that the second order term in the expansion of $\Mc_\gamma^+$ is of order $\gamma^{2(1-\theta)}$ which would imply that the limit
\begin{equation}
\label{E:conj_crossing_probability}
c_{\mathrm{conj}} := \lim_{\gamma \to 0} \frac{Z_\gamma}{\gamma^{2(1-\theta)}} \in (0,\infty)
\end{equation}
exists and is nondegenerate. This would also identify $h_\theta^\mathrm{conj}$ with $c_\mathrm{conj} h_\theta$. Finally, with the help of a Tauberian theorem, it may be the case that the above convergence \eqref{E:conj_crossing_probability} implies the convergence of the crossing probability
\begin{equation}
\lim_{r \to 0} | \log r |^{1-\theta} \P \Big( e^{-1} \partial \D \overset{\Lc_\D^\theta}{\longleftrightarrow} r \D \Big) \in (0,\infty).
\end{equation}
This would be a considerable strengthening of the main result of \cite{JLQ23a} (see Theorem \ref{T:large_crossing}).